\documentclass[11pt]{article}
\usepackage{enumerate}
\usepackage[OT1]{fontenc}
\usepackage[usenames]{color}
\usepackage{smile}
\usepackage[colorlinks,
linkcolor=red,
anchorcolor=blue,
citecolor=blue
]{hyperref}
\usepackage{mathrsfs}
\usepackage{fullpage}
\usepackage{hyperref}
\usepackage[protrusion=true, expansion=true]{microtype}
\usepackage{float}
\usepackage{subfigure}
\usepackage{amsfonts,amsmath,amssymb,amsthm,url,xspace}
\usepackage{tikz}
\usepackage{verbatim}
\usetikzlibrary{arrows,shapes}
\usepackage{mathtools}
\usepackage{authblk}
\usepackage[bottom]{footmisc}
\usepackage{enumitem}
\usepackage{mathtools}
\usepackage{lscape}
\usepackage{caption}

\usepackage{algorithm}
\usepackage{algorithmicx,algpseudocode}

\ifx\counterwithout\undefined\usepackage{chngcntr}\fi
\counterwithout{equation}{section}

\allowdisplaybreaks[1]

\newcommand{\myred}{\color{black}}

\newcommand{\mR}{\mathbb{R}}

\newcommand{\mE}{\mathbb{E}}

\newcommand{\mL}{\mathcal{L}}

\newcommand{\mF}{\mathcal{F}}

\def\I{\mathcal{I}}

\def\P{{\mathcal P}}

\def\O{{\mathcal O}}

\def\L{{\mathcal L}}

\def\0{{\boldsymbol 0}}

\def\F{{\mathcal F}}
\def\A{{\mathcal A}}
\def\B{{\mathcal B}}
\def\C{{\mathcal C}}

\def\bv{{\boldsymbol{v}}}

\def\b1{{\boldsymbol{1}}}
\def\bx{{\boldsymbol{x}}}
\def\bu{{\boldsymbol{u}}}
\def\bd{{\boldsymbol{d}}}

\def\blambda{{\boldsymbol{\lambda}}}

\def\barH{\bar{H}}
\def\bartau{\bar{\tau}}

\def \b0{{\boldsymbol{0}}}
\def\barg{{\bar{g}}}
\def\barf{{\bar{f}}}
\def\barH{{\bar{H}}}
\def\bnabla{{\bar{\nabla}}}
\def\bargamma{{\bar{\gamma}}}

\def\barblambda{{\bar{\boldsymbol{\lambda}}}}
\def\barmu{{\bar{\mu}}}

\def\bartau{{\bar{\tau}}}

\def\barbzeta{\bar{\boldsymbol{\zeta}}}
\def\bzeta{\boldsymbol{\zeta}}

\mathtoolsset{showonlyrefs=true}

\usepackage{xargs}
\usepackage[colorinlistoftodos,prependcaption,textsize=tiny]{todonotes}
\newcommandx{\unsure}[2][1=]{\todo[linecolor=red,backgroundcolor=red!25,bordercolor=red,#1]{#2}}
\newcommandx{\change}[2][1=]{\todo[linecolor=blue,backgroundcolor=blue!25,bordercolor=blue,#1]{#2}}
\newcommandx{\info}[2][1=]{\todo[linecolor=OliveGreen,backgroundcolor=OliveGreen!25,bordercolor=OliveGreen,#1]{#2}}
\newcommandx{\improvement}[2][1=]{\todo[linecolor=Plum,backgroundcolor=Plum!25,bordercolor=Plum,#1]{#2}}

\usepackage{alphalph}

\allowdisplaybreaks

\begin{document}

\title{High Probability Complexity Bounds of Trust-Region Stochastic Sequential Quadratic Programming with Heavy-Tailed Noise}

\author[1]{Yuchen Fang}
\author[2]{Javad Lavaei}
\author[3]{Sen Na}
\affil[1]{Department of Mathematics, University of California, Berkeley}
\affil[2]{Department of Industrial Engineering and Operations Research, University of California, Berkeley}
\affil[3]{H. Milton School of Industrial and Systems Engineering, Georgia Institute of Technology}

\date{}

\maketitle

\begin{abstract}

In this paper, we consider nonlinear optimization problems with a stochastic objective and deterministic equality constraints. We propose a Trust-Region Stochastic Sequential Quadratic Programming (TR-SSQP) method and establish its high-probability iteration complexity bounds for identifying \textit{first- and second-order} $\epsilon$-stationary points. In our algorithm, we assume that exact objective values, gradients, and Hessians are not directly accessible but can be estimated via zeroth-, first-, and second-order probabilistic oracles. Compared to existing complexity studies of SSQP methods that rely on a zeroth-order oracle with sub-exponential tail noise (i.e., light-tailed) and focus mostly on first-order stationarity, our analysis accommodates biased (also referred to as irreducible in the literature) and heavy-tailed noise in the zeroth-order oracle, and significantly extends the analysis to second-order stationarity. 
We show that under heavy-tailed noise conditions, our SSQP method achieves the same high-probability first-order iteration complexity bounds as in the light-tailed noise setting, while further exhibiting promising second-order iteration complexity bounds. 
Specifically, the method identifies a first-order $\epsilon$-stationary point in $\mathcal{O}(\epsilon^{-2})$ iterations and a second-order $\epsilon$-stationary point in $\mathcal{O}(\epsilon^{-3})$ iterations with high probability, provided that $\epsilon$ is lower bounded by a constant determined by the bias magnitude (i.e., the irreducible noise) in the estimation. We validate our theoretical findings and evaluate practical performance of our method on CUTEst benchmark test set.

\end{abstract}

\section{Introduction}\label{sec:1}

In this paper, we consider stochastic optimization problems with deterministic equality constraints of the following form:
\begin{equation}\label{Intro_StoProb}
\min_{\bx\in\mR^d}\;f(\bx),\quad\;\;\;\text{s.t.}\quad c(\bx)=\b0,
\end{equation}
where $f:\mR^d\rightarrow \mR$ is the stochastic objective and $c:\mR^d\rightarrow\mR^m$ are deterministic equality constraints. For Problem \eqref{Intro_StoProb}, we assume that the exact objective value $f(\bx)$, together with its gradient $\nabla f(\bx)$ and Hessian $\nabla^2 f(\bx)$, cannot be evaluated accurately but can be estimated via stochastic probabilistic oracles. For the sake of generality, we do not specify the underlying true sampling distribution of $f$, since the oracles may have different sampling distributions, leading to intrinsically biased estimates. Constrained stochastic optimization problems of form \eqref{Intro_StoProb} appear ubiquitously in scientific and engineering areas, such as optimal control \citep{Betts2010Practical}, portfolio optimization \citep{Cakmak2005Portfolio}, multi-stage optimization \citep{Shapiro2021Lectures}, supply chain network design \citep{Santoso2005stochastic}, and statistical machine learning \citep{Cuomo2022Scientific}.

Deterministic constrained optimization has been extensively studied for decades. While numerous methods have been developed under favorable problem settings, Sequential Quadratic Programming (SQP) has proven to be robust and effective, especially for highly nonlinear problems \citep{Bertsekas1982Constrained, Boggs1995Sequential, Nocedal2006Numerical}. Recent studies on optimization problems where the objective function and constraints are evaluated subject to bounded noise further highlight the advantages of SQP methods \citep{Sun2023trust, Lou2024Noise, Oztoprak2023Constrained, Sun2024Trust}. Inspired by the success of SQP, a series of Stochastic SQP (SSQP) methods has been proposed to solve Problem \eqref{Intro_StoProb}, where the exact computations of objective quantities are replaced by the corresponding estimates obtained through sampling. The quality of the estimates is assumed to satisfy specific conditions that depend on different sampling frameworks.

The first widely used framework is the \textit{fully stochastic} setup, where at each iteration a single sample or a fixed batch of samples is drawn to estimate the objective gradient, and a prespecified sequence is utilized to assist the stepsize selection. Gradient estimates are typically assumed to be unbiased, with variance either bounded or subject to a growth condition. 
Under this setup, \citet{Berahas2021Sequential} introduced the first line-search SSQP method, while \citet{Fang2024Fully} introduced the first trust-region SSQP method. Building on these designs, subsequent extensions in algorithm design and analysis have been reported \citep[see, e.g.,][]{Berahas2023Stochastic, Berahas2023Accelerating, Curtis2024Sequential, Curtis2024Stochastic, Beiser2023Adaptive, Bollapragada2023adaptive,  Berahas2025Sequential, Curtis2025Almost, Curtis2025Stochastic, Curtis2026Single, Kuang2025Online}.
The non-asymptotic iteration complexity has also been investigated under this setup. Based on the algorithm in \citet{Berahas2021Sequential}, \citet{Curtis2023Worst} established a worst-case iteration complexity of $\mathcal{O}(\epsilon^{-4})$ for first-order $\epsilon$-stationarity with high probability, assuming that the trial step is solved exactly in each iteration and the input prespecified stepsize-related sequence is constant. Meanwhile, \citet{Na2025Statistical} established the local convergence guarantee of the SSQP method with decaying stepsize, demonstrating the same $\mathcal{O}(\epsilon^{-4})$ iteration complexity for first-order $\epsilon$-stationarity while allowing the trial step to be solved inexactly in each iteration. In addition, \citet{Lu2024Variance} measured the $\epsilon$-stationarity by expected optimality residual and deterministic feasibility residual (i.e., constraint violation), and established $\tilde{\O}(\epsilon^{-3})$ sample complexity and $\tilde{\O}(\epsilon^{-4})$ iteration complexity for first-order stationarity, where $\tilde{\O}(\cdot )$ hides the logarithmic term.

In contrast to the fully stochastic setup, methods in the \textit{random model} setup do not require unbiased gradient estimates, and the stepsizes are also selected adaptively. This framework leverages zeroth-, first-, and second-order probabilistic oracles to estimate objective values, gradients, and Hessians, respectively, ensuring that the estimates satisfy proper, adaptive accuracy conditions with a high but fixed probability. The random model setup has a relatively long history in the literature on unconstrained problems, starting from the STORM method and its subsequent designs \citep{Chen2017Stochastic, Ghanbari2017Black, Blanchet2019Convergence, Jin2025Sample}. Under this setup, \cite{Na2022adaptive, Fang2024Trust} recently introduced some of the early SSQP methods for constrained problems. The studies have been further generalized in \cite{Na2023Inequality, Berahas2022Adaptive, Fang2026Trust,Qiu2023sequential} to accommodate inequality constraints, adaptive sampling mechanisms, and robust designs. However, the above studies in the random model setup share two key limitations: (1) they assume the estimation noise vanishes with fixed probability, and (2) they lack a non-asymptotic, high-probability iteration complexity analysis.

Recently, \citet{Berahas2025Sequential} addressed the two limitations in constrained optimization by analyzing a line-search (also referred to as step-search) method that allows for an biased, irreducible noise in the estimation. The authors conducted a non-asymptotic complexity analysis and established an iteration complexity of $\mathcal{O}(\epsilon^{-2})$ for first-order $\epsilon$-stationarity with high probability, provided that $\epsilon$ exceeds a threshold determined by the irreducible noise level. However, their zeroth-order oracle assumes the noise follows a sub-exponential distribution, hence excluding all heavy-tailed noise. Moreover, their algorithm does not utilize second-order Hessian information, so that theoretical analysis offered in that work does not cover second-order convergence property. Identifying only first-order stationary points can be sometimes unsatisfactory in nonlinear optimization, as such points may correspond to saddle points or even local maxima, leading to highly sensitive solutions \citep{Dauphin2014Identifying, Choromanska2015Loss, Jain2017Global}. To the best of our knowledge, in the random model setup, high-probability iteration complexity for second-order stationarity has only been explored in the context of unconstrained optimization. In this setting, \citet{Cao2023First} performed a non-asymptotic analysis for a trust-region method, demonstrating an iteration complexity of $\mathcal{O}(\epsilon^{-2})$ for first-order $\epsilon$-stationarity and $\mathcal{O}(\epsilon^{-3})$ for second-order $\epsilon$-stationarity. However, their zeroth-order probabilistic oracle still requires the noise to follow a sub-exponential distribution.

Assuming the noise in the zeroth-order oracle following a sub-exponential distribution leads to a significant gap between non-asymptotic and asymptotic analyses, as such an oracle condition is not required in asymptotic analysis \citep{Na2022adaptive, Na2023Inequality, Fang2024Trust}. The sub-exponential noise condition can be restrictive, as it implies the existence of infinite-order moments, thereby excluding noise from any heavy-tailed distribution where only a finite number of moments are bounded. More importantly, there is no clear justification for the discrepancy between the noise conditions in the zeroth-order oracle and those in the first- and second-order oracles; the latter oracles do allow for heavy-tailed noise (e.g., only a finite second moment). This discrepancy highlights the desire to relax light-tailed noise in the zeroth-order oracle to heavy-tailed noise when analyzing the complexity of the methods.

In this paper, we make progress toward that goal. We introduce a trust-region SSQP method for constrained optimization in the random model setup that accommodates \textit{irreducible and heavy-tailed noise}, and establish its high-probability iteration complexity bounds for achieving both \textit{first- and second-order $\epsilon$-stationarities}. Our key contributions are summarized as follows.
\begin{enumerate}[label=\textbf{(\alph*)},topsep=2pt, wide,labelindent=0pt]
\setlength\itemsep{0.2em}
\item Compared to \cite{Na2022adaptive, Na2023Inequality, Fang2024Trust}, we allow irreducible noise in the zeroth-, first-, and second-order oracles, while their framework requires the noise to diminish with fixed probability. In addition, we extend the second-order asymptotic convergence guarantee from \cite{Fang2024Trust} to establish the first second-order non-asymptotic guarantee for constrained problems.$\hskip1cm$

\item 
Compared to the step-search SSQP method in \cite{Berahas2025Sequential}, we study a trust-region SSQP method and provide a complexity analysis for second-order stationarity. As a classical complement to the line-search scheme, the trust-region scheme computes the search direction and selects the stepsize jointly. One significant difference from the line-search scheme is that the trust-region SQP subproblems remain well-defined without any Hessian regularizations, which are often necessary in the line-search scheme to ensure the search direction to be a descent direction.

\item Compared to \cite{Cao2023First, Berahas2025Sequential} for unconstrained or constrained first-order complexity analysis, we refine the definition of the stopping time (Definition \ref{def:stopping_time}) of the iteration by using the deterministic KKT residual for first-order stationarity, and the maximum of deterministic KKT residual and deterministic negative curvature of the reduced Hessian for second-order stationarity. \cite{Cao2023First} adopted a stopping time definition that depends on the Lipschitz constant and the irreducible noise, while \cite{Berahas2025Sequential} defined the stopping time based on the local model reduction of the merit function. Although different stopping time definitions lead to the same complexity bounds, our definition seems more natural while necessitating additional efforts in the analysis. See the discussion after Definition \ref{def:stopping_time} for details.

\item Compared to the complexity analysis in \cite{Cao2023First, Berahas2025Sequential}, we significantly relax the parametric sub-exponential distribution assumption on the zeroth-order oracle noise. Specifically, instead of requiring infinite-order moments, we assume the noise has a bounded $1 + \delta$ moment for any (small) $\delta > 0$. The key technical tools we leverage are the \textit{Burkholder-type inequality} \citep{Burkholder1973Distribution,Chen2020Rosenthal} and the \textit{martingale Fuk-Nagaev inequality} (Lemma \ref{Fuk–Nagaev}) \citep{Fuk1973Certain,Nagaev1979Large,Fan2017Deviation}, which help us quantify the tail behavior of the accumulated oracle noise. For any $\delta > 0$, our complexity results suggest that the method can find a (first-order or second-order) stationary point within finite iterations almost surely. This matches the result in \cite{Berahas2025Sequential}, where infinite-order moments are imposed. We highlight that our relaxation of the noise condition is achieved purely through sharper analysis and is independent of various line-search or trust-region designs, suggesting that the results in the aforementioned literature should also hold under heavy-tailed noise condition utilizing our analytical techniques. 
Our technique has been used in the follow-up work \citep{Scheinberg2025Stochastic}, which unifies the complexity analysis of unconstrained line-search and trust-region methods for achieving first-order stationarity, under a zeroth-order oracle noise with at least bounded second moments (i.e., $\delta \geq 1$).$\quad$

\end{enumerate}

Under irreducible and heavy-tailed oracle noise conditions, we establish \textit{first- and second-order} high-probability iteration complexity bounds for our trust-region SSQP method.
In particular, we establish that the method achieves a first-order $\epsilon$-stationary point in $\mathcal{O}(\epsilon^{-2})$ iterations and a second-order $\epsilon$-stationary point in $\mathcal{O}(\epsilon^{-3})$ iterations with high probability, provided that $\epsilon$ is chosen above a threshold determined by the irreducible noise levels. We summarize the results, along with the dependence of $\epsilon$ on irreducible noises in Table \ref{tab:complexity_result}. These results align with existing arguments for irreducible and light-tailed noise. Specifically, our results corroborate the first-order high-probability complexity bound established by \cite{Berahas2025Sequential} for a step-search SSQP method in constrained stochastic optimization, as well as the first- and second-order high-probability complexity bounds derived by \cite{Cao2023First} for a trust-region method in unconstrained stochastic optimization. We implement our method on problems from the CUTEst test set. The actual performance validates our theoretical analysis.

\renewcommand{\arraystretch}{1.5}
\begin{table}[t!]
\centering
\begin{tabular}{c c c c}
\hline
Stationarity  & Iteration Complexity & Sample Complexity* & Conditions of $\epsilon$ \\
\hline
First-order  & $\mathcal{O}(\epsilon^{-2})$ & $\mathcal{O}(\epsilon^{-6})$ & $\epsilon \geq \mathcal{O}(\sqrt{\epsilon_f} + \epsilon_g)$ \\
\hline
Second-order  & $\mathcal{O}(\epsilon^{-3})$ & $\mathcal{O}(\epsilon^{-9})$ & $\epsilon \geq \mathcal{O}(\sqrt[3]{\epsilon_f} + \sqrt{\epsilon_g} + \epsilon_h)$\\
\hline
\end{tabular}

* Sample complexity refers to the number of samples for estimating the objective value, which dominates the number of samples for estimating the gradient and Hessian. See Remark \ref{rem:4.21}. $\quad$\vskip2pt
\caption{Summary of complexity results. Each row corresponds a type of stationarity. $\epsilon_f,\epsilon_g,\epsilon_h$ denote the irreducible noises in objective value estimation, gradient estimation, and Hessian estimation.}
\label{tab:complexity_result}
\end{table}

\subsection{Notation}

We use $\|\cdot\|$ to denote the $\ell_2$ norm for vectors and the operator norm for matrices. $I$ denotes the identity matrix and $\b0$ denotes the zero vector/matrix, whose dimensions are clear from the context. For the constraints $c(\bx): \mR^d\rightarrow\mR^m$, we let $G(\bx) \coloneqq\nabla c(\bx) \in\mR^{m\times d}$ denote its Jacobian matrix and let $c^i(\bx)$ denote the $i$-th constraint for $1\leq i\leq m$ (the subscript indexes the iteration). Define $P(\bx)=I-G(\bx)^T[G(\bx)G(\bx)^T]^{-1}G(\bx)$ to be the projection matrix onto the null space of $G(\bx)$. Then, we let $Z(\bx)\in\mR^{d\times (d-m)}$ form an orthonormal basis of $\text{ker}(G(\bx))$ such that $Z(\bx)^TZ(\bx)=I$ and $Z(\bx)Z(\bx)^T=P(\bx)$.  Throughout the paper, we use an overline to denote a stochastic estimate of a quantity. For example, $\barf(\bx)$ denotes an estimate of $f(\bx)$.

\vspace{-0.15cm}
\subsection{Structure of the paper}

We introduce the computation of trial steps and the algorithm design in Section \ref{sec:2_short}. Probabilistic oracles that allow irreducible and heavy-tailed noise are introduced in Section \ref{sec:3}. In Section \ref{sec:4}, we establish first- and second-order high-probability complexity bounds. Numerical experiments are presented in Section \ref{sec:5}, and conclusions are presented in Section \ref{sec:6}. To make the main paper concise, we defer some proofs to the appendix.

\section{A Trust-Region SSQP Method }\label{sec:2_short}

Let $\L(\bx,\blambda)=f(\bx)+\blambda^Tc(\bx)$ be the Lagrangian function of Problem \eqref{Intro_StoProb}, where $\blambda\in\mR^m$ represents the Lagrangian multipliers. For any $\epsilon>0$, we call $(\bx^*,\blambda^*)$ a first-order $\epsilon$-stationary point of \eqref{Intro_StoProb} if
\begin{equation}\label{1st_order_point}
\| \nabla \L(\bx^*,\blambda^*) \| = \left\lVert \begin{pmatrix}
\nabla_\bx \L(\bx^*,\blambda^*)\\
\nabla_{\blambda} \L(\bx^*,\blambda^*)
\end{pmatrix} \right\rVert =
\left\lVert \begin{pmatrix}
\nabla f(\bx^*)+G(\bx^*)^T\blambda^*\\
c(\bx^*)
\end{pmatrix} \right\rVert \leq \epsilon.
\end{equation}
We call $(\bx^*,\blambda^*)$ a second-order $\epsilon$-stationary point if, in addition to \eqref{1st_order_point}, 
\begin{equation*}
\tau(\bx^*,\blambda^*)\geq -\epsilon,
\end{equation*}
where $\tau(\bx^*,\blambda^*)$ denotes the smallest eigenvalue of the reduced Lagrangian Hessian $Z(\bx^*)^T\nabla^2_{\bx}\L(\bx^*,\blambda^*)Z(\bx^*)$. Here, we define $\nabla^2_{\bx}\L(\bx,\blambda)=\nabla^2 f(\bx)+\sum_{i=1}^{m}\blambda^i\nabla^2 c^{i}(\bx)$ as the Lagrangian Hessian with respect to the primal variable $\bx$.
We denote the optimality residual by $\|\nabla_\bx \L(\bx,\blambda)\|$, the feasibility residual by $\|\nabla_\blambda \L(\bx,\blambda)\|$ (i.e., $\|c(\bx)\|$), and the KKT residual by $\|\nabla \L(\bx,\blambda)\|$. Given the $k$-th iterate $(\bx_k,\blambda_k)$, we  denote $g_k= \nabla f(\bx_k)$, $\nabla^2 f_k= \nabla^2 f(\bx_k)$, and $\barg_k$, $\bar{\nabla}^2 f_k$ as their estimates, which are constructed via probabilistic oracles in Section \ref{sec:3}. Similarly, we denote $c_k,G_k, \{\nabla^2 c_k^i\}_{i=1}^m$. We define the \textit{estimated} Lagrangian gradient as $\bar{\nabla} \L_k = (\bar{\nabla}_{\bx}\L_k,c_k)$ with $\bar{\nabla}_{\bx}\L_k=\barg_k+G_k^T\blambda_k$, and the \textit{estimated} Lagrangian Hessian (with respect to $\bx$) as $\bar{\nabla}^2_{\bx}\L_k= \bar{\nabla}^2 f_k+\sum_{i=1}^{m}\blambda_k^i\nabla^2 c^{i}_k$. Throughout the paper, $\bar{H}_k$ is used to denote the approximation of the Lagrangian Hessian $\nabla_{\bx}^2 \mathcal{L}_k$. When only first-order stationarity is desired, $\bar{H}_k$ can be constructed using various methods, such as setting it to the identity matrix, employing the SR1 update, defining it as $\bar{H}_k \coloneqq \bar{\nabla}_{\bx}^2 \mathcal{L}_k$, or taking the average of $\bar{\nabla}_{\bx}^2 \mathcal{L}_k$. However, to ensure second-order stationarity, we specifically set $\bar{H}_k$ to $\bar{H}_k \coloneqq \bar{\nabla}_{\bx}^2 \mathcal{L}_k$, since we will leverage the curvature information of $\bar{\nabla}_{\bx}^2 \mathcal{L}_k$ to explore its eigen directions. Further details on this construction are in Sections \ref{sec:2.2} and \ref{sec:5}.

In the remainder of this section, we first introduce step computation, a key component of our algorithm, and then describe the algorithm.

\subsection{Step computation}\label{sec:2.1}

Given the iterate $\bx_k$ and the trust-region radius $\Delta_k$ in the $k$-th iteration, we formulate the trust-region SSQP subproblem to solve for the trial step $\Delta\bx_k$ by using a quadratic approximation of the objective and a linear approximation of the constraints in \eqref{Intro_StoProb}, with an additional trust-region constraint:
\begin{align}\label{def:SQPsubproblem}
\min_{\Delta\bx\in\mR^d}&\quad \frac{1}{2}\Delta\bx^T\barH_k\Delta\bx+\barg_k^T\Delta\bx,\quad\quad\text{s.t.}\quad c_k+G_k\Delta\bx=\b0,\;\;\;\|\Delta\bx\|\leq\Delta_k.
\end{align} 
Due to the presence of the trust-region constraint, we need to deal with a potential infeasibility issue:
\begin{equation*}
\{\Delta\bx\in\mR^d:c_k+G_k\Delta\bx=\b0\}\cap\{\Delta\bx\in\mR^d:\|\Delta\bx\|\leq\Delta_k\}=\emptyset.
\end{equation*}
To address this issue, we follow the existing literature \citep{Vardi1985Trust, Byrd1987Trust, Omojokun1989Trust, Fang2024Fully, Fang2024Trust} by relaxing the linear constraint in \eqref{def:SQPsubproblem} via the trial step decomposition. In particular, we decompose the trial step $\Delta\bx_k$ into two orthogonal components, $\Delta\bx_k = \bw_k + \bt_k$, which are computed separately. The normal step $\bw_k \in \text{im}(G_k^T)$ accounts for feasibility, while the tangential step $\bt_k \in \ker(G_k)$ accounts for optimality. To satisfy the trust-region constraint, we further decompose the trust-region radius to control the lengths of the normal and tangential steps, respectively. Depending on whether we aim to achieve first- or second-order convergence, the trial step $\Delta\bx_k$ is computed using either a gradient step or an eigen step. The gradient step reduces the KKT residual to achieve first-order convergence, whereas the eigen step explores the negative curvature of the reduced Lagrangian Hessian to achieve second-order convergence. However, when searching for second-order stationary points, the Maratos effect \citep{Conn2000Trust} can cause iterates to stagnate at saddle points. To mitigate this special issue in constrained optimization, we compute a second-order correction (SOC) step and append it to the trial step whenever the Maratos effect is detected.

\subsubsection{Gradient steps}\label{subsec:2.1.1}

When computing the gradient step, we decompose the trust-region radius as 
\begin{equation*}
\breve{\Delta}_k=\frac{\|c_k^{RS}\|}{\|\bar{\nabla}\L_k^{RS}\|}\cdot\Delta_k\quad\quad\quad\text{ and  }\quad\quad\quad\tilde{\Delta}_k=\frac{\|\bar{\nabla}_{\bx}\L_k^{RS}\|}{\|\bar{\nabla}\L_k^{RS}\|}\cdot\Delta_k
\end{equation*}
to control the lengths of $ \bw_k$ and $\bt_k$, respectively, and have $\Delta\bx_k = \bw_k + \bt_k$. Here, $c^{RS}_k\coloneqq c_k/\|G_k\|$, $\bar{\nabla}_{\bx}\L_k^{RS} \coloneqq \bar{\nabla}_{\bx}\L_k/\|\barH_k\|$, $\bar{\nabla}\L_k^{RS} \coloneqq (\bar{\nabla}_{\bx}\L_k^{RS},c^{RS}_k)$ denote \textit{rescaled} feasibility, optimality, and KKT residual vectors, respectively. Decomposing the radius based on the rescaled residuals preserves the scale invariance property, so that if the objective and/or constraints are scaled by a (positive) scalar, the radius decomposition and further step computation remain unchanged. Such a desirable property cannot be achieved if using original residuals $\|c_k\|$ and $\|\bnabla_\bx\mL_k\|$ in the above radius decomposition. \;\;

Suppose $G_k$ has full row rank (as assumed in Assumption \ref{assump:4-1}), we define 
\begin{equation*}
\bv_k \coloneqq -G_k^T[G_kG_k^T]^{-1}c_k.
\end{equation*}
When $G_k$ is rank deficient, the normal step is defined either as the least-squares solution $\bv_k = \arg\min_{\bv\in\text{im}(G_k^T)}\|c_k + G_k\boldsymbol{v}\|^2$ or as the trust-region solution $\bv_k = \arg\min_{\bv\in\text{im}(G_k^T)}  \|c_k+ G_k\boldsymbol{v}\|^2$ s.t. $\|\boldsymbol{v}\|\leq w \|G_k^Tc_k\|$ for some $w>0$. However, in this rank-deficient case, the method may not aim to converge to a stationarity point but rather to minimize the infeasibility $\|c(\bx)\|^2$. We point to \cite{Berahas2023Stochastic} for further details, while restrict our attention here to the full-row-rank case.
Without the trust-region constraint, $c_k+G_k\Delta\bx_k=\b0$ would imply $\bw_k = \bv_k$ since $G_k\bt_k=\b0$. However, the trust-region constraint enforces us to set the normal step $\bw_k$ by shrinking $\bv_k$ as
\begin{equation}\label{eq:normal_step}
\bw_k = \bargamma_k\bv_k \quad\text{with}\quad \bargamma_k \coloneqq \min\{ \breve{\Delta}_k/\|\bv_k\|,\; 1 \}.
\end{equation}
The tangential step $\bt_k$ is solved from
\begin{equation}\label{eq:Sto_tangential_step1}
{\min_{\bt\in\mR^{d}}\; m(\bt) \coloneqq\frac{1}{2}\bt^T\barH_k\bt+(\barg_k+\barH_k\bw_k)^T\bt,\quad  \text{s.t.}\quad G_k\bt=\b0,\quad \|\bt\|\leq\tilde{\Delta}_k.}
\end{equation}
Since any feasible solution $\bt_k$ of \eqref{eq:Sto_tangential_step1} can be expressed as $\bt_k=Z_k\bu_k$, where $Z_k\in \mR^{d\times (d-m)}$ forms an orthonormal basis of $\ker(G_k)$ and $\bu_k\in\mR^{d-m}$.
We can rewrite \eqref{eq:Sto_tangential_step1} as
\begin{equation}\label{eq:Sto_tangential_step}
\min_{\bu\in\mR^{d-m}}\quad \tilde{m}(\bu) \coloneqq\frac{1}{2}\bu^TZ_k^T\barH_kZ_k\bu+(\barg_k+\barH_k\bw_k)^TZ_k\bu,\quad\quad \text{s.t.}\quad\|\bu\|\leq\tilde{\Delta}_k,
\end{equation}
which is a trust-region problem for unconstrained optimization.  For the trust-region constraint in \eqref{eq:Sto_tangential_step} we use the observation that $\|\bt\|^2=\bu^TZ_k^TZ_k\bu=\|\bu\|^2$.

We only require $\bt_k$ to achieve a fixed fraction $\kappa_{fcd}\in(0,1]$ of the  Cauchy reduction. In particular, let $\bt_k^C = Z_k\bu_k^C$ denote the Cauchy point  \cite[see][(4.11) for the formula]{Nocedal2006Numerical}, then we require \cite[see][Lemma 4.3 for the last equality]{Nocedal2006Numerical}:
\begin{multline}\label{eq:cauchy1}
m(\bt_k)-m(\b0) \leq \kappa_{fcd}(m(\bt_k^C)-m(\b0)) = \kappa_{fcd}(\tilde{m}(\bu_k^C)-\tilde{m}(\b0))
\\
=-\frac{\kappa_{fcd}}{2}\|Z_k^T(\barg_k+\barH_k\bw_k)\|\min\left\{\tilde{\Delta}_k,\frac{\|Z_k^T(\barg_k+\barH_k\bw_k)\|}{\|Z_k^T\barH_kZ_k\|}\right\}.
\end{multline}
Condition \eqref{eq:cauchy1} is standard in the trust-region literature, which is achievable by computing the Cauchy point of \eqref{eq:Sto_tangential_step} or applying various approaches to solve \eqref{eq:Sto_tangential_step}, including the two-dimensional subspace minimization method, the dogleg method, and the Steihaug's algorithm \citep[see][Chapter 4 for more details]{Nocedal2006Numerical}.

\subsubsection{Eigen steps}\label{subsec:2.1.2}

We define $\bartau_k$ to be the smallest eigenvalue of $Z_k^T\barH_kZ_k$ and $\bartau_k^+\coloneqq |\min\{\bartau_k,0\}|$. Parallel to the gradient step $\Delta\bx_k = \bw_k + \bt_k$, the normal step $\bw_k$ is still computed as in \eqref{eq:normal_step} but the radius decomposition is replaced by
\begin{equation*}
\breve{\Delta}_k=\frac{\|c_k^{RS}\|}{\|(c_k^{RS},\bartau_k^{RS+})\|}\cdot\Delta_k\quad\quad\quad\text{ and  }\quad\quad\quad\tilde{\Delta}_k=\frac{\bartau_k^{RS+}}{\|(c_k^{RS},\bartau_k^{RS+})\|}\cdot\Delta_k,
\end{equation*}
where $\bartau_k^{RS+}\coloneqq \bartau_k^+/\|\barH_k\|$ is the rescaled negative curvature and $(c_k^{RS},\bartau_k^{RS+})\in\mR^{m+1}$ denotes the vertical concatenation of $c_k^{RS}$ and $\bartau_k^{RS+}$. The tangential step is still $\bt_k=Z_k\bu_k$ with $\bu_k$ (inexactly) solved from \eqref{eq:Sto_tangential_step}, but the reduction condition is adjusted to
\begin{equation}\label{eq:reduction_eigen_step}
(\barg_k+\barH_k\bw_k)^TZ_k\bu_k\leq 0,\quad\quad \|\bu_k\| \leq \tilde{\Delta}_k,\quad\quad \bu_k^TZ_k^T\barH_kZ_k\bu_k \leq-\kappa_{fcd}\cdot\bar{\tau}_k^+\tilde{\Delta}_k^2.
\end{equation}
When the above conditions are satisfied, $\bt_k=Z_k\bu_k$ achieves the curvature reduction:$\quad$
\begin{equation}\label{nequ:2}
{m(\bt_k)} - m(\b0) =  \tilde{m}(\bu_k) - \tilde{m}(\b0) \leq -\frac{\kappa_{fcd}}{2}\bar{\tau}_k^+\tilde{\Delta}_k^2<0.
\end{equation}
where we use $\kappa_{fcd}\in(0,1]$ to denote the fraction in both gradient steps and eigen steps for simplicity.

Conditions in \eqref{eq:reduction_eigen_step} are standard in the literature \citep[e.g., see Chapter 7.5 in][]{Conn2000Trust} and can be satisfied by computing $\bu_k$ in different approaches. For example, we can compute $\barbzeta_k$, the eigenvector of $Z_k^T\barH_kZ_k$ corresponding to the eigenvalue $\bartau_k$, and then rescale it as $\barbzeta_k^{RS}\coloneqq \pm \barbzeta_k\cdot \tilde{\Delta}_k/\|\barbzeta_k\|$. The rescaled vector $\barbzeta_k^{RS}$ satisfies \eqref{eq:reduction_eigen_step} with $\kappa_{fcd}=1$, provided that $\pm$ is properly chosen. In practice, it suffices to compute a good approximation of $\barbzeta_k$, which can be achieved by employing methods such as truncated conjugate gradient and truncated Lanczos methods.

\subsubsection{Second-order correction steps}\label{subsec:2.1.3}

It is observed by \cite{Byrd1987Trust} that when $\bx_k$ is close to a saddle point, the (gradient or eigen) step $\Delta\bx_k$ may increase $f(\bx)$ and $\|c(\bx)\|$ simultaneously. Since no reduction is made by $\Delta\bx_k$, the current iterate is not updated. Even worse, this issue cannot be resolved by reducing the stepsize, implying that the iterate is trapped at the saddle point. This phenomenon, known as the \textit{(second-order) Maratos effect}, arises from the inaccurate linear approximation of constraints.
To address this, the second-order correction (SOC) step $\bd_k$ is computed to better capture the curvature of the constraints:
\begin{equation}\label{def:correctional_step}
\bd_k=-G_k^T[G_kG_k^T]^{-1}\cbr{c(\bx_k+\Delta\bx_k)-c_k-G_k\Delta\bx_k}.
\end{equation}
We will not compute SOC steps in every iteration, but only when the Maratos effect is likely to occur, see Section \ref{sec:2.2} Case 2 below for details. 

\begin{remark}
In this remark, we briefly discuss the per-iteration computational cost of trust-region methods. We first consider the gradient step for achieving first-order stationarity. Without the trust-region constraint, the subproblem in line-search methods is equivalent to solving a linear system of size $d+m$. For a standard linear system solver such as conjugate gradient (CG), the computational cost arises from matrix-vector products, and is $\mathcal{O}((d+m)^2)$ flops per CG iteration with at most $d+m$ CG iterations in total (it often terminates earlier).
As a comparison, trust-region methods involve computing a normal step \eqref{eq:normal_step} and a tangential step \eqref{eq:Sto_tangential_step1}, with the latter dominating the computational cost. In particular, the tangential step requires solving a trust-region-constrained subproblem \eqref{eq:Sto_tangential_step1}, which can be handled by a projected CG solver (as used in Knitro). 
As preparation for the projected CG method, we form $G_k G_k^T$ ($\mathcal{O}(m^2 d)$) and compute its Cholesky factorization ($\mathcal{O}(m^3)$), leading to a \textit{one-time} cost of $\mathcal{O}(m^2 d + m^3)$. Each projected CG iteration then involves matrix-vector products ($\mathcal{O}(d^2)$) and projections onto $\ker(G_k)$ via triangular solves ($\mathcal{O}(dm + m^2)$), leading to $\mathcal{O}((d+m)^2)$ flops per projected CG iteration. Thus, the total cost of solving the tangential step is
\begin{equation*}
\mathcal{O}\big(m^2 d + m^3 + (d+m)^2 \times \text{\# of projected CG iterations}\big).
\end{equation*}
Note that the number of projected CG iterations is typically much smaller than the worst-case bound of $d-m$ (which corresponds to solving the subproblem exactly). In our setting, we only require achieving a fraction $\kappa_{fcd} \in (0,1]$ of the Cauchy reduction and never solve the tangential subproblem exactly.

We now turn to the eigen step for achieving second-order stationarity. We focus on trust-region methods, as line-search methods with second-order convergence are less explored in the literature. The eigen step has to additionally compute the smallest eigenvalue of the reduced Lagrangian Hessian, and the normal and tangential steps then follow the same computation as discussed above. The smallest eigenvalue (and corresponding eigenvector) can be derived using a matrix-free Lanczos procedure. Similar to projected CG method, it requires forming $G_kG_k^T$ ($\mathcal{O}(m^2d)$), performing a Cholesky factorization ($\mathcal{O}(m^3)$), and each Lanczos iteration costs $\mathcal{O}((d+m)^2)$ flops in the general dense setting. Again, the total number of Lanczos iterations depends on the spectral gap and the desired accuracy, and is often smaller than $d-m$ in practice. 

\end{remark}

\subsection{Trust-region SSQP design}\label{sec:2.2}

We now present the design of our trust-region SSQP method and summarize it in Algorithm \ref{Alg:STORM}. Each iteration of the method comprises four components: (i) estimating the objective value, gradient, and Hessian using the probabilistic oracles introduced later in Section \ref{sec:3}; (ii) computing a trial step and/or an SOC step as described in Section \ref{sec:2.1}; (iii) updating the parameter of the merit function; and (iv) updating the iterate and the trust-region radius.

We use $\alpha$ to indicate the final goal of stationarity, with $\alpha=0$ for first-order stationarity and $\alpha=1$ for second-order stationarity. In the $k$-th iteration, given the iterate $\bx_k$, the trust-region radius $\Delta_k$, and the merit parameter $\barmu_k$, we first obtain the gradient estimate $\barg_k$. Then, we compute the Lagrangian multiplier $\barblambda_k=-[G_kG_k^T]^{-1}G_k\barg_k$ and the Lagrangian gradient $\bar{\nabla}\L_k = \barg_k + G_k^T  \barblambda_k$. For the Hessian estimate, 

\noindent $\bullet$ \textbf{if $\alpha=0$:} we generate any matrix $\barH_k$ to approximate the Lagrangian Hessian $\nabla_{\bx}^2\L_k$ and set $\bartau_k^+=0$.

\vskip0.1cm
\noindent $\bullet$ \textbf{if $\alpha=1$:} we obtain the Hessian estimate $\bar{\nabla}^2 f_k$, compute $\barH_k=\bar{\nabla}^2 f_k+ \sum_{i=1}^{m}\barblambda_k^i\nabla^2 c_k^i$, and set $\bartau_k$ as the smallest eigenvalue of $Z_k^T\barH_kZ_k$ and $\bartau_k^+=|\min\{\bartau_k,0\}|$.

\vskip0.1cm

Then, we decide whether to perform a gradient step or an eigen step by checking
\begin{equation}\label{eq:cri_step_comput}
\|\bar{\nabla}\L_k\|\min\left\{\Delta_k,\frac{\|\bar{\nabla}\L_k\|}{\|\barH_k\|}\right\}\geq  \bartau_k^+\Delta_k\left(\Delta_k+\|c_k\|\right).
\end{equation}
If \eqref{eq:cri_step_comput} holds, we compute the gradient step $\Delta\bx_k$ as in Section \ref{subsec:2.1.1}; otherwise, we compute the eigen step $\Delta\bx_k$ as in Section \ref{subsec:2.1.2}.

Next, we define the predicted (local) model reduction of an (estimated) $\ell_2$ merit function $\L_{\mu}(\bx)=f(\bx)+\mu\|c(\bx)\|$ as
\begin{equation}\label{def:Pred_k}
\text{Pred}_k=\barg_k^T\Delta\bx_k+\frac{1}{2}\Delta\bx_k^T\barH_k\Delta\bx_k+\barmu_k(\|c_k+G_k\Delta\bx_k\|-\|c_k\|),
\end{equation}
and update the merit parameter $\barmu_k\leftarrow\rho\barmu_k$ until
\begin{equation}\label{eq:threshold_Predk}
\text{Pred}_k\leq -\frac{\kappa_{fcd}}{2} \max\left\{\|\bar{\nabla}\L_k\|\min\left\{\Delta_k,\frac{\|\bar{\nabla}\L_k\|}{\|\barH_k\|}\right\},\bartau_k^+\Delta_k\left(\Delta_k+\|c_k\|\right)\right\}.
\end{equation}
With the updated merit parameter, we compute the (estimated) actual reduction of the  merit function $\L_{\barmu_k}(\bx)$. Specifically, we set $\bx_{s_k}=\bx_k+\Delta\bx_k$, obtain the function estimates $\barf_k$ and $\barf_{s_k}$, and compute the (estimated) actual reduction as
\begin{equation}\label{def:Ared_k}
\text{Ared}_k= \bar{\L}_{\barmu_k}(\bx_{s_k})-\bar{\L}_{\barmu_k}(\bx_k) =\barf_{s_k}-\barf_k+\barmu_k(\|c_{s_k}\|-\|c_k\|).
\end{equation}
Finally, we check the following conditions to update $\bx_k$ and $\Delta_k$:
\begin{equation}\label{nequ:3}
\text{(a):}\;\;\; \frac{\text{Ared}_k-\vartheta_\alpha}{\text{Pred}_k}\geq\eta\hskip0.8cm \text{and}\hskip0.8cm \text{(b):} \;\;\; \max\left\{\frac{\|\bnabla\L_k\|}{\max\{1,\|\barH_k\|\}},\bartau_k^+\right\}\geq \eta\Delta_k, 
\end{equation}
where $\vartheta_\alpha = 2\epsilon_f$ if $\alpha=0$ and $\vartheta_\alpha =2 \epsilon_f+\epsilon_g^{3/2}$ if $\alpha=1$. ($\epsilon_f,\epsilon_g>0$ are irreducible noises in the objective value and gradient estimates, respectively. See Section \ref{sec:3} for details).$\quad\;\;$

\vskip0.2cm

\noindent $\bullet$ \textbf{Case 1: (\ref{nequ:3}a) holds.} We update the iterate as $\bx_{k+1}=\bx_{s_k}$. Furthermore, if (\ref{nequ:3}b) holds, we increase the trust-region radius by $\Delta_{k+1}=\min\{\gamma\Delta_k,\Delta_{\max}\}$. Otherwise, we decrease the trust-region radius by $\Delta_{k+1}=\Delta_k/\gamma$.

\vskip0.2cm

\noindent $\bullet$ \textbf{Case 2: (\ref{nequ:3}a) does not hold and $\alpha = 1$.} In this case, we decide whether to perform a SOC step to recheck (\ref{nequ:3}a). If $\|c_k\|\leq r$, we compute a SOC step $\bd_k$ (cf. Section \ref{subsec:2.1.3}), set $\bx_{s_k}=\bx_k + \Delta\bx_k+\bd_k$ and re-estimate $\barf_{s_k}$. Then, we recompute $\text{Ared}_k$ as in \eqref{def:Ared_k} and recheck (\ref{nequ:3}a). If (\ref{nequ:3}a) holds, we go to \textbf{Case 1} above; if  (\ref{nequ:3}a) does not hold, we go to \textbf{Case 3} below. Otherwise, if $\|c_k\|>r$, the SOC step is not triggered and we directly go to \textbf{Case 3} below.

\vskip0.2cm

\noindent $\bullet$ \textbf{Case 3: (\ref{nequ:3}a) does not hold and $\alpha = 0$.} 
We do not update the current iterate by setting $\bx_{k+1}=\bx_k$, and decrease the trust-region radius as $\Delta_{k+1}=\Delta_k/\gamma$.

When $\vartheta_\alpha=0$, (\ref{nequ:3}a) reduces to the standard condition in both deterministic and stochastic trust-region methods \citep[see][]{Powell1990trust, Byrd1987Trust, Omojokun1989Trust, Heinkenschloss2014Matrix,Bandeira2014Convergence,Fang2024Trust}. Inspired by \cite{Cao2023First}, we introduce a relaxation term $\vartheta_\alpha$ in (\ref{nequ:3}a) to accommodate the irreducible noise in the probabilistic oracles. Since both $\text{Ared}_k$ and $\text{Pred}_k$ are negative, subtracting a $\vartheta_\alpha>0$ in the numerator increases the value on the left-hand side of (\ref{nequ:3}a), making it more likely for the iterate to be updated. When $\bx_k$ is updated, we determine whether to increase the trust-region radius based on (\ref{nequ:3}b). 
In principle, (\ref{nequ:3}b) quantifies how confident we are about the trust-region approximation at the given stationarity.
When (\ref{nequ:3}b) is satisfied, the trust-region radius is increased to allow for a larger trial step in the next iteration, potentially enabling greater progress. Conversely, if (\ref{nequ:3}b) is not satisfied, the trust-region radius is reduced to ensure more cautious movement in the subsequent iteration.

\begin{algorithm}[t]
\caption{A Trust-Region SSQP Method}\label{Alg:STORM}
\begin{algorithmic}[1]
\State \textbf{Input:} Initial iterate $\bx_0$ and radius $\Delta_0=\Delta_{\max}$, and parameters $\eta\in(0,1)$, $\kappa_{fcd}\in(0,1]$, $\barmu_0, r, \epsilon_f,\epsilon_g>0$, $\rho,\gamma>1$. 
\State Set $\alpha=0$ for first-order stationarity and $\alpha=1$ for second-order stationarity.
\For {$k=0,1,\cdots$}
\State Obtain $\barg_k$ and compute $\barblambda_k$ and $\bar{\nabla}\L_k$. 
\State If $\alpha=1$, obtain $\bar{\nabla}^2 f_k$, compute $\barH_k$ and the smallest eigenvalue $\bartau_{k}$ of $Z_k^T\barH_kZ_k$, and set $\bartau_{k}^+=|\min\{\bartau_{k},0\}|$. Otherwise, let $\barH_{k}$ be certain approximation of $\nabla^2\mL_k$ and set $\bartau_{k}^+ = 0$. 
\State If \eqref{eq:cri_step_comput} holds, compute $\Delta\bx_k$ as a gradient step; otherwise, compute $\Delta\bx_k$ as an eigen step.
\State Perform $\barmu_k \leftarrow \rho\barmu_k$ until $\text{Pred}_k$ satisfies \eqref{eq:threshold_Predk}.  
\State Set $\bx_{s_k}=\bx_k+\Delta\bx_k$, obtain $\barf_k, \barf_{s_k}$, and compute $\text{Ared}_k$ as in \eqref{def:Ared_k}.
\If {(\ref{nequ:3}a) holds}\Comment{\textbf{(Case 1)}}
\State Set $\bx_{k+1}=\bx_{s_k}$. 
\If{(\ref{nequ:3}b) holds} 
\State Set $\Delta_{k+1}=\min\{\gamma\Delta_k,\Delta_{\max}\}$. 
\Else
\State Set $\Delta_{k+1}=\Delta_k/\gamma$.
\EndIf 
\ElsIf{$\alpha=1$ and $\|c_k\|\leq r$} \Comment{\textbf{(Case 2)}}
\State Compute SOC step $\bd_k$, set $\bx_{s_k}=\bx_k+\Delta\bx_k+\bd_k$, re-estimate $\barf_{s_k}$, and recompute $\text{Ared}_k$. 
\State If (\ref{nequ:3}a) holds, perform Lines 10-15; otherwise, perform Line 20.
\Else \Comment{\textbf{(Case 3)}}
\State Set $\bx_{k+1}=\bx_k$, $\Delta_{k+1}=\Delta_k/\gamma$.
\EndIf
\State Set $\barmu_{k+1}=\barmu_k$.
\EndFor
\end{algorithmic}
\end{algorithm}

\vspace{-0.1cm}

\section{Probabilistic Oracles with Irreducible and Heavy-Tailed Noise}\label{sec:3}

In this section, we introduce probabilistic oracles that accommodate irreducible and heavy-tailed noise. In each iteration, the zeroth-, first-, and second-order probabilistic oracles generate estimates of the objective value $\bar{f}(\bx,\xi)$, the objective gradient $\bar{g}(\bx,\xi)$, and the objective Hessian $\bar{\nabla}^2 f(\bx,\xi)$, respectively, where $\xi$ denotes a random variable defined on some probability space. These estimates are required to satisfy certain adaptive accuracy conditions with a high but fixed probability. Our oracle conditions do not prespecify a specific method for generating estimates and allow for estimates to have irreducible noise ($\epsilon_f, \epsilon_g, \epsilon_h$ below), leading to biased estimates. This is a significant extension of \cite{Bandeira2014Convergence, Blanchet2019Convergence, Fang2024Trust, Na2022adaptive, Na2023Inequality}, which required the estimation noise to diminish with a fixed probability.$\hskip1.5cm$

Furthermore, we remove all parametric assumptions about the estimation noise. In particular, \cite{Cao2023First, Berahas2025Sequential} assumed that the noise in the zeroth-order oracle follows a sub-exponential distribution in each iteration (see \eqref{eq4} below). Such a condition essentially assumes the existence of infinite-order moments, thereby excluding all distributions with heavy tails and limiting the scope of problems their analysis can cover. Additionally, the light-tailed noise condition in the zeroth-order oracle contrasts with the noise conditions in the first- and second-order oracles, which indeed allow for heavy-tailed noise. To resolve this discrepancy, as shown in \eqref{eq3}, we relax the noise condition to require only the existence of a $(1+\delta)$-order moment for any small $\delta > 0$, thereby covering noise from a broad class of heavy-tailed distributions.

Before introducing oracle conditions, we first formalize the randomness of the method. Define $\F_{-1}\subseteq\F_0\subseteq\F_1\cdots$ as a filtration of $\sigma$-algebras, where $\mF_{k-1} = \sigma(\{\bx_i\}_{i=0}^{k})$, $\forall k\geq 0$ contains all the randomness before performing the $k$-th iteration. Let $\F_{k-0.5} = \sigma\big(\{\bx_i\}_{i=0}^{k}$ $\cup \{\barg_k, \bnabla^2 f_k\}\big)$; we find for $k\geq 0$, $\sigma(\bx_k,\Delta_k)\subseteq \F_{k-1}$ and $\sigma(\Delta\bx_k,\barblambda_k,\barmu_k,\bd_k)\subseteq\F_{k-0.5}$.$\quad$

\subsection{Probabilistic oracle conditions}\label{sec:3.1}

Let $\epsilon_h, \epsilon_g, \epsilon_f>0$, $\tilde{\epsilon}_f\in(0, \epsilon_f]$, $\kappa_h, \kappa_g, \kappa_f \geq 0$, and $p_h,p_g,p_f\in(0,1)$ be user-specified parameters, $\Delta_k$ denote the trust-region radius, and recall that $\alpha=0$ corresponds to finding first-order stationary points and $\alpha=1$ corresponds to finding second-order stationary points.

We first introduce the probabilistic second-order oracle defined at the current iterate $\bx_k$. The oracle indicates that the noise of the objective Hessian estimate satisfies an accuracy condition of $\epsilon_h+\O(\Delta_k)$ with probability at least $1-p_h$; that is, it allows for irreducible noise no less than $\epsilon_h$ even if $\Delta_k\rightarrow 0$ \cite[cf.][Corollary 4.14]{Fang2024Trust}.

\begin{definition}[Probabilistic second-order oracle]\label{second-order oracle}
When $\alpha=1$, given $\bx_k$, the oracle computes $\bar{\nabla}^2f_k \coloneqq   \bar{\nabla}^2f(\bx_k,\xi_k^h)$, an estimate of the objective Hessian $\nabla^2 f_k $, such that $\quad\quad$
\begin{equation}\label{def:Ak}
\A_k=\left\{\|\bar{\nabla}^2f_k-\nabla^2 f_k\|\leq\epsilon_h + \kappa_h\Delta_k\right\} \quad\quad \text{satisfies}\quad\quad P(\A_k \mid \F_{k-1})\geq 1- p_h.
\end{equation}
\end{definition}

Recall from Section \ref{sec:2.2} that for first-order convergence ($\alpha=0$), we do not have to estimate the objective Hessian. Instead, we only require an arbitrary matrix $\barH_k$ to superficially approximate the Lagrangian Hessian, as long as its norm is bounded with a fixed probability. Specifically, when $\alpha=0$, we impose \eqref{def:Ak} but re-define $\A_k=\left\{\|\barH_k\|\leq \kappa_B \right\}$ for some constant $\kappa_B \geq 1$.

Next, we introduce the probabilistic first-order oracle defined at $\bx_k$. It indicates that the noise of the gradient estimate satisfies an accuracy condition of $\epsilon_g+\O(\Delta_k^{\alpha+1})$ with probability at least $1-p_g$; thus, it allows for irreducible noise with a level at least $\epsilon_g$.$\quad\quad$

\begin{definition}[Probabilistic first-order oracle]\label{first-order oracle}
Given $\bx_k$, the oracle computes $\barg_k\coloneqq \barg(\bx_k,\xi_k^g)$, an estimate of the objective gradient $g_k$, such that
\begin{equation}\label{def:Bk}
\B_k=\{\|\barg_k-g_k\|\leq\epsilon_g + \kappa_g\Delta_k^{\alpha+1}\} \quad\quad \text{satisfies}\quad\quad P(\B_k \mid \F_{k-1})\geq 1 - p_g.
\end{equation}
\end{definition}

Finally, we introduce the zeroth-order probabilistic oracle, which is defined at both the current iterate $\bx_k$ and the trial iterate $\bx_{s_k}$ (since we need to evaluate the objective value for both points). Here, $\bx_{s_k}= \bx_k+\Delta\bx_k$ if the SOC step is not performed and $\bx_{s_k}= \bx_k+ \Delta\bx_k+ \bd_k$ if the SOC step is performed.

\begin{definition}[Probabilistic zeroth-order oracle]\label{def:heavy-tailed oracle}
Given $\bx_k$ and $\bx_{s_k}$, the oracle computes $\barf(\bx_k,\xi_k^f)$ and $\barf(\bx_{s_k},\xi_{s_k}^f)$, which are estimates of the objective function values $f(\bx_k)$ and $f(\bx_{s_k})$. Let $e(\bx,\xi) = |\barf(\bx,\xi)-f(\bx)|$ be the absolute error with $e_k\coloneqq e(\bx_k,\xi_k^f)$ and $e_{s_k}\coloneqq e(\bx_{s_k},\xi_{s_k}^f)$. The zeroth-order oracle is assumed to satisfy the following three conditions.

\noindent$\bullet$\textbf{(i)} The absolute errors $e_k$ and $e_{s_k}$ are sufficiently small with a fixed probability:
\begin{equation}\label{def:Ck}
\C_k=\left\{\max\left(e_k,e_{s_k} \right)\leq \epsilon_f + \kappa_f\Delta_k^{\alpha+2}\right\} \quad\quad \text{satisfies}\quad\quad P(\C_k\mid \F_{k-1/2})\geq 1 - p_f.
\end{equation}
\noindent$\bullet$\textbf{(ii)} The mean absolute errors are sufficiently small:
\begin{equation}\label{eq1}
\max\left\{\mE[ e_k \mid \F_{k-1}],\mE[e_{s_k} \mid \F_{k-1/2}]\right\}  \leq \tilde{\epsilon}_f.
\end{equation}	
\noindent$\bullet$\textbf{(iii)} One of the following tail conditions is satisfied:  
\vskip4pt 
\noindent \textbf{(iii.1) Heavy-tailed condition.} $e_k$ and $e_{s_k}$ have bounded $1+\delta$ moment for some $\delta > 0$. In other words, for some constant $\Upsilon_f>0$,  we have
\begin{equation}\label{eq3}
\max \left\{\mE \left[ \left|e_k-\mE\left[e_k\mid \F_{k-1}\right] \right|^{1+\delta}\mid \F_{k-1}\right],\mE \left[ \left|e_{s_k}-\mE\left[e_{s_k}\mid \F_{k-1/2}\right] \right|^{1+\delta}\mid \F_{k-1/2}\right] \right\} \leq \Upsilon_f.
\end{equation}
\noindent\textbf{(iii.2) Sub-exponential tail condition.} For some constants $v, b>0$, $e_k$ and $e_{s_k}$ satisfy	
\begin{multline}\label{eq4} 
\max \left\{ \mE\left[e^{\lambda(e_k-\mE[e_k\mid\F_{k-1}])}\mid \F_{k-1}\right],	\mE\left[e^{\lambda(e_{s_k}-\mE[e_{s_k}\mid \F_{k-1/2}])}\mid \F_{k-1/2}\right]  \right\} \\
\leq \exp\left( \frac{\lambda^2v^2}{2}\right), \quad \forall \lambda\in[0,\frac{1}{b}].
\end{multline}
\end{definition}

When \textbf{(iii.1) heavy-tailed condition} is satisfied, we say $e_k,e_{s_k}$ are generated via the \textit{probabilistic heavy-tailed zeroth-order  oracle}, while when \textbf{(iii.2) sub-exponential condition} is satisfied, we say $e_k,e_{s_k}$ are generated via the \textit{probabilistic sub-exponential zeroth-order  oracle}.

\begin{remark}\label{rem:1}

We further discuss the probabilistic zeroth-order oracle in this remark.$\quad$

\noindent $\bullet$ The conditions \textbf{(i)} and \textbf{(ii)} in \eqref{def:Ck} and \eqref{eq1} are standard in the literature \citep{Berahas2025Sequential, Cao2023First}, which, however, do not provide any tail information regarding the noise. Thus, existing literature often additionally imposed \textbf{(iii.2)} for the considered zeroth-order oracle.$\hskip2.5cm$

\vskip3pt
\noindent $\bullet$ For the sub-exponential oracle, if $\tilde{\epsilon}_f < \epsilon_f$, then condition \eqref{def:Ck} can be implied by \eqref{eq4}, provided $v, b$ are sufficiently small (i.e., the probability mass of $e_k, e_{s_k}$ is almost concentrated around their means). See Remark \ref{remark:pf} for a rigorous discussion. However, \eqref{eq4} should only be used to capture the tail behavior of $e_k, e_{s_k}$, and requiring infinitesimally small scalars of $v, b$ is overly stringent for the oracle (though the tail behavior remains sub-exponential). Thus, one may not need to consider the interconnections among conditions \eqref{def:Ck}--\eqref{eq4}.$\quad\quad$

\vskip3pt
\noindent $\bullet$
The heavy-tailed condition in \textbf{(iii.1)} significantly relaxes the sub-exponential condition in \textbf{(iii.2)} by reducing the requirement from the existence of infinite-order moments to a finite $1+\delta$ moment. Leveraging the \textit{Burkholder-type inequality} \citep{Burkholder1973Distribution,Chen2020Rosenthal} and the \textit{martingale Fuk-Nagaev inequality} \citep{Fuk1973Certain,Nagaev1979Large,Fan2017Deviation}, we establish high-probability complexity bounds for any $\delta > 0$.  
By Theorem \ref{thm: 1st} and its subsequent discussion, we know our high-probability complexity bounds directly imply that, for any $\delta > 0$, the algorithm finds a (first- or second-order) $\epsilon$-stationary point in a finite number of iterations \emph{almost surely}. If we suppress the irreducible noise by setting $\epsilon_h=\epsilon_g=\epsilon_f=0$, this result further implies a liminf-type almost-sure convergence; that is, for any run of the method, there exists a subsequence of iterates for which the (first- or second-order) stationarity residuals vanish (cf. Definition \ref{def:stopping_time}). 
We note that although liminf-type almost-sure convergence for second-order stationarity with heavy-tailed oracles is common in the literature \citep{Blanchet2019Convergence, Fang2024Trust}, the above asymptotic result implied from non-asymptotic analysis has a gap compared to the convergence results directly from asymptotic analyses in \cite[Theorem 4.18]{Chen2017Stochastic} and \cite[Theorem 4]{Blanchet2019Convergence}, where the authors established lim-type convergence to first-order stationarity \textit{without imposing any moment conditions} on the zeroth-order oracle. 
That being said, those methods do not incorporate irreducible noise in all orders of the oracles, and have a substantially different convergence analysis based on an ad-hoc random function \cite[e.g.,][(20)]{Blanchet2019Convergence}. 
Overall, it remains an open question whether one can strengthen our high-probability complexity bounds by completely removing the moment conditions in \textbf{(ii)} and \textbf{(iii.1)}.

\vskip3pt
\noindent $\bullet$
In the zeroth-order oracle, we follow the existing literature and impose conditions on the absolute objective value estimation errors $e_k$ and $e_{s_k}$ separately. However, we note that imposing accuracy conditions on the error difference $e_{s_k}-e_k$ is sufficient for the analysis. That is, we can redefine $e_k\coloneqq | (\barf_{s_k} -\barf_k) - (f_{s_k} - f_k) | $ and impose oracle conditions on the newly defined $e_k$.
\end{remark}

\begin{remark}\label{rmk:heavy_tail}
We introduce some existing iteration complexity results under heavy-tailed noise with bounded $1+\delta$ moment. In fact, the bounded $(1+\delta)$-moment condition has been extensively studied in the stochastic gradient descent (SGD) literature for unconstrained nonconvex optimization, although the moment condition in those works refers to the gradient noise.
To handle heavy-tailed gradient noise, existing SGD methods typically incorporate gradient clipping, gradient normalization, momentum, or combinations of these techniques.
For example, \citet{Zhang2020Adaptive} analyzed SGD with gradient clipping and established an iteration complexity of $\mathcal{O}(\epsilon^{-(3\delta+1)/\delta})$; \citet{Cutkosky2021High} proved a high-probability convergence rate using a combination of gradient normalization and clipping; \citet{Nguyen2023Improved} studied SGD with adaptive clipping; \citet{Liu2025Nonconvex,Sun2025Revisiting} demonstrated that gradient normalization alone suffices to achieve the optimal rate; and \cite{Fang2026Normalization} extended the analysis to SGD with stochastic preconditioners.
These methods also yield liminf-type almost-sure convergence to first-order stationary points. Our setting similarly allows heavy-tailed gradient noise; indeed, we impose no moment conditions at all (only the high-probability condition \eqref{def:Bk}) on the gradient noise.

Despite these similarities, the aforementioned approaches differ from our method in two respects. 
First, clipping and normalization in SGD operate on gradient estimates, while our analysis requires analogous control on objective value estimates. 
Second, their methods rely crucially on unbiased gradient estimates, while our model includes irreducible noise in objective value (and gradient) estimates, preventing a direct adaptation of their techniques. Whether clipping- or normalization-type ideas can be effectively applied to objective value estimates remains an open question.

\end{remark}

In the next subsection, we briefly discuss how to construct estimates to satisfy the oracle conditions when the estimation noise is heavy-tailed.

\vspace{-0.1cm}

\subsection{Construction of oracles with heavy-tailed noise}\label{subsec:3.2}

Heavy-tailed noise with a finite $1+\delta$ moment is prevalent across numerous application domains, including deep learning \citep{Mahoney2019Traditional}, reinforcement learning \citep{Medina2016No}, robust regression \citep{Pensia2024Robust}, and quantitative finance \citep{Muller1998Heavy}. Although the variance becomes infinite when $\delta < 1$, we show that probabilistic zeroth-, first-, and second-order oracles defined in Section \ref{sec:3.1} can still be constructed.
We denote the sample sets used to estimate $\bar\nabla^2 f_k, \barg_k, \barf_k, \barf_{s_k}$ by $\xi^h_k$, $\xi_k^g$, $\xi_k^f$, $\xi_{s_k}^f$, respectively, and denote their corresponding sample sizes by $|\cdot|$.
Samples are drawn independently from a distribution $\P$, and each realization is assumed to be unbiased with a finite $1+\delta$ moment. Without loss of generality, we restrict attention to $\delta \in (0,1]$, since bounded higher-order moments $(\delta>1)$ imply bounded lower-order moments.

We explore three estimation strategies: the sample average method, the Median-of-Means (MoM) method, and the finite-difference method (with access only to noisy objective value estimates).
We summarize the sample complexity for each method below, while defer detailed proofs to Appendix \ref{Appendix_1}.

\vskip3pt

\noindent$\bullet$ \textbf{Sample average.} For objective value estimate $\barf_k$, we draw $|\xi_k^f|$ i.i.d. samples at $\bx_k$ with each single realization $F(\bx_k,\xi^f)$ satisfying 
\begin{equation*}
\mE[F(\bx_k,\xi^f)\mid \F_{k-1}]=f_k \quad\quad \text{ and }\quad\quad \mE[|F(\bx_k,\xi^f)-f_k|^{1+\delta}\mid \F_{k-1}]\leq \tilde\Upsilon_f,
\end{equation*}
and set the estimate as $\barf_k = \barf(\bx_k,\xi_{k}^f) = \frac{1}{|\xi_k^f|}\sum_{\xi^f\in \xi_k^f}F(\bx_k,\xi^f)$. Then, by Markov inequality and Burkholder-type inequality (Lemma \ref{append:lemma1}), the zeroth-order oracle conditions \eqref{def:Ck}, \eqref{eq1}, and \eqref{eq3} are satisfied as long as $|\xi_k^f|$ is respectively greater than
\begin{equation}\label{sample_avg_1}
\mathcal{O}\left(\left[\frac{1}{p_f}\right]^{\frac{1}{\delta}}\left[\frac{1}{ \epsilon_f+\kappa_f\Delta_k^{\alpha+2} }\right]^{\frac{1+\delta}{\delta}}\right),\quad\quad 
\mathcal{O}\left( \left[\frac{1}{\tilde{\epsilon}_f}\right]^{\frac{1+\delta}{\delta}}\right), \quad\quad 
\mathcal{O}\left(\left[\frac{\tilde{\Upsilon}_f}{\Upsilon_f}\right]^{\frac{1}{\delta}}\right).
\end{equation}
Noting that the last term is independent of the noise $\epsilon_f, \tilde{\epsilon}_f$ and the same requirements apply to $\bar{f}_{s_k}$, we obtain the zeroth-order sample complexity:
\begin{equation}\label{sample_avg_obj_val_1}
\min\{|\xi_k^f|,|\xi_{s_k}^f|\} \geq \mathcal{O}\left( \left[\frac{1}{p_f}\right]^{\frac{1}{\delta}} \left[\frac{1}{\min\{\epsilon_f+\kappa_f\Delta_k^{\alpha+2},\tilde{\epsilon}_f\}}\right]^{\frac{1+\delta}{\delta}}\right).
\end{equation}
To satisfy first- and second-order oracle conditions in \eqref{def:Ak} and \eqref{def:Bk}, same derivations yield
\begin{equation}\label{sample_avg_gradhess_val_1}   
|\xi_k^g| \geq \mathcal{O}\left( \left[\frac{d}{p_g}\right]^{\frac{1}{\delta}} \left[\frac{\sqrt{d}}{\epsilon_g+\kappa_g\Delta_k^{\alpha+1}}\right]^{\frac{1+\delta}{\delta}}\right),\; |\xi_k^h| \geq \mathcal{O}\left(\left[\frac{d^2}{p_h}\right]^{\frac{1}{\delta}}\left[\frac{d}{\epsilon_h+\kappa_h\Delta_k}\right]^{\frac{1+\delta}{\delta}}\right). 
\end{equation}
\noindent$\bullet$ \textbf{Median-of-Means (MoM).} 
The MoM estimator provides robustness to outliers and heavy-tailed noise. Samples are partitioned into equal-size groups; the empirical mean of each group is computed; and the (coordinate-wise) median of these means is used as the estimator. The MoM method guarantees exponential-type concentration bounds even in finite-moment settings. Essentially, by applying standard MoM results (e.g., \cite[Theorem 3.1]{Devroye2016Sub}, \cite[Lemma 2]{Bubeck2013Bandits}, \cite[Theorem 3]{Lugosi2019Mean}), we can improve the dependence of the sample complexity on failure probabilities from $1/p_f^{1/\delta}, (d/p_g)^{1/\delta}, (d^2/p_h)^{1/\delta}$ in \eqref{sample_avg_obj_val_1} and \eqref{sample_avg_gradhess_val_1} to $\log(1/p_f)$, $\log(d/p_g)$, $\log(d^2/p_h)$. In particular, we have all the oracle conditions satisfied as long as 
\begin{align}\label{MoM_obj_val_1}
& \quad\quad\quad \min\{|\xi_k^f|,|\xi_{s_k}^f|\} \geq \mathcal{O}\bigg(\log\left(\frac{1}{p_f}\right)\left[\frac{1}{\min\{\epsilon_f+\kappa_f\Delta_k^{\alpha+2},\tilde{\epsilon}_f\}}\right]^{\frac{1+\delta}{\delta}}\bigg),\\
& |\xi_k^g| \geq \mathcal{O}\bigg(\log\left(\frac{d}{p_g}\right)\left[\frac{\sqrt{d}}{\epsilon_g+\kappa_g\Delta_k^{\alpha+1}}\right]^{\frac{1+\delta}{\delta}}\bigg),\;\;|\xi_k^h| \geq \mathcal{O}\bigg(\log\left(\frac{d^2}{p_h}\right)\left[\frac{d}{\epsilon_h+\kappa_h\Delta_k}\right]^{\frac{1+\delta}{\delta}}\bigg).
\end{align}

\noindent $\bullet$ \textbf{Finite difference.} We now consider the setting where only noisy objective evaluations are available. In this case, we estimate gradients and Hessians via finite differences. Given $\bx_k$, for any direction $\by$, we assume that each sample realization $F(\bx_k+\boldsymbol{y},\xi^f)$ satisfies 
\begin{equation*}
\mE[F(\bx_k+\by,\xi^f)\mid \F_{k-1},\by] = f(\bx_k+\by), \quad  \mE[|F(\bx_k+\boldsymbol{y},\xi^f)-f(\bx_k+\boldsymbol{y})|^{1+\delta}\mid\F_{k-1},\by]\leq \tilde\Upsilon_f.
\end{equation*}
We also assume that the true gradient $\nabla f$ and Hessian $\nabla^2f$ are Lipschitz continuous with constants $L_{\nabla f}$ and $L_{\nabla^2 f}$, respectively. 

For the finite-difference case, we can estimate the objective values $f_k, f_{s_k}$ using either sample average or MoM method, and their sample complexities are given by \eqref{sample_avg_obj_val_1} and \eqref{MoM_obj_val_1}. Furthermore, the $j$-th gradient entry and the $(i,j)$-th Hessian entry are estimated by 
\begin{align*}
[\barg_k]_j & = \frac{1}{\sigma}\rbr{\barf(\bx_k + \sigma \boldsymbol{e}_j, \xi_k^g)- \barf(\bx_k,\xi_k^g)},\\
[\bar{\nabla}^2 f_k]_{i,j} & = 
\frac{1}{\sigma^2}\rbr{\barf(\bx_k + \sigma \boldsymbol{e}_i + \sigma \boldsymbol{e}_j, \xi_k^h) - \barf(\bx_k + \sigma \boldsymbol{e}_i, \xi_k^h) - \barf(\bx_k + \sigma \boldsymbol{e}_j, \xi_k^h) + \barf(\bx_k, \xi_k^h)},
\end{align*}
where $\be_i, \be_j\in\mR^{d}$ are the $i$-th, $j$-th canonical basis of $\mR^d$. Again, each objective value estimate in the above finite differences can be obtained using either sample average or MoM method. We should mention that, for simplicity, we do not consider more advanced simultaneous perturbation or Gaussian smoothing techniques for finite differences \citep{Spall1998overview}, which may need fewer, dimension-independent number of function evaluations (but also introduce additional randomness through randomized perturbations). In addition, we do not require the independence between $\barg_k$ and $\bar{\nabla}^2 f_k$, meaning that we can use the same set of samples for estimating $\barf(\bx_k + \sigma \boldsymbol{e}_j, \xi_k^g)$ and $\barf(\bx_k + \sigma \boldsymbol{e}_j, \xi_k^h)$ to reduce function evaluations. We write $\xi_k^g, \xi_k^h$ to simply distinguish their sample complexities here.

Let us set $\sigma = \frac{\epsilon_g+\kappa_g\Delta_k^{\alpha+1}}{\sqrt{d}L_{\nabla f}}$ for gradient estimation and $\sigma = \frac{\epsilon_h+\kappa_h\Delta_k}{dL_{\nabla^2 f}}$ for Hessian estimation. If we use sample average to estimate $\bar{f}(\cdot,\xi_k^g)$ and $\bar{f}(\cdot, \xi_k^h)$, then the first- and second-order oracle conditions in \eqref{def:Ak} and \eqref{def:Bk}) are satisfied as long as
\begin{equation}\label{sn:16}
 |\xi_{k}^g| \geq \mathcal{O} \left(\left[\frac{d}{p_g}\right]^{\frac{1}{\delta}}\left[\frac{L_{\nabla f}d}{(\epsilon_g + \kappa_g\Delta_k^{\alpha+1})^2}\right]^{\frac{1+\delta}{\delta}}\right), |\xi_{k}^h| \geq \mathcal{O} \left(\left[\frac{d^2}{p_h}\right]^{\frac{1}{\delta}}\left[\frac{L_{\nabla^2 f}^2d^3}{(\epsilon_h + \kappa_h\Delta_k)^3}\right]^{\frac{1+\delta}{\delta}}\right).
\end{equation}
The failure probabilities $(d/p_g)^{1/\delta}, (d^2/p_h)^{1/\delta}$ are replaced by $\log(d/p_g)$, $\log(d^2/p_h)$ if the MoM method is used to estimate $\bar{f}(\cdot,\xi_k^g)$ and $\bar{f}(\cdot, \xi_k^h)$ in finite differences instead.

\section{High-Probability Complexity Bounds} \label{sec:4}

In this section, we analyze the first- and second-order high-probability complexity bounds for the present trust-region SSQP method. In Section \ref{sec:4.1}, we introduce fundamental lemmas. In Sections \ref{sec:4_heavytail} and \ref{sec:4_subexp}, we establish the complexity bounds under probabilistic heavy-tailed and sub-exponential zeroth-order oracles, respectively. We first state the \mbox{assumption}.

\begin{assumption}\label{assump:4-1}

Let $\Omega\subseteq\mR^d$ be an open convex set containing the iterates and trial points $\{\bx_k, \bx_{s_k}\}$. The objective $f(\bx)$ is twice continuously differentiable and bounded below by $f_{\inf}$ over $\Omega$. The gradient $\nabla f(\bx)$ and Hessian $\nabla^2 f(\bx)$ are both Lipschitz continuous over $\Omega$, with constants $L_{\nabla f}$ and $L_{\nabla^2 f}$, respectively. Analogously, the constraint $c(\bx)$ is twice continuously differentiable, its Jacobian $G(\bx)$ is Lipschitz continuous over $\Omega$ with the constant $L_G$. 
For $1 \leq i \leq m$, the Hessian of the $i$-th constraint, $\nabla^2 c^i(\bx)$, is Lipschitz continuous over $\Omega$ with the constant $L_{\nabla^2 c}$. We also assume that there exist constants $\kappa_c$, $\kappa_{\nabla f}$, $\kappa_{1,G}$, $\kappa_{2,G}>0$ such that
\begin{equation*}
\|c_k\| \leq \kappa_{c},\quad\quad \|\nabla f_k\| \leq \kappa_{\nabla f}, \quad\quad \kappa_{1,G} \cdot I \preceq G_k G_k^T \preceq \kappa_{2,G} \cdot I, \quad\quad\forall k\geq 0.
\end{equation*}

\end{assumption}

Assumption \ref{assump:4-1} is standard in the SQP literature \citep{Byrd1987Trust, Powell1990trust, ElAlem1991Global, Conn2000Trust, Berahas2021Sequential, Berahas2023Stochastic, Curtis2024Stochastic, Fang2024Fully,Fang2024Trust}. 
To analyze first-order complexity, it suffices to assume that $f(\bx)$ and $c(\bx)$ are continuously differentiable, without conditions on Hessians $\nabla^2 f(\bx)$ and $\nabla^2 c^i(\bx)$.
Assumption \ref{assump:4-1} implies that $G_k$ has full row rank, $\sqrt{\kappa_{1,G}}\leq\|G_k\|\leq\sqrt{\kappa_{2,G}}$ and $\|G_k^T[G_kG_k^T]^{-1}\| \leq 1/\sqrt{\kappa_{1,G}}$. Consequently, both the true Lagrangian multiplier $\blambda_k = -[G_kG_k^T]^{-1}G_k\nabla f_k$ and the estimated counterpart $\barblambda_k = -[G_kG_k^T]^{-1}G_k\barg_k$ are well defined. Additionally, $\|\nabla^2 f(\bx)\| \leq L_{\nabla f}$ and $\|\nabla^2 c^i(\bx)\| \leq L_G$ for $1\leq i\leq m$ over $\Omega$.

\subsection{Fundamental lemmas}\label{sec:4.1}

In this section, we analyze basic properties of the trust-region SSQP method. Lemmas in this section apply to both the heavy-tailed and the sub-exponential oracles.
We first show that on the event $\A_k\cap\B_k$ (cf. \eqref{def:Ak} and \eqref{def:Bk}), the Hessian estimate $\barH_k$ constructed for second-order convergence has an upper bound. 

\begin{lemma}\label{lemma:bounded_Hessian}
Under Assumption \ref{assump:4-1} with $\alpha=1$, there exists a positive constant $\kappa_B\geq 1$ such that $\|\barH_k\|\leq \kappa_B$ on the event $\A_k\cap\B_k$.
\end{lemma}

\begin{proof}
Recall that $\barH_k = \bar{\nabla}^2f_k+\sum_{i=1}^m\barblambda_k^i\nabla^2c^i_k$, we have
\begin{align*}
\|\barH_k\| & \leq \|\bar{\nabla}^2f_k-\nabla^2f_k\|+\|\nabla^2f_k\| + \|\sum_{i=1}^m(\barblambda_k^i-\blambda_k^i)\nabla^2c^i_k\| +\|\sum_{i=1}^m\blambda_k^i\nabla^2c^i_k\| \nonumber\\
& \leq \|\bar{\nabla}^2f_k-\nabla^2f_k\|+\|\nabla^2f_k\|+\|\barblambda_k-\blambda_k\| \big\{\sum_{i=1}^m\|\nabla^2c^i_k\|^2\big\}^{1/2}+\|\blambda_k\| \big\{\sum_{i=1}^m\|\nabla^2c^i_k\| \big\}^{1/2}\nonumber\\
& \leq \|\bar{\nabla}^2f_k-\nabla^2f_k\|+L_{\nabla f}+\frac{\sqrt{m} L_G}{\sqrt{\kappa_{1,G}}}\|\barg_k-g_k\|+\frac{\sqrt{m} L_G\kappa_{\nabla f}}{\sqrt{\kappa_{1,G}}},
\end{align*}
where the last inequality follows from Assumption \ref{assump:4-1} and the definitions of $\blambda_k$ and $\barblambda_k$. On the event $\A_k\cap\B_k$, $\|\bar{\nabla}^2f_k - \nabla^2 f_k\|\leq\epsilon_h+\kappa_h\Delta_k$ and $\|\barg_k - g_k\|\leq\epsilon_g+\kappa_g\Delta_k^2$. Since $\Delta_k\leq\Delta_{\max}$, it follows that  
\begin{equation*}
\|\barH_k\| \leq \epsilon_h+ \kappa_h\Delta_{\max}+L_{\nabla f}+\frac{\sqrt{m} L_G}{\sqrt{\kappa_{1,G}}}(\epsilon_g+\kappa_g\Delta_{\max}^2+\kappa_{\nabla f}).
\end{equation*}\vskip -0.2cm
\hskip-0.5cm By setting $\kappa_B=\max\{1,\epsilon_h+\kappa_h\Delta_{\max}+L_{\nabla f}+\sqrt{m} L_G/\sqrt{\kappa_{1,G}}\cdot(\epsilon_g+\kappa_g\Delta_{\max}^2+\kappa_{\nabla f})\}$, we complete the proof.
\end{proof}

In the next lemma, we demonstrate that for the second-order stationarity, the difference between the true Lagrangian Hessian $\nabla_{\bx}^2\L_k$ and its estimate $\barH_k$ is bounded by $\mathcal{O}(\Delta_k)+\mathcal{O}(\epsilon_h+\epsilon_g)$ on the event $\A_k\cap\B_k$. Furthermore, the difference between $\tau_k^+\coloneqq |\min\{\tau_k,0\}|$ and its estimate $\bartau_k^+$ is bounded by the same quantity. This lemma ensures that $\bartau_k^+$ is an accurate estimate of $\tau_k^+$ provided that both the objective gradient and Hessian estimates are accurate.

\begin{lemma}\label{lemma:tau_accurate} 
	
Under Assumption \ref{assump:4-1} with $\alpha=1$, we have $\|\nabla_{\bx}^2\L_k-\barH_k\| \leq \epsilon_H + \kappa_H\Delta_k$ and $| \tau_k^+-\bartau_k^+|\leq \epsilon_H + \kappa_H\Delta_k$ on the event $\A_k\cap\B_k$, where $\kappa_H \coloneqq  \kappa_h+\frac{\sqrt{m}\kappa_gL_G\Delta_{\max}}{\sqrt{\kappa_{1,G}}}$ and $\epsilon_H \coloneqq \epsilon_h+\frac{\sqrt{m}L_G}{\sqrt{\kappa_{1,G}}}\epsilon_g$.
	
\end{lemma}

\vspace{-0.22cm}
\begin{proof}
We have 
\begin{align*}
\|\nabla_{\bx}^2\L_k-\barH_k\| & = \|\nabla^2f_k-\bar{\nabla}^2f_k+\sum_{i=1}^m(\blambda_k^i-\barblambda_k^i)\nabla^2c_k^i\| \\
& \leq \|\nabla^2f_k-\bar{\nabla}^2f_k\|+\|\barblambda_k-\blambda_k\|\big\{\sum_{i=1}^m\|\nabla^2c_k^i\|^2\}^{1/2} \\
& \leq \|\nabla^2f_k-\bar{\nabla}^2f_k\|+\frac{\sqrt{m} L_G}{\sqrt{\kappa_{1,G}}}\|g_k-\barg_k\| \quad\quad (\text{Assumption \ref{assump:4-1}})\\
& \leq\underbrace{\left(\kappa_h+\frac{\sqrt{m}\kappa_gL_G\Delta_{\max}}{\sqrt{\kappa_{1,G}}}\right)}_{\kappa_H}\Delta_k + \underbrace{\left(\epsilon_h+\frac{\sqrt{m}L_G}{\sqrt{\kappa_{1,G}}}\epsilon_g\right)}_{\epsilon_H},
\end{align*}
where the last inequality is due to the event $\A_k\cap\B_k$ and $\Delta_k\leq \Delta_{\max}$. Next, we show $|\tau_k-\bartau_k|\leq\kappa_H\Delta_k+ \epsilon_H$. Let $\barbzeta_k$ be a normalized eigenvector corresponding to $\bartau_k$, then
\begin{equation*}
\tau_k-\bartau_k \leq \barbzeta_k^T\left[Z_k^T(\nabla_{\bx}^2\L_k-\barH_k)Z_k\right]\barbzeta_k \leq\|\nabla_{\bx}^2\L_k-\barH_k\|\leq\kappa_H\Delta_k+ \epsilon_H.
\end{equation*}
Let $\bzeta_k$ be a normalized eigenvector corresponding to $\tau_k$, then
\begin{equation*}
\bartau_k-\tau_k \leq \bzeta_k^T\left[Z_k^T(\barH_k-\nabla_{\bx}^2\L_k)Z_k\right]\bzeta_k \leq\|\nabla_{\bx}^2\L_k-\barH_k\|\leq\kappa_H\Delta_k+ \epsilon_H.
\end{equation*}
Combining the above two displays, we have $|\tau_k-\bartau_k|\leq\kappa_H\Delta_k+ \epsilon_H$, which implies $|\tau_k^+-\bartau_k^+|\leq\kappa_H\Delta_k+ \epsilon_H$. We complete the proof.
\end{proof}

Let us define $\L_{\barmu_{k}}^{s_k} \coloneqq \L_{\barmu_{k}}(\bx_{s_k})$ and $\L_{\barmu_{k}}^{k} \coloneqq \L_{\barmu_{k}}(\bx_{k})$, where $\barmu_k$ is the merit parameter in the $k$-th iteration after the update \eqref{eq:threshold_Predk}. Here, $\bx_{s_k}=\bx_k+\Delta\bx_k$ if the SOC step is not performed and $\bx_{s_k}=\bx_k+\Delta\bx_k+\bd_k$ if the SOC step is performed. The following two lemmas examine the difference between the reduction in the merit function (i.e., $\L_{\barmu_{k}}^{s_k}-\L_{\barmu_{k}}^{k}$) and the predicted reduction $\text{Pred}_k$ in \eqref{def:Pred_k} among the first $T$ iterations, for any finite $T$. 
We first show that on the event $\A_k\cap\B_k$ and without the SOC step, the difference has an upper bound $\epsilon_g\Delta_k+\O(\Delta_k^2)$.

\begin{lemma}\label{lemma:diff_ared_pred_wo_corr_step}

Under Assumption \ref{assump:4-1}, for any $T\geq 1$ and any $0 \leq k \leq T-1$, on the event $\A_k\cap\B_k$ and suppose the SOC step is not performed, we have 
\begin{equation*}
\big|\L_{\barmu_{k}}^{s_k}-\L_{\barmu_{k}}^{k}-\text{Pred}_k\big| \leq \epsilon_g\Delta_k + \Upsilon_1\Delta_k^2,
\end{equation*}
where $\Upsilon_1=\kappa_g\max\{1,\Delta_{\max}\}+\frac{1}{2}(L_{\nabla f}+\kappa_B+\barmu_{T-1} L_G)$.
\end{lemma}

\begin{proof}
Since the SOC step is not performed, $\bx_{s_k}=\bx_k+\Delta\bx_k$. By the definition of $\ell_2$ merit function $\mL_{\mu}(\bx)$ and \eqref{def:Pred_k}, we have
\begin{equation*}
\big|\L_{\barmu_{k}}^{s_k}-\L_{\barmu_{k}}^{k}-\text{Pred}_k\big| = \left|f_{s_k}+\barmu_{k}\|c_{s_k}\|-f_k-\barg_k^T\Delta\bx_k-\frac{1}{2}\Delta\bx_k^T\barH_k\Delta\bx_k-\barmu_{k}\|c_k+G_k\Delta\bx_k\|\right|.
\end{equation*} 
By the Taylor expansion of $f(\bx)$ and the Lipschitz continuity of $\nabla f(\bx)$, we have
\begin{equation*}
f_{s_k}-f_k-\barg_k^T\Delta\bx_k \leq (g_k-\barg_k)^T\Delta\bx_k+\frac{1}{2}L_{\nabla f}\|\Delta\bx_k\|^2.
\end{equation*}
Similarly, we have
\begin{equation*}
\big|\|c_{s_k}\|-\|c_k+G_k\Delta\bx_k\|\big|\leq \|c_{s_k}-c_k-G_k\Delta\bx_k\|\leq \frac{1}{2}L_G\|\Delta\bx_k\|^2.
\end{equation*}
Recall that the monotonic updating scheme of the merit parameter implies $\barmu_{k}\leq\barmu_{T-1}$ for all $k\leq T-1$, and Assumption \ref{assump:4-1} and Lemma \ref{lemma:bounded_Hessian} imply $\|\barH_k\|\leq\kappa_B$ under $\A_k$. Combining the above two displays and using the Cauchy-Schwartz inequality lead to
\begin{equation*}
\big|\L_{\barmu_{k}}^{s_k}-\L_{\barmu_{k}}^{k}-\text{Pred}_k\big| \leq \|g_k-\bar{g}_k\|\|\Delta\bx_k\|+\frac{1}{2}(L_{\nabla f}+\kappa_B+\barmu_{T-1} L_G)\|\Delta\bx_k\|^2. 
\end{equation*}
On the event $\B_k$, we have $\|g_k-\bar{g}_k\|\leq \epsilon_g+\kappa_g\max\{1,\Delta_{\max}\}\Delta_k$. Combining $\|\Delta\bx_k\|\leq \Delta_k$ with the above display, we complete the proof.
\end{proof}

Next, we show that on the event $\A_k\cap\B_k$ with the SOC step, the difference between the two reductions is $\epsilon_g^{3/2}+\mathcal{O}(\Delta_k^2+\Delta_k^3)$. Compared to Lemma \ref{lemma:diff_ared_pred_wo_corr_step}, the error bound includes an additional constant $\epsilon_g^{3/2}$, while the linear term of $\Delta_k$ is removed. The proof is deferred to Appendix \ref{lemma:append:A.4}.

\begin{lemma}\label{lemma:diff_ared_pred_w_corr_step}
	
Under Assumption \ref{assump:4-1}, for any $T\geq 1$ and any $0 \leq k \leq T-1$, on the event $\A_k\cap\B_k$ and suppose the SOC step is performed, we have  
\begin{equation*}
\big|\L_{\barmu_{k}}^{s_k}-\L_{\barmu_{k}}^{k}-\text{Pred}_k\big|\leq \epsilon_g^{3/2}+\left(\frac{1}{2}\epsilon_h+\frac{\sqrt{m}L_G}{2\sqrt{\kappa_{1,G}}}\epsilon_g\right)\Delta_k^2+ \Upsilon_2\Delta_k^3,
\end{equation*}
where  
\begin{multline*}
\Upsilon_2 = \kappa_g+1+\frac{L_{\nabla^2 f}+\kappa_h}{2} +\frac{L_G^2\Delta_{\max}(0.5L_{\nabla f} + \sqrt{m}\barmu_{T-1} L_G)}{\kappa_{1,G}}\\
+ \frac{0.5\sqrt{m}L_{\nabla^2 c}(L_{\nabla f}\Delta_{\max} + \kappa_{\nabla f}) + 0.5 \sqrt{m}L_G(\kappa_g\Delta_{\max} + L_{\nabla f} + 2\barmu_{T-1} L_G) }{\sqrt{\kappa_{1,G}}}.
\end{multline*}
\end{lemma}

\begin{proof}
See Appendix \ref{lemma:append:A.4}.
\end{proof}

For our non-asymptotic analysis, we recall the definitions of first- and second-order stationarity points in Section \ref{sec:2_short} and define the stopping time based on $\epsilon$-stationary points.

\begin{definition}[Stopping time]\label{def:stopping_time}

For first-order stationarity ($\alpha=0$), the stopping time is defined as the index of the first iteration in which the true KKT residual enters the neighborhood of radius $\epsilon$: 
\begin{equation*}
T_{\epsilon} =\min \{k: \|\nabla\L_k\|\leq \epsilon\}.
\end{equation*}
For second-order stationarity ($\alpha=1$), the stopping time is defined as the index of the first iteration in which the true KKT residual and the true eigenvalue both enter the neighborhood of radius $\epsilon$: 
\begin{equation*}
T_{\epsilon} =\min \{k: \max(\|\nabla\L_k\|,\tau_k^+) \leq \epsilon\}.
\end{equation*}
\end{definition}

We define the stopping time directly using the deterministic KKT residual and the negative eigenvalue of the reduced Lagrangian Hessian. In comparison, \citep[{Definition 3.7}]{Berahas2025Sequential} defines the stopping time for first-order stationarity based on the reduction of the linear approximation of the merit function. 
In their Remark 3.8, the authors show that this definition is equivalent to our Definition \ref{def:stopping_time}, but with the threshold $\epsilon$ depending on the lower bound of an \textit{unknown} true merit parameter. Such a definition may not be direct, since even after the algorithm terminates under that definition, the true KKT residual is only controlled with an unknown upper bound.
In addition, the stopping time in \cite{Cao2023First} depends on the Lipschitz constants and the irreducible noises.

Similar to \cite{Berahas2025Sequential, Cao2023First}, we require $\epsilon$ to exceed a threshold determined by irreducible noises $\epsilon_f$, $\epsilon_g$, and $\epsilon_h$. Specifically, we impose the following assumption.

\begin{assumption}\label{assump:epsilon}

For the stationarity level $\epsilon$, we assume
\begin{align*}
\text{If }\alpha=0:\quad \epsilon & > \frac{1}{\Upsilon_3}\sqrt{\frac{8\gamma^4\barmu_{T-1}\epsilon_f}{\kappa_{fcd}\eta^3\barmu_0(p-1/2)}}+\frac{\Upsilon_4}{\Upsilon_3}\epsilon_g,\\
\text{If } \alpha=1:\quad \epsilon & > \frac{1}{\Upsilon_5}\sqrt[3]{\frac{2\barmu_{T-1}\gamma^6\max\{\Delta_{\max},1\}(4\epsilon_f+\epsilon_g^{3/2})}{\kappa_{fcd}\eta^3\barmu_0(p-1/2)}}+ \frac{\Upsilon_6}{\Upsilon_5}\epsilon_g+\frac{\Upsilon_7}{\Upsilon_5}\epsilon_h,
\end{align*}
where $\Upsilon_3, \Upsilon_4, \Upsilon_5, \Upsilon_6,\Upsilon_7$ are constants independent of irreducible noise levels and are defined in Lemmas \ref{lemma:guarantee_succ_step_KKT} and  \ref{lemma:guarantee_succ_step_eigen}. Furthermore, the failure oracle probabilities $p_h,p_g,p_f$ are small enough such that $1-p_h-p_g-2p_f\eqqcolon p\in(1/2,1)$. 
\end{assumption}

Here, $p$ represents the lower bound of probability that $\A_k\cap \B_k \cap \C_k \cap \C_k'$ holds, where $ \C_k'$ denotes the event that accurate $\barf_k,\barf_{s_k}$ are regenerated after computing the SOC step. For notational consistency, we assume $\C_k'$ holds when the SOC step is not performed. As demonstrated in Assumption \ref{assump:epsilon}, for first-order stationarity, we need $\epsilon \geq \mathcal{O}(\sqrt{\epsilon_f}+\epsilon_g)$; for second-order stationarity, we need $\epsilon \geq \mathcal{O}(\sqrt[3]{\epsilon_f}+\sqrt{\epsilon_g}+\epsilon_h) $. These orders matches the ones in \cite{Berahas2025Sequential,Cao2023First}.

\begin{remark}\label{rem:4.8}
The dependence of $\epsilon$ on $\epsilon_f$, $\epsilon_g$, and $\epsilon_h$ is standard in existing literature \citep{Cao2023First,Berahas2025Sequential,Sun2023trust}. In the presence of biased, irreducible noise in objective estimation, the iterates can only be guaranteed to converge to an $\epsilon$-neighborhood of a stationary point, where $\epsilon$ is lower bounded by a quantity determined by the magnitude of the bias. This dependence arises for the following reasons.  
First, the first- and second-order stationarity conditions are characterized by the KKT residual and the negative curvature, both of which rely on gradient and Hessian information. When their estimates contain biased irreducible noise, different true gradients and Hessians are not identifiable whenever their differences lie within $\epsilon_g$ and $\epsilon_h$, respectively.
Second, the algorithm employs a merit function to adaptively update the trust-region radius, which depends on the objective value that can only be estimated up to an irreducible noise $\epsilon_f$. This adaptive mechanism leads to the dependence of $\epsilon$ on $\epsilon_f$. We note that for non-adaptive algorithms, where the stepsize or trust-region radius sequences are prespecified and thus do not use objective value information, the dependence of $\epsilon$ on $\epsilon_f$ can be suppressed.
\end{remark}

In the next two lemmas, we show that before the algorithm terminates, if $\A_k\cap \B_k \cap \C_k \cap \C_k'$ holds and the trust-region radius is sufficiently small, then both two conditions in \eqref{nequ:3} hold. We first consider $\alpha=0$.

\begin{lemma}\label{lemma:guarantee_succ_step_KKT}

Under Assumptions \ref{assump:4-1}, \ref{assump:epsilon}, and the event $\A_k\cap \B_k \cap \C_k \cap \C_k'$, for $\alpha=0$ and $k<T_{\epsilon}$, if  
\begin{equation}\label{delta:lemma_guarantee_succ_step_KKT}
\Delta_k \leq  \Upsilon_3\|\nabla\L_k\|-\Upsilon_4 \epsilon_g
\end{equation}
with $\Upsilon_3$, $\Upsilon_4$ given by 
\begin{equation*}
\Upsilon_3 = \frac{\kappa_{fcd}(1-\eta)}{4\kappa_f+\kappa_g+2\Upsilon_1+\kappa_B} \quad \text{and}\quad
\Upsilon_4  = \frac{\kappa_{fcd}(1-\eta)+2}{4\kappa_f+\kappa_g+2\Upsilon_1+\kappa_B},
\end{equation*}
then (\ref{nequ:3}a) and  (\ref{nequ:3}b) both hold.
	
\end{lemma}

\begin{proof}
We first note that for $\alpha=0$, Assumption \ref{assump:epsilon} implies $\epsilon>\Upsilon_4/\Upsilon_3\cdot\epsilon_g$. Moreover, when $k<T_\epsilon$, Definition \ref{def:stopping_time} ensures that $\|\nabla \L_k\|>\epsilon$. Combining these two facts yields $\|\nabla \L_k\|> \Upsilon_4/\Upsilon_3 \cdot \epsilon_g$, which can be rearranged to obtain $\Upsilon_3\|\nabla \L_k\|> \Upsilon_4\epsilon_g$. Therefore, the right-hand side of \eqref{delta:lemma_guarantee_succ_step_KKT} is positive and the inequality is well defined.
We first prove (\ref{nequ:3}a) holds. By the algorithm design, we have $\vartheta_\alpha = 2\epsilon_f$ for $\alpha=0$. Thus, we only need to show that $\text{Ared}_k-2\epsilon_f\leq \eta\text{Pred}_k$ can be satisfied by \eqref{delta:lemma_guarantee_succ_step_KKT} (recalling that $\text{Pred}_k<0$). Since no 
SOC step will be performed when $\alpha =0$, $\bx_{s_k}=\bx_k+\Delta\bx_k$ and we have
\begin{align*}
\text{Ared}_k - 2\epsilon_f & = \bar{\L}_{\barmu_{k}}^{s_k}-\bar{\L}_{\barmu_{k}}^k - 2\epsilon_f = \bar{\L}_{\barmu_{k}}^{s_k}- \L_{\barmu_{k}}^{s_k} +\L_{\barmu_{k}}^{s_k} - \L_{\barmu_{k}}^k +\L_{\barmu_{k}}^k-\bar{\L}_{\barmu_{k}}^k - 2\epsilon_f\\
& = \bar{f}_{s_k} - f_{s_k} +\L_{\barmu_{k}}^{s_k} - \L_{\barmu_{k}}^k +f_k-\bar{f}_k - 2\epsilon_f\\
& \leq  |\bar{f}_{s_k} - f_{s_k}| + |f_k-\bar{f}_k| +\text{Pred}_k +\epsilon_g\Delta_k+\Upsilon_1\Delta_k^2 - 2\epsilon_f. \qquad( \text{Lemma }\ref{lemma:diff_ared_pred_wo_corr_step})
\end{align*}
Under $\C_k$, when $\alpha=0$ we have $ |f_k-\bar{f}_k|+ |\bar{f}_{s_k} - f_{s_k}|\leq 2\epsilon_f+2\kappa_f\Delta_k^2$. Since 
\begin{equation*}
\text{Pred}_k\leq -\frac{\kappa_{fcd}}{2}\|\bnabla\L_k\|\min\left\{\Delta_k,\frac{\|\bnabla\L_k\|}{\|\barH_k\|}\right\} \stackrel{\eqref{delta:lemma_guarantee_succ_step_KKT}}{=}-\frac{\kappa_{fcd}}{2}\|\bnabla\L_k\| \Delta_k,
\end{equation*}
we only need to show 
\begin{align}\label{neq:5}
2\kappa_f\Delta_k^2 +\epsilon_g\Delta_k+\Upsilon_1\Delta_k^2 \leq \frac{\kappa_{fcd}}{2}(1-\eta)\|\bnabla\L_k\| \Delta_k.
\end{align}
To this end, we note that
\begin{align*}
\eqref{delta:lemma_guarantee_succ_step_KKT}\Rightarrow & \{(4\kappa_f+\kappa_g)+2\Upsilon_1+\kappa_B\}\Delta_k\leq \kappa_{fcd}(1-\eta)\|\nabla\L_k\|-\left\{\kappa_{fcd}(1-\eta)+2\right\}\epsilon_g\\
\Rightarrow&\{(4\kappa_f+\kappa_{fcd}(1-\eta)\kappa_g)+2\Upsilon_1\}\Delta_k \leq \kappa_{fcd}(1-\eta)\|\nabla\L_k\|-\left\{\kappa_{fcd}(1-\eta)+2\right\}\epsilon_g,
\end{align*}
since $\kappa_{fcd}(1-\eta)<1$. Rearranging the terms, we have
\begin{equation}\label{neq:7}
(2\kappa_f +\Upsilon_1) \Delta_k \leq  \frac{\kappa_{fcd}}{2}(1-\eta)(\|\nabla\L_k\|-\kappa_g\Delta_k - \epsilon_g) -\epsilon_g.
\end{equation}
Under $\B_k$, when $\alpha=0$ we have $\|\bnabla\L_k\|\geq \|\nabla\L_k\|-\kappa_g\Delta_k-\epsilon_g$. 
This, together with \eqref{neq:7}, implies
\begin{equation*}
(2\kappa_f +\Upsilon_1)\Delta_k\leq \frac{\kappa_{fcd}}{2}(1-\eta)\|\bnabla\L_k\| - \epsilon_g.
\end{equation*}
Multiplying $\Delta_k$ and rearranging the terms, we have shown \eqref{neq:5}, thus proved (\ref{nequ:3}a). Next, we prove (\ref{nequ:3}b). When $\alpha=0$, we have $\max\{1,\|\barH_k\|\}\leq\kappa_B$ under $\A_k$. Therefore it suffices to show $\|\bnabla\L_k\|\geq \eta \kappa_B\Delta_k$. Recalling $\kappa_{fcd}\leq 1, \eta<1$, we have
\begin{multline*}
(\eta\kappa_B+ \kappa_g)\Delta_k \leq (4\kappa_f+\kappa_g+2\Upsilon_1+\kappa_B) \Delta_k  \\
\stackrel{\eqref{delta:lemma_guarantee_succ_step_KKT}}{\leq} \kappa_{fcd}(1-\eta)\|\nabla\L_k\|-(\kappa_{fcd}(1-\eta)+2)\epsilon_g \leq \|\nabla\L_k\| - \epsilon_g.
\end{multline*}
Combining the above inequality with the relation $\|\bnabla\L_k\| \geq \|\nabla\L_k\| - \kappa_g\Delta_k - \epsilon_g$, we obtain $\|\bnabla\L_k\|\geq \eta \kappa_B\Delta_k$, thus  (\ref{nequ:3}b) holds. We complete the proof.
\end{proof}

In the next lemma, we show the same result holds for $\alpha=1$. By Assumption \ref{assump:epsilon} and Definition \ref{def:stopping_time}, when $k<T_\epsilon$, we know $\Upsilon_5 \max\{ \tau_k^+,\|\nabla\L_k\|\} - \Upsilon_6\epsilon_g -\Upsilon_7\epsilon_h> \Upsilon_5\epsilon- \Upsilon_6\epsilon_g -\Upsilon_7\epsilon_h>0$; i.e., the right-hand side of \eqref{delta:lemma_guarantee_succ_step_eigen} is positive. Thus the inequality is well-defined. The proof of Lemma \ref{lemma:guarantee_succ_step_eigen} is deferred to Appendix \ref{lemma:append:A.6}.

\begin{lemma}\label{lemma:guarantee_succ_step_eigen}
	
Under Assumptions \ref{assump:4-1}, \ref{assump:epsilon}, and the event $\A_k\cap\B_k\cap\C_k\cap\C_k'$, for $\alpha=1$ and $k<T_{\epsilon}$, if  
\begin{equation}\label{delta:lemma_guarantee_succ_step_eigen}
\Delta_k\leq \Upsilon_5 \max\{ \tau_k^+,\|\nabla\L_k\|\} - \Upsilon_6\epsilon_g -\Upsilon_7\epsilon_h
\end{equation} 
with $\Upsilon_5$, $\Upsilon_6$, $\Upsilon_7$ given by 
\begin{align}
\Upsilon_5 &=  \frac{(1-\eta)\kappa_{fcd}\min\{1,r\}}{(4\kappa_f +\kappa_g) \max\{\Delta_{\max},1\}+\kappa_B+2\Upsilon_1+2\Upsilon_2+(\kappa_H+\eta)(1-\eta)\kappa_{fcd}\min\{1,r\}},\\ 
\Upsilon_6 &=  \frac{\kappa_{fcd}(1-\eta)+2+2\sqrt{m}L_G/\sqrt{\kappa_{1,G}}}{(4\kappa_f +\kappa_g)\max\{\Delta_{\max},1\}+\kappa_B+2\Upsilon_1+2\Upsilon_2+(\kappa_H+\eta)(1-\eta)\kappa_{fcd}\min\{1,r\}},\\ 
\Upsilon_7 & =  \frac{2}{4\kappa_f\max\{\Delta_{\max},1\}+2\Upsilon_1+2\Upsilon_2+(\kappa_H+\eta)(1-\eta)\kappa_{fcd}\min\{1,r\}},
\end{align}
then (\ref{nequ:3}a) and  (\ref{nequ:3}b)  both hold. Here, $\kappa_H$ is derived in Lemma \ref{lemma:tau_accurate}, $\Upsilon_1$ is defined in Lemma \ref{lemma:diff_ared_pred_wo_corr_step}, and $\Upsilon_2$ is defined in Lemma \ref{lemma:diff_ared_pred_w_corr_step}.
\end{lemma}

\begin{proof}
See Appendix \ref{lemma:append:A.6}.
\end{proof}

In the next lemma, we consider the reduction in the stochastic merit function in iterations where (\ref{nequ:3}a) and (\ref{nequ:3}b) both hold. For brevity, we define the function $h_\alpha(\cdot)$ as 
\begin{equation*}
h_\alpha(\Delta) = \left\{ \begin{aligned}
&\frac{\kappa_{fcd}}{2}\eta^3\Delta^2 \qquad\qquad\qquad \text{for }\alpha=0,\\
&\frac{\kappa_{fcd}}{2\max\{\Delta_{\max},1\}}\eta^3\Delta^3\quad \text{for }\alpha=1.
\end{aligned}\right.
\end{equation*}

\begin{lemma}\label{lemma:suff_iter}

Under Assumption \ref{assump:4-1} and the event $ \A_k\cap \B_k\cap \C_k \cap \C_k'$, if (\ref{nequ:3}a) and (\ref{nequ:3}b) both hold in the $k$-th iteration, then we have
\begin{equation*}
\bar{\L}_{\barmu_k}^{k+1}-\bar{\L}_{\barmu_k}^{k} \leq - h_\alpha(\Delta_k) + \vartheta_\alpha.
\end{equation*}
\end{lemma}

\begin{proof}
Since (\ref{nequ:3}a) holds, we have $    \bar{\L}_{\barmu_k}^{k+1}-\bar{\L}_{\barmu_k}^{k} \leq \eta \text{Pred}_k + \vartheta_\alpha$. When $\alpha=0$, we have
\begin{equation*}
\text{Pred}_k\leq -\frac{\kappa_{fcd}}{2}\|\bnabla\L_k\|\min\left\{\Delta_k,\frac{\|\bnabla\L_k\|}{\|\barH_k\|}\right\} \stackrel{(\ref{nequ:3}\text{b})}{\leq} -\frac{\kappa_{fcd}}{2}\eta^2\Delta_k^2.
\end{equation*}
When $\alpha=1$, we similarly have
\begin{align*}
\text{Pred}_k & \leq -\frac{\kappa_{fcd}}{2} \max\left\{\|\bar{\nabla}\L_k\|\min\left\{\Delta_k,\frac{\|\bar{\nabla}\L_k\|}{\|\barH_k\|}\right\},\bartau_k^+\Delta_k\left(\Delta_k+\|c_k\|\right)\right\} \\
& \stackrel{\mathclap{(\ref{nequ:3}\text{b})}}{\leq} -\frac{\kappa_{fcd}\eta^3}{2\max\{\Delta_{\max},1\}}\Delta_k^3. 
\end{align*}
We complete the proof by combining the above two displays.
\end{proof}

To further our analysis, we manipulate the lower bound of $\epsilon$ in Assumption \ref{assump:epsilon} and define the following threshold
\begin{equation}\label{equ:delta_hat}
\hat{\Delta} =\left\{ \begin{aligned}
& \Upsilon_3\epsilon - \Upsilon_4\epsilon_g \quad\quad\quad\quad\; \text{for }\alpha=0,\\
& \Upsilon_5 \epsilon - \Upsilon_6\epsilon_g - \Upsilon_7\epsilon_h \quad\text{for }\alpha=1.
\end{aligned} 
\right.
\end{equation}
Note that for $k<T_\epsilon$, we have $\|\nabla \L_k\|>\epsilon$ when $\alpha=0$ and $\max\{\|\nabla\L_k\|,\tau_k^+\}>\epsilon$ when $\alpha=1$. Therefore, $\hat{\Delta}$ is smaller than the upper bounds of conditions in Lemmas \ref{lemma:guarantee_succ_step_KKT} and \ref{lemma:guarantee_succ_step_eigen}. This implies that when $\Delta_k\leq \hat{\Delta}$ and $\A_k\cap\B_k\cap\C_k\cap \C_k'$ holds, conclusions in both lemmas will apply. Without loss of generality, we suppose $\hat{\Delta} \leq \Delta_{\max}$. Otherwise, we can always increase $\Delta_{\max}$ to ensure this condition. In fact, when $\alpha = 0$, $\hat{\Delta}$ is independent of $\Delta_{\max}$; thus, a large $\Delta_{\max}$ trivially ensures $\hat{\Delta} \leq \Delta_{\max}$. When $\alpha = 1$, by Assumption \ref{assump:epsilon}, $\hat{\Delta}$ grows at a rate of $\sqrt[3]{\Delta_{\max}}$; thus, increasing $\Delta_{\max}$ still ensures $\hat{\Delta} \leq \Delta_{\max}$.

The trust-region radius $\Delta_k$, $k\geq 0$ always has the formula $\Delta_k = \gamma^l\Delta_0$ with $\Delta_0=\Delta_{\max}$, where $l$ is a non-positive integer. Therefore, we define the following threshold:
\begin{equation*}
\hat{\Delta}' = \max_{l \in \mathbb{Z}}\; \{\gamma^l\Delta_0:  \gamma^l\Delta_0\leq \gamma^{-1}\hat{\Delta} \}.
\end{equation*}
The threshold $\hat{\Delta}'$ is the largest trust-region radius that is less or equal to $ \gamma^{-1}\hat{\Delta}$.

\begin{remark}
We require $\hat{\Delta}'\leq \gamma^{-1}\hat{\Delta}$ instead of $\hat{\Delta}'\leq \hat{\Delta}$. This additional $\gamma^{-1}$ ensures that when $\Delta_k\leq \hat{\Delta}'$ and the trust-region radius increases in the $k$-th iteration, we will have $\Delta_{k+1} = \gamma\Delta_k\leq \Delta_{\max}$. This property is essential to our analysis, as utilized in Lemma \ref{lemma: count_iter} (c).
\end{remark}

To derive the complexity bounds, we categorize iterations $k=0,1,\cdots,T-1$ into different classes by defining the following indicator variables:
\begin{itemize}
\item \textbf{$I_k$ for accurate iteration}: 	We set the indicator $I_k=1$ if $\A_k\cap\B_k\cap\C_k\cap\C_k' $ holds and call the $k$-th iteration \textit{accurate}. Otherwise we set $I_k=0$ and call the $k$-th iteration \textit{inaccurate}. 
\item \textbf{$\Theta_k$ for sufficient iteration}: 	We set the indicator $\Theta_k=1$ if (\ref{nequ:3}a) and (\ref{nequ:3}b) both hold and call the $k$-th iteration \textit{sufficient}. Otherwise we set $\Theta_k=0$ and call the $k$-th iteration \textit{insufficient}.
\item \textbf{$\Lambda_k$ for large iteration}: We set the indicator $\Lambda_k=1$ if  $\min\{\Delta_k,\Delta_{k+1}\}\geq \hat{\Delta}'$ and call the $k$-th iteration \textit{large}. We set $\Lambda_k=0$ if $\max\{\Delta_k,\Delta_{k+1}\}\leq \hat{\Delta}'$ and call the $k$-th iteration \textit{small}.
\end{itemize}

Note that $\min\{\Delta_k, \Delta_{k+1}\} \geq \hat{\Delta}'$ and $\max\{\Delta_k, \Delta_{k+1}\} \leq \hat{\Delta}'$ are mutually disjoint. To see why, consider the only possible case where $\min\{\Delta_k, \Delta_{k+1}\} = \hat{\Delta}' = \max\{\Delta_k, \Delta_{k+1}\}$, which implies $\Delta_k = \Delta_{k+1} = \hat{\Delta}'$. However, by the definition of $\hat{\Delta}'$, we have $\hat{\Delta}' \leq \gamma^{-1} \Delta_{\max}$. According to the algorithm design, the trust-region radius updates as either $\Delta_{k+1} = \gamma \Delta_k$ or $\Delta_{k+1} = \Delta_k / \gamma$, leading to a contradiction.

In the next lemma, we characterize the relations of iterations in different classes.

\begin{lemma}\label{lemma: count_iter}

Under Assumptions \ref{assump:4-1}, \ref{assump:epsilon}, for any $T\geq 1$, we have
\begin{enumerate}[label={(\alph*)},topsep=4pt, wide,labelindent=0pt]
\setlength\itemsep{0.0em}
\item $\sum_{k=0}^{T-1}\Lambda_k\Theta_k \geq \sum_{k=0}^{T-1}\Lambda_k(1-\Theta_k) - \log_\gamma (\Delta_0/\hat{\Delta}')$;
\item $\sum_{k=0}^{T-1}\Lambda_k\Theta_k \geq 0.5\sum_{k=0}^{T-1}\Lambda_k - 0.5 \log_\gamma (\Delta_0/\hat{\Delta}')$;
\item $\sum_{k=0}^{T-1}(1-\Lambda_k)(1-\Theta_k) \geq \sum_{k=0}^{T-1}(1-\Lambda_k)\Theta_k$;
\item $\sum_{k=0}^{T}(1-\Lambda_k)I_k \leq \sum_{k=0}^{T}(1-\Lambda_k)(1- I_k )$.
\end{enumerate}
	
\end{lemma}

\begin{proof}
See Appendix \ref{lemma:append:A.8}.
\end{proof}

\begin{remark} We discuss the insights behind the above lemma.
 
\noindent $\bullet$ In Lemma \ref{lemma: count_iter} (a), we demonstrate that among the first $T$ iterations, the number of iterations that are large and sufficient (i.e., $\Lambda_k\Theta_k =1$) is no less that the number of iterations that are large and insufficient (i.e., $\Lambda_k(1-\Theta_k)=1$), subtracted by a constant $ \log_\gamma (\Delta_0/\hat{\Delta}')$. This conclusion is intuitive: when an iteration is large, its trust-region radius is lower bounded by a constant $\hat{\Delta}'$. Thus, to maintain this lower bound, among all large iterations, the number of sufficient iterations (in which the radius increases) must be no less than the number of insufficient iterations (in which the radius decreases), except for a constant $ \log_\gamma (\Delta_0/\hat{\Delta}')$ representing the smallest number of iterations to reduce the radius from $\Delta_0$ to $\hat{\Delta}'$.

\noindent $\bullet$ Lemma \ref{lemma: count_iter} (b) follows directly by rearranging the terms in (a).

\noindent $\bullet$ In Lemma \ref{lemma: count_iter} (c), we show that the number of iterations that are small and insufficient (i.e., $ (1-\Lambda_k)(1-\Theta_k)=1$) is no less than the number of iterations that are small and sufficient (i.e., $(1-\Lambda_k)\Theta_k=1 $). In fact, when an iteration is small, its trust-region radius is upper bounded by $\hat{\Delta}'$. Thus, to maintain the upper bound, among all small iterations, the number of iterations that are sufficient (in which the radius increases) must be no more than the number of iterations that are insufficient (in which the radius decreases).

\noindent $\bullet$ In Lemma \ref{lemma: count_iter} (d), we show that the number of iterations that are small and accurate (i.e., $ (1-\Lambda_k)I_k=1$) is no greater than the number of iterations that are small and inaccurate (i.e., $(1-\Lambda_k)(1-I_k) = 1 $). The intuition is also straightforward: when an iteration is small and accurate, Lemmas \ref{lemma:guarantee_succ_step_KKT} and \ref{lemma:guarantee_succ_step_eigen} guarantee that the step will be sufficient, leading to an increase in the trust-region radius. Consequently, to maintain the upper bound $\hat{\Delta}'$, the number of small and accurate  iterations (in which the radius increases) must not exceed the number of small and inaccurate iterations (only in this case the radius may decrease).

\end{remark}

The next lemma follows from the  conclusion that $P(I_k=1\mid \F_{k-1/2})\geq p$ (recall $p=1-p_h-p_g-2p_f$)
and the Azuma-Hoeffding inequality (cf. Lemma \ref{Azuma-Hoeffding} in Appendix \ref{append:F}), which establishes a probabilistic bound for the number of accurate iterations.

\begin{lemma}\label{lemma:I_k}
For all $T\geq 1$, $\hat{p} \in[0,p)$ with $p\in(1/2,1)$, and both $\alpha=0$ and $\alpha=1$, we have 
\begin{equation*}
P\left[\sum_{k=0}^{T-1}I_k<\hat{p}\; T\right]\leq \exp\left\{-\frac{(\hat{p}-p)^2}{2p^2}T\right\}.
\end{equation*}
\end{lemma}

\begin{proof}
It follows from $P(I_k=1\mid \F_{k-1/2})\geq p$ that $\{\sum_{k=0}^{T-1}I_k- pT\}$ is a submartingale. Moreover, 
\begin{equation*}
\left| \left(\sum_{k=0}^{T}I_k- p(T+1)\right) -\left( \sum_{k=0}^{T-1}I_k- pT\right) \right| = \left| I_T - p\right| \leq \max\{|1-p|, |0-p| \} = p.
\end{equation*}
The last equality follows from $p >1/2$. By the Azuma-Hoeffding inequality (see Lemma \ref{Azuma-Hoeffding}) with $X_T= \sum_{k=0}^{T-1}I_k - pT$ and $X_0=0$, for any $T\geq 1$ and $c\geq 0$, we have
\begin{equation*}
P\left[\sum_{k=0}^{T-1}I_k - pT\leq -c \right] \leq \exp\left\{ -\frac{c^2}{2Tp^2}\right\}.
\end{equation*}
Letting $c = (p-\hat{p})T$, we complete the proof.
\end{proof}

In the following lemma, we demonstrate that prior to algorithm termination (i.e., $T_\epsilon>T-1$), if the number of accurate iterations (i.e., $I_k=1$) is sufficiently large, then the number of iterations that are simultaneously sufficient, accurate, and large (i.e., $\Theta_kI_k\Lambda_k=1$) cannot be too small.

\begin{lemma}\label{lemma: prob=0}
For all $T\geq 1, p\in(1/2,1)$, and both $\alpha=0$ and $\alpha=1$, we have
\begin{equation*}
P\left\{ T_{\epsilon}> T-1, \sum_{k=0}^{T-1}I_k \geq pT, \sum_{k=0}^{T-1}\Theta_kI_k\Lambda_k< \left(p-\frac{1}{2}\right) T- \frac{1}{2}\log_\gamma\left(\frac{\Delta_0}{\hat{\Delta}}\right)-1\right\}=0.
\end{equation*}
\end{lemma}

\begin{proof}
It suffices to show that $T_{\epsilon}> T-1$ and $\sum_{k=0}^{T-1}I_k \geq pT$ imply $\sum_{k=0}^{T-1}\Theta_kI_k\Lambda_k \geq  \left(p-\frac{1}{2}\right) T- \frac{1}{2}\log_\gamma\left(\frac{\Delta_0}{\hat{\Delta}}\right)-1$.
For $k=0,1,\cdots,T-1$, we define the following sets and use $|\cdot|$ to denote their cardinality:
\begin{itemize}
\item $ \A_{\L}=\{k: I_k\Lambda_k=1\}$ as the set of iterations that are large and accurate; 
\item $\I_{\L} = \{k  : (1-I_k)\Lambda_k =1 \}$ as the set of iterations that are large and inaccurate; 
\item $ \A_{\mathcal{S}}=\{k : I_k(1-\Lambda_k)=1\}$ as the set of iterations that are small and accurate; 
\item $\I_{\mathcal{S}} = \{k  : (1-I_k)(1-\Lambda_k) =1 \}$ as the set of iterations that are small and  inaccurate.
\end{itemize}
This classification is valid, noticing that $| \I_{\mathcal{S}} | + | \A_{\mathcal{S}} | +| \I_{\L} | + | \A_{\L} |=T$. It follows from Lemma \ref{lemma: count_iter} (b) and the observation $ \sum_{k=0}^{T-1}\Lambda_k = | \I_{\L} |+ | \A_{\L} |$ that
\begin{equation}\label{nequ:22}
\sum_{k=0}^{T-1}\Lambda_k\Theta_k \geq  \frac{1}{2}( | \I_{\L} | + | \A_{\L} | ) - \frac{1}{2}\log_\gamma \left(\frac{\Delta_0}{\hat{\Delta}'}\right).
\end{equation}
Since $\sum_{k=0}^{T-1}I_k \geq pT$, by the fact that an iteration is either large or small, we have
\begin{equation}\label{nequ:23}
| \I_{\L} | + | \I_{\mathcal{S}} | = \sum_{k=0}^{T-1}(1-I_k)= T-\sum_{k=0}^{T-1}I_k \leq T-pT.
\end{equation}
Moreover, it follows from Lemma \ref{lemma: count_iter} (d) that \begin{equation}\label{nequ:24}
| \A_{\mathcal{S}} | = \sum_{k=0}^{T-1}(1-\Lambda_k)I_k \leq \sum_{k=0}^{T-1}(1-\Lambda_k)(1- I_k ) = | \I_{\mathcal{S}} |  .
\end{equation}
Combining the above displays, we have
\begin{align*}
& \sum_{k=0}^{T-1}\Lambda_k\Theta_kI_k   = \sum_{k=0}^{T-1}\Lambda_k\Theta_k - \sum_{k=0}^{T-1}\Lambda_k\Theta_k( 1 - I_k )  \stackrel{\mathclap{\eqref{nequ:22}}}{\geq} \; \; \frac{1}{2}( | \I_{\L} | + | \A_{\L} | ) - \frac{1}{2} \log_\gamma \left(\frac{\Delta_0}{\hat{\Delta}'}\right)- \sum_{k=0}^{T-1}\Lambda_k ( 1 - I_k ) \\
& =  \frac{1}{2}( | \I_{\L} | + | \A_{\L} | ) - \frac{1}{2} \log_\gamma \left(\frac{\Delta_0}{\hat{\Delta}'}\right) - | \I_{\L} |  = \frac{1}{2} ( T- | \I_{\mathcal{S}} | - | \A_{\mathcal{S}} | )- \frac{1}{2} \log_\gamma \left(\frac{\Delta_0}{\hat{\Delta}'}\right) - | \I_{\L} | \\
& \stackrel{\mathclap{\eqref{nequ:23}}}{\geq} \; \; \frac{1}{2} ( T- | \I_{\mathcal{S}} | - | \A_{\mathcal{S}} | )- \frac{1}{2} \log_\gamma \left(\frac{\Delta_0}{\hat{\Delta}'}\right) +  | \I_{\mathcal{S}} | -T +pT\\
& = \frac{1}{2}(| \I_{\mathcal{S}} | - | \A_{\mathcal{S}} | ) + \left(p-\frac{1}{2}\right) T- \frac{1}{2} \log_\gamma \left(\frac{\Delta_0}{\hat{\Delta}'}\right)  \stackrel{\mathclap{\eqref{nequ:24}}}{\geq} \; \; \left(p-\frac{1}{2}\right) T- \frac{1}{2} \log_\gamma \left(\frac{\Delta_0}{\hat{\Delta}'}\right) \\
& \geq \left(p-\frac{1}{2}\right) T- \frac{1}{2}\log_\gamma\left(\frac{\Delta_0}{\hat{\Delta}}\right)-1.
\end{align*}
In the last inequality, we use the relation $\gamma^{-2}\hat{\Delta}\leq \hat{\Delta}' $, which implies $\Delta_0/\hat{\Delta}' \leq \Delta_0/(\gamma^{-2}\hat{\Delta})$ and thus
\begin{equation*}
\log_\gamma\left(\frac{\Delta_0}{\hat{\Delta}'}\right) \leq \log_\gamma\left(\frac{\Delta_0}{\gamma^{-2}\hat{\Delta}}\right)  = \log_\gamma\left(\frac{\Delta_0}{\hat{\Delta}}\right) +2.
\end{equation*}
We complete the proof.
\end{proof}

Until now, we have presented all the preparatory lemmas that hold for both heavy-tailed and sub-exponential zeroth-order oracles. In the next two subsections, we establish the first- and second-order complexity bounds for these two oracles separately.

\subsection{Complexity bounds with heavy-tailed zeroth-order oracle}\label{sec:4_heavytail}

In this section, we establish the first and second-order high-probability complexity bounds when the noise of the objective value estimate follows the probabilistic heavy-tailed zeroth-order oracle (i.e., conditions \textbf{(i)}, \textbf{(ii)}, and \textbf{(iii.1)} of Definition \ref{def:heavy-tailed oracle}). We will show that$\quad$
\begin{equation*}
P\{T_{\epsilon}\leq T-1 \} \geq \text{ a probability that converges to 1 as }T \rightarrow\infty.
\end{equation*}

In the next theorem, we establish a generic iteration complexity bound of our method, which unifies the analysis of first and second-order stationarity. The corresponding complexity bounds are then derived in subsequent corollaries.

\begin{theorem}\label{thm: 1st}
Under Assumptions \ref{assump:4-1}, \ref{assump:epsilon}, and the probabilistic heavy-tailed zeroth-order oracle (i.e., bounded $1+\delta$ moment), for any $s\geq 0$, when  
\begin{equation}\label{thm:p}
\hat{p}\in\left(\frac{1}{2}+\frac{\barmu_{T-1}(2\epsilon_f+\vartheta_\alpha+2s)}{\barmu_0h_\alpha(\gamma^{-2}\hat{\Delta})},p\right)
\end{equation}
and
\begin{equation}\label{thm:T}
T>  \left(\frac{\hat{p}-0.5}{\barmu_{T-1}}-\frac{2\epsilon_f+\vartheta_\alpha+2s}{\barmu_0h_\alpha(\gamma^{-2}\hat{\Delta})}\right)^{-1}\cdot\left(\frac{f(\bx_0)+\barmu_0\|c(\bx_0)\|-f_{\inf}}{\barmu_0 h_\alpha(\gamma^{-2}\hat{\Delta})}+\frac{ \log_\gamma (\Delta_0/\hat{\Delta})+2}{2\barmu_{T-1}}\right),
\end{equation}
we have  
\begin{multline}\label{prob_heavy_tailed}
P\{T_\epsilon\leq T-1\} \geq  1 - \exp\rbr{-\frac{(p-\hat{p})^2}{2p^2}T} \\ 
- 2 \exp \left( - \frac{2(\epsilon_f-\tilde{\epsilon}_f+s)^2}{(3+\delta)^2 e^{1+\delta} \Upsilon_f^{\frac{2}{1+\delta}}} \cdot T \right)
- \frac{ 3^{2+\delta} \Upsilon_f }{(\epsilon_f -\tilde{\epsilon}_f + s)^{1+\delta}} \cdot T^{-\delta}.
\end{multline}	

\end{theorem}

\begin{proof}
See Appendix \ref{append:B}.
\end{proof}

In the above theorem, the existence of $s\geq 0$ is ensured by Assumption \ref{assump:epsilon} along with the definitions of $\hat{\Delta}$ in \eqref{equ:delta_hat} and $\vartheta_\alpha$ in \eqref{nequ:3}. In particular, we have $0.5+\barmu_{T-1}(2\epsilon_f+\vartheta_\alpha)/\{\barmu_0h_\alpha(\gamma^{-2}\hat{\Delta})\}<p$ for both $\alpha=0$ and $\alpha=1$. 
With Theorem \ref{thm: 1st}, we derive in the following corollary the complexity bound for first-order $\epsilon$-stationarity. 
The corollary follows directly from substituting $h_{\alpha}(\gamma^{-2}\hat{\Delta}) = \frac{\kappa_{fcd}}{2}\eta^3\gamma^{-4}(\Upsilon_3\epsilon - \Upsilon_4\epsilon_g)^2$ and $\vartheta_\alpha=2\epsilon_f$ for $\alpha=0$ into the conclusion in Theorem \ref{thm: 1st}.

\begin{corollary}\label{coro:1st}

Under conditions of Theorem \ref{thm: 1st}, for first-order stationarity ($\alpha=0$), the conclusion \eqref{prob_heavy_tailed} holds given
\begin{equation}\label{coro1:p}
\hat{p}\in\left(\frac{1}{2}+\frac{4\gamma^4 \barmu_{T-1}(2\epsilon_f+s)}{\barmu_0\kappa_{fcd}\eta^3(\Upsilon_3\epsilon-\Upsilon_4\epsilon_g)^2},p\right)
\end{equation}
and
\begin{multline}\label{coro1:T}
T >  \left(\frac{\hat{p} - 0.5}{\barmu_{T-1}}-\frac{4\gamma^4(2\epsilon_f+s)}{\barmu_0\kappa_{fcd}\eta^3(\Upsilon_3\epsilon-\Upsilon_4\epsilon_g)^2 }\right)^{-1} \\
\cdot\left(\frac{2\gamma^4(f(\bx_0)+\barmu_0\|c(\bx_0)\|-f_{\inf})}{\barmu_0 \kappa_{fcd}\eta^3(\Upsilon_3\epsilon-\Upsilon_4\epsilon_g)^2}+\frac{ \log_\gamma\{\Delta_0/(\Upsilon_3\epsilon-\Upsilon_4\epsilon_g) \}+2}{2\barmu_{T-1}}\right).
\end{multline}
\end{corollary}

Next, we show the high-probability complexity bound for second-order $\epsilon$-stationarity by substituting $h_{\alpha}(\gamma^{-2}\hat{\Delta}) = \frac{\kappa_{fcd}}{2\max\{\Delta_{\max},1\}}\eta^3\gamma^{-6}(\Upsilon_5\epsilon - \Upsilon_6\epsilon_g-\Upsilon_7\epsilon_h)^3$ and $\vartheta_\alpha=2\epsilon_f+\epsilon_g^{3/2}$ for $\alpha=1$ into the conclusion in Theorem \ref{thm: 1st}.

\begin{corollary}\label{coro:2nd}

Under conditions of Theorem \ref{thm: 1st}, for second-order stationarity ($\alpha=1$), the conclusion \eqref{prob_heavy_tailed} holds given
\begin{equation}\label{coro2:p}
\hat{p} \in\left(\frac{1}{2}+\frac{2\gamma^6\barmu_{T-1}\max\{\Delta_{\max},1\} (4\epsilon_f+\epsilon_g^{3/2}+2s)}{\barmu_0\kappa_{fcd}\eta^3(\Upsilon_5\epsilon-\Upsilon_6\epsilon_g-\Upsilon_7\epsilon_h)^3},p\right)
\end{equation}
and
\begin{multline}\label{coro2:T}
T >  \left(\frac{\hat{p}-0.5}{\barmu_{T-1}}-\frac{2\gamma^6 \max\{\Delta_{\max},1\} (4\epsilon_f+\epsilon_g^{3/2}+2s)}{\barmu_0\kappa_{fcd}\eta^3(\Upsilon_5\epsilon - \Upsilon_6\epsilon_g-\Upsilon_7\epsilon_h)^3 }\right)^{-1} \\
\cdot\left(\frac{2\gamma^6\max\{\Delta_{\max},1\}(f(\bx_0)+\barmu_0\|c(\bx_0)\|-f_{\inf})}{\barmu_0 \kappa_{fcd}\eta^3(\Upsilon_5\epsilon - \Upsilon_6\epsilon_g-\Upsilon_7\epsilon_h)^3}+\frac{ \log_\gamma\frac{\Delta_0}{\Upsilon_5\epsilon - \Upsilon_6\epsilon_g-\Upsilon_7\epsilon_h}+2}{2\barmu_{T-1}}\right).
\end{multline}
\end{corollary}

The above corollaries indicate that our method achieves an iteration complexity of $\mathcal{O}(\epsilon^{-2})$ for finding a first-order $\epsilon$-stationary point and $\mathcal{O}(\epsilon^{-3})$ for finding a second-order $\epsilon$-stationary point with high probability. 
The above complexity bounds are consistent with the results established in \cite{Nesterov2004Introductory,Hollender2023computational,Wang2020Cubic} for deterministic unconstrained nonconvex optimization, which are achieved using constant-stepsize gradient-based methods or cubic-regularized Newton methods.
These complexity bounds also match the results in \cite{Berahas2025Sequential} for a step-search SSQP method in identifying a first-order $\epsilon$-stationary point in constrained optimization, and the results in \cite{Cao2023First} for a trust-region method in identifying both first- and second-order $\epsilon$-stationary points in unconstrained optimization. Notably, our analysis only requires $e_k$ and $e_{s_k}$ to have a finite $1+\delta$ moment, rather than the sub-exponential tail condition assumed in \cite{Cao2023First,Berahas2025Sequential}, thereby allowing for noise with heavy tails.

\begin{remark}

In this remark, we demonstrate the techniques used in establishing the iteration complexity and discuss the sample complexity.

\vskip3pt
\noindent $\bullet$ 
The heavy-tailed nature of the noise introduces significant challenges in complexity analysis. To this end, we employ the Burkholder-type inequality (Lemma \ref{append:lemma1}) \citep{Burkholder1973Distribution, Chen2020Rosenthal} and martingale Fuk–Nagaev inequality (Lemma \ref{Fuk–Nagaev}) \citep{Fuk1973Certain,Nagaev1979Large, Fan2017Deviation} to study the tail behavior of $e_k$ and $e_{s_k}$. In contrast, \cite{Cao2023First, Berahas2025Sequential} utilized concentration techniques for sub-exponential variables, which we also apply in Section \ref{sec:4_subexp}. We mention that the techniques used in analyzing heavy-tailed noise are independent of the algorithm designs, implying that same analysis can be applied to trust-region and line-search methods \citep{Cao2023First,Berahas2025Sequential} and derive the same complexity bounds.

\vskip3pt
\noindent $\bullet$ 
Due to the heavy-tailed noise, $P\{T_{\epsilon} \leq T-1\}$ no longer grows to 1 at an exponential rate as in \cite{Cao2023First,Berahas2025Sequential}, but is instead reduced to a polynomial rate of $\O(T^{-\delta})$ for $\delta > 0$ (cf. \eqref{prob_heavy_tailed}). 
Despite this slower rate, we still obtain for any $\delta>0$, $P[T_{\epsilon}=\infty]=P[\cap_{T=0}^\infty \{T_\epsilon > T\}] = \lim\limits_{T\rightarrow\infty}P[T_{\epsilon}>T]=0$, where the second equality is due to the monotonicity of the events $\{T_\epsilon>T+1\}\subseteq\{T_\epsilon>T\}$. Thus, for any $\delta>0$, our algorithm reaches a (first- or second-order) $\epsilon$-stationary point in finite time \textit{almost surely}. If we suppress the irreducible noise by setting $\epsilon_f=\epsilon_g=\epsilon_h=0$, this result further yields a liminf-type almost-sure convergence: for any realization of the algorithm, there exists a subsequence of iterates converging to a (first- or second-order) stationary point. 
This liminf-type second-order convergence with heavy-tailed oracles matches existing literature \citep{Blanchet2019Convergence, Fang2024Trust}; however, this asymptotic result implied from our non-asymptotic analysis still leaves a gap compared with the stronger lim-type first-order convergence established in \cite[Theorem 4.18]{Chen2017Stochastic} and \cite[Theorem 4]{Blanchet2019Convergence}, which requires no moment conditions on the objective value estimates. Closing this gap remains an open question.

\end{remark}

\begin{remark}\label{rem:4.21}
We analyze the total sample complexity of the method following the introduction of oracle constructions in Section \ref{subsec:3.2}.
By Assumption \ref{assump:epsilon}, achieving first-order stationarity requires $\epsilon \geq \O(\sqrt{\epsilon_f}+\epsilon_g)$, leading to the relations $\epsilon_f \approx \O(\epsilon^2)$ and $\epsilon_g \approx \O(\epsilon)$. Similarly, for second-order stationarity where $\epsilon \geq \O(\sqrt[3]{\epsilon_f} + \sqrt{\epsilon_g} + \epsilon_h)$, we obtain $\epsilon_f \approx \O(\epsilon^3)$, $\epsilon_g \approx \O(\epsilon^2)$, and $\epsilon_h \approx \O(\epsilon)$. 

Suppose the noise of all orders of oracles has bounded variance (i.e., $\delta=1$). As shown in Section \ref{subsec:3.2}, estimating the objective value, gradient, and Hessian in each iteration requires sample sizes of $\O(\tilde{\epsilon}_f^{ -2})$, $\O(\epsilon_g^{-2})$, and $\O(\epsilon_h^{-2})$. Thus, when $\tilde{\epsilon}_f = \O(\epsilon_f)$, achieving first-order $\epsilon$-stationarity requires $\O(\epsilon^{-4})$ and $\O(\epsilon^{-2})$ samples per iteration for estimating the objective value and gradient, respectively. Given $\O(\epsilon^{-2})$ iteration complexity, the total sample complexity amounts to $\O(\epsilon^{-6})$. 
Analogously, for second-order $\epsilon$-stationarity, $\O(\epsilon^{-6})$, $\O(\epsilon^{-4})$, and $\O(\epsilon^{-2})$ samples are required per iteration for estimating the objective value, gradient, and Hessian, respectively. Combined with $\O(\epsilon^{-3})$ iteration complexity, the total sample complexity amounts to $\O(\epsilon^{-9})$.

We mention that the total sample complexity of $\O(\epsilon^{-6})$ for first-order $\epsilon$-stationarity matches the result in \cite{Jin2025Sample}, which analyzed a trust-region method for identifying first-order $\epsilon$-stationary points in unconstrained optimization. However, this result is worse than that in \cite{Curtis2023Worst}, which established a total sample complexity of $\O(\epsilon^{-4})$ for a line-search method. The discrepancy arises because both our method and \cite{Jin2025Sample} require objective value estimation at each iteration, leading to a per-iteration sample complexity of $\O(\epsilon^{-4})$ (cf. Remark \ref{rem:4.8}).
If we adopt a non-adaptive, objective-value-free design as in \cite{Curtis2023Worst}, where only gradient estimates are required, the per-iteration sample complexity will reduce to $\O(\epsilon^{-2})$, yielding a total sample complexity of $\O(\epsilon^{-4})$. However, as demonstrated in \cite{Berahas2025Sequential}, \mbox{objective-value-free} methods tend to perform worse in practice than those using objective value estimates, especially in the presence of irreducible noise.

Reducing the above sample complexity of adaptive SSQP methods is still a challenging problem. Some helpful insights can be drawn from well-studied stochastic methods such as stochastic gradient descent (SGD). Numerous variance-reduction techniques have been proposed for SGD in nonconvex settings and have demonstrated significant improvements over vanilla SGD. For example, Stochastic Variance-Reduced Gradient (SVRG) \citep{Johnson2013Accelerating, Reddi2016Stochastic} achieves an improved sample complexity of $\mathcal{O}(N + N^{2/3}\epsilon^{-2})$ for finite-sum optimization problems with $N$ samples. Building on this idea, \cite{Fang2026Trsvr} generalize SVRG to stochastic trust-region methods, attaining the same sample complexity. In the online setting, STOchastic Recursive Momentum (STORM) \citep{Cutkosky2019Momentum} achieves a sample complexity of $\mathcal{O}(\epsilon^{-3})$. Similarly, Stochastic Path-Integrated Differential EstimatoR (SPIDER) \citep{Fang2018Spider} provides a unifying framework, achieving $\mathcal{O}(\epsilon^{-3})$ in the online setting and $\mathcal{O}(N + N^{1/2}\epsilon^{-2})$ in the finite-sum setting with $N$ samples.
In constrained stochastic optimization, \cite{Lu2024Variance} applies both truncated recursive momentum and truncated Polyak momentum for penalty methods, attaining $\mathcal{O}(\epsilon^{-3})$ sample complexity. \cite{Berahas2023Accelerating} integrates SVRG into SSQP for finite-sum problems, attaining $\mathcal{O}(N + N^{2/3}\epsilon^{-2})$ sample complexity and matching the result in \cite{Reddi2016Stochastic} for unconstrained optimization. However, the SVRG technique requires periodically computing a full gradient as a ``checkpoint'', which presents challenges for directly adapting it to the online setting.  
We leave the development of variance-reduced trust-region SSQP methods for improving sample complexity to future research.

\end{remark}

\subsection{Complexity bounds with sub-exponential zeroth-order oracle}\label{sec:4_subexp}

In this section, we suppose that the noise in the objective value estimate follows the probabilistic sub-exponential zeroth-order oracle (i.e., conditions \textbf{(i)}, \textbf{(ii)}, and \textbf{(iii.2)} of Definition \ref{def:heavy-tailed oracle}), and we demonstrate that the same high-probability complexity bounds hold as in the case of heavy-tailed oracles.

In the next theorem, we establish a generic iteration complexity bound for both first- and second-order $\epsilon$-stationarity under the sub-exponential oracle.

\begin{theorem}\label{thm: 1st_subexp}

Under Assumptions \ref{assump:4-1}, \ref{assump:epsilon}, and the probabilistic sub-exponential zeroth-order oracle, for any $s\geq 0$, when $\hat{p}$ satisfies condition \eqref{thm:p} and $T$ satisfies condition  \eqref{thm:T}, we have
\begin{multline}\label{prob_subexp}
P\{T_\epsilon\leq T-1\} \geq  1- \exp\left\{-\frac{(p-\hat{p})^2}{2p^2}T\right\} \\
- \exp\left\{-\frac{1}{2}\min\left(\frac{(\epsilon_f -\tilde{\epsilon}_f + s)^2}{v^2},\frac{(\epsilon_f -\tilde{\epsilon}_f + s)}{b}\right)T\right\}.
\end{multline}	
\end{theorem}

\begin{proof}
See Appendix \ref{thm:Append:C.2}.
\end{proof}

Plugging the formulas of $\vartheta_\alpha$ and $h_{\alpha}(\gamma^{-2}\hat{\Delta})$ for $\alpha=0$ or $1$ into \eqref{thm:p} and \eqref{thm:T}, we immediately know that Corollaries \ref{coro:1st} and \ref{coro:2nd} remain true for the sub-exponential oracle. We combine them into the following corollary.

\begin{corollary}
Under conditions of Theorem \ref{thm: 1st_subexp}, for first-order stationarity (i.e., $\alpha=0$), the conclusion \eqref{prob_subexp} holds when $\hat{p}$ satisfies condition \eqref{coro1:p} and $T$ satisfies condition \eqref{coro1:T}.
For second-order stationarity (i.e., $\alpha=1$), the conclusion \eqref{prob_subexp} holds when $\hat{p}$ satisfies condition \eqref{coro2:p} and $T$ satisfies condition \eqref{coro2:T}.

\end{corollary}

The corollary suggests that the method enjoys an iteration complexity of $\mathcal{O}(\epsilon^{-2})$ for finding a first-order $\epsilon$-stationary point and $\mathcal{O}(\epsilon^{-3})$ for finding a second-order $\epsilon$-stationary point with high probability. These results align with the conclusions in \cite{Berahas2025Sequential,Cao2023First} as well as those under the heavy-tailed oracle in Section \ref{sec:4_heavytail}. Comparing \eqref{prob_subexp} with \eqref{prob_heavy_tailed}, we observe that $P\{T_\epsilon\leq T-1\}$ approaches 1 at an exponential rate. This faster convergence is attributed to the more restrictive sub-exponential condition \eqref{eq4}. Technically, applying the Chernoff bound to analyze the concentration of the accumulated oracle noise $e_k$ and $e_{s_k}$ leads to an exponential decay rate in $T$.

\begin{remark}\label{rem:4.24}
In this remark, we discuss how the estimation quality affects the iteration complexity. In our algorithm, the estimation quality is governed by the zeroth-, first-, and second-order probabilistic oracles. The parameters in these oracles directly determine the quality of the stochastic estimates. For example, smaller values of $\kappa_f$, $\kappa_g$, $\kappa_h$ and $\epsilon_f$, $\epsilon_g$, $\epsilon_h$ correspond to higher-quality estimates. These parameters also appear explicitly in our complexity bounds. 
For first-order stationarity (cf. \eqref{coro1:T}), the definitions of $\Upsilon_3$ and $\Upsilon_4$ imply that decreasing $(\kappa_f, \kappa_g, \epsilon_f, \epsilon_g)$ yields a smaller lower bound on $T$, thereby improving the constant factors in the complexity bound. Similarly, for second-order stationarity (cf. \eqref{coro2:T}), the definitions of $\Upsilon_5$, $\Upsilon_6$, $\Upsilon_7$ imply that decreasing $(\kappa_f, \kappa_g, \kappa_h, \epsilon_f, \epsilon_g, \epsilon_h)$ yields a smaller lower bound on $T$. Thus, although the iteration complexity in $\epsilon$ remains of the same order, the constant factors are influenced by the quality of the stochastic estimates.

\end{remark}

\begin{remark}\label{remark:pf}

In this remark, we echo Remark \ref{rem:1} and revisit the probabilistic sub-exponential zeroth-order oracle by exploring the connection between \eqref{def:Ck} and \eqref{eq4}. 
We show that when $\tilde{\epsilon}_f<\epsilon_f$ and $v,b$ are sufficiently small in \eqref{eq4} --- that is, the probability mass of $e_k, e_{s_k}$ is largely concentrated around their means --- then \eqref{def:Ck} holds with a proper choice of $p_f$. 

\begin{lemma}\label{lemma:p_f}
Under the probabilistic sub-exponential zeroth-order oracle (i.e., conditions \textbf{(i)}, \textbf{(ii)}, and \textbf{(iii.2)} of Definition \ref{def:heavy-tailed oracle}), the event $\C_k$ satisfies
\begin{equation}\label{eq5}
P(\C_k\mid \F_{k-1/2})\geq 1 -  2 \exp\left\{-\frac{1}{2}\min\left(\frac{(\epsilon_f-\tilde{\epsilon}_f)^2}{v^2} ,\frac{\epsilon_f-\tilde{\epsilon}_f}{b} \right)\right\}.
\end{equation} 
\end{lemma}	

\begin{proof}
See Appendix \ref{lemma:Append:C.1}.
\end{proof}
	
It follows from the lemma that when $\tilde{\epsilon}_f<\epsilon_f$ and $v,b$ are sufficiently small, \eqref{eq4} implies \eqref{def:Ck} with $p_f \leq 2 \exp\{-0.5\min ((\epsilon_f-\tilde{\epsilon}_f)^2/v^2 , (\epsilon_f-\tilde{\epsilon}_f)/b)\}$.

\end{remark}

\section{Numerical Experiments}\label{sec:5}

In this section, we demonstrate the empirical performance of the trust-region SSQP method (Algorithm \ref{Alg:STORM}). We apply the methods to a subset of 35 equality-constrained problems from the CUTEst test set \citep{Gould2014CUTEst} to find first- and second-order stationary points, referred to as TR-SSQP and TR-SSQP2, respectively. The 35 problems in our experiment have $d\leq10$ and $m\leq 5$.
We introduce the algorithm setup in Section \ref{subsec:5.1}. In Section \ref{subsec:5.2}, we examine the performance of the algorithms for various values of $\epsilon$ without irreducible noise. In Section \ref{subsec:5.3}, we examine their performance under different combinations of irreducible noise.$\quad\quad$

\subsection{Algorithm setup}\label{subsec:5.1}

For all methods, we generate samples and construct estimates of objective quantities in each iteration by sample average. Given a sample $\xi$, we denote the realizations of objective value, gradient, and Hessian at $\bx$ as $F(\bx,\xi), \nabla F(\bx,\xi)$, and $\nabla^2 F(\bx,\xi)$.
We allow samples to be dependent within the iteration, but independent across iterations. Additionally, the sample size for constructing each estimate is limited to at most $10^4$.

For TR-SSQP and TR-SSQP2, we construct estimates of objective Hessians, gradients, and values by generating samples $(\xi_k^h, \xi_k^g, \xi_k^f,\xi_{s_k}^f)$ following \eqref{sample_avg_obj_val_1} and \eqref{sample_avg_gradhess_val_1}. We set $\barmu_0=1$, $\Delta_0=\Delta_{\max}=5$, $\rho=1.2$, $\gamma=1.5$, $\kappa_f=\kappa_h=\kappa_g=0.05$, $p_f=p_h=p_g=0.1$, $\eta=0.4$, $C_h=C_g=C_f=5$, $r=0.01$. To solve \eqref{eq:Sto_tangential_step}, we use the \texttt{IPOPT} solver \citep{Waechter2005Implementation}. Since trust-region methods allow indefinite Hessian matrices, same as \cite{Fang2024Fully,Fang2024Trust}, we try four different $\barH_k$ for TR-SSQP in all problems.
\begin{enumerate}
\item Identity matrix (Id). It has been used in various existing SSQP literature, especially for line-search methods \citep[see, e.g.,][] {Berahas2021Sequential,Berahas2023Stochastic, Berahas2025Sequential,Na2022adaptive,Na2023Inequality}.

\item Symmetric rank-one (SR1) update.  We initialize $\barH_{0}=I$ and for $k\geq 1$, $\barH_k$ is updated as
\begin{equation*}
\barH_{k}=\barH_{k-1}+\frac{(\boldsymbol{y}_{k-1}-\barH_{k-1}\Delta\bx_{k-1})(\boldsymbol{y}_{k-1}-\barH_{k-1}\Delta\bx_{k-1})^T}{(\boldsymbol{y}_{k-1}-\barH_{k-1}\Delta\bx_{k-1})^T\Delta\bx_{k-1}}
\end{equation*}
where $\boldsymbol{y}_{k-1}=\bar{\nabla}_{\bx}\L_{k}-\bar{\nabla}_{\bx}\L_{k-1}$, and $\Delta\bx_{k-1}=\bx_{k}-\bx_{k-1}$. We employ SR1 instead of BFGS since SR1 can generate indefinite matrices, which may be preferable for constrained problems \citep{Khalfan1993Theoretical}.

\item Estimated Hessian (EstH). In the $k$-th iteration, we estimate the objective Hessian matrix $\bar{\nabla}^2f_k$ using \textit{one} single sample and set $\barH_k=\bar{\nabla}_{\bx}^2\L_k=\bar{\nabla}^2f_k+\sum_{i=1}^m \barblambda_k^i\nabla^2c_k^i$, where $\barblambda_k=-[G_kG_k^T]^{-1}G_k\barg_k$ is the least-squares Lagrangian multiplier introduced in Section \ref{sec:2.2}, with $\barg_k$ obtained from the first-order oracle.

\item Averaged Hessian (AveH). In the $k$-th iteration, we estimate the Hessian matrix $\bar{\nabla}_{\bx}^2\L_k$ as in EstH but set $\barH_k=\frac{1}{50}\sum_{i=k-49}^{k}\bar{\nabla}_{\bx}^2\L_i$. \cite{Na2022Hessian, Na2025Statistical} demonstrated that Hessian averaging facilitates stochastic Newton methods achieving faster (local) convergence.
\end{enumerate}

We recall from Section \ref{sec:2.2} $(\alpha = 1)$ that, for TR-SSQP2, $\bar{H}_k$ is defined in the same form as EstH but involves a second-order oracle. That is, $\barH_k = \bar{\nabla}^2 f_k+ \sum_{i=1}^{m}\bar{\boldsymbol{\lambda}}_k^i\nabla^2 c_k^i$, where 
$\bar{\nabla}^2 f_k$ is obtained from the second-order oracle using samples $\xi_k^h$ and $\barblambda_k=-[G_kG_k^T]^{-1}G_k\barg_k$ with $\barg_k$ obtained from the first-order oracle using samples $\xi_k^g$. 
We clarify that in EstH and AveH, the objective Hessians are estimated using a single sample; hence, they do not satisfy the second-order oracle condition \eqref{def:Ak}. However, this does not pose any issue since these Hessian estimates are used only within the first-order TR-SSQP method, for which access to a second-order oracle is not required anyway. Thus, for both TR-SSQP and TR-SSQP2, all oracle conditions required in the analysis are fulfilled by our constructions.

For all methods, the objective model estimates are generated based on deterministic evaluations provided by the CUTEst package, with entrywise independent noise (denoted by $\texttt{rand}$) following six different distributions: (i) standard normal distribution, (ii) $t$-distribution with degree of freedom 4 (denoted as $t_4$), (iii) $t$-distribution with degree of freedom 2 (denoted as $t_2$), (iv) log-normal distribution with location-scale parameters $(\mu=0,\sigma=1)$, (v) Weibull distribution with scale-shape parameters $(\lambda=1,k=1)$, and (vi) Cauchy distribution with location-scale parameters $(x_0=0,\gamma=1)$. Noise from the normal distribution exhibits sub-exponential tail, while noise from the other five distributions are heavy-tailed. We generate noise from two different $t$-distributions; noise from the $t_4$-distribution has bounded variance, while noise from the $t_2$-distribution only has a bounded $1+\delta$ moment for $\delta\in(0,1)$. Noise from the Cauchy distribution has neither a finite mean nor a bounded $1+\delta$ moment for any $\delta>0$. 
This distribution violates the conditions imposed by the probabilistic oracles, and we incorporate it solely to examine its impact on empirical performance. Specifically, for distributions (i), (ii), (iii) and (vi), we generate $F(\bx_k,\xi) = f_k+\sigma\cdot\texttt{rand}$, $[\nabla F(\bx_k,\xi)]_i= [\nabla f_k]_i + \sigma \cdot \texttt{rand}$, and $[\nabla^2 F(\bx_k,\xi)]_{i,j} = (\nabla^2 f_k)_{i,j} \\ + \sigma \cdot \texttt{rand}$. 
For distributions (iv) and (v), $\texttt{rand}$ is replaced by $\delta \cdot \texttt{rand}$, with $\delta$ being a Rademacher variable (i.e., $P(\delta=1)=P(\delta=-1)=0.5$), so that the noise becomes symmetric. Throughout the experiments, we set $\sigma=10^{-2}$.

We set the maximum iteration budget to be $10^5$ and define the stopping time as in Definition \ref{def:stopping_time}. The values of $\epsilon$ in stopping time are specified in each experiment below. For each algorithm and each problem, under every combination of irreducible noise level and $\epsilon$, we report the average of $T_{\epsilon}$ over five independent runs.

\subsection{Algorithms performance without irreducible noise}\label{subsec:5.2}

In this section, we aim to validate our complexity bounds in Theorems \ref{thm: 1st} and \ref{thm: 1st_subexp}, which state that achieving first-order $\epsilon$-stationarity requires at most $\mathcal{O}(\epsilon^{-2})$ iterations, while achieving second-order $\epsilon$-stationarity requires at most $\mathcal{O}(\epsilon^{-3})$ iterations with high probability. To allow for arbitrarily small $\epsilon$, we set $\epsilon_f = \epsilon_g = \epsilon_h = 0$. We vary $\epsilon \in \{10^{-1}, 10^{-2}$, $10^{-3}, 10^{-4}\}$ and report the average stopping time in Figure \ref{fig:cutest}.

\begin{figure}[t]
\centering
\subfigure[Normal distribution]{\includegraphics[width=0.32\textwidth]{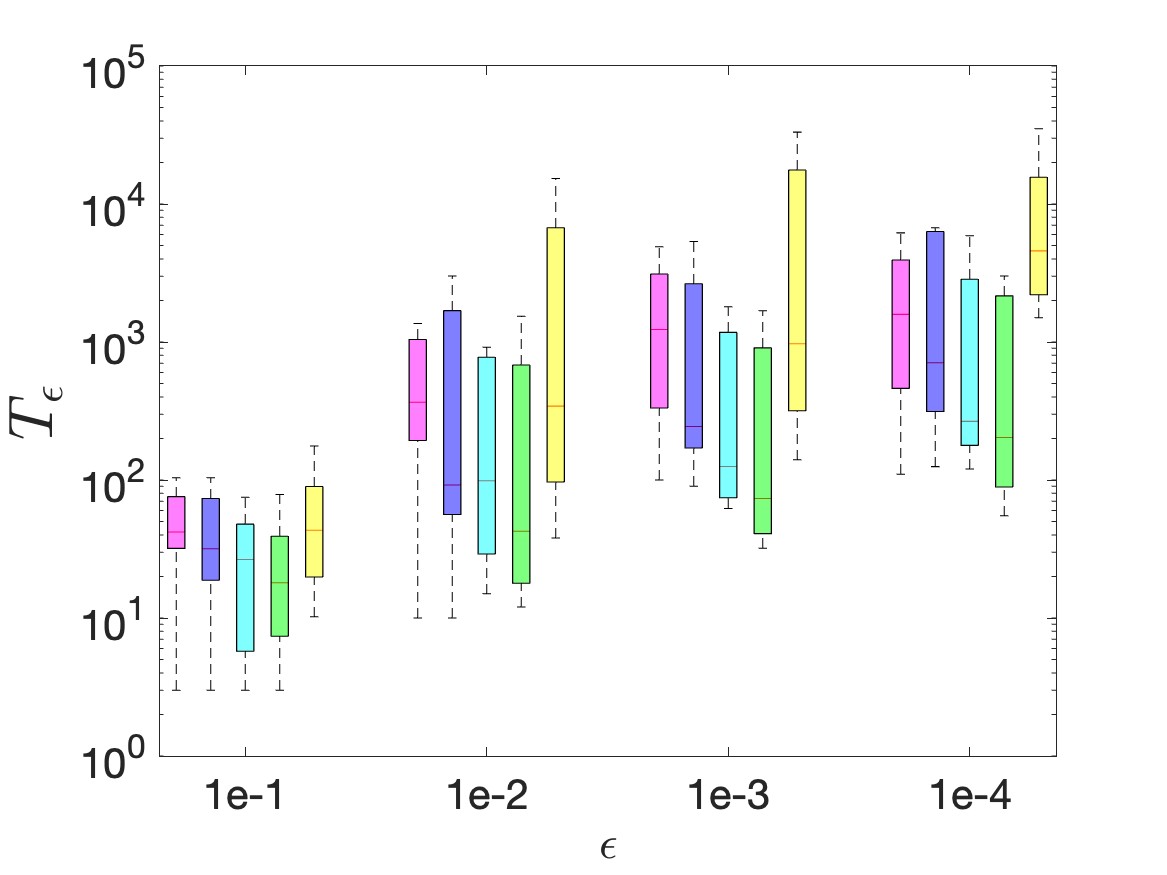}}
\subfigure[$t_4$-distribution]{\includegraphics[width=0.32\textwidth]{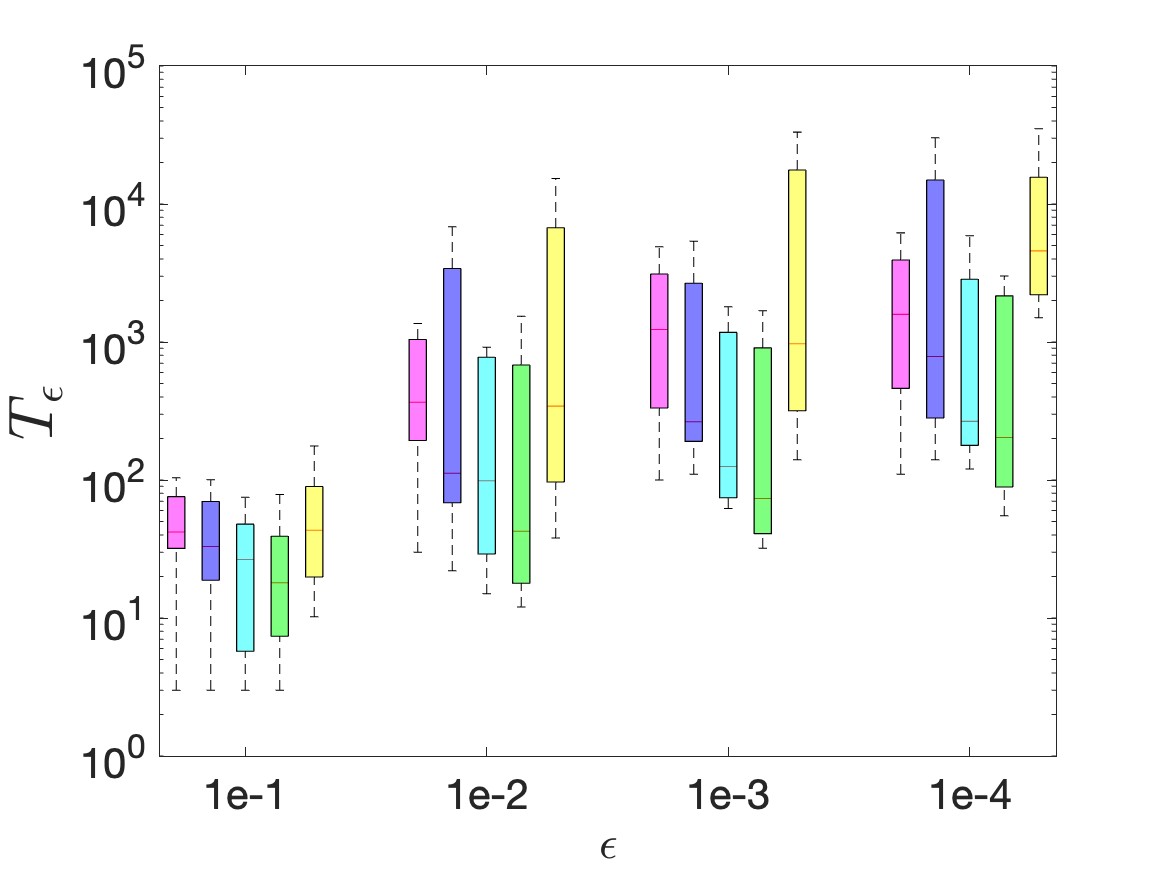}}
\subfigure[$t_2$-distribution]{\includegraphics[width=0.32\textwidth]{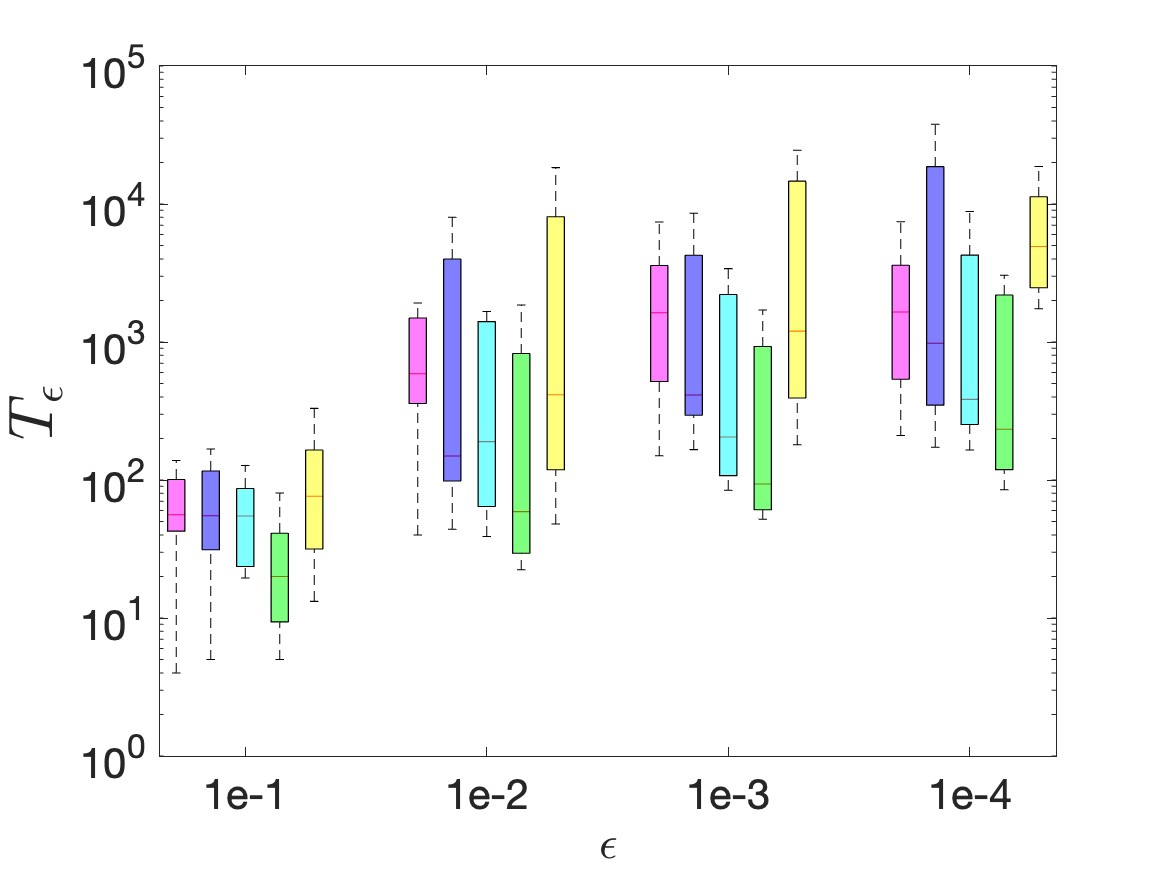}}
\subfigure[Log-normal distribution]{\includegraphics[width=0.32\textwidth]{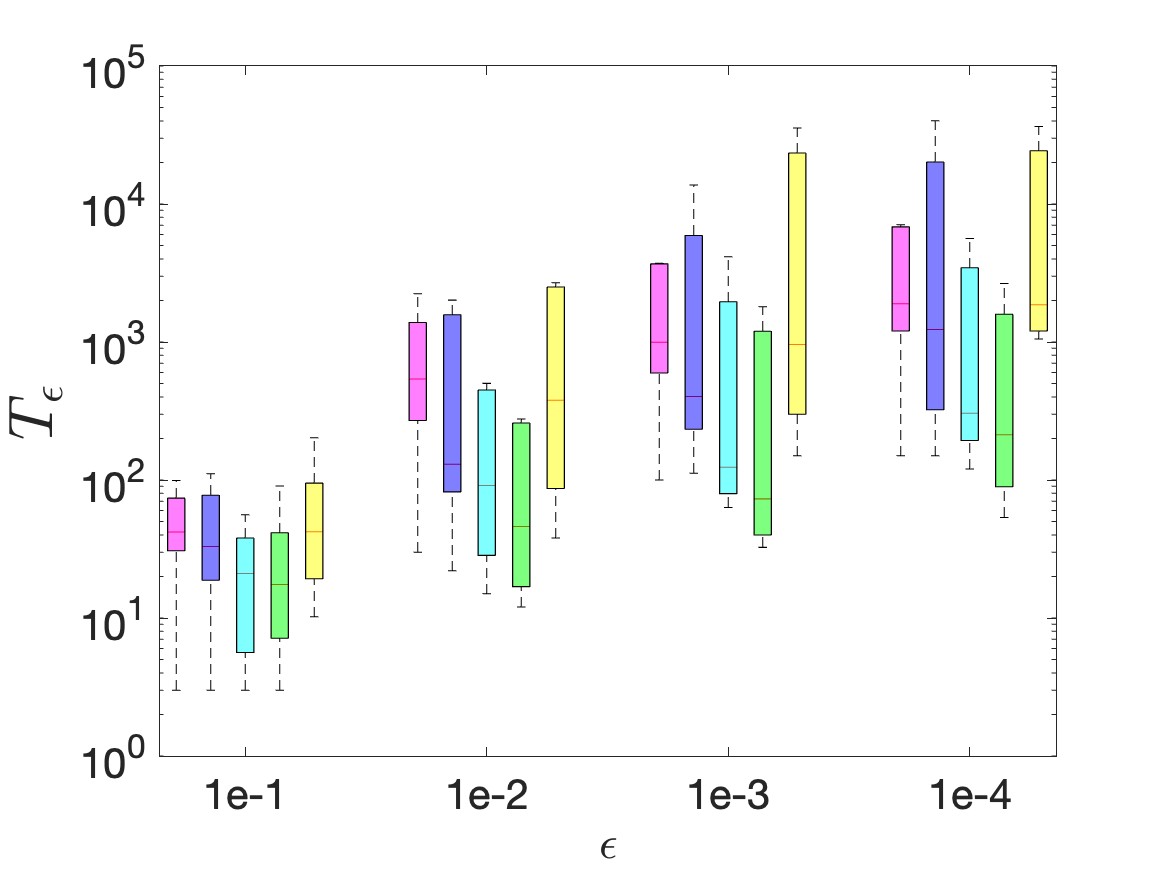}}
\subfigure[Weibull distribution]{\includegraphics[width=0.32\textwidth]{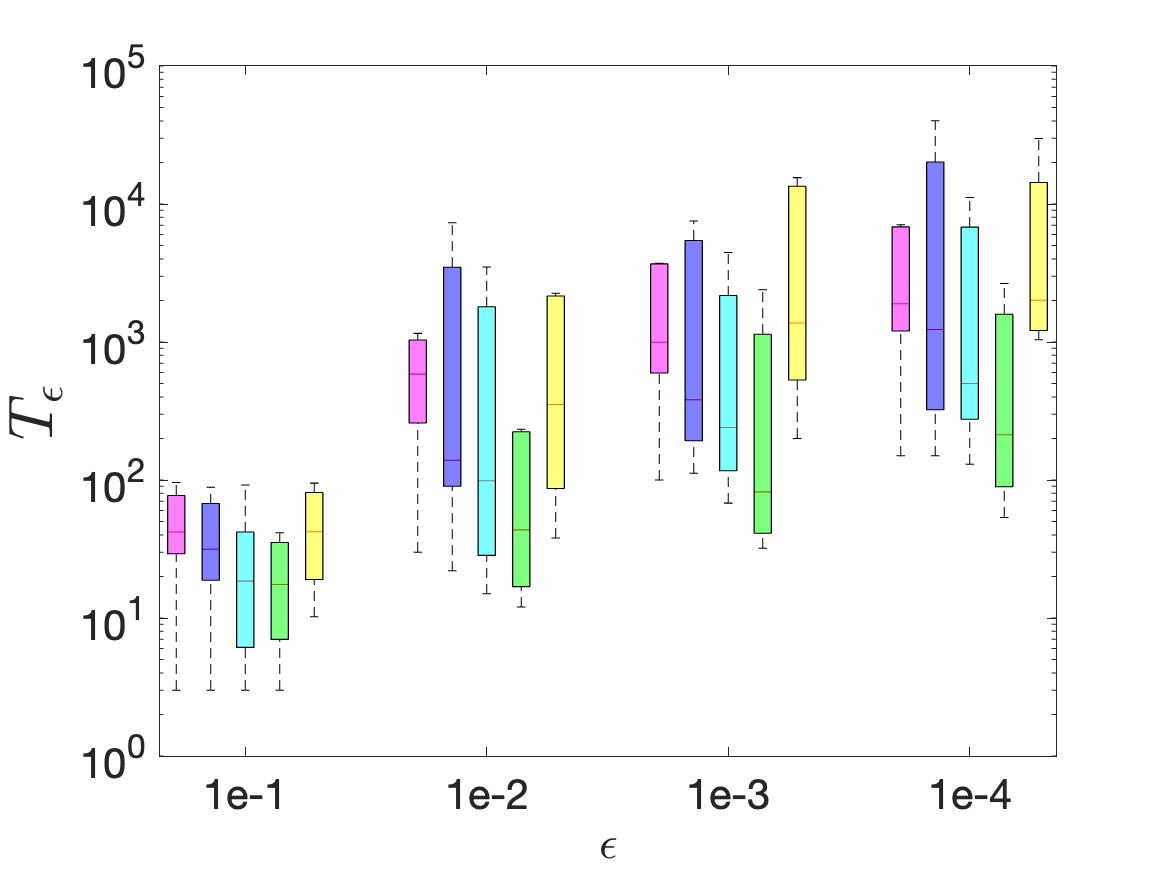}}
\subfigure[Cauchy distribution]{\includegraphics[width=0.32\textwidth]{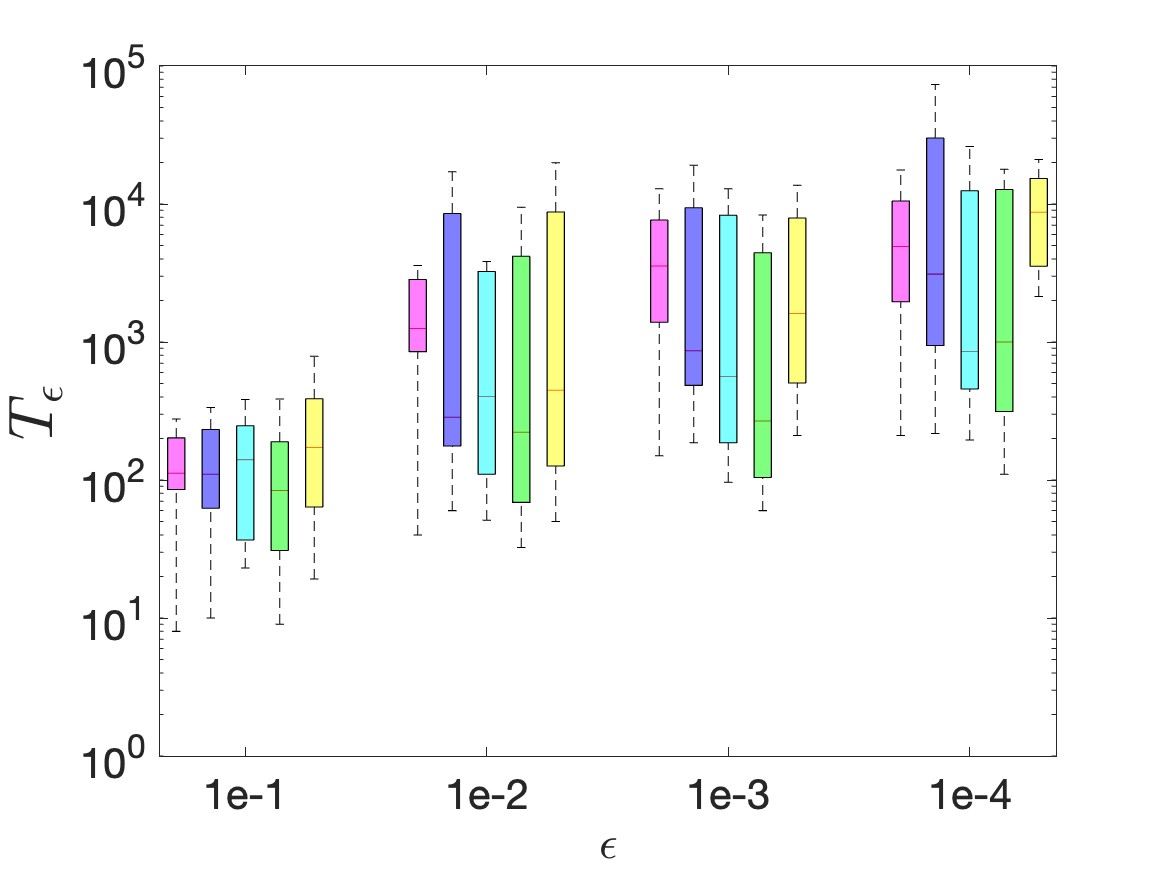}}
\subfigure{\includegraphics[width=0.9\textwidth]{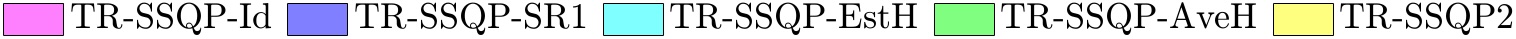}}
\caption{Averaged stopping time $T_{\epsilon}$ with noise from six different distributions. In every plot, the first four boxes correspond to TR-SSQP with different choices of $\barH_k$, and the fifth box corresponds to TR-SSQP2.}
\label{fig:cutest}	\vspace{-0.5cm}
\end{figure}

First, we observe that as $\epsilon$ decreases, all methods require more iterations to converge. Specifically, when $\epsilon$ is reduced by a factor of 10, Figure \ref{fig:cutest} shows that the stopping time of TR-SSQP increases by a factor between 10 and 100, while that of TR-SSQP2 increases by approximately 100. This behavior aligns with our theoretical results, though it is important to emphasize that our analysis provides only a \textit{worst-case guarantee}. Due to this worst-case nature, the actual iteration complexity may grow more slowly than predicted when gradually decreasing $\epsilon$.

Second, we observe that for TR-SSQP methods, performance varies across the four Hessian approximations. TR-SSQP-SR1 exhibits the most unstable performance when $\epsilon$ is small. This instability is largely attributed to the Hessian updating scheme, which can accumulate estimation errors over iterations. In contrast, TR-SSQP with the averaged Hessian (TR-SSQP-AveH) achieves the best performance, followed by TR-SSQP-EstH, with their advantage becoming particularly pronounced as $\epsilon$ decreases. This highlights the benefits of fully exploring Hessian information in algorithm designs. We discuss the insights behind the differences between EstH and AveH in more detail in Remark \ref{Hessian_approximation}.

Third, we observe that when the noise satisfies the conditions specified in the probabilistic oracles, the noise distribution has a limited impact on the stopping time for most methods. As shown in Figure \ref{fig:cutest} (a-e), changes in the noise distribution result in no significant variation in the average stopping time for most methods. In contrast, Figure \ref{fig:cutest} (f) displays a substantial deviation from this pattern, with the average stopping time increasing significantly compared to (a)-(e), even for large $\epsilon$. This deterioration is attributed to the extremely heavy-tailed nature of Cauchy noise, which does not even have a finite mean, so the sample average cannot reliably estimate the true quantity even as the sample size increases. 
Due to the unbounded mean, TR-SSQP-AveH is affected by Cauchy noise more severely than the other methods. This method terminates within $10^3$ to $10^4$ iterations under Cauchy noise with $\epsilon=10^{-4}$, while it terminates within $10^2$ to $10^3$ iterations under the other five noises with the same $\epsilon$. This phenomenon is less apparent for the other methods, as they generally require more iterations to converge across different noise distributions.

\begin{remark}\label{Hessian_approximation}

We further discuss the impact of Hessian approximation quality on both the theory and practice of our method. For first-order stationarity, our theoretical iteration complexity results do not require any high-accuracy Hessian approximation assumptions; that is, the quality of the Hessian estimate does not affect the theoretical guarantees. This is reasonable, since the first-order convergence analysis relies on establishing sufficient decrease guarantees, which depend only on gradient information and do not involve second-order oracles. Nevertheless, as our experiments indicate, the Hessian approximation can have a substantial effect on practical performance, as reflected in the noticeably different behavior produced by different Hessian construction methods in Figure \ref{fig:cutest}.

We highlight the contrast between the estimated Hessian (EstH) and the averaged Hessian (AveH). In stochastic settings, using the most recent Hessian estimate yields an unbiased approximation but suffers from high variance, which can significantly impair performance, sometime even performing worse than simply using the identity matrix. In contrast, averaging Hessian estimates over iterations dramatically reduces variance due to the noise concentration effect, though at the expense of introducing bias. Notably, as the iterate approaches a stationary point, the bias of the averaged Hessian estimate diminishes, while the variance of the last-iterate Hessian estimate still remains at a constant level. This bias-variance trade-off explains why AveH typically produces more stable and reliable performance than EstH. This behavior is also consistent with deterministic SQP theory, where higher-quality Hessian approximations are better able to satisfy the Dennis-Moré condition \citep{Dennis1974characterization}, which underpin local superlinear convergence, even though such a condition is not required for global convergence.

Finally, we note that for achieving second-order stationarity, the quality of the Hessian approximation becomes essential and \textit{explicitly appears in the theoretical lower bound on $T$}, as discussed in Remark \ref{rem:4.24}.
\end{remark}

\begin{figure}[t]
\centering
\subfigure{\includegraphics[width=0.3\textwidth]{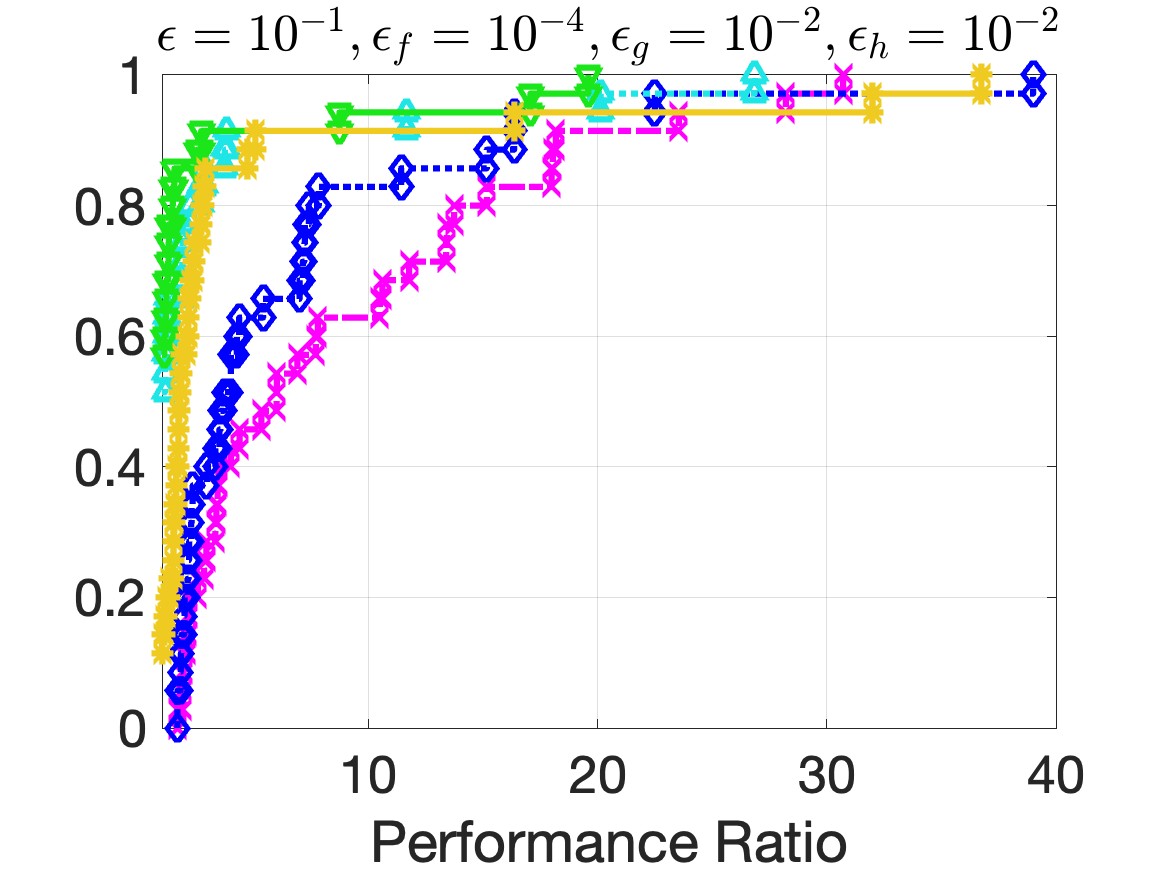}}
\subfigure{\includegraphics[width=0.3\textwidth]{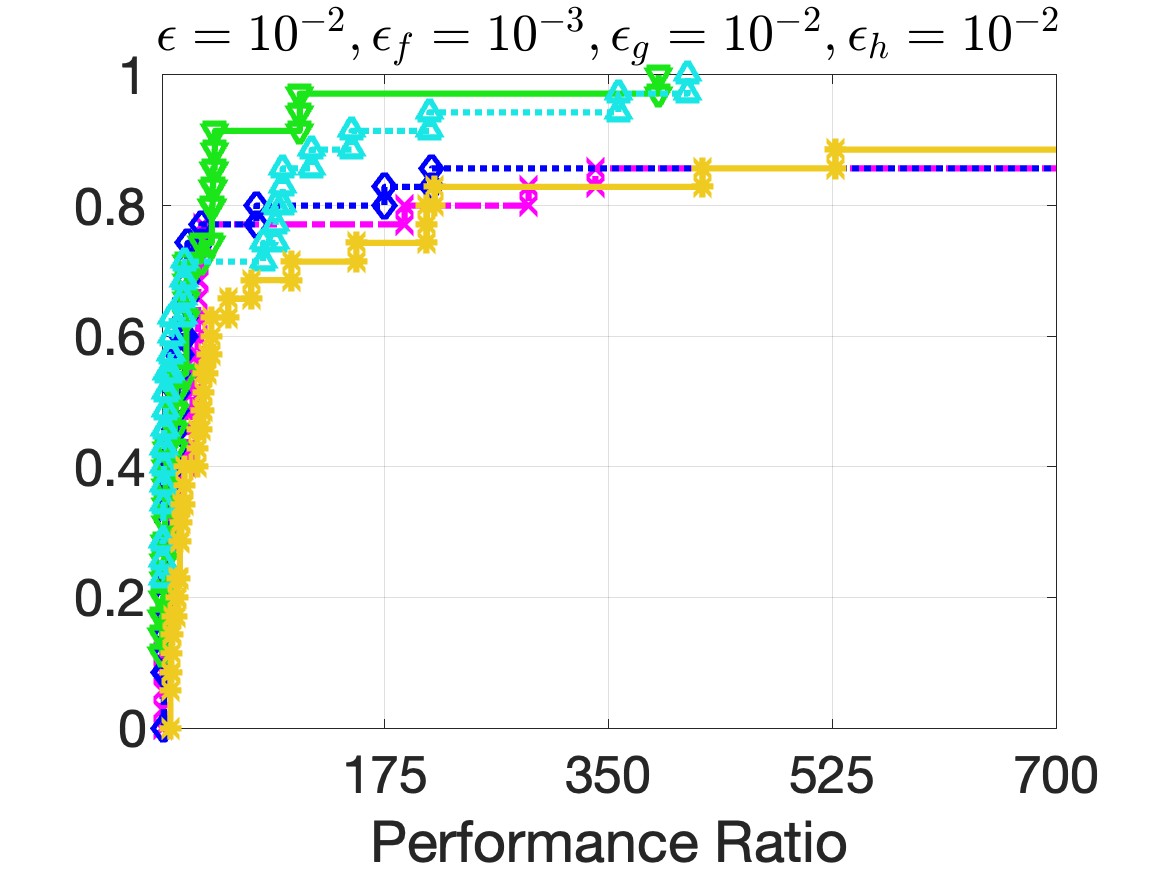}}
\subfigure{\includegraphics[width=0.3\textwidth]{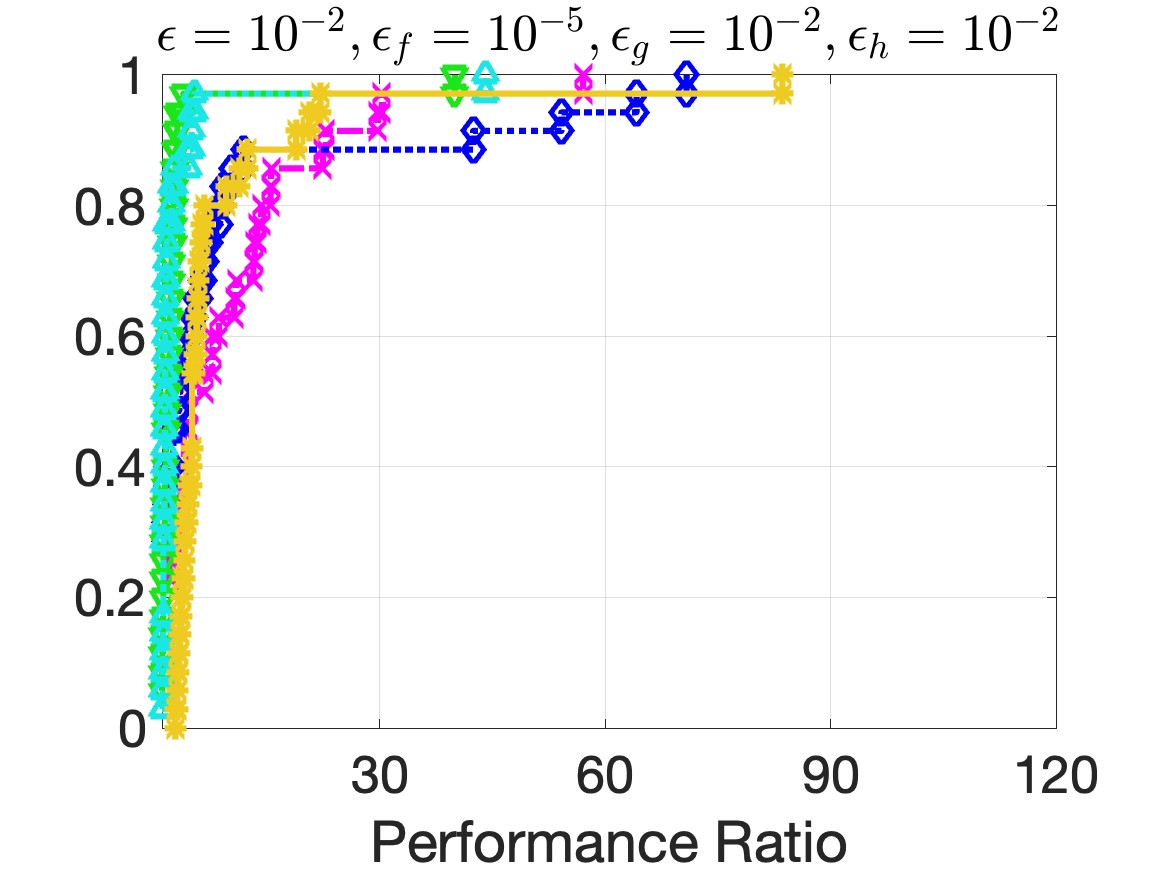}}\\	
\subfigure{\includegraphics[width=0.3\textwidth]{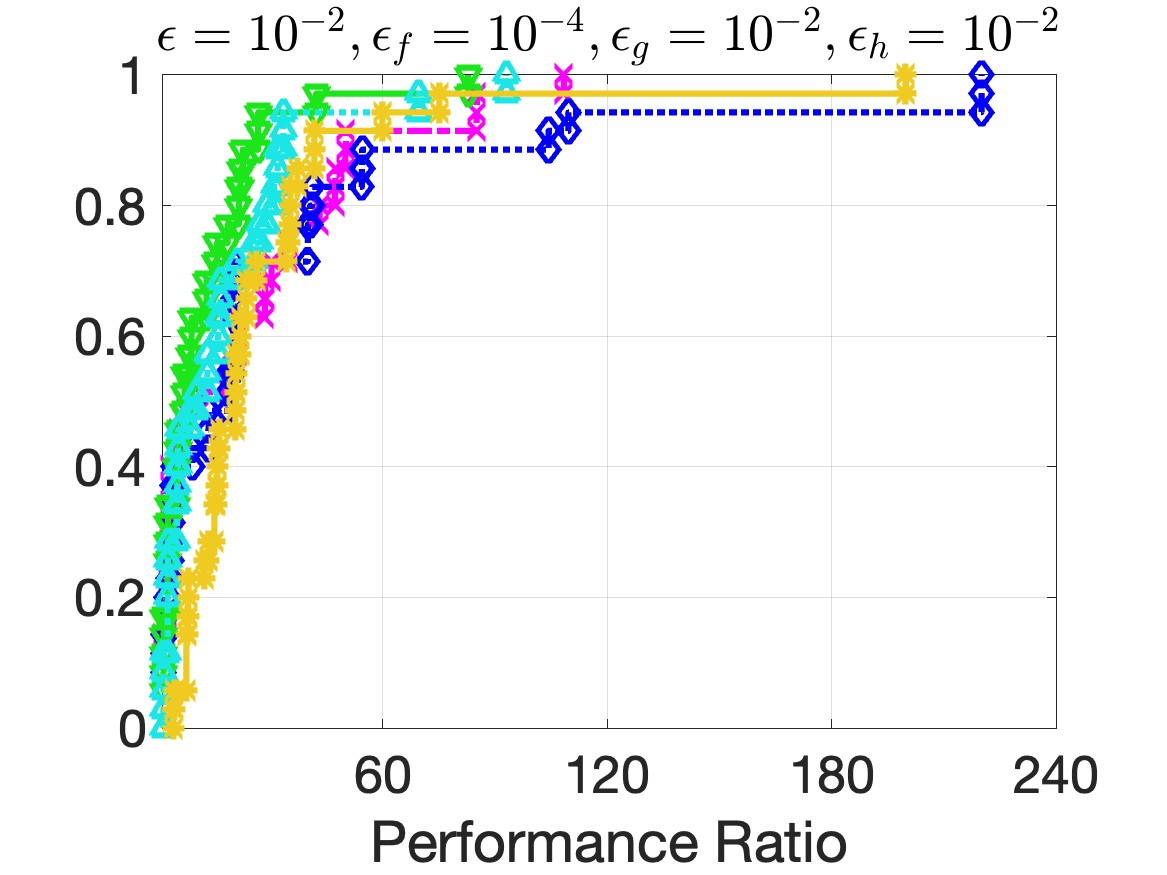}}
\subfigure{\includegraphics[width=0.3\textwidth]{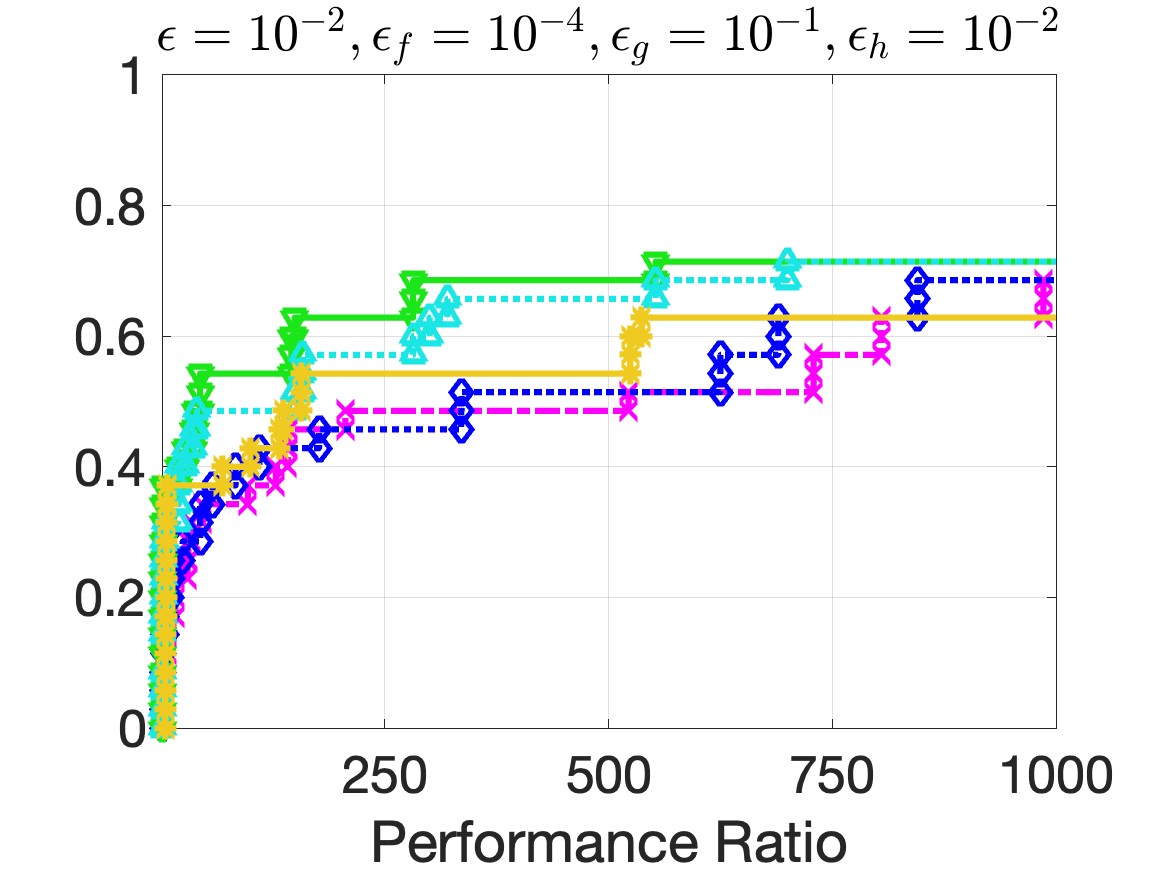}}  \subfigure{\includegraphics[width=0.3\textwidth]{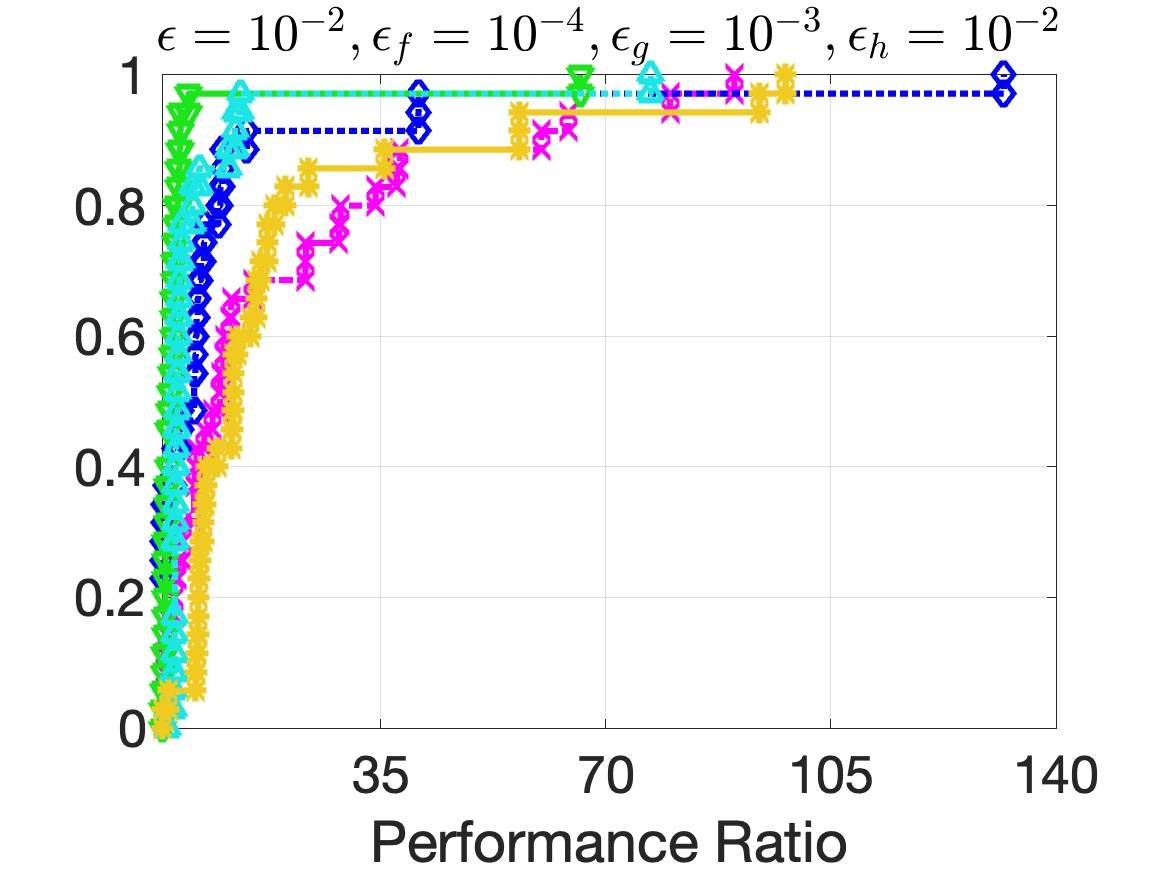}}\\
\subfigure{\includegraphics[width=0.3\textwidth]{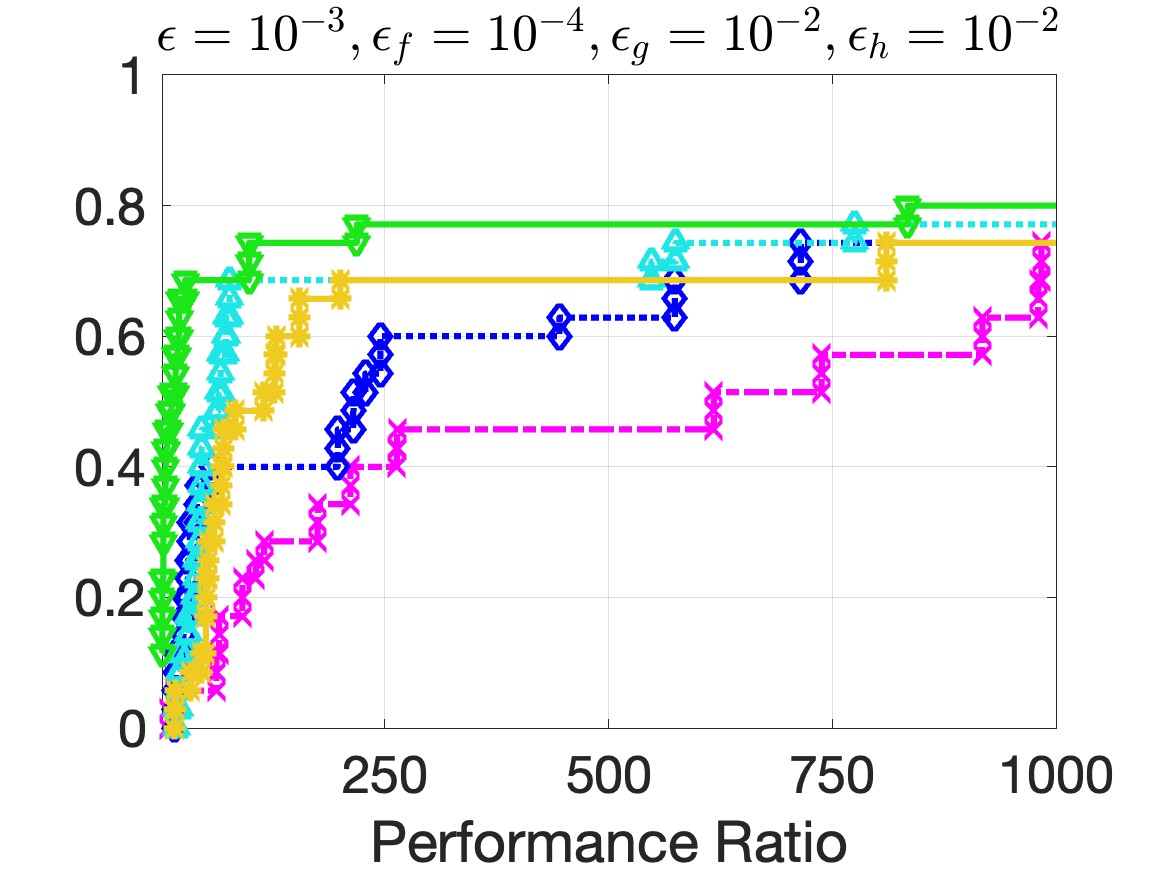}}
\subfigure{\includegraphics[width=0.3\textwidth]{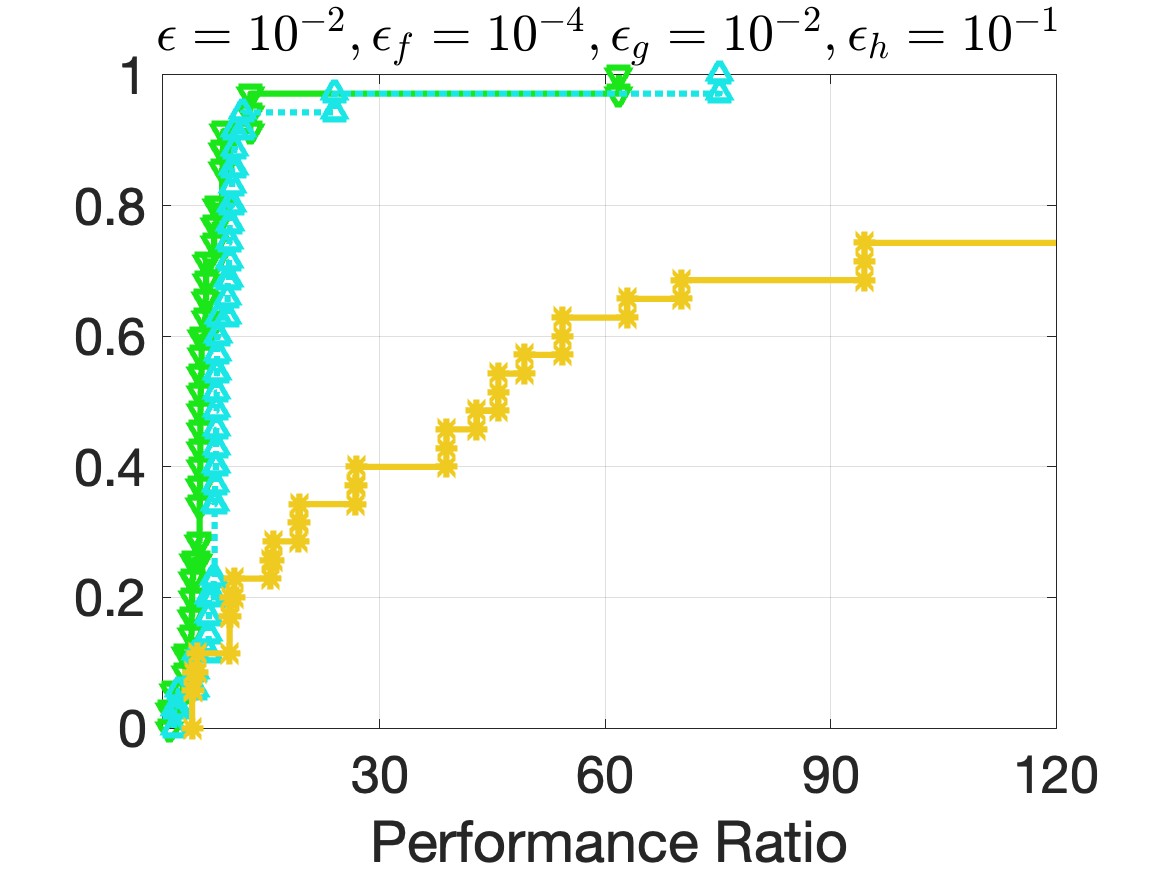}}
\subfigure{\includegraphics[width=0.3\textwidth]{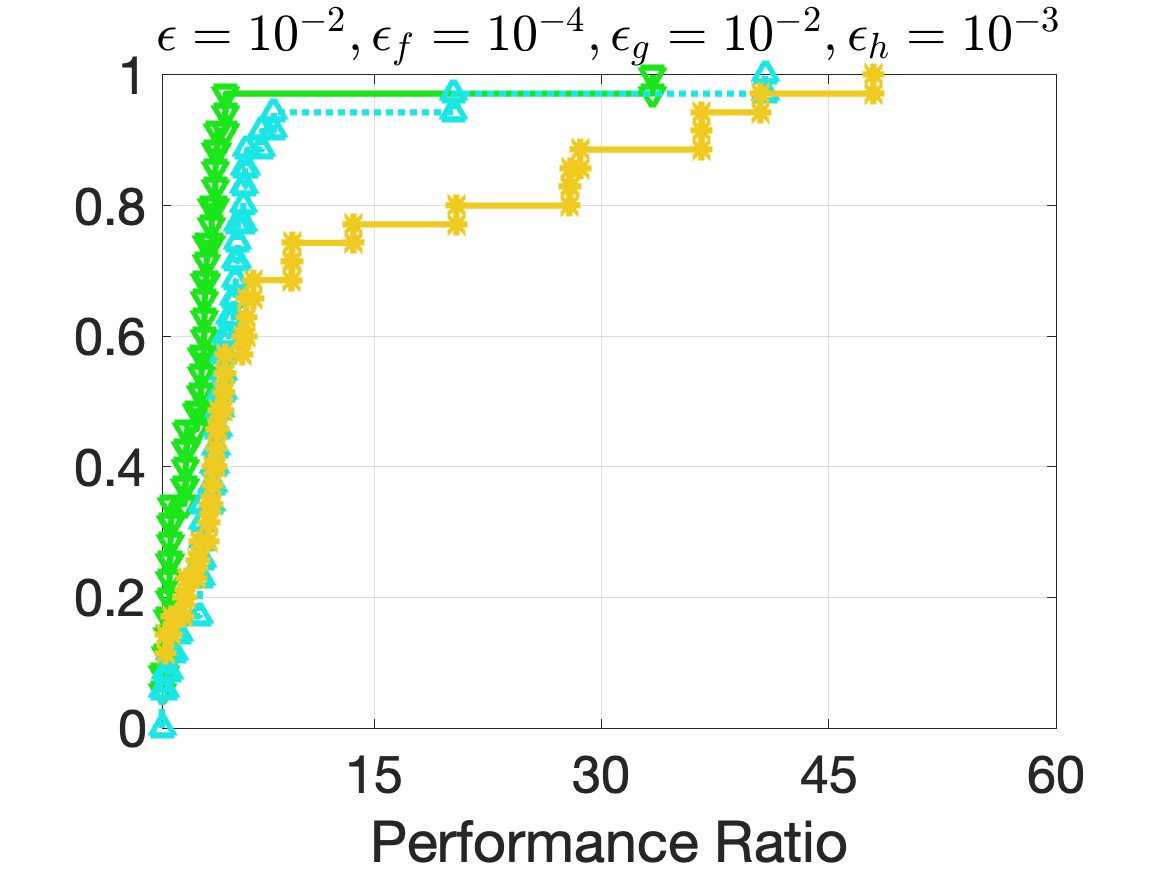}}\\
\subfigure{\includegraphics[width=0.95\textwidth]{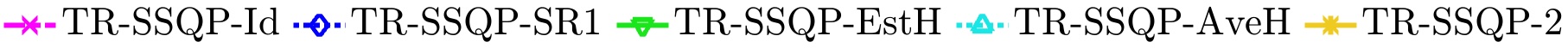}}
\caption{Performance profiles with noise following a normal distribution. Each line represents a different method. The first column corresponds to the default irreducible noise with varying $\epsilon$. Each row of the last two columns corresponds to varying $\epsilon_f$, $\epsilon_g$, $\epsilon_h$, respectively.}
\label{fig:cutest2}	\vskip-0.2cm
\end{figure}

\subsection{Algorithms performance with irreducible noise}\label{subsec:5.3}

In this section, we evaluate the performance of all methods under different combinations of irreducible noise levels. To limit the number of total combinations, we set the default values as $\epsilon = 10^{-2}$, $\epsilon_f = 10^{-4}$, $\epsilon_g = 10^{-2}$, $\epsilon_h = 10^{-2}$, and vary one parameter at a time while keeping the others fixed. In particular, we vary $\epsilon \in \{10^{-1}, 10^{-3}\}$, $\epsilon_f \in \{10^{-3}, 10^{-5}\}$, $\epsilon_g \in \{10^{-1}, 10^{-3}\}$, and $\epsilon_h \in \{10^{-1}, 10^{-3}\}$. The default setup ensures that the condition $\epsilon \geq \mathcal{O}(\sqrt{\epsilon_f} + \epsilon_g)$ for first-order stationarity is satisfied (cf. Assumption \ref{assump:epsilon}). For consistency, we use the same default setup for second-order stationarity.
When introducing irreducible noise to each method, we sample a Rademacher variable $\delta$ and directly add the noise $\delta \epsilon_h, \delta \epsilon_g, \delta \epsilon_f$ to the corresponding objective quantity estimates. The Rademacher variable ensures that the irreducible noise is two-side distributed.

\begin{figure}[t]
\centering
\subfigure{\includegraphics[width=0.3\textwidth]{ 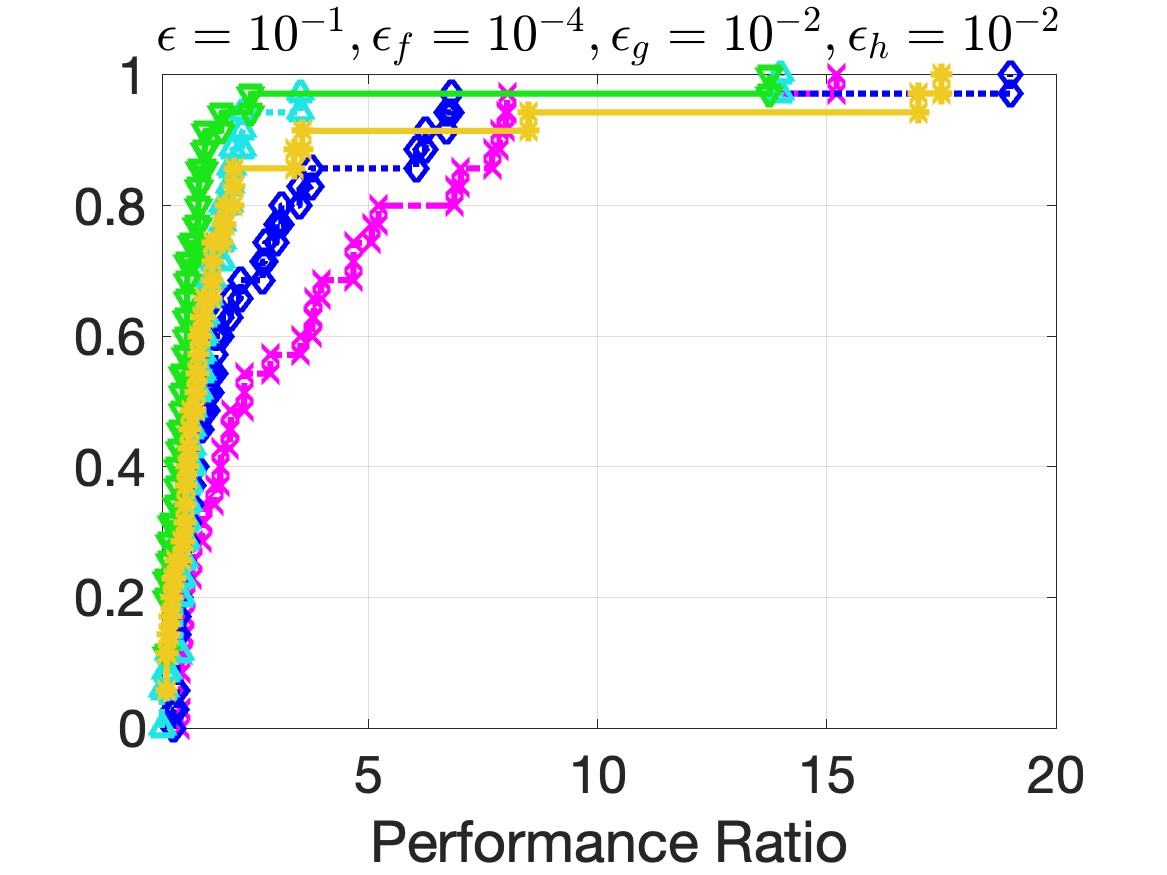}}
\subfigure{\includegraphics[width=0.3\textwidth]{ 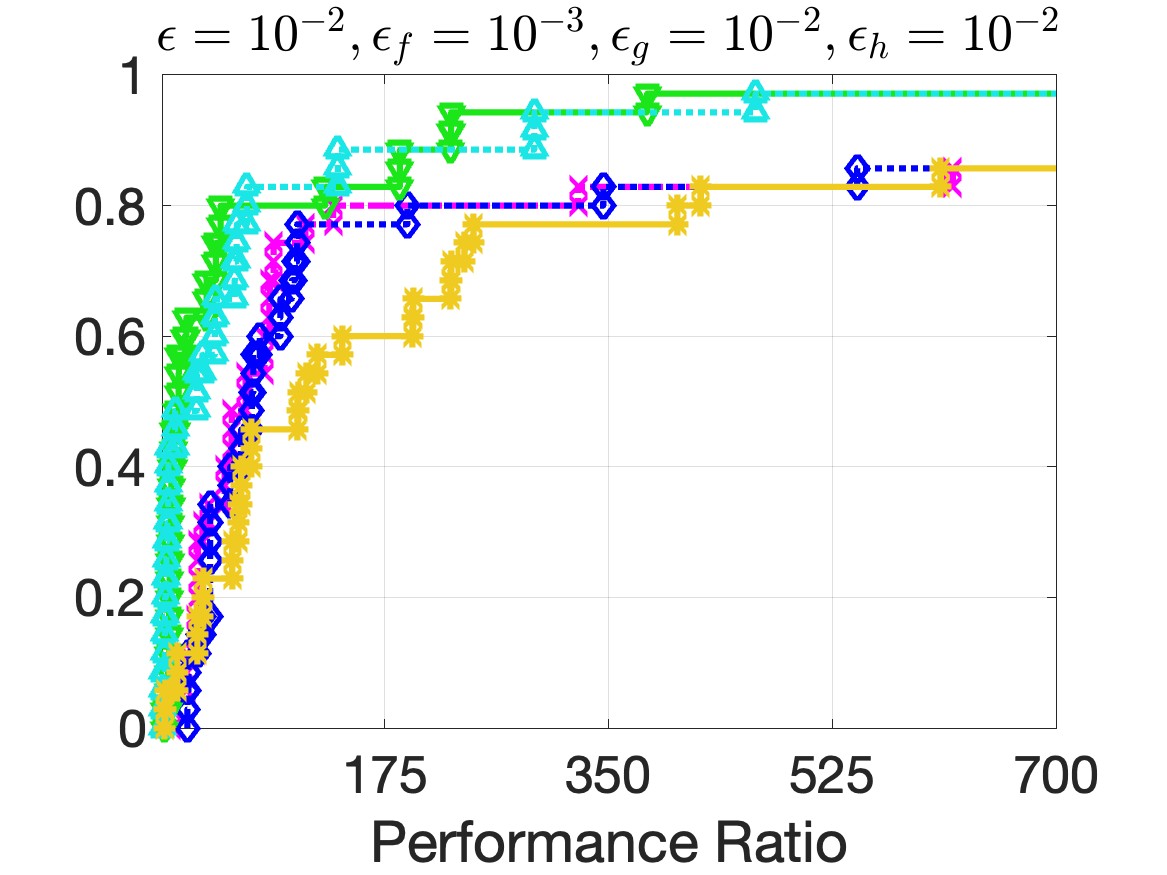}}
\subfigure{\includegraphics[width=0.3\textwidth]{ 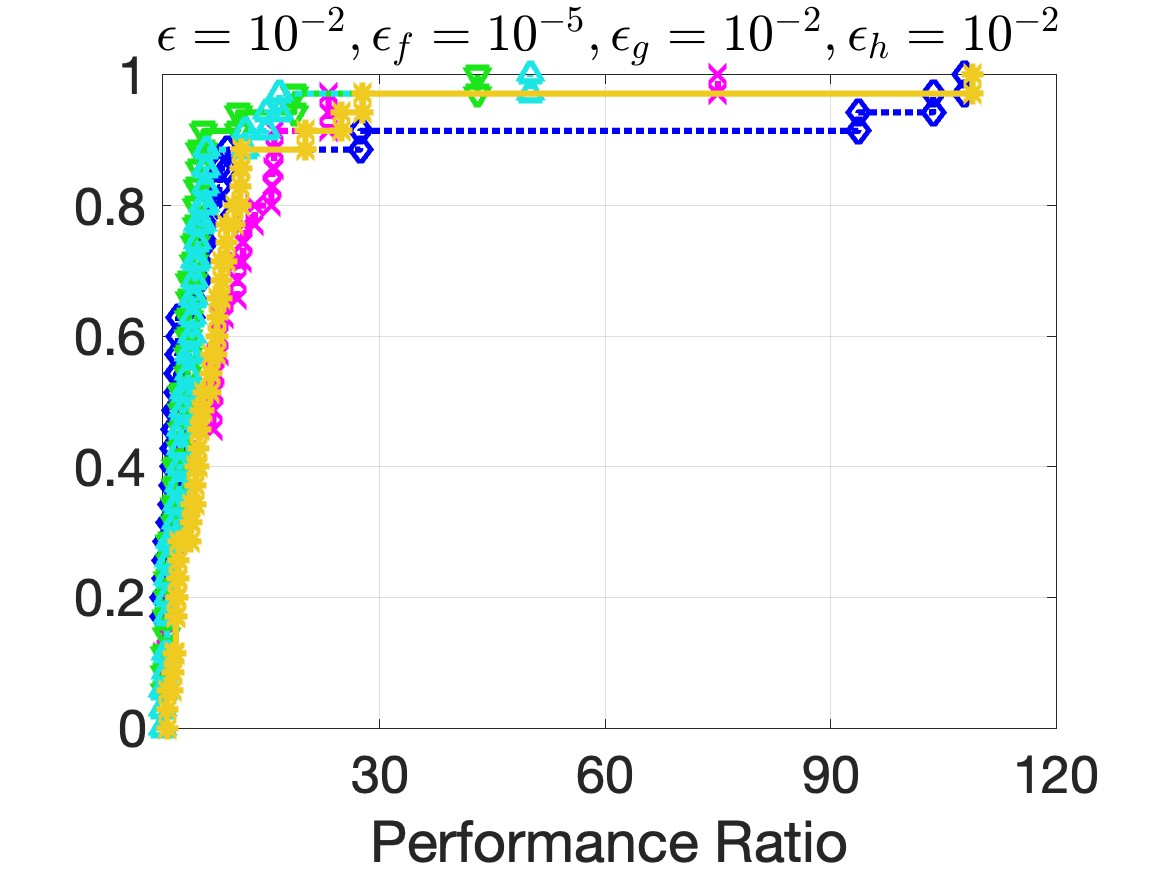}}\\	
\subfigure{\includegraphics[width=0.3\textwidth]{ 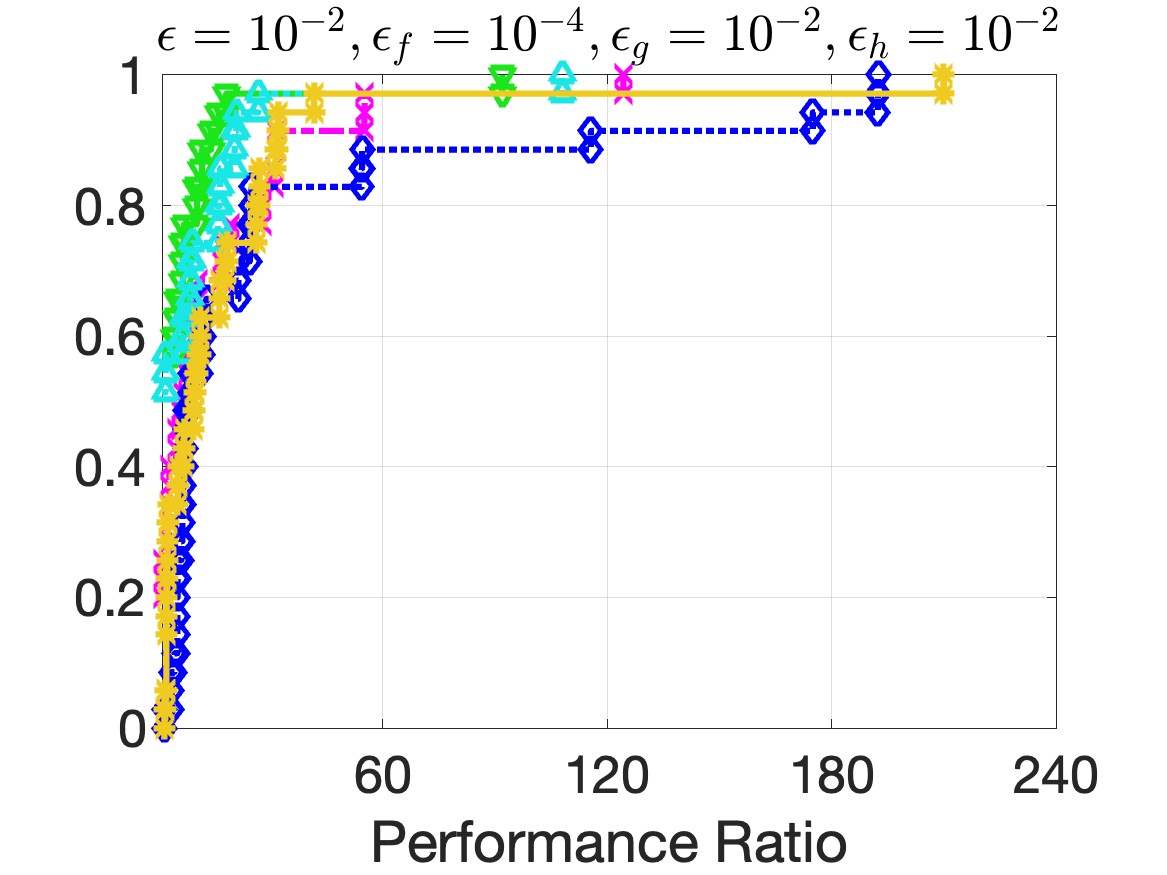}}
\subfigure{\includegraphics[width=0.3\textwidth]{ 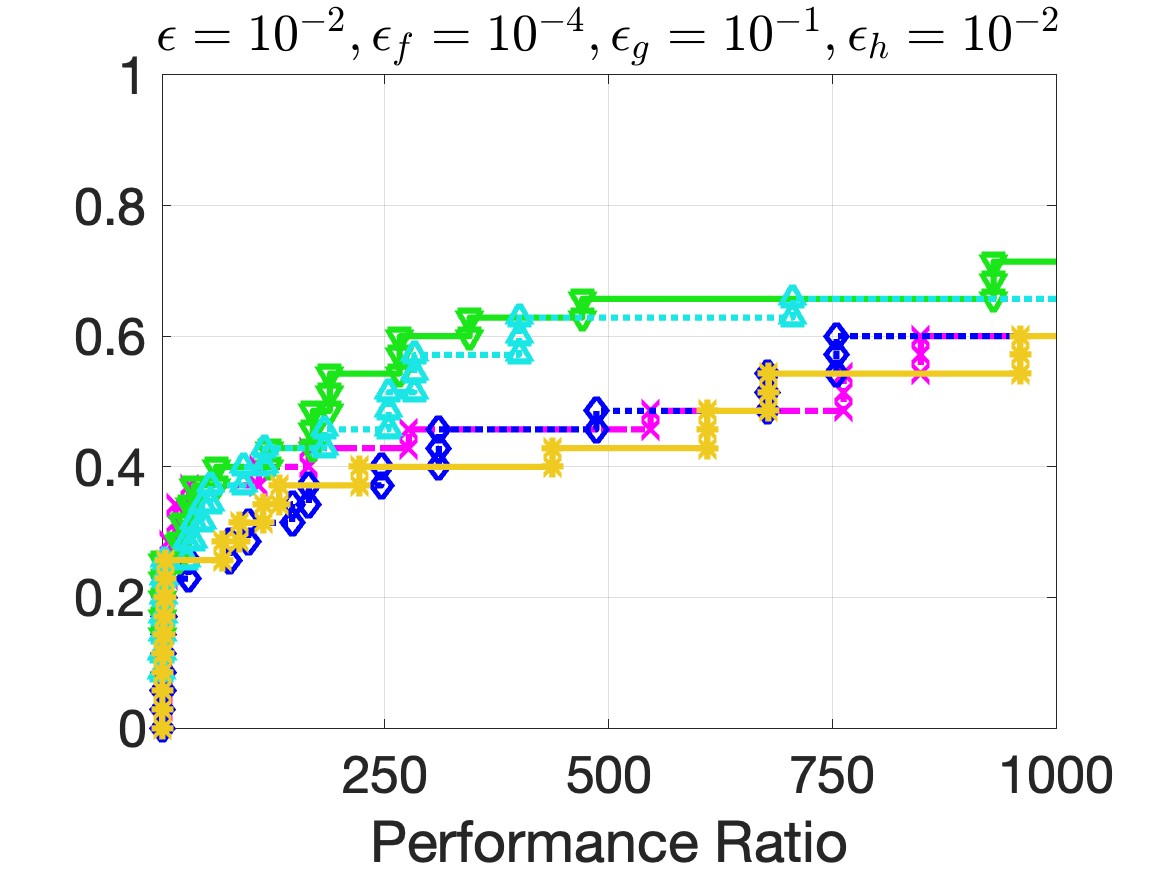}}  \subfigure{\includegraphics[width=0.3\textwidth]{ 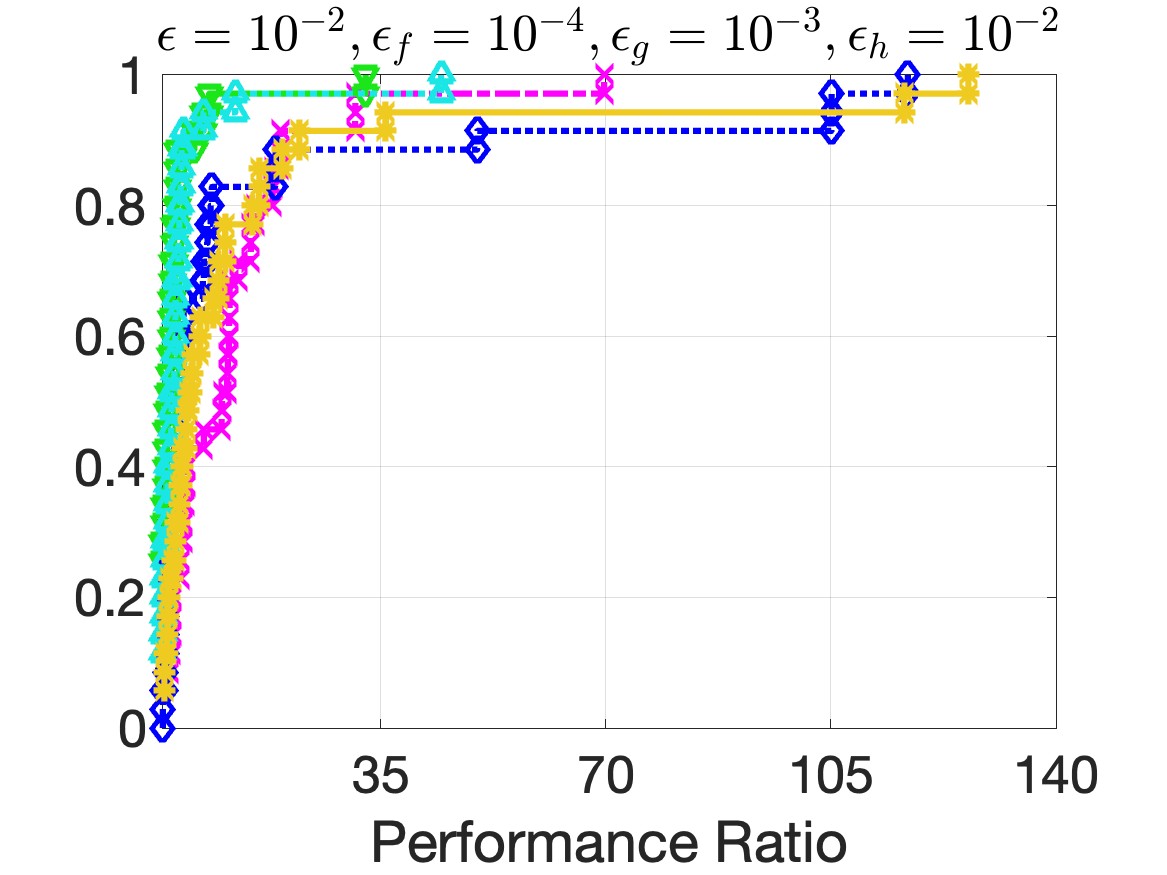}}\\
\subfigure{\includegraphics[width=0.3\textwidth]{ 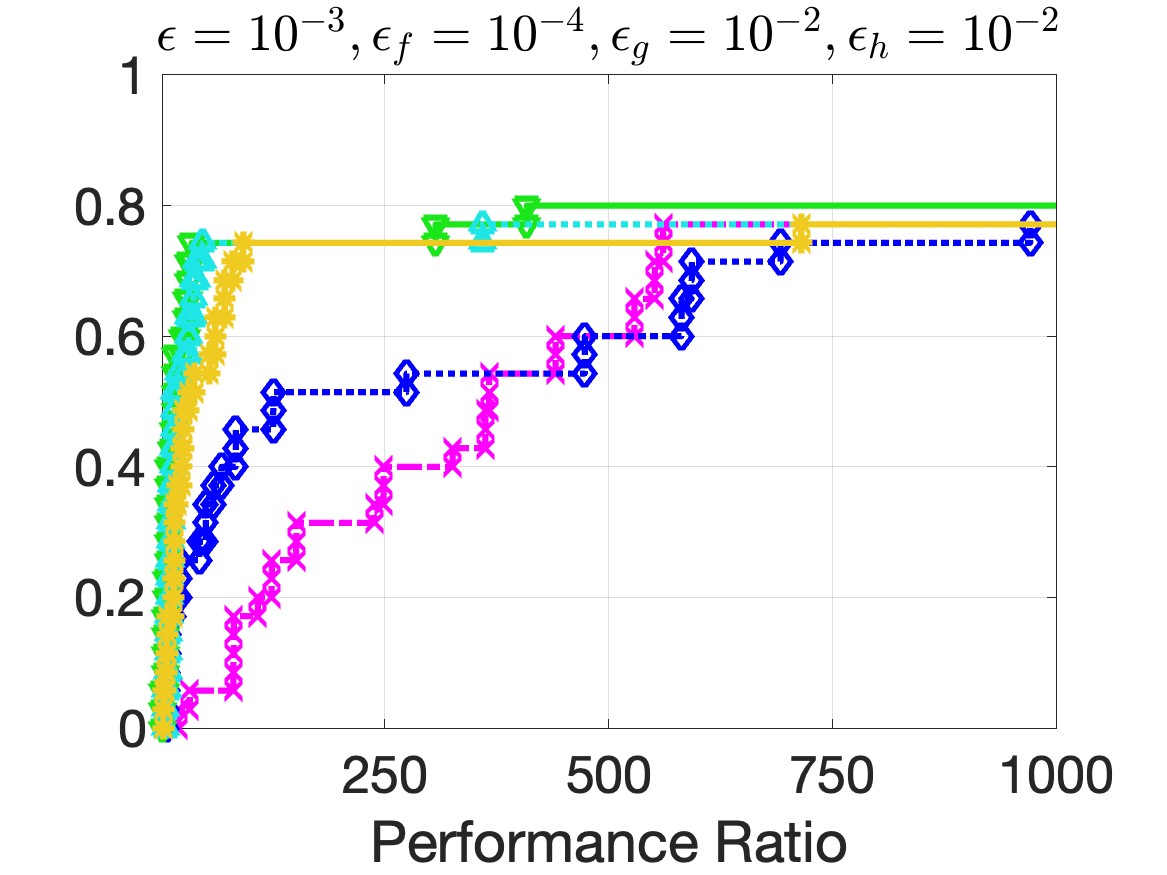}}
\subfigure{\includegraphics[width=0.3\textwidth]{ 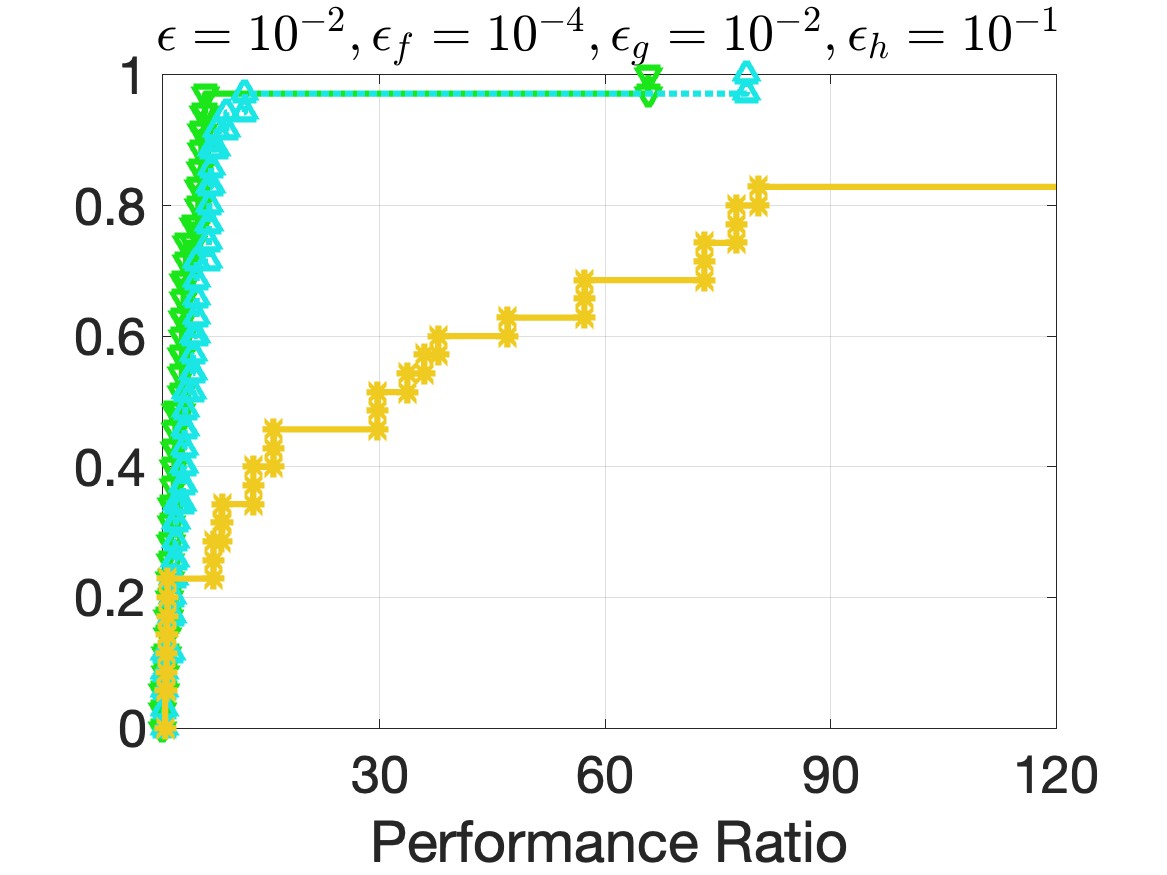}}
\subfigure{\includegraphics[width=0.3\textwidth]{ 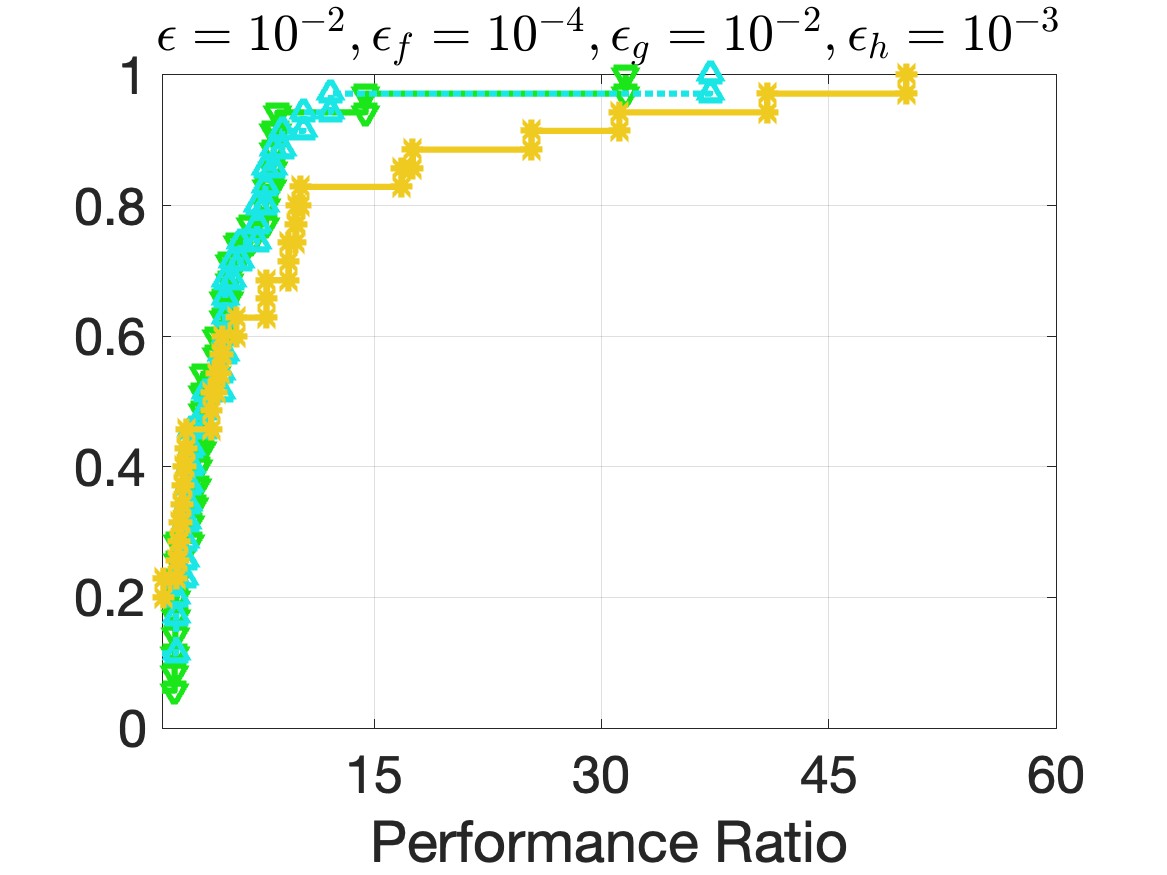}}\\
\subfigure{\includegraphics[width=0.95\textwidth]{ label2}}
\caption{Performance profiles with noise following a $t_4$-distribution. Each line represents a different method. The first column corresponds to the default irreducible noise with varying $\epsilon$. Each row of the last two columns corresponds to varying $\epsilon_f$, $\epsilon_g$, $\epsilon_h$, respectively.}
\label{fig:cutest3}\vskip-0.2cm
\end{figure}

Results are presented as performance profiles with respect to iterations. As in \cite{Berahas2025Sequential}, we employ the convergence metric $\|\nabla\L(\bx_0)\| - \|\nabla\L(\bx)\| \geq (1 - \varepsilon_{pp})  \left(\|\nabla\L(\bx_0)\| - \|\nabla\L_b\|\right) $. Here, $\bx_0$ is the initial iterate, and $\|\nabla\L_b\|$ represents the best KKT residual value identified by any algorithm for the given problem instance within the iteration budget in the experiments of Section \ref{subsec:5.2} (no irreducible noise). The prescribed tolerance is set to $\varepsilon_{pp} = 10^{-3}$.
This convergence metric extends the approach in \cite{More2009Benchmarking} to accommodate the constrained optimization setting and the nonconvexity of the objective function. The algorithm terminates when any of the following criteria are met: the stopping time is reached, the convergence metric is satisfied, or the iteration budget is exhausted.

\begin{figure}[t]
\centering
\subfigure{\includegraphics[width=0.3\textwidth]{ 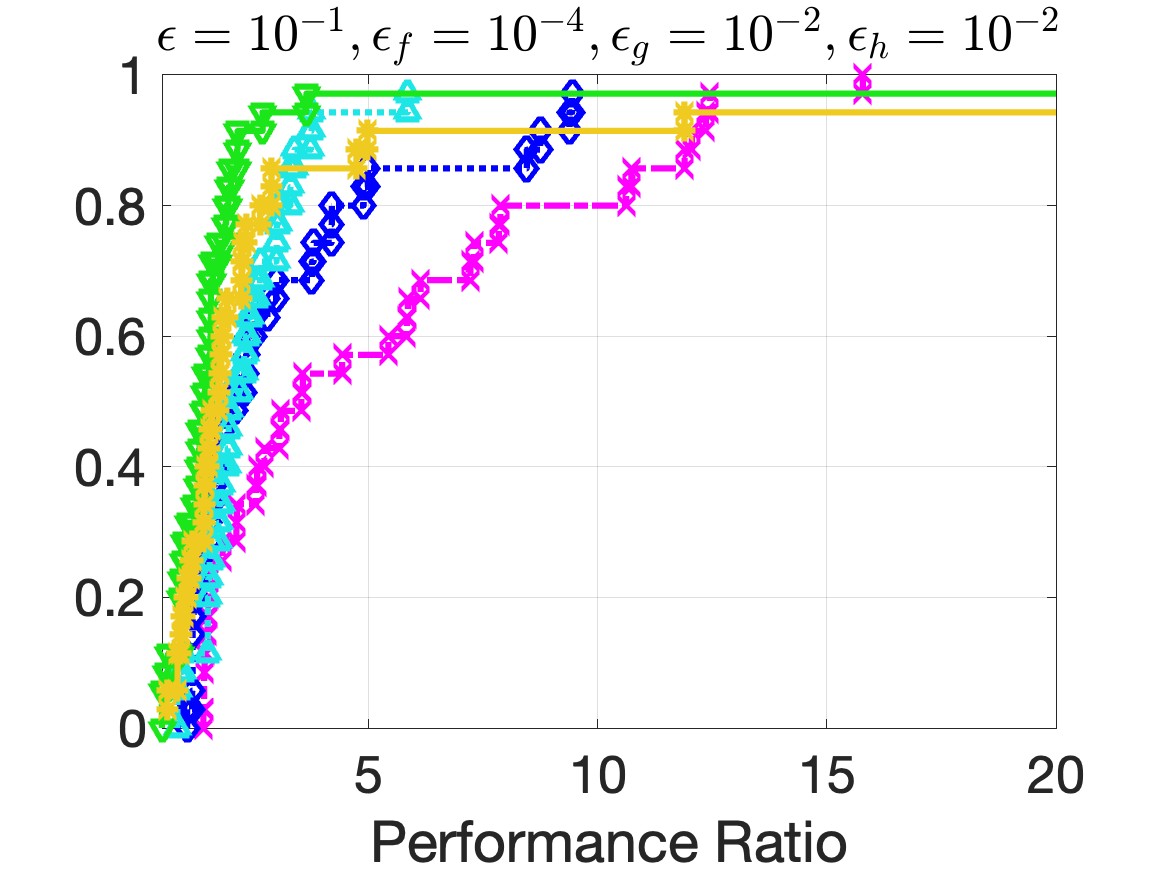}}
\subfigure{\includegraphics[width=0.3\textwidth]{ 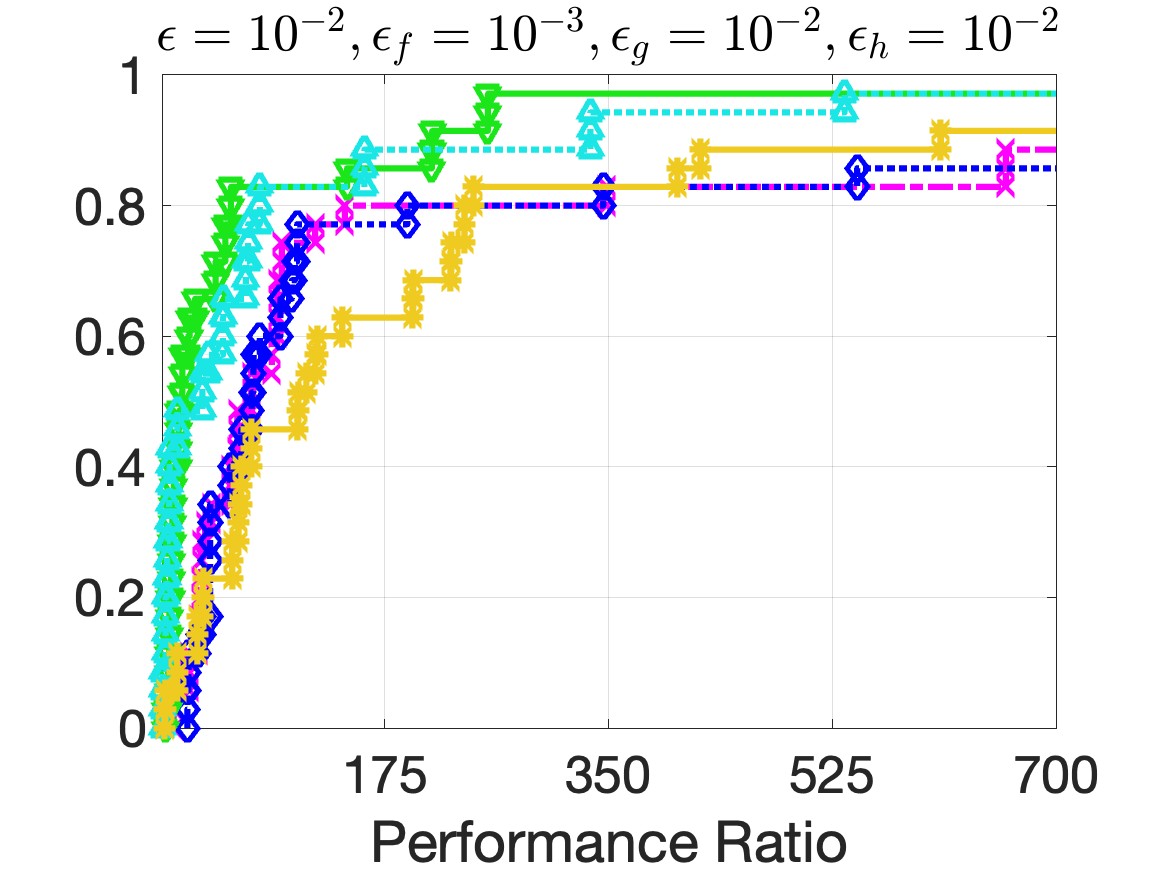}}
\subfigure{\includegraphics[width=0.3\textwidth]{ 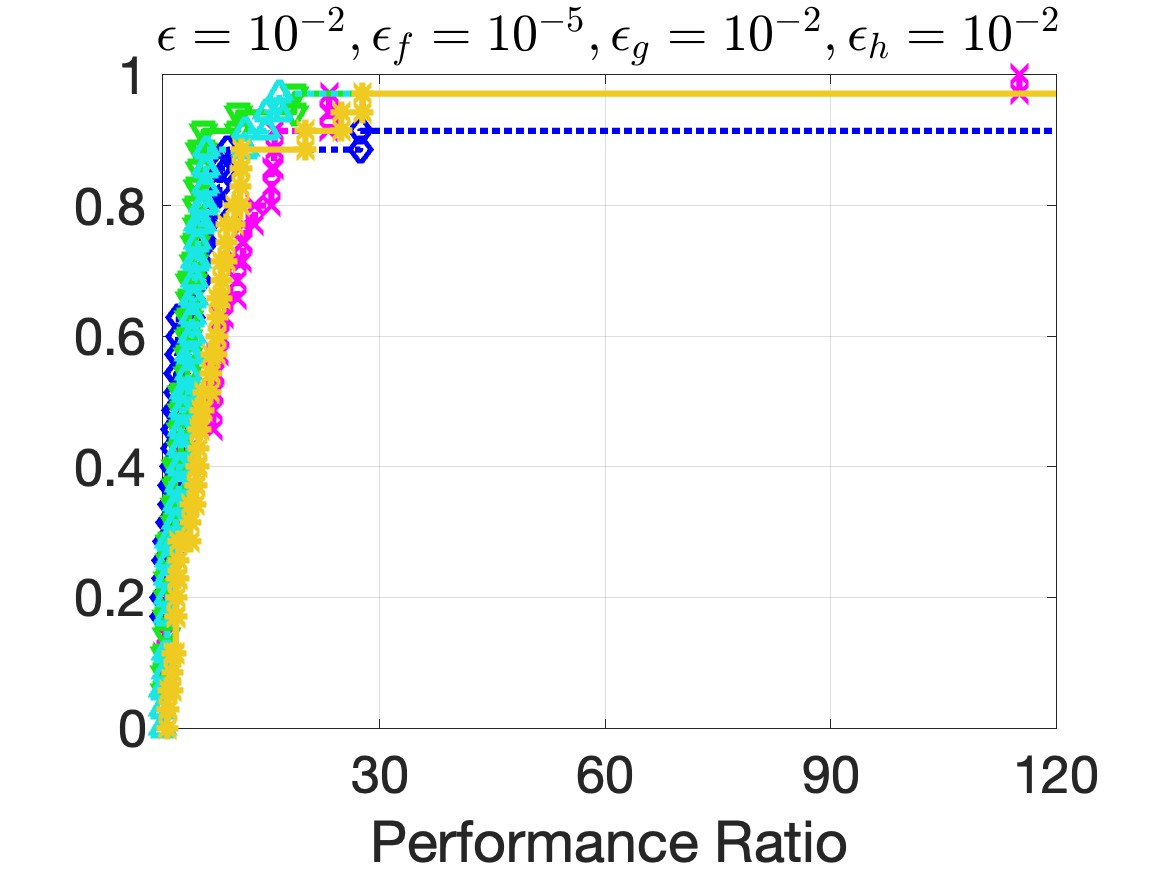}}\\	
\subfigure{\includegraphics[width=0.3\textwidth]{ 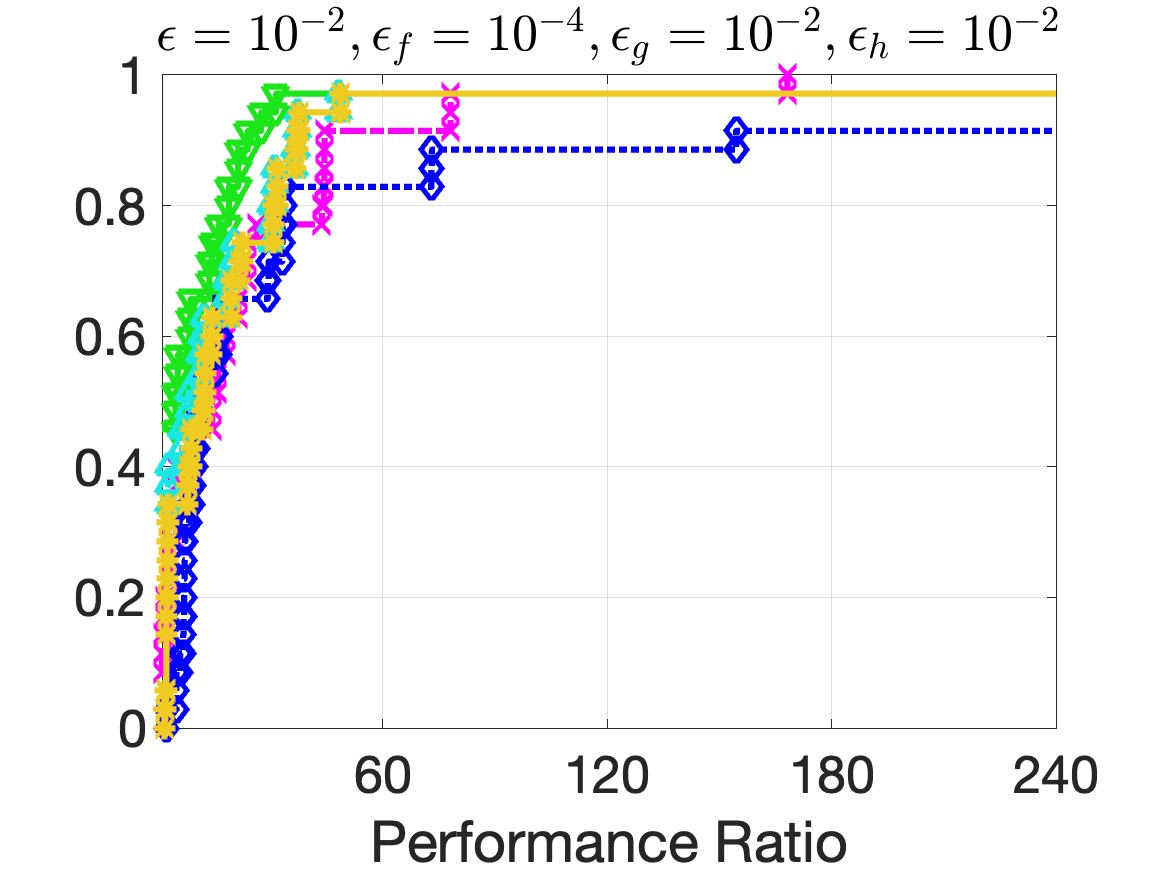}}
\subfigure{\includegraphics[width=0.3\textwidth]{ 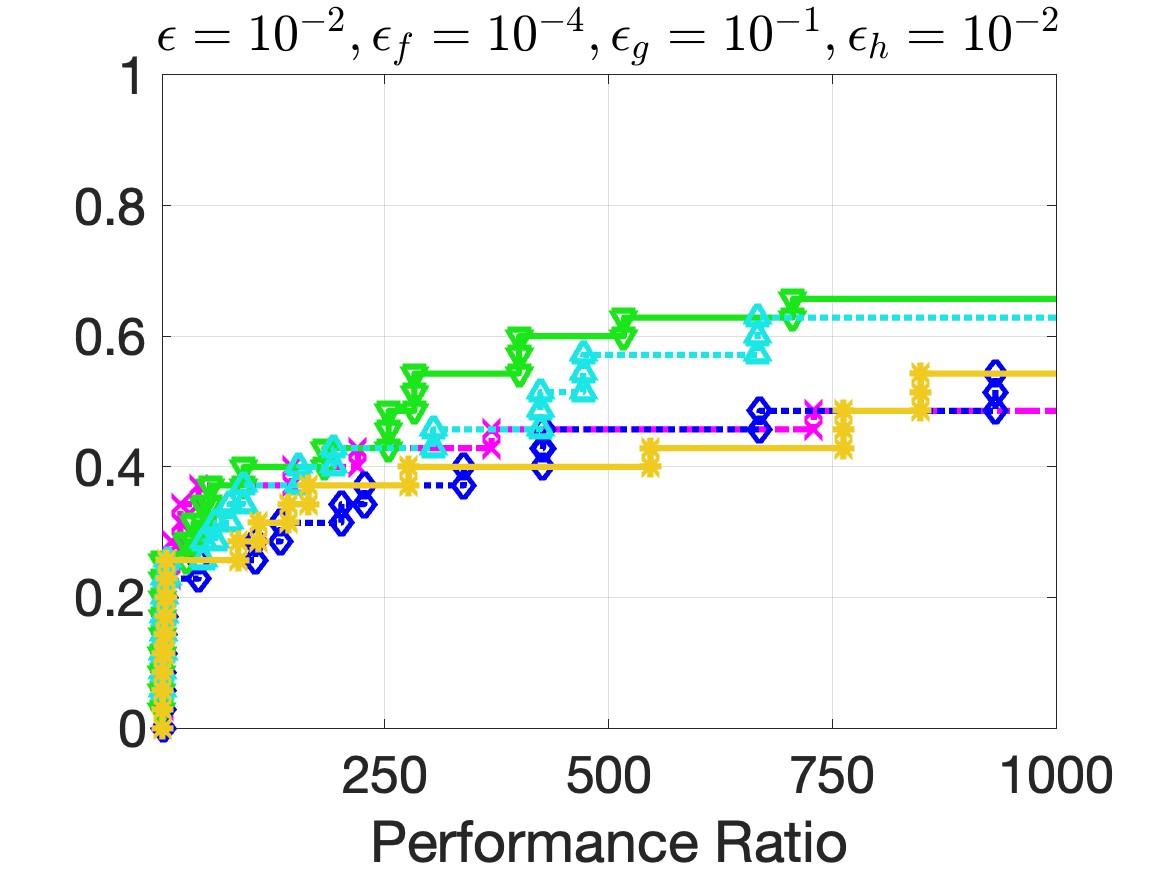}}  \subfigure{\includegraphics[width=0.3\textwidth]{ 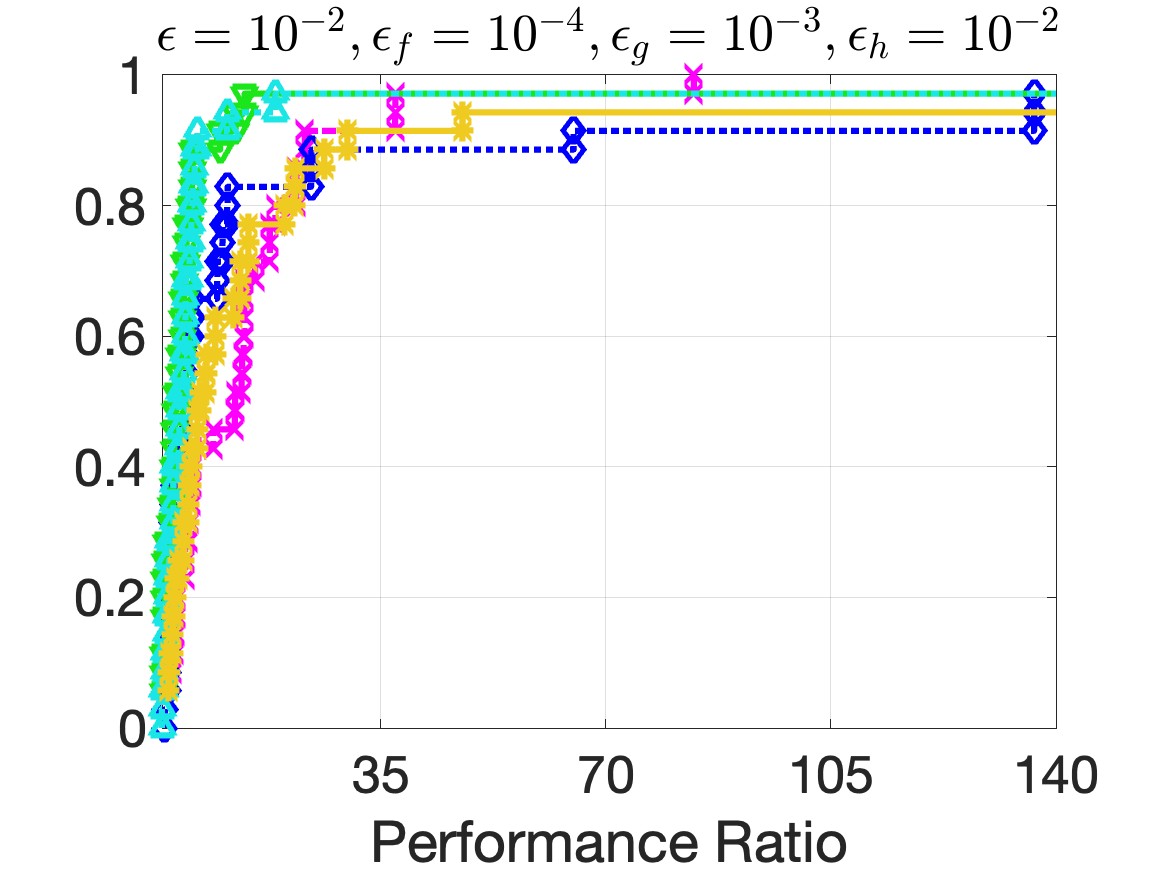}}\\
\subfigure{\includegraphics[width=0.3\textwidth]{ 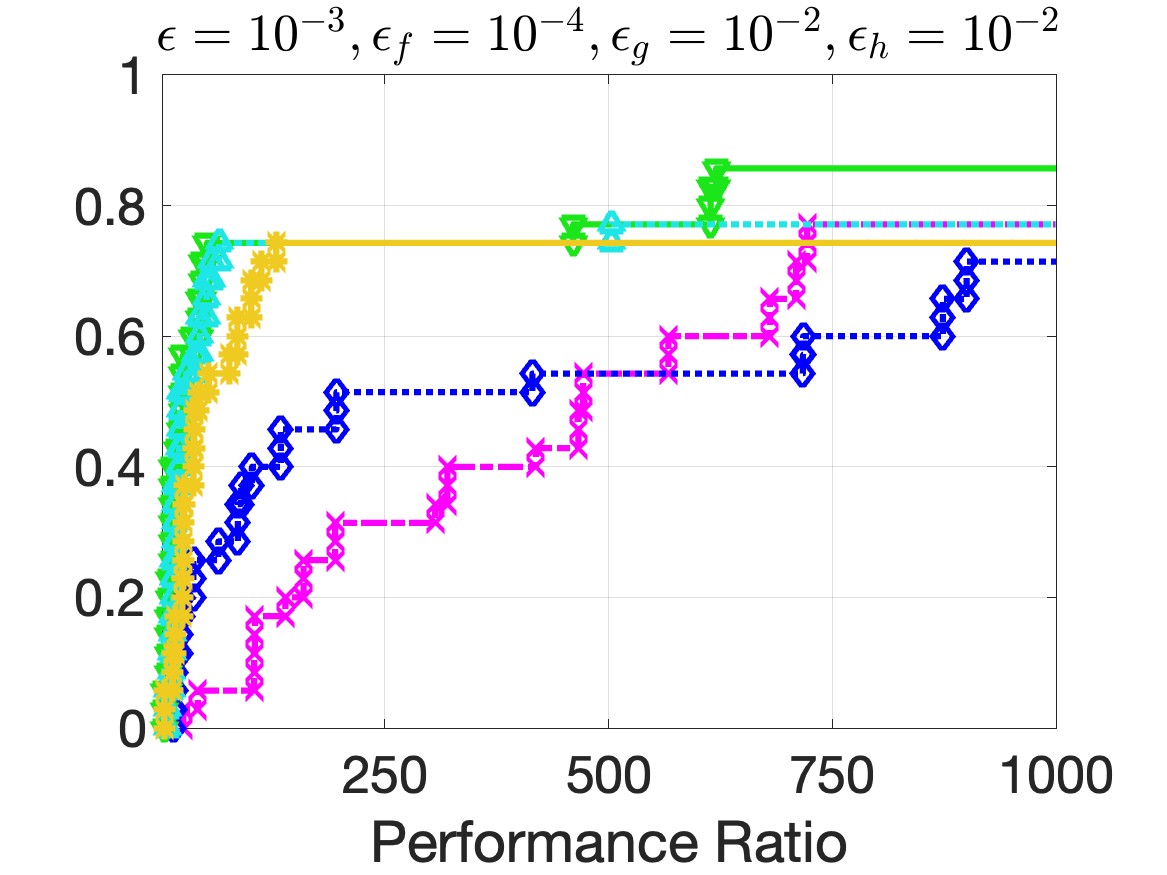}}
\subfigure{\includegraphics[width=0.3\textwidth]{ 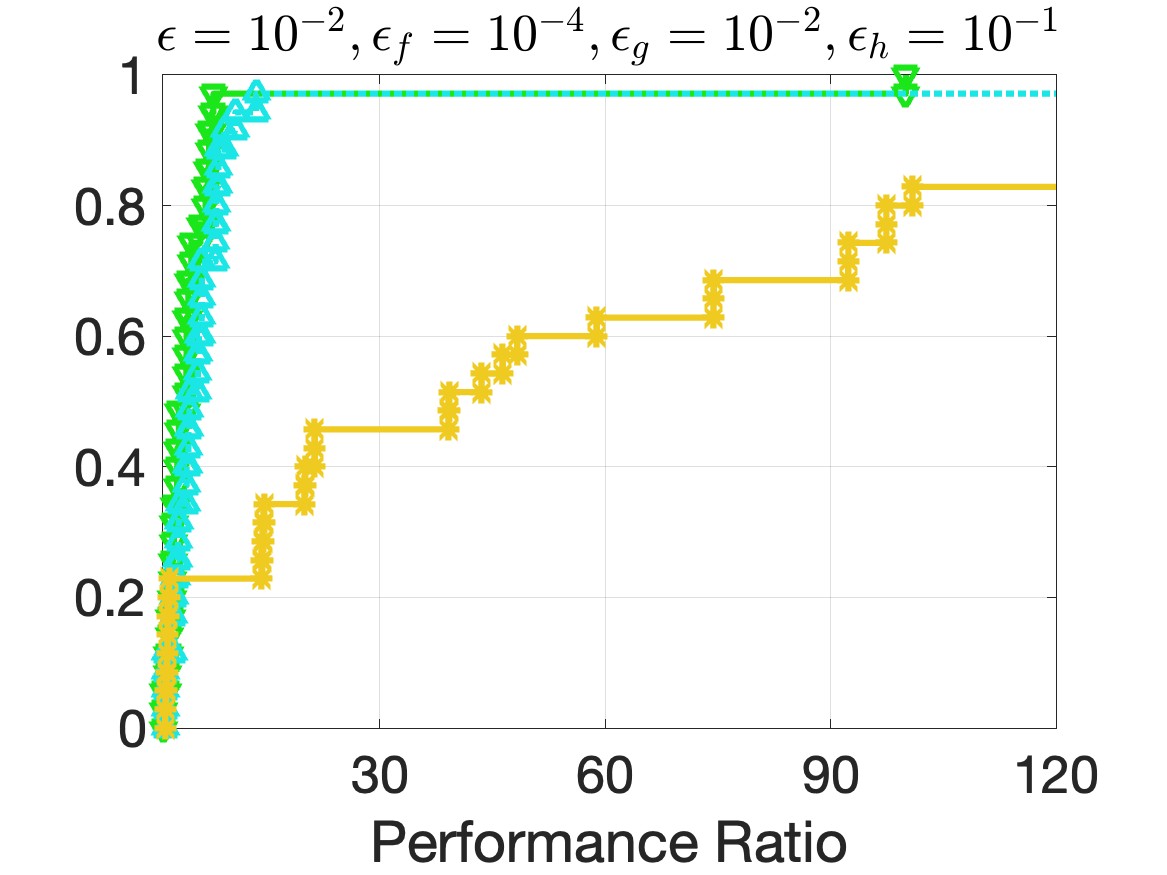}}
\subfigure{\includegraphics[width=0.3\textwidth]{ 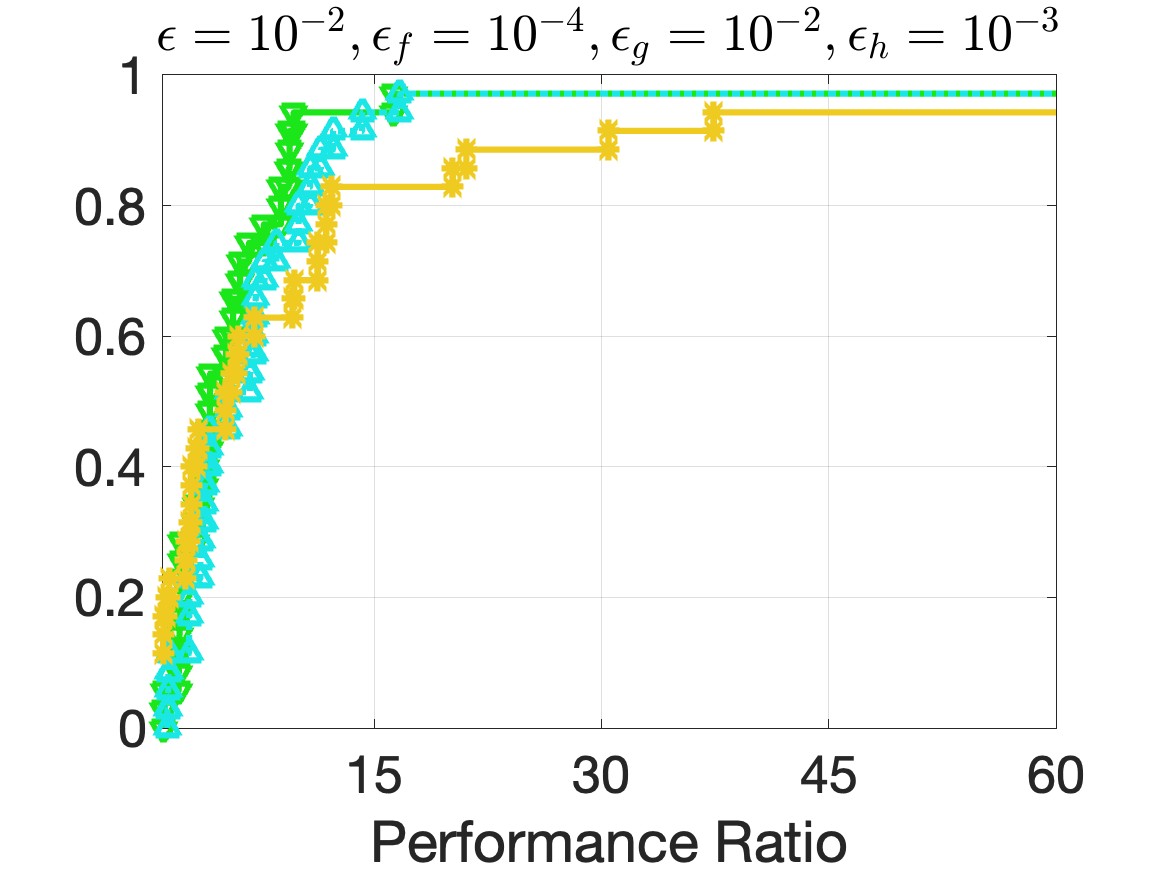}}\\
\subfigure{\includegraphics[width=0.95\textwidth]{ label2}}
\caption{Performance profiles with noise following a $t_2$-distribution. Each line represents a different method. The first column corresponds to the default irreducible noise with varying $\epsilon$. Each row of the last two columns corresponds to varying $\epsilon_f$, $\epsilon_g$, $\epsilon_h$, respectively.}
\label{fig:cutest4}\vskip-0.2cm
\end{figure}

\begin{figure}[t]
\centering
\subfigure{\includegraphics[width=0.3\textwidth]{ 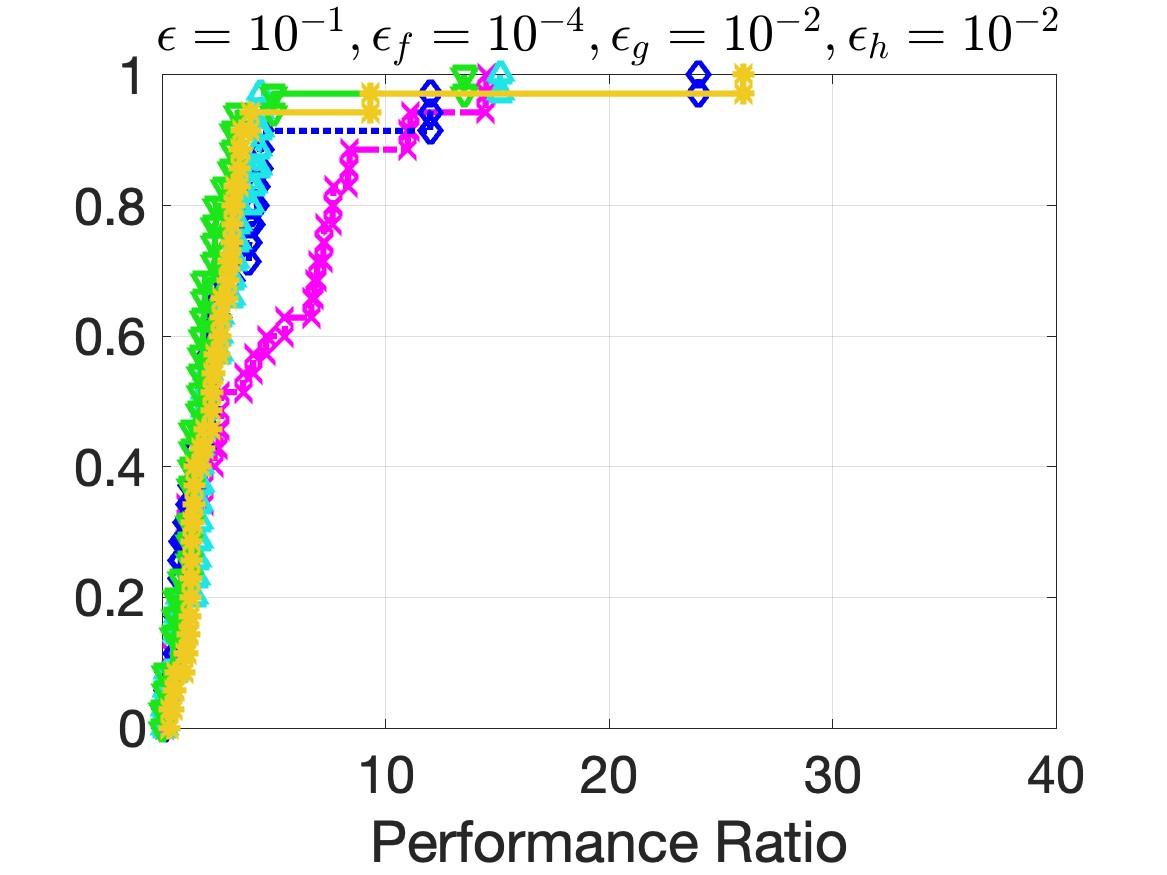}}
\subfigure{\includegraphics[width=0.3\textwidth]{ 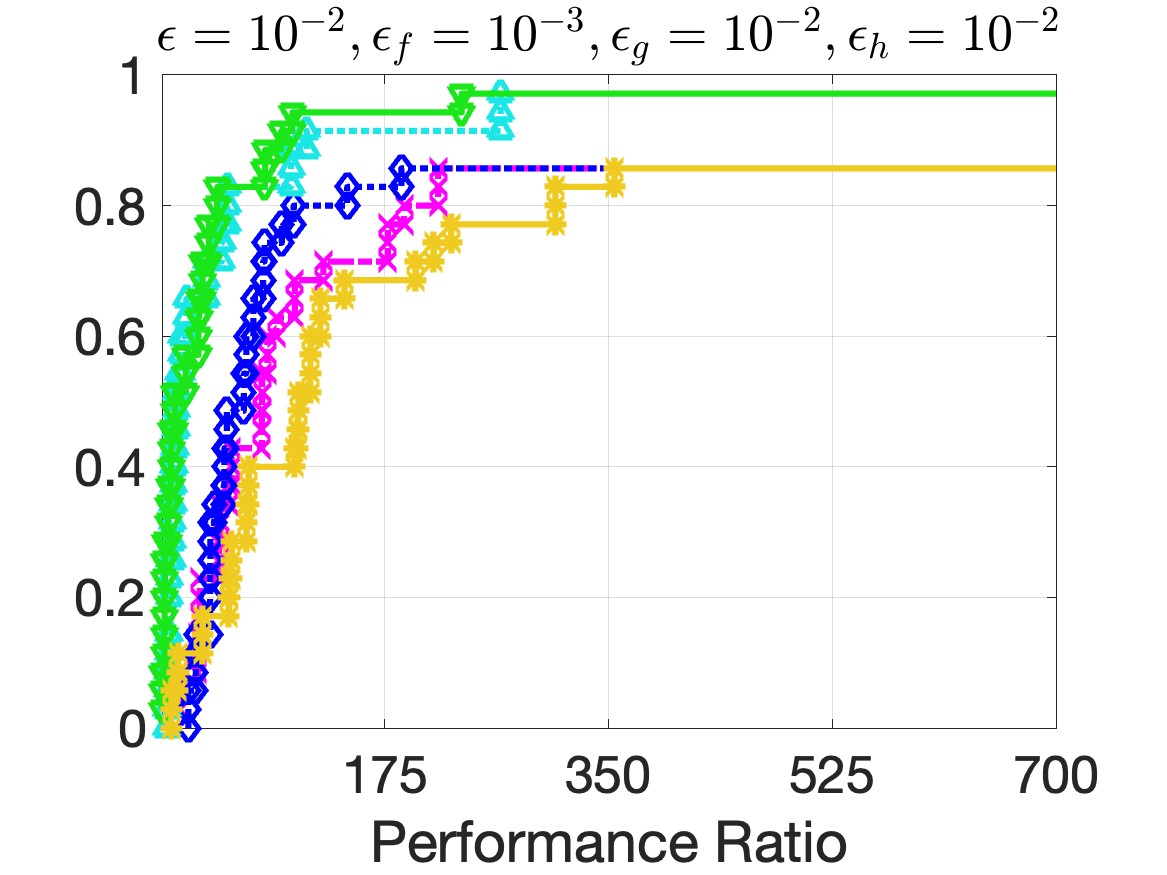}}
\subfigure{\includegraphics[width=0.3\textwidth]{ 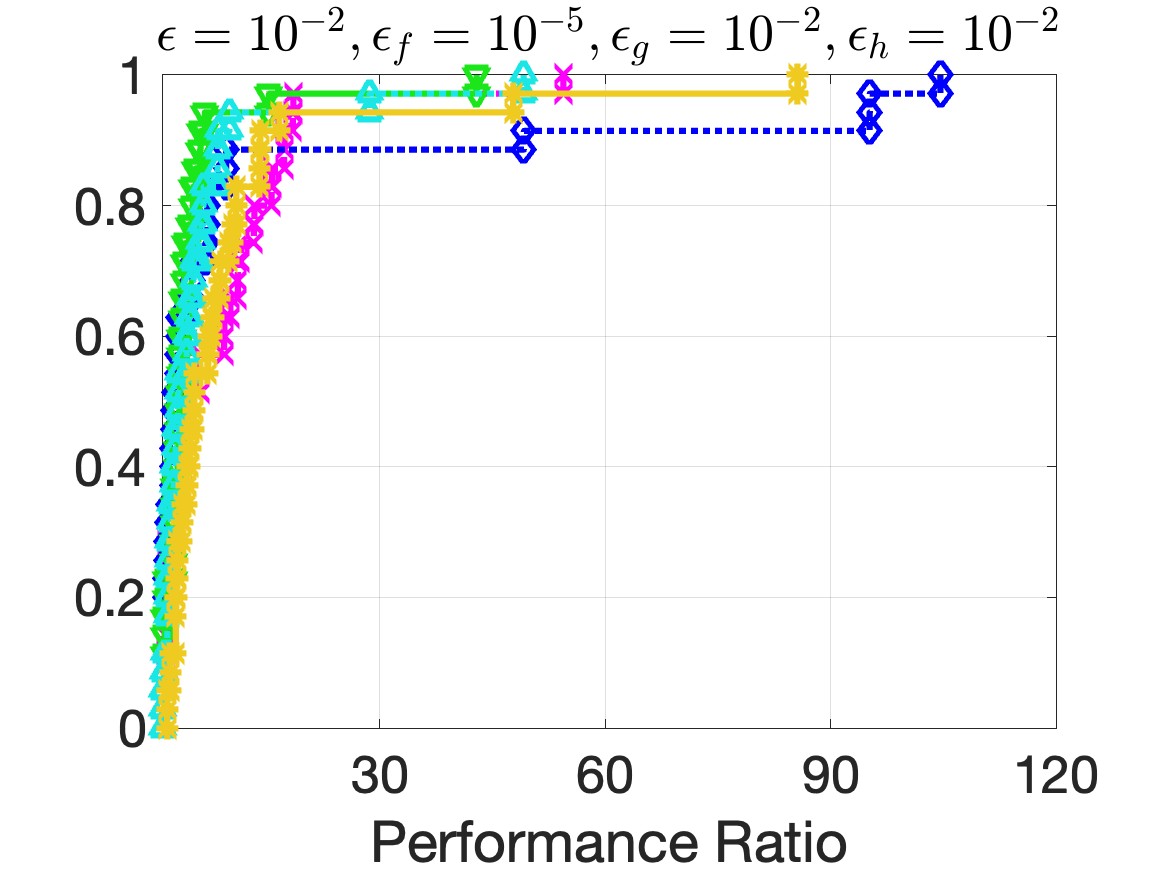}}\\	
\subfigure{\includegraphics[width=0.3\textwidth]{ 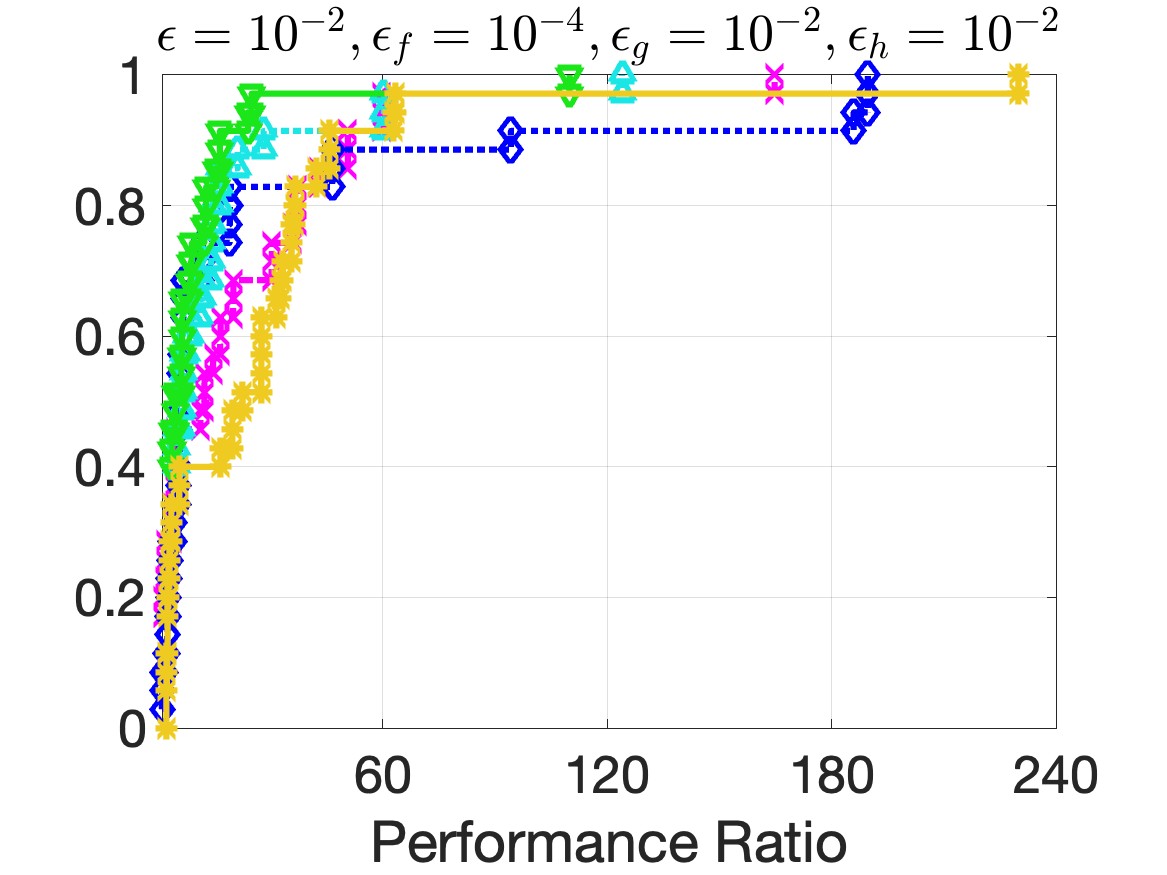 }}
\subfigure{\includegraphics[width=0.3\textwidth]{ 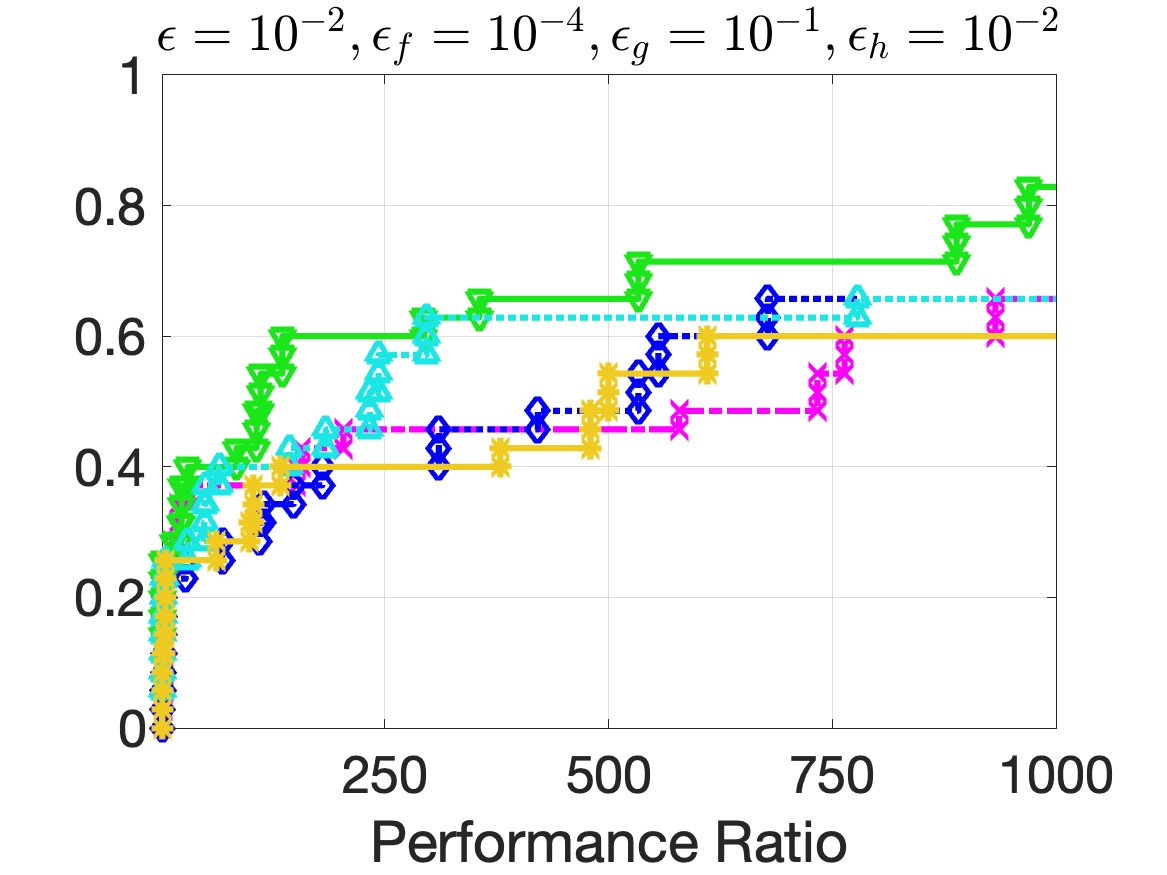 }}  \subfigure{\includegraphics[width=0.3\textwidth]{ 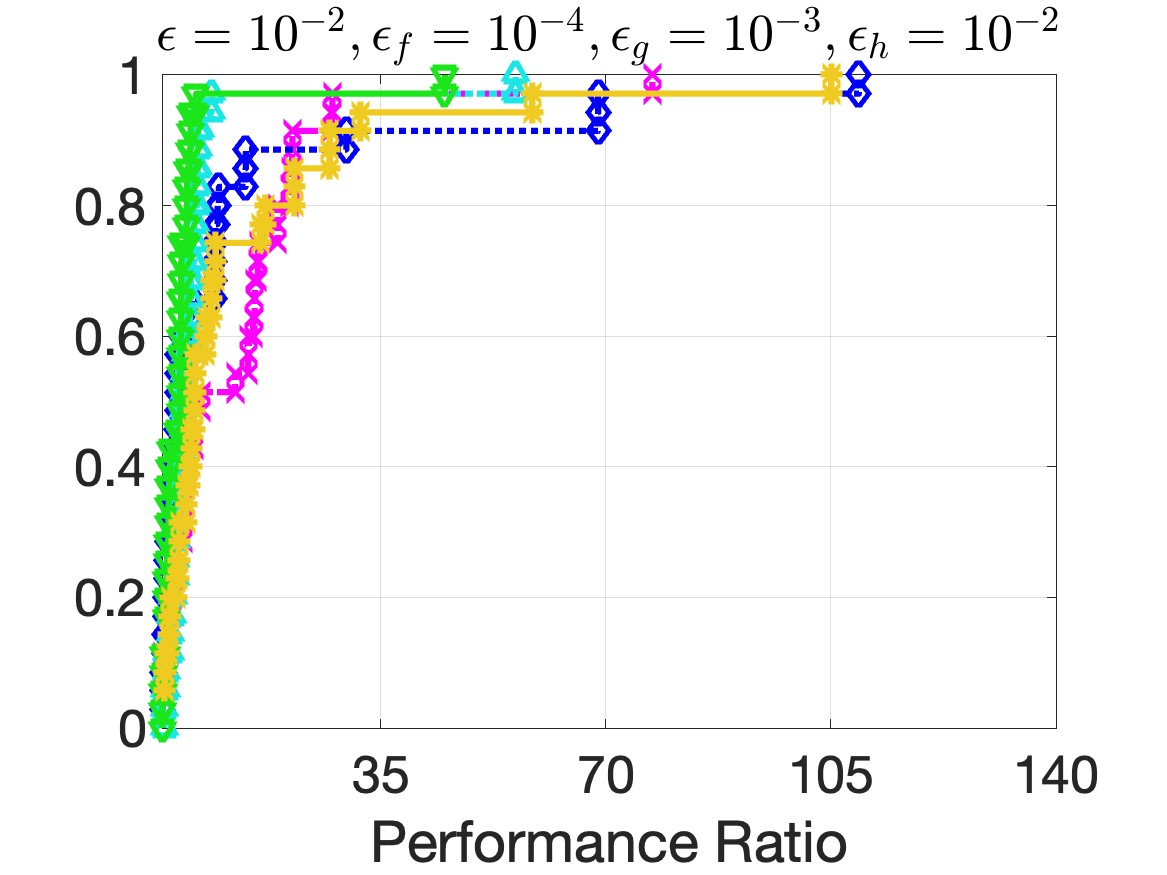 }}\\
\subfigure{\includegraphics[width=0.3\textwidth]{ 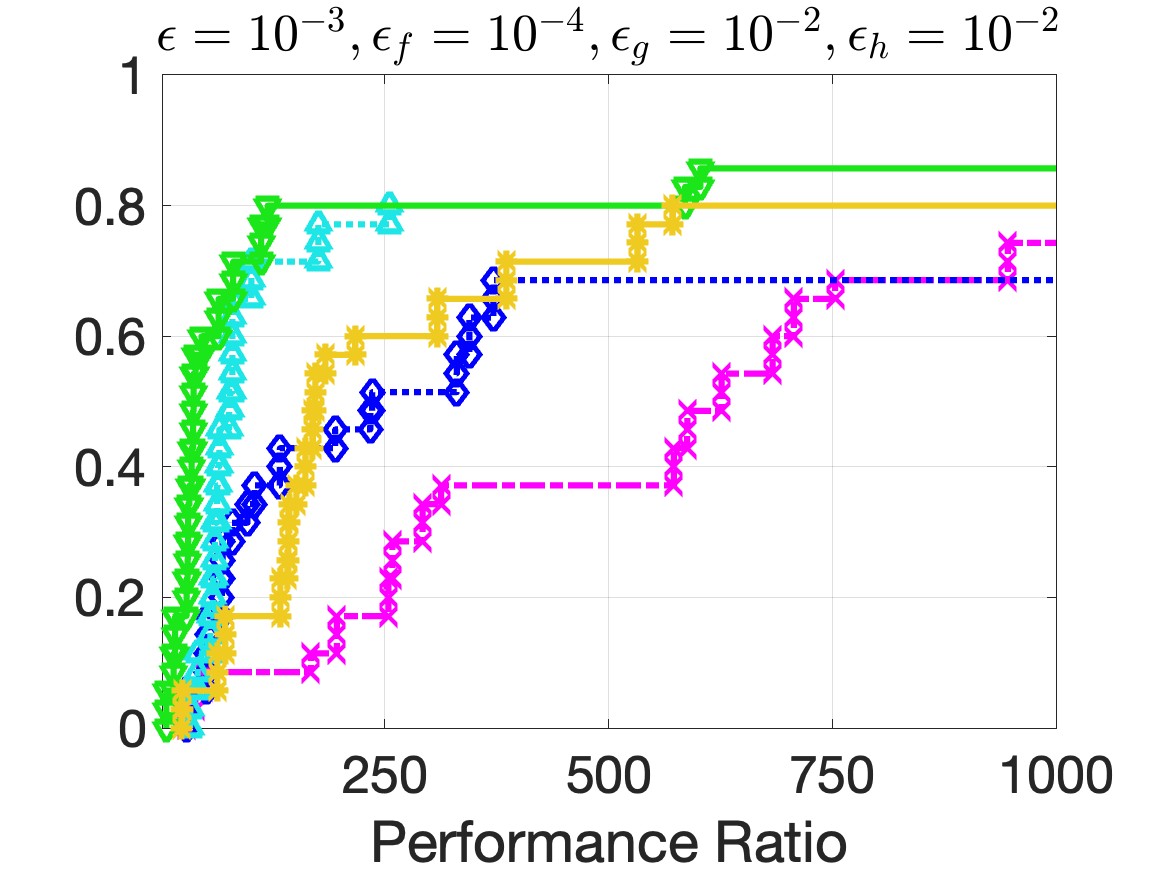 }}
\subfigure{\includegraphics[width=0.3\textwidth]{ 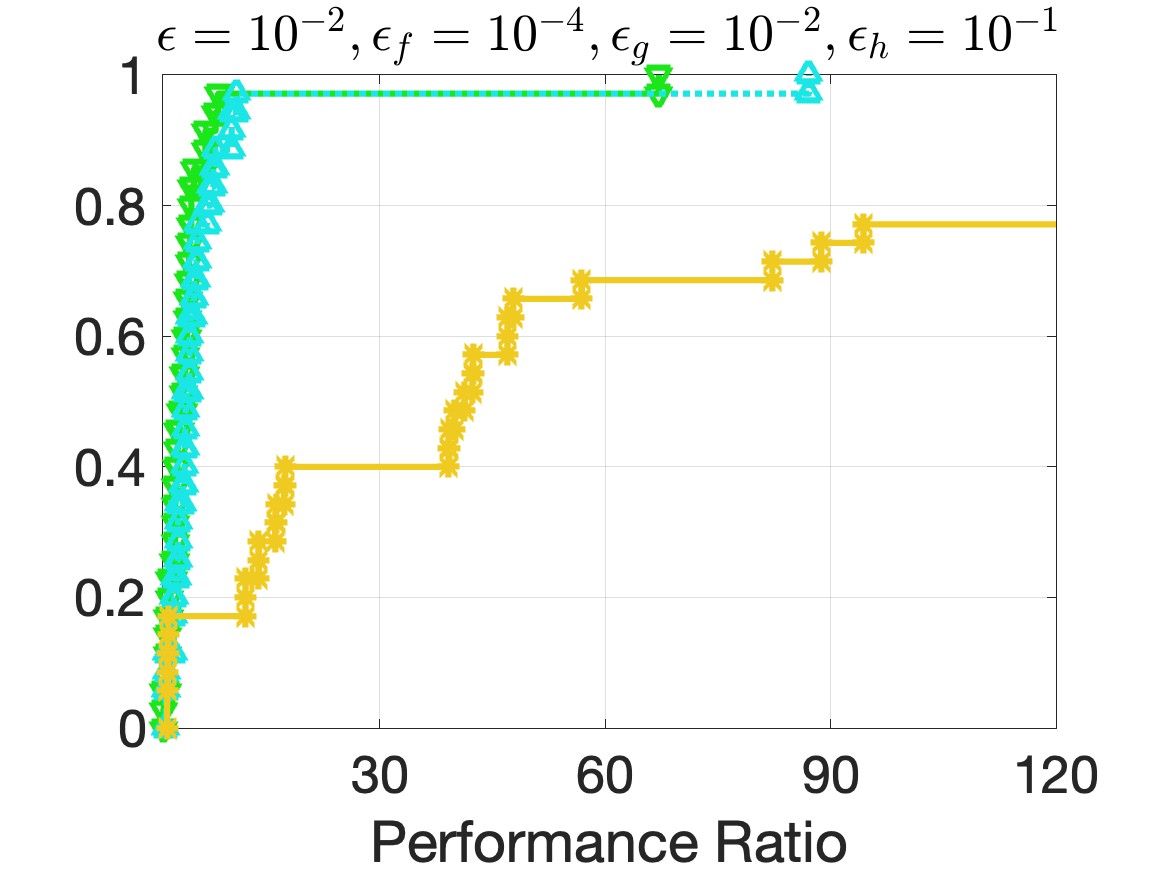 }}
\subfigure{\includegraphics[width=0.3\textwidth]{ 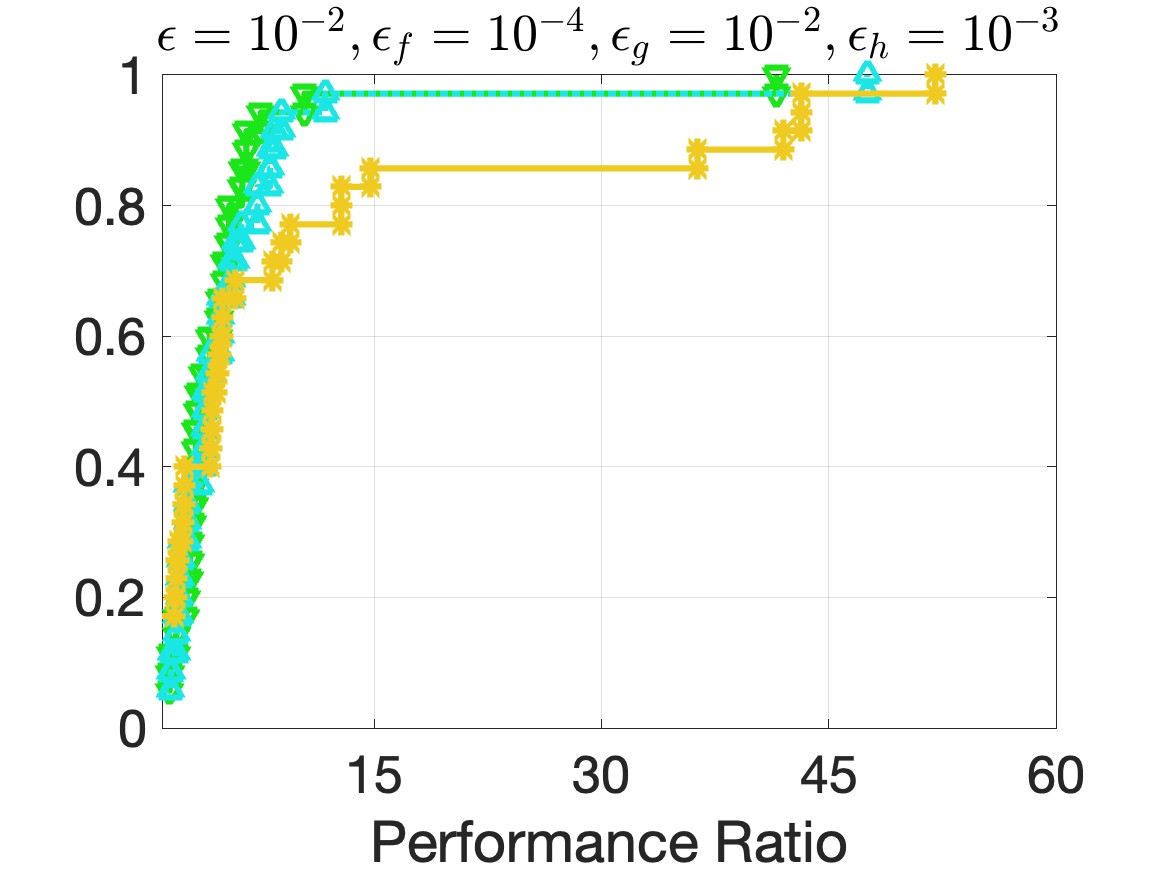 }}\\
\subfigure{\includegraphics[width=0.95\textwidth]{ label2}}
\caption{Performance profiles with noise following a log-normal-distribution. Each line represents a different method. The first column corresponds to the default irreducible noise with varying $\epsilon$. Each row of the last two columns corresponds to varying $\epsilon_f$, $\epsilon_g$, $\epsilon_h$, respectively.}
\label{fig:cutest5}\vskip-0.2cm
\end{figure}

We present the experimental results in Figures \ref{fig:cutest2}, \ref{fig:cutest3}, \ref{fig:cutest4}, \ref{fig:cutest5}, \ref{fig:cutest6}, and \ref{fig:cutest7}, corresponding to noise from the normal distribution, $t_4$-distribution, $t_2$-distribution, log-normal distribution, Weibull distribution, and Cauchy distribution, respectively. In each figure, the first column represents the performance of all methods under the default irreducible noise levels with varying $\epsilon$. The last two columns represent the performance of TR-SSQP and TR-SSQP2 under varying irreducible noise levels with the default $\epsilon$.
For clarity, we trim the performance ratio ($x$-axis) based on two conditions: (1) if all methods converge, we trim the performance ratio after the last method converges; (2) if some methods diverge on some problems, we trim the performance ratio when no method converges on new problems within the iteration budget.

\begin{figure}[t]
\centering
\subfigure{\includegraphics[width=0.3\textwidth]{ 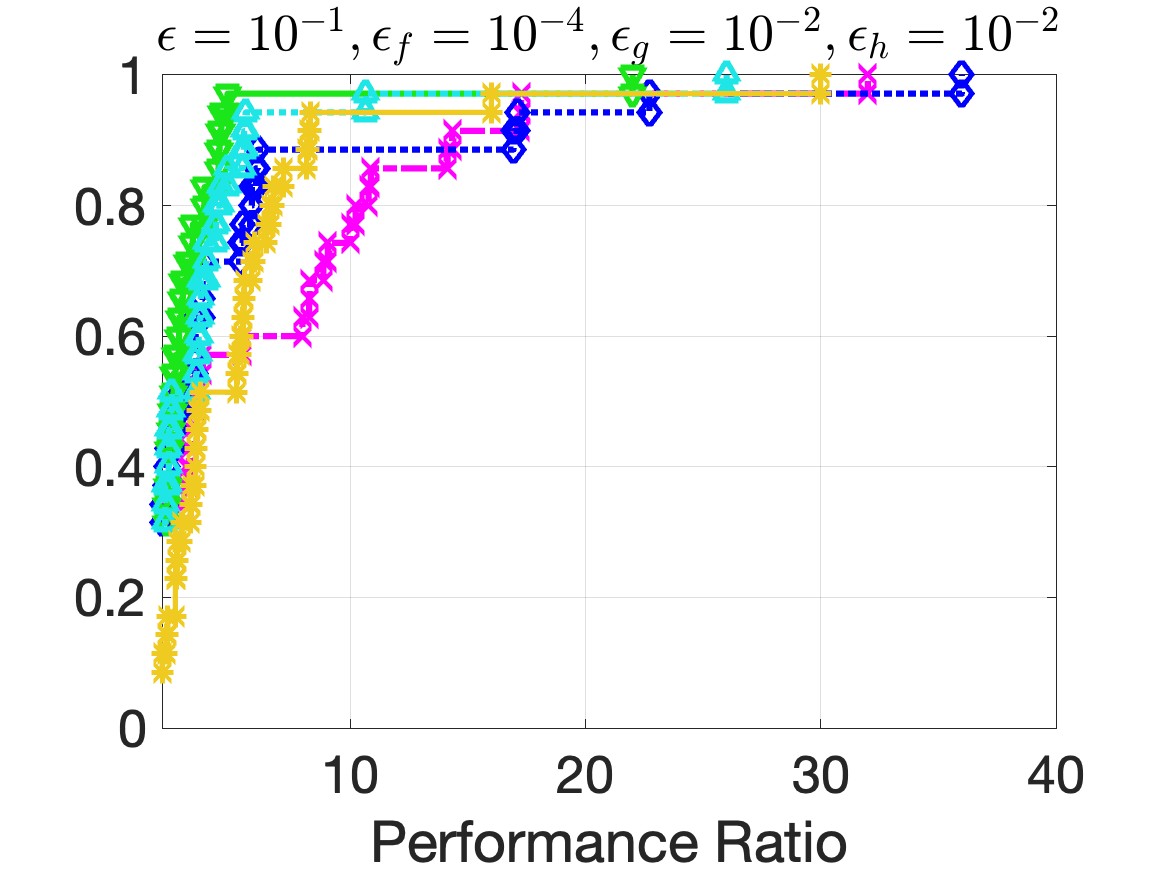}}
\subfigure{\includegraphics[width=0.3\textwidth]{ 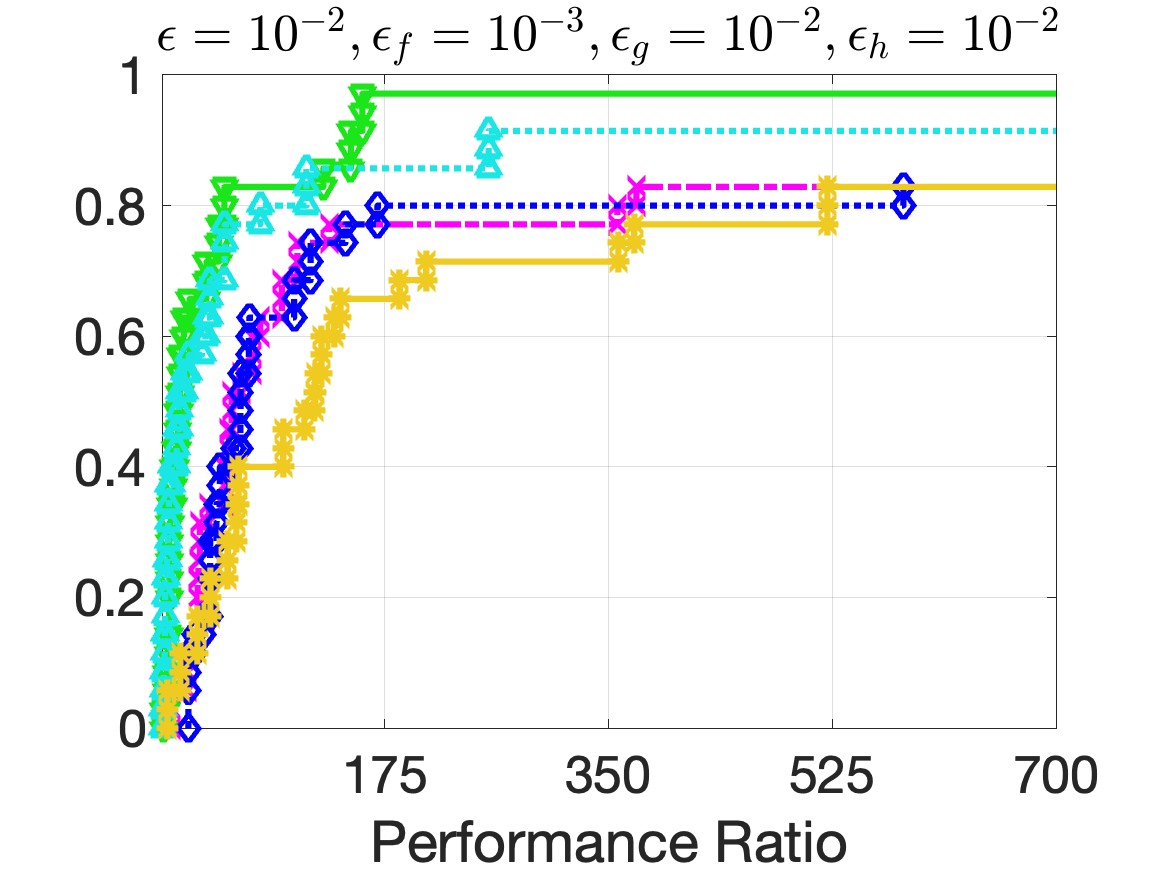}}
\subfigure{\includegraphics[width=0.3\textwidth]{ 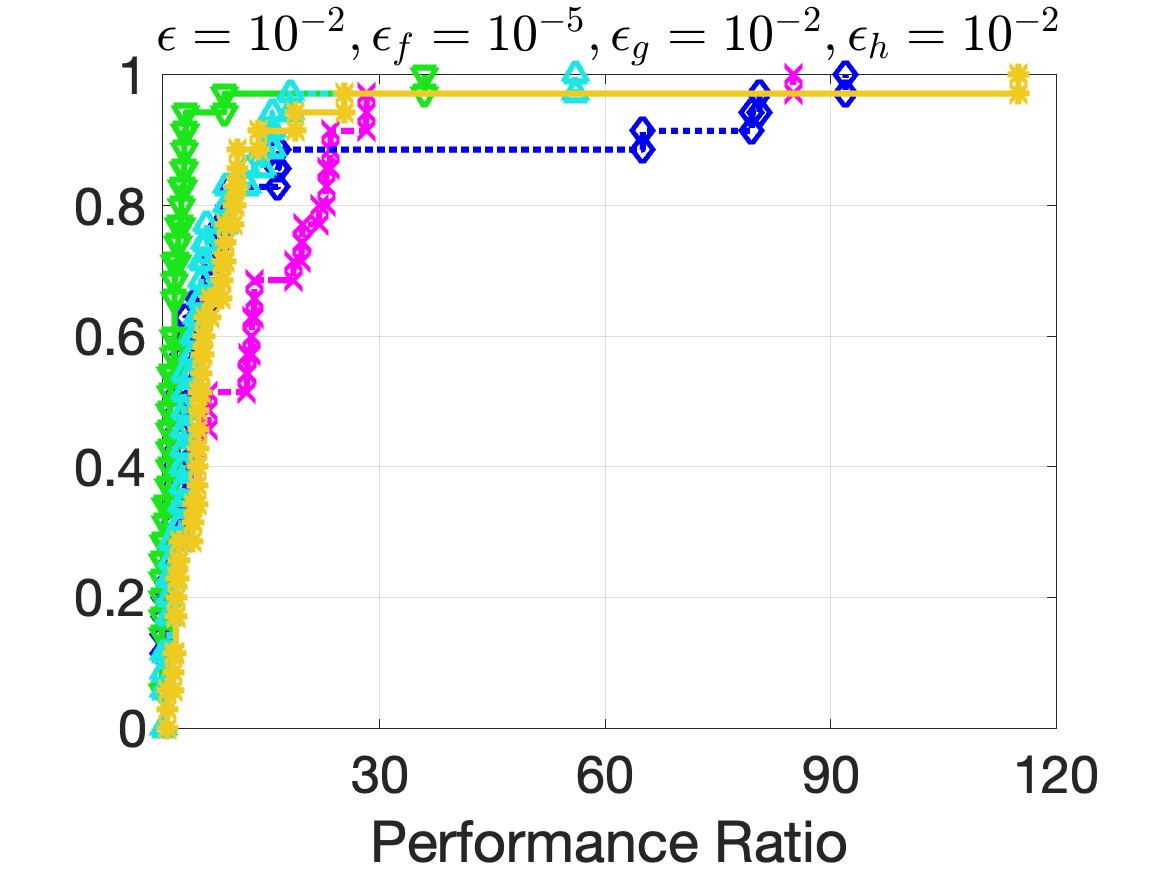}}\\	
\subfigure{\includegraphics[width=0.3\textwidth]{ 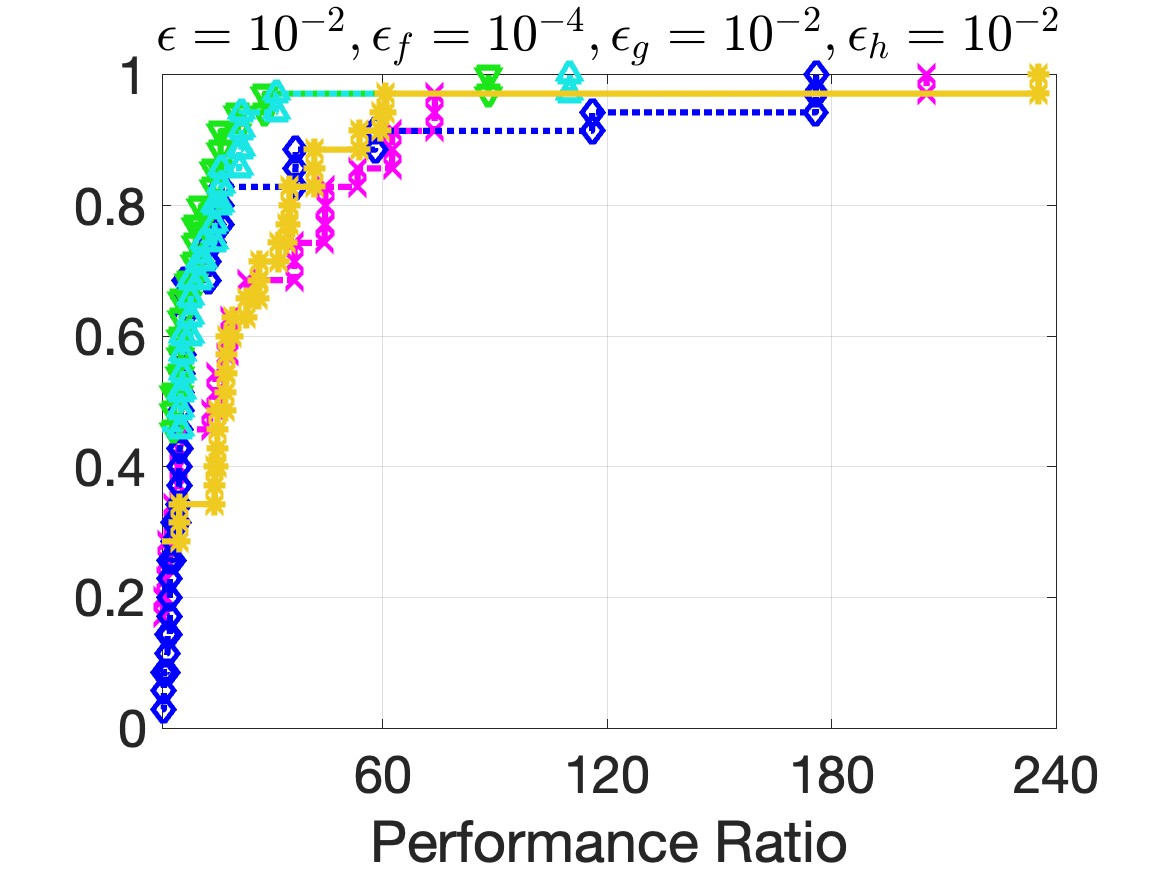}}
\subfigure{\includegraphics[width=0.3\textwidth]{ 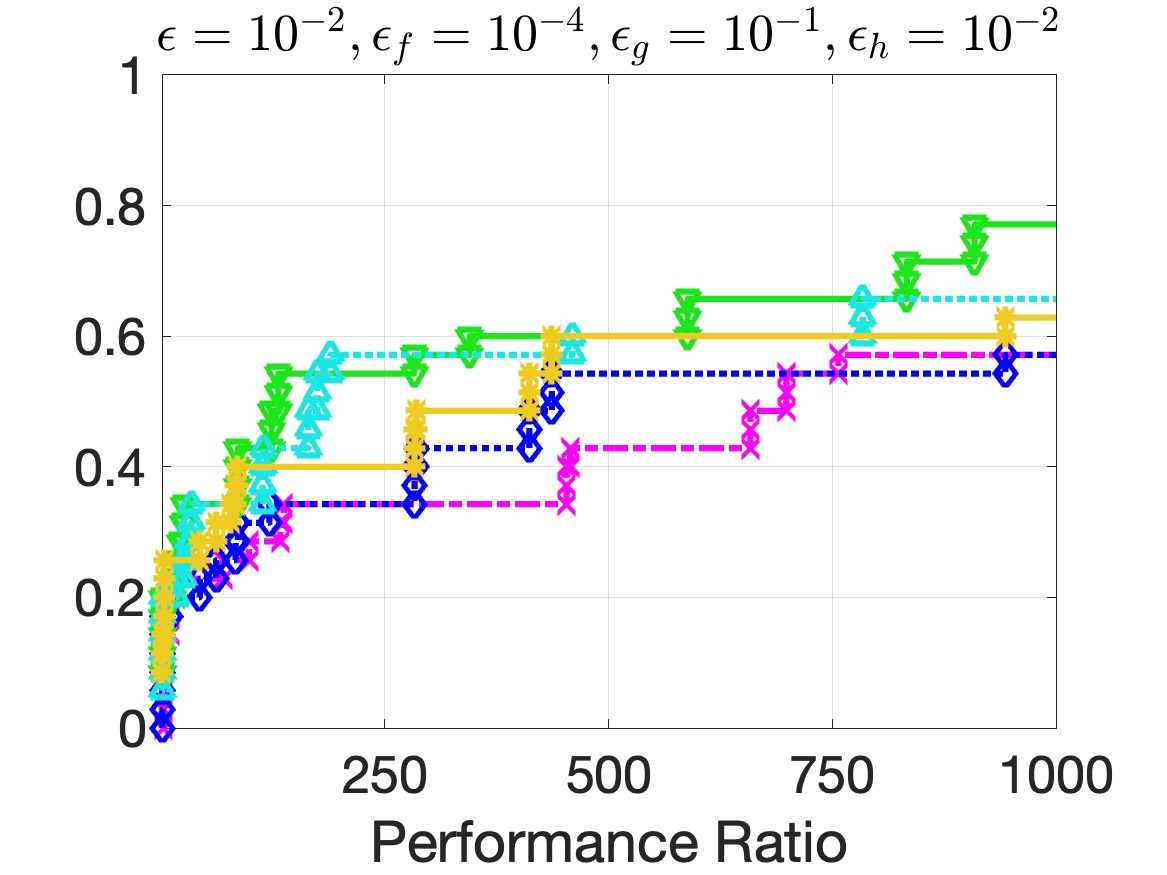}}  \subfigure{\includegraphics[width=0.3\textwidth]{ 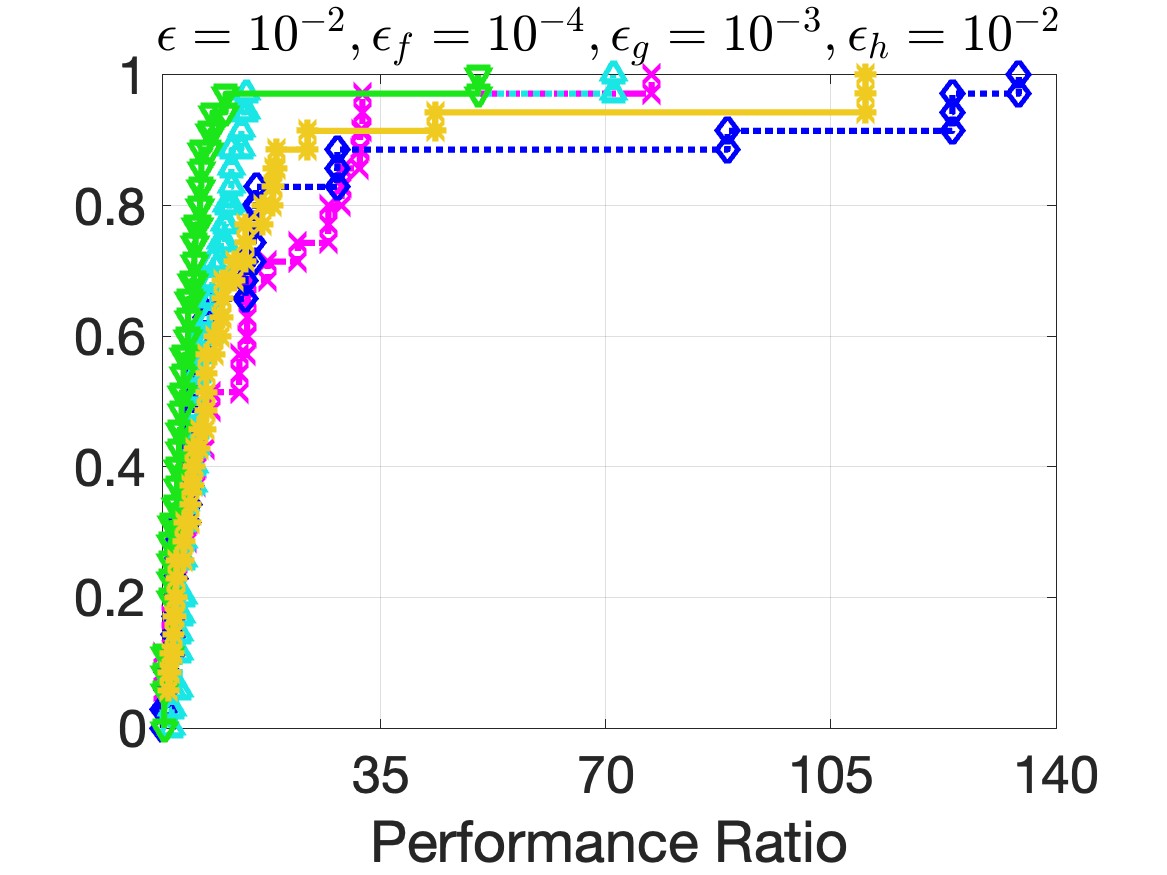}}\\
\subfigure{\includegraphics[width=0.3\textwidth]{ 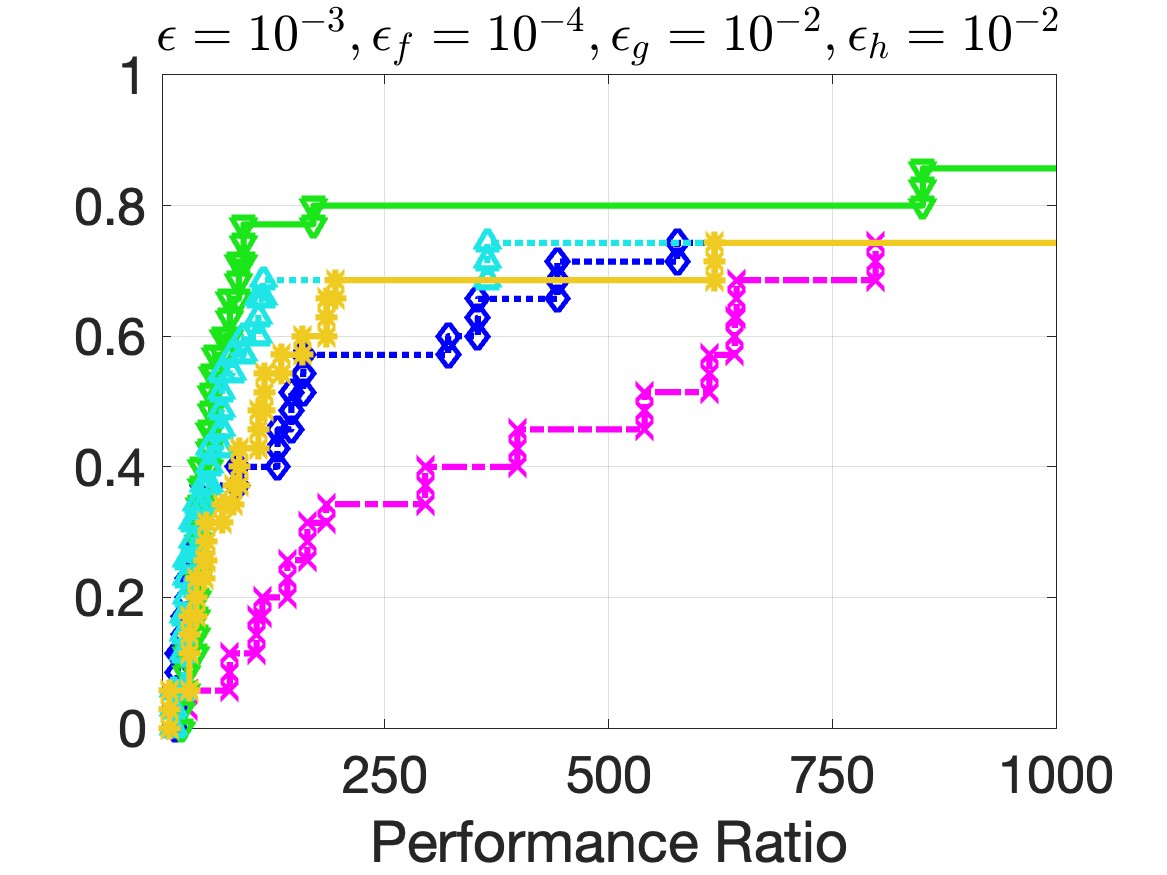}}
\subfigure{\includegraphics[width=0.3\textwidth]{ 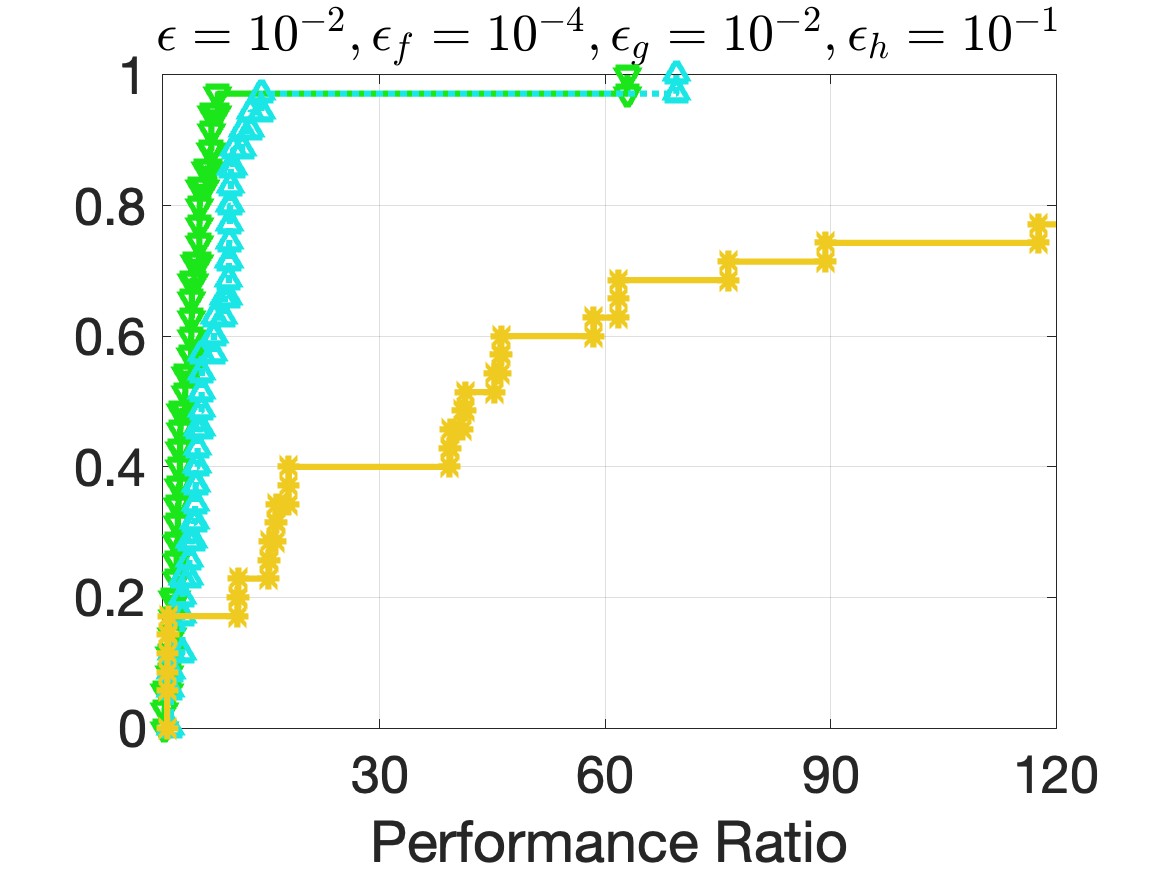}}
\subfigure{\includegraphics[width=0.3\textwidth]{ 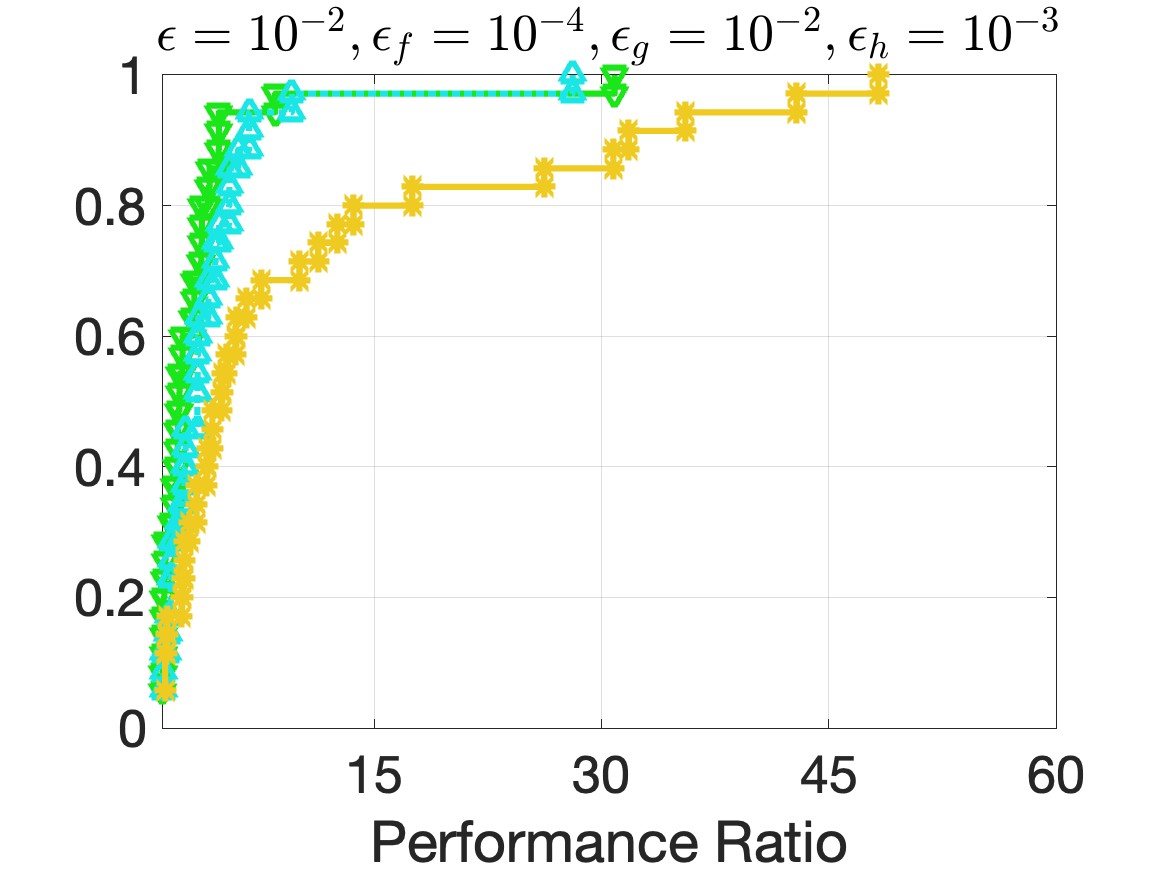}}\\
\subfigure{\includegraphics[width=0.95\textwidth]{ label2 }}
\caption{Performance profiles with noise following a Weibull distribution. Each line represents a different method. The first column corresponds to the default irreducible noise with varying $\epsilon$. Each row of the last two columns corresponds to varying $\epsilon_f$, $\epsilon_g$, $\epsilon_h$, respectively.}
\label{fig:cutest6}\vskip-0.2cm
\end{figure}

Comparing corresponding plots across the six figures, we first observe that all methods are robust to the noise distribution. For any fixed combination of $\epsilon, \epsilon_f, \epsilon_g, \epsilon_h$, changing the noise distribution results in limited differences in the performance ratios. Notably, even when the noise follows a Cauchy distribution, which does not satisfy the bounded $(1+\delta)$-moment condition, the performance ratios degrade only by a limited amount.
Next, we investigate the impact of varying $\epsilon$, $\epsilon_f$, $\epsilon_g$, $\epsilon_h$.

\vskip 0.2cm
\noindent $\bullet$ $\boldsymbol{\epsilon}:$
From the first column of the six figures, we observe that as $\epsilon$ decreases from $10^{-1}$ to $10^{-2}$, all methods require more iterations to converge, increasing by a factor of 10 to 100. This observation aligns well with our theoretical complexity bounds. However, when $\epsilon$ further decreases to $10^{-3}$, all methods fail to converge on some datasets within the iteration budget. This is primarily because, given the default irreducible noise levels, $\epsilon = 10^{-3}$ no longer satisfies the lower bound in Assumption \ref{assump:epsilon}. That being said, the impact of decreasing $\epsilon$ is less notable for TR-SSQP-AveH and TR-SSQP-EstH, highlighting the advantages of exploring Hessian information in the algorithm design.

\vskip 0.2cm
\noindent $\bullet$ $\boldsymbol{\epsilon_f}:$
Comparing the first row of the second and third columns with the default setting (middle plot in the first column) in each figure, we observe that increasing $\epsilon_f$ from $10^{-5}$ to $10^{-4}$ approximately triples the performance ratio. This is expected since, by $\epsilon \approx \mathcal{O}(\sqrt{\epsilon_f} + \epsilon_g)$, a 10-fold increase in $\epsilon_f$ leads to at most a $\sqrt{10}$-fold increase in the performance ratio. This pattern no longer holds when $\epsilon_f$ is further increased, as all methods fail to converge on some datasets (due to violation of Assumption \ref{assump:epsilon}).

\begin{figure}[t]
\centering
\subfigure{\includegraphics[width=0.3\textwidth]{ 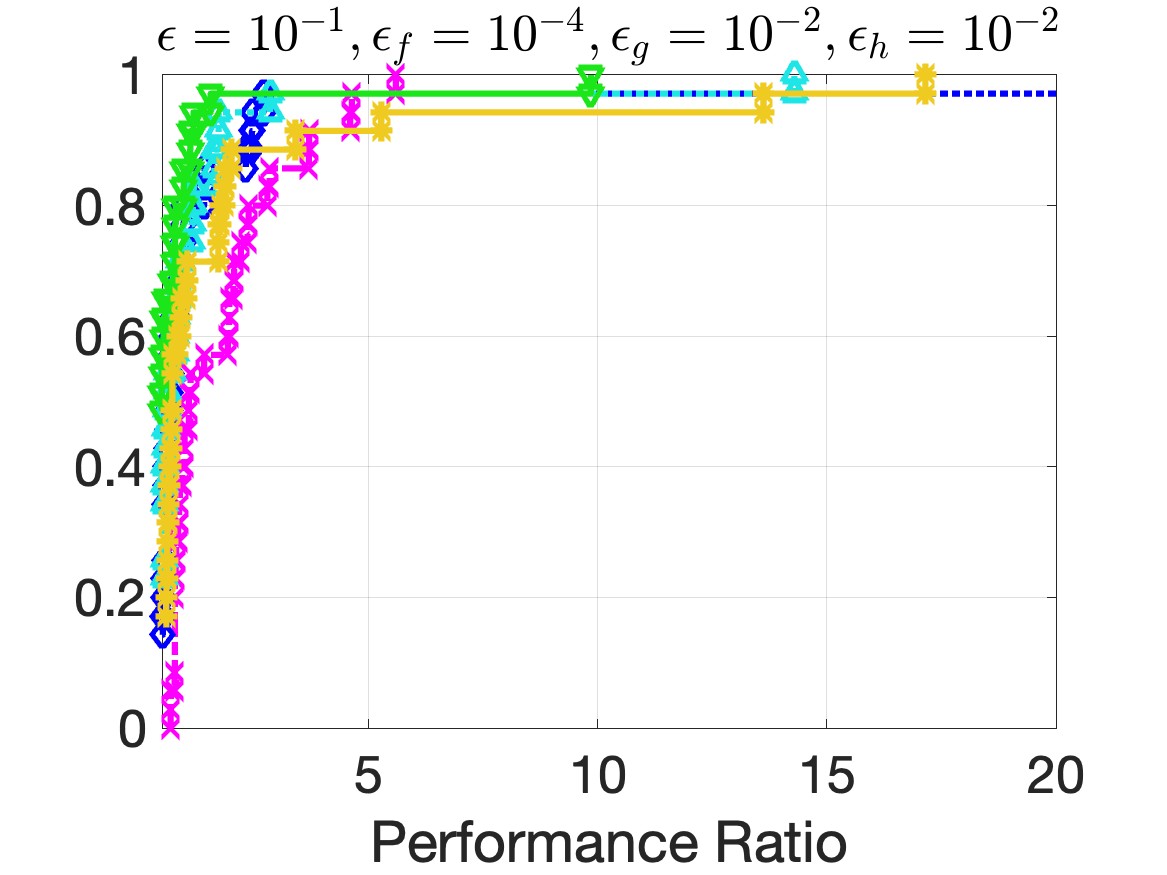}}
\subfigure{\includegraphics[width=0.3\textwidth]{ 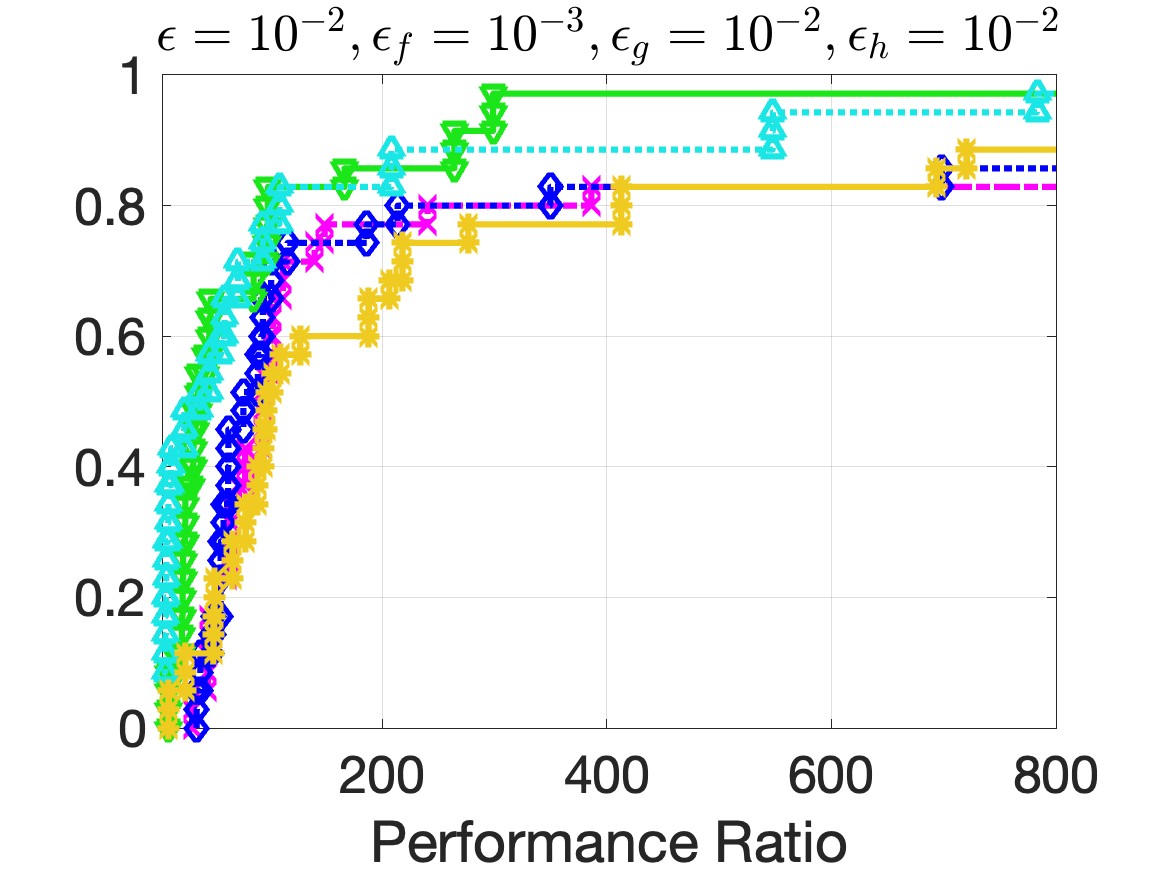}}
\subfigure{\includegraphics[width=0.3\textwidth]{ 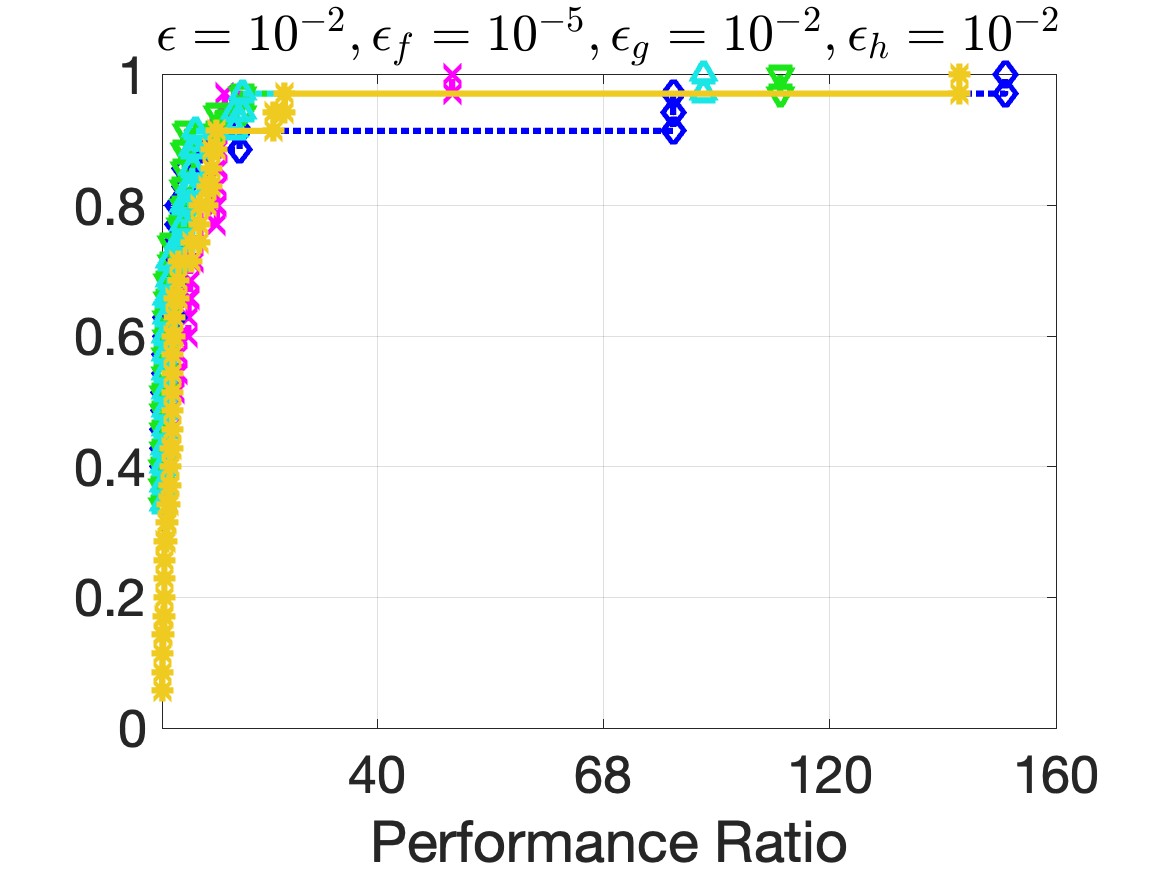}}\\	
\subfigure{\includegraphics[width=0.3\textwidth]{ 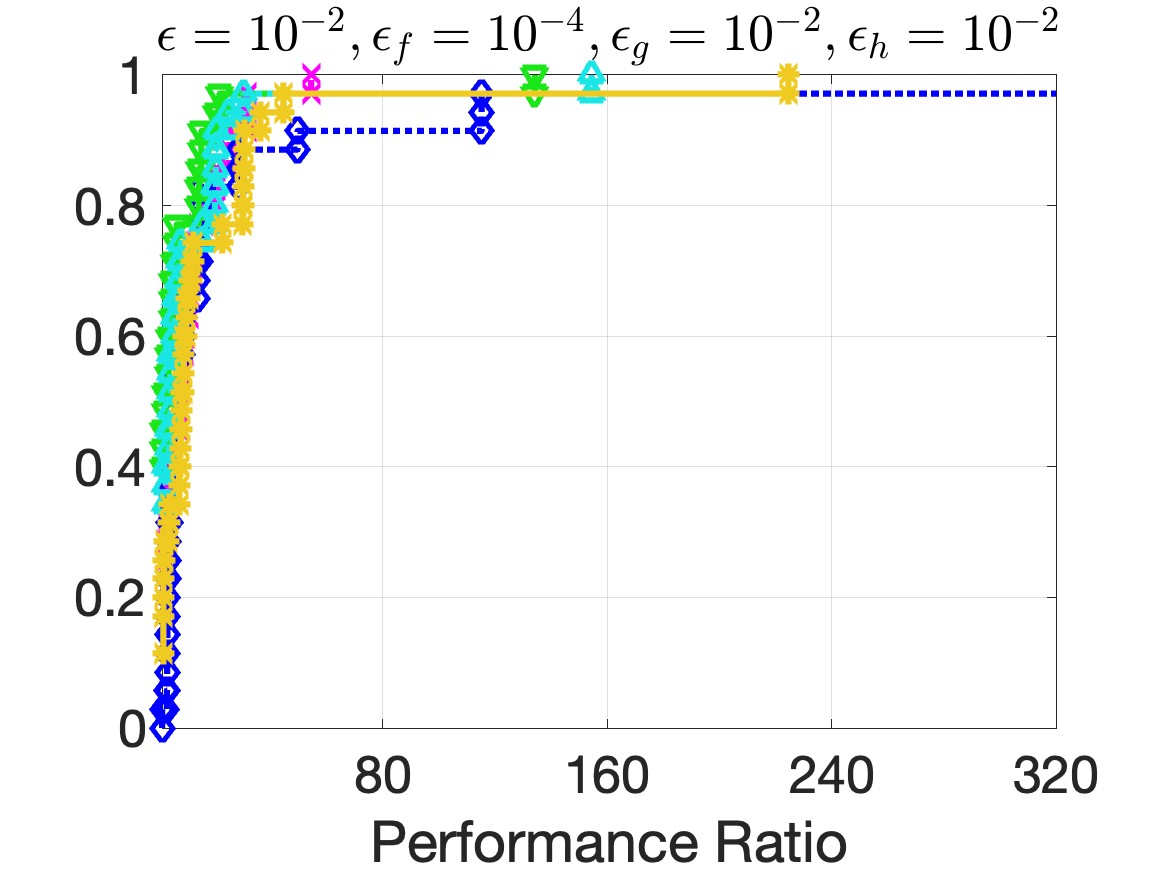}}
\subfigure{\includegraphics[width=0.3\textwidth]{ 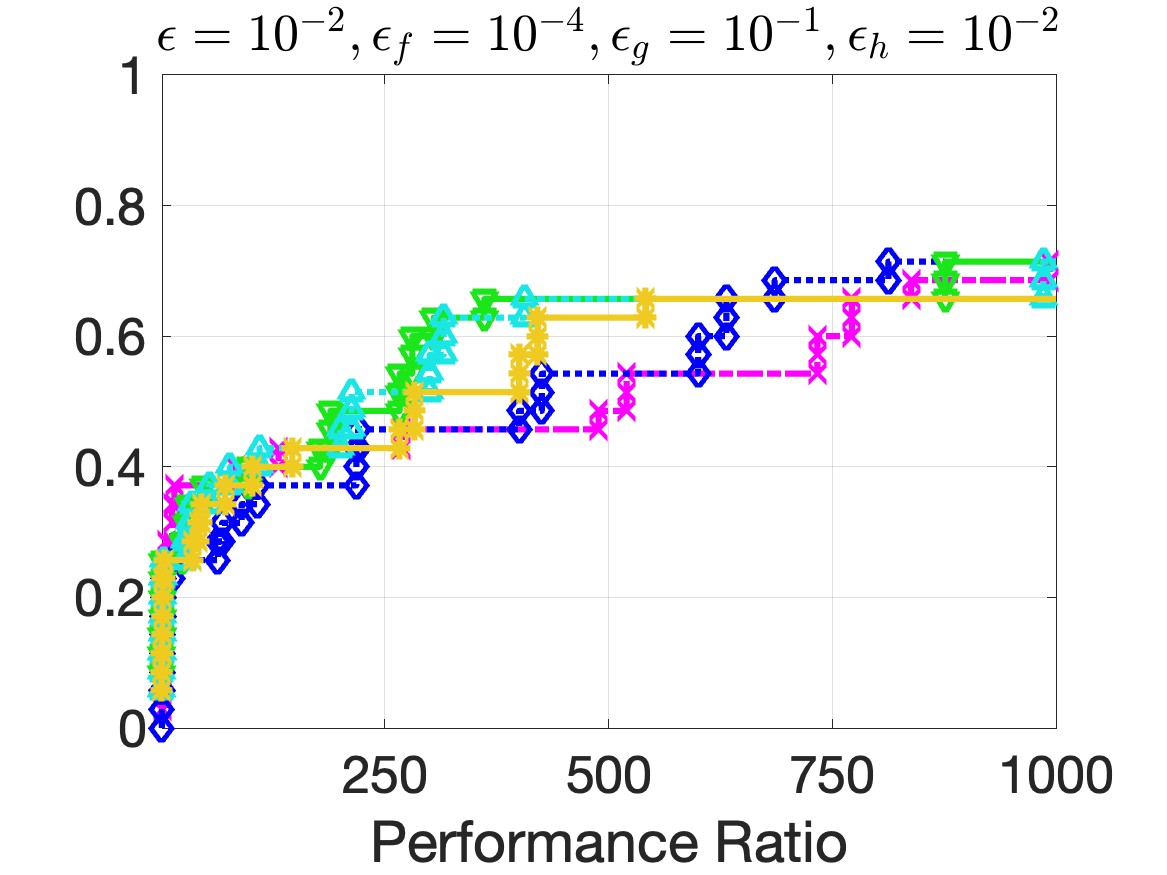}}  \subfigure{\includegraphics[width=0.3\textwidth]{ 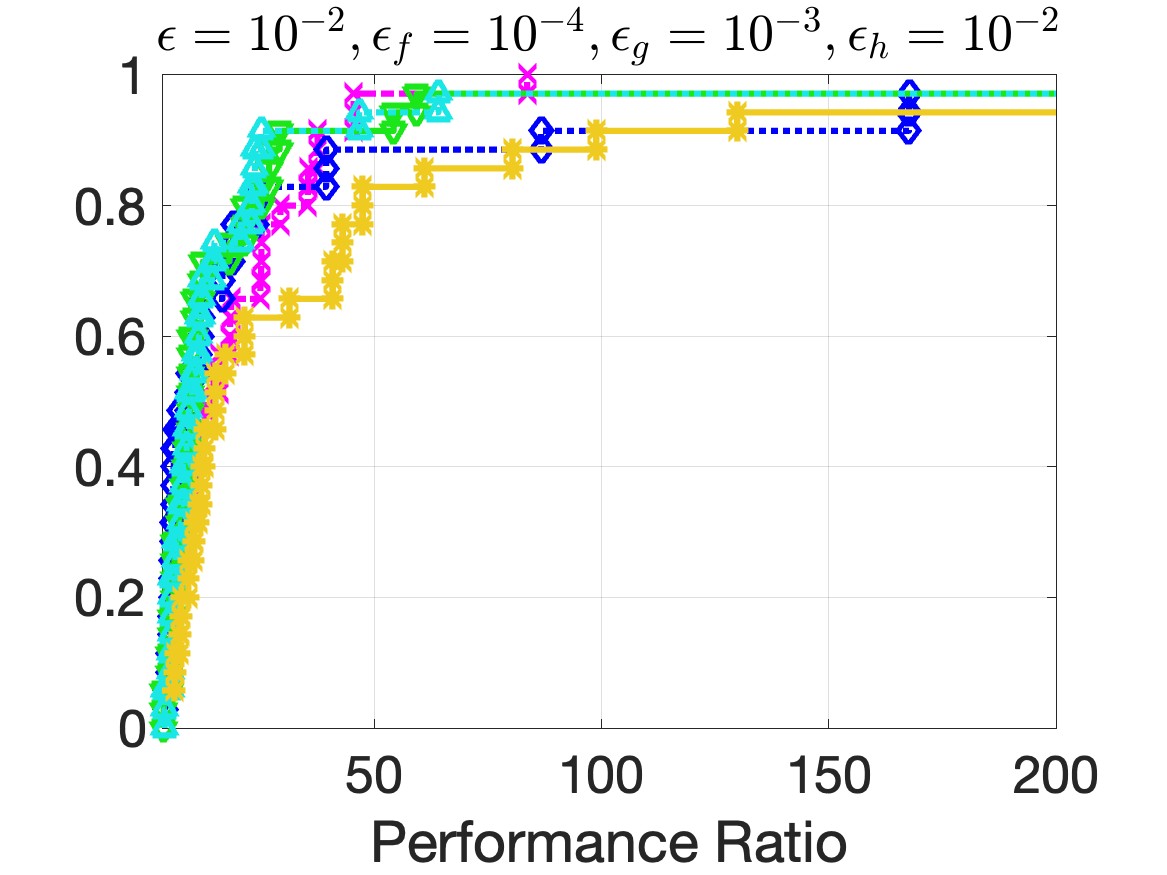}}\\
\subfigure{\includegraphics[width=0.3\textwidth]{ 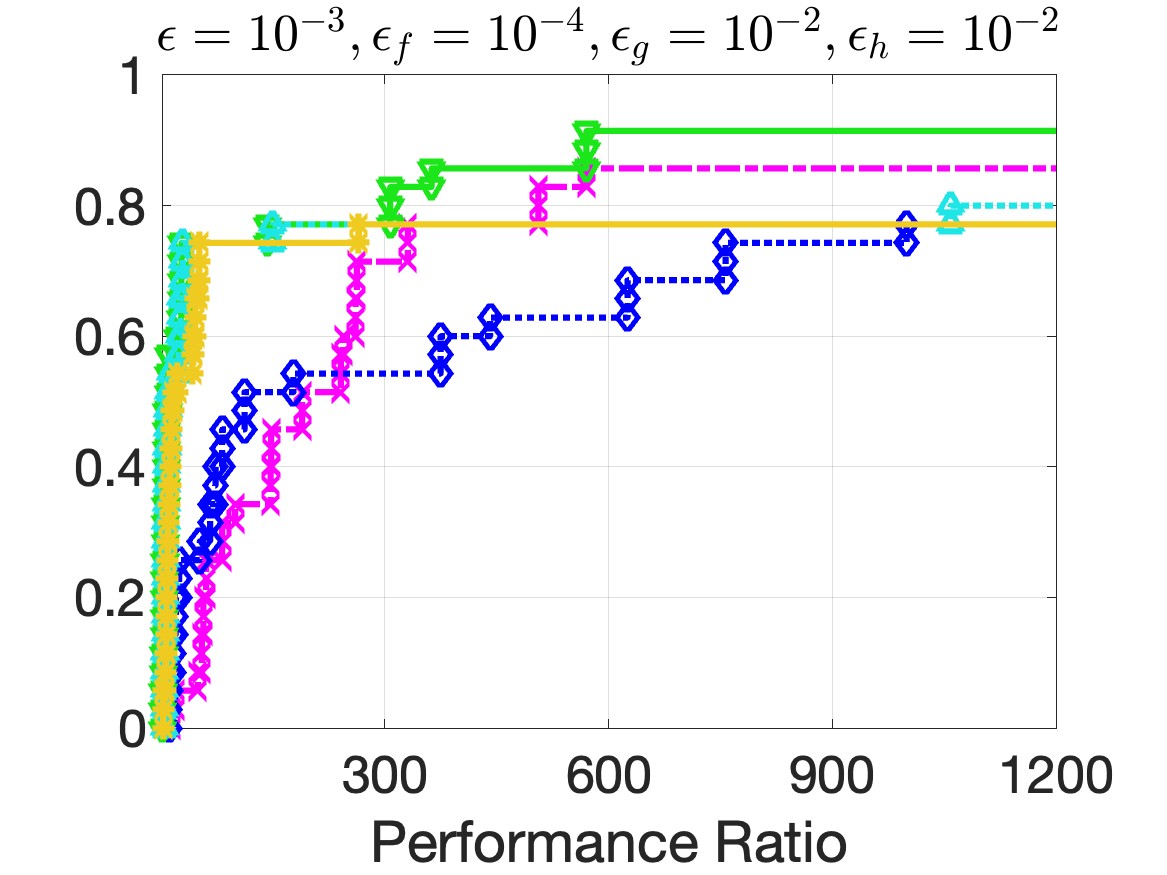}}
\subfigure{\includegraphics[width=0.3\textwidth]{ 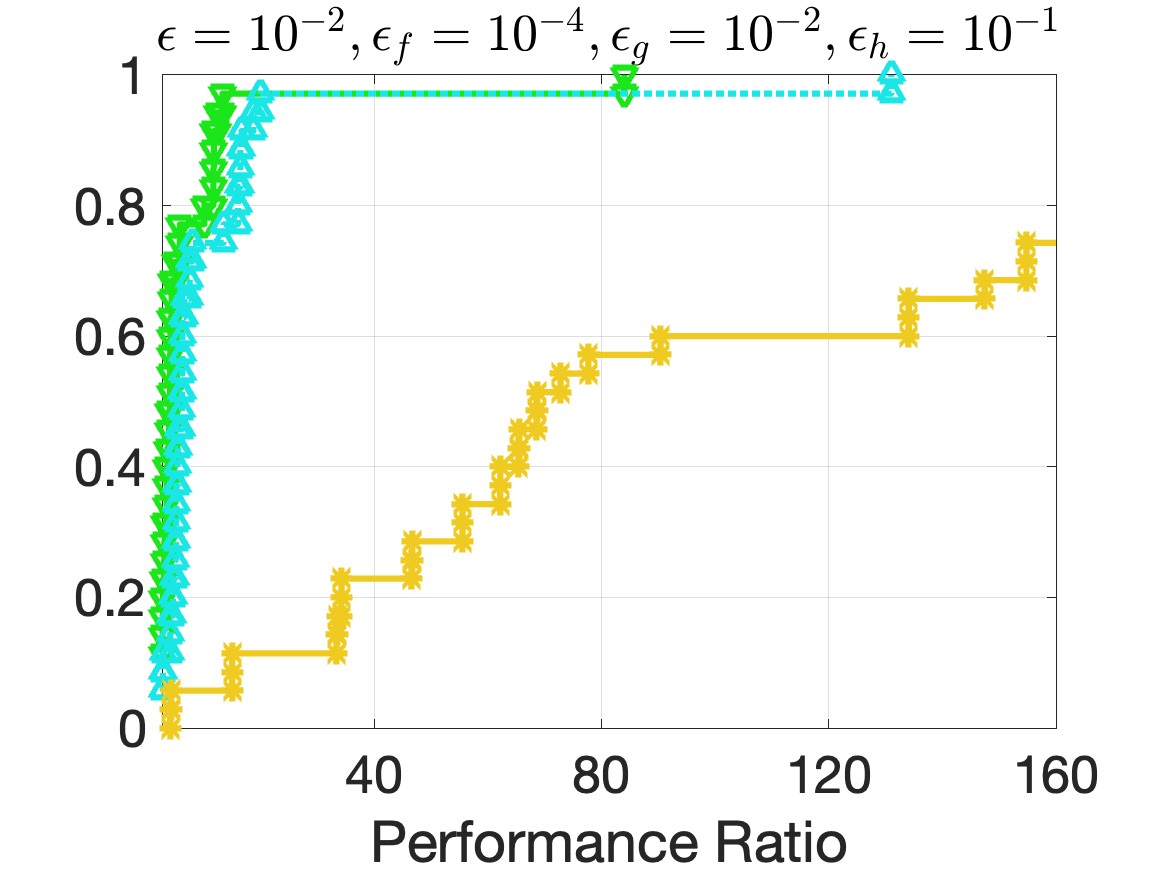}}
\subfigure{\includegraphics[width=0.3\textwidth]{ 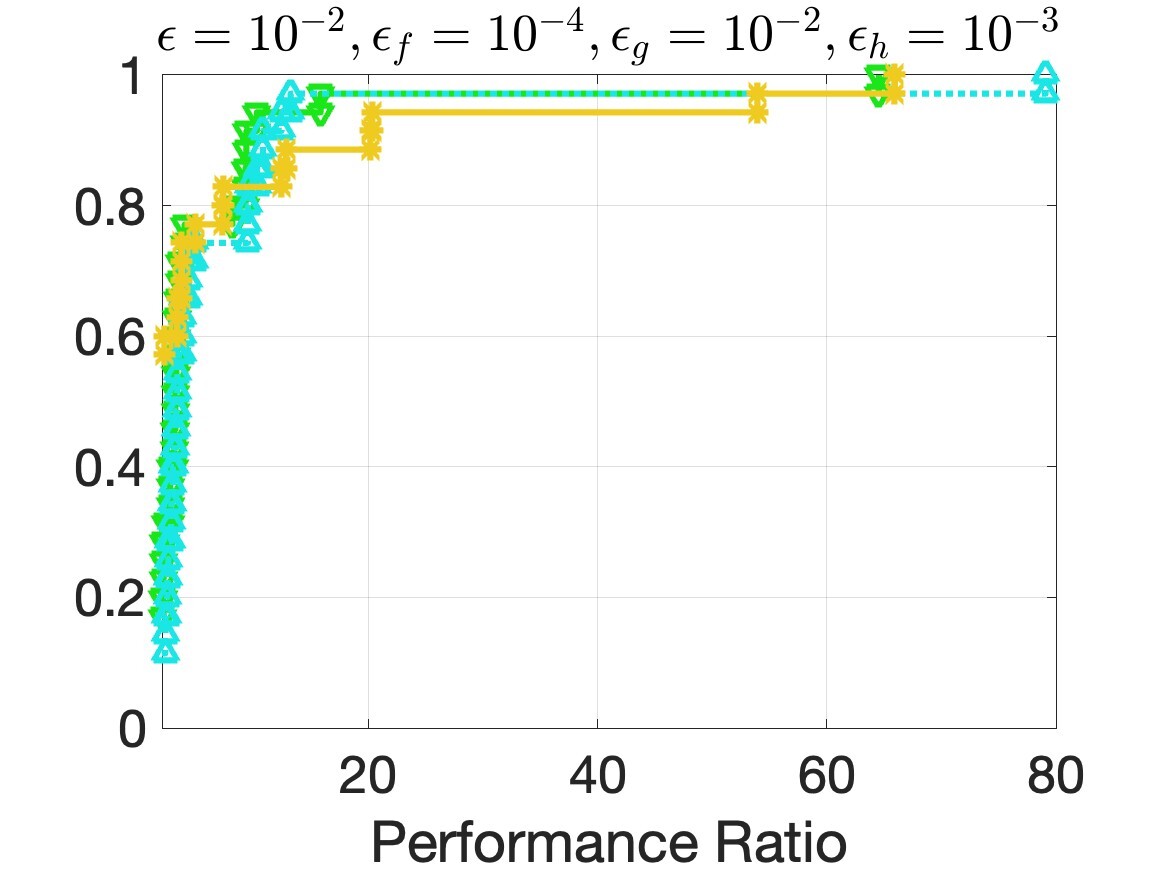}}\\
\subfigure{\includegraphics[width=0.95\textwidth]{ label2}}
\caption{Performance profiles with noise following a Cauchy distribution. Each line represents a different method. The first column corresponds to the default irreducible noise with varying $\epsilon$. Each row of the last two columns corresponds to varying $\epsilon_f$, $\epsilon_g$, $\epsilon_h$, respectively.}
\label{fig:cutest7}\vskip-0.2cm
\end{figure}

\vskip 0.2cm
\noindent $\bullet$ $\boldsymbol{\epsilon_g}:$
Comparing the second row of the last two columns with the default setting, we observe that when $\epsilon_g$ increases from $10^{-3}$ to $10^{-2}$, the performance ratio also increases by a factor of 2 to 3. This observation aligns with our analysis, as under Assumption \ref{assump:epsilon}, a 10-fold increase in $\epsilon_g$ should not result in more than a 10-fold increase in the performance ratio. However, when $\epsilon_g$ further increases from $10^{-2}$ to $10^{-1}$, the performance of all methods deteriorates rapidly. As shown in the middle plot, all methods fail to converge on more than 20\% of problem instances.

\vskip 0.2cm
\noindent $\bullet$ $\boldsymbol{\epsilon_h}:$
The third row of the last two columns illustrates the impact of varying $\epsilon_h$. We observe that TR-SSQP-AveH and TR-SSQP-EstH behave robustly to changes in $\epsilon_h$, primarily because eigenvalues are not included in the definition of the stopping time for first-order stationarity. In contrast, the performance of TR-SSQP2 is significantly affected by the increase in $\epsilon_h$.
When $\epsilon_h$ increases from $10^{-3}$ to $10^{-2}$, TR-SSQP2 requires approximately $5$ times more iterations to converge. This observation also aligns with our theoretical analysis, as for second-order stationarity, $\epsilon \geq \mathcal{O}(\sqrt[3]{\epsilon_f} + \sqrt{\epsilon_g} + \epsilon_h)$ implies that a 10-fold increase in $\epsilon_h$ should lead to no more than a 10-fold increase in the performance ratio. Furthermore, we observe that TR-SSQP2 fails to converge on approximately 25\% of problem instances when $\epsilon_h$ is further increased from $10^{-2}$ to $10^{-1}$, while TR-SSQP-AveH and TR-SSQP-EstH still perform well in this case for finding first-order stationary points.

\section{Conclusion}\label{sec:6}

In this paper, we derived high-probability complexity bounds for a trust-region Stochastic Sequential Quadratic Programming (SSQP) method in identifying first- and second-order $\epsilon$-stationary points in equality-constrained optimization problems. Our method extends the existing literature by allowing irreducible and heavy-tailed noise in objective function estimation. Under weaker oracle conditions, we established that our method achieves an iteration complexity of $\O(\epsilon^{-2})$ for identifying a first-order $\epsilon$-stationary point and $\O(\epsilon^{-3})$ for identifying a second-order $\epsilon$-stationary point with high probability, provided that $\epsilon$ is above a threshold determined by the irreducible noise. We validated our theoretical analysis through numerical experiments on problems from the CUTEst dataset.

\section*{Acknowledgement}
The authors sincerely appreciate Katya Scheinberg's in-depth discussions on the paper and her valuable contributions to relaxing the second-order oracle condition and pointing out the gap between heavy-tailed zeroth-order oracle and high-probability first-order oracle in the initial version of the manuscript. Her comments significantly helped refine our moment condition from $2+\delta$ to $1+\delta$.

\bibliographystyle{my-plainnat}
\bibliography{ref}

\appendix
\numberwithin{equation}{section}
\numberwithin{theorem}{section}

\section{Summary of Constant Definitions}

In this appendix, we provide the definitions for all constants that appear in Section \ref{sec:4}.

\begin{table}[h]
\centering
\begin{tabular}{m{2.5cm} m{12.5cm}}
\hline
Constant & Definition \\
\hline
$\Upsilon_1$ & $\kappa_g\max\{1,\Delta_{\max}\}+\frac{1}{2}(L_{\nabla f}+\kappa_B+\barmu_{T-1} L_G)$ \\
\hline\\[-10pt]
$\Upsilon_2$ & 
$\kappa_g+1+${\Large $\frac{L_{\nabla^2 f}+\kappa_h}{2}$ }+ {\Large $\frac{L_G^2\Delta_{\max}(0.5L_{\nabla f} + \sqrt{m}\barmu_{T-1} L_G)}{\kappa_{1,G}}$ }\\
& {\large $\quad\;+ \frac{0.5\sqrt{m}L_{\nabla^2 c}(L_{\nabla f}\Delta_{\max} + \kappa_{\nabla f}) + 0.5 \sqrt{m}L_G(\kappa_g\Delta_{\max} + L_{\nabla f} + 2\barmu_{T-1} L_G) }{\sqrt{\kappa_{1,G}}}$} \\ [8pt]
\hline\\[-15pt]
$\Upsilon_3$ & {\Large $\frac{\kappa_{fcd}(1-\eta)}{4\kappa_f+\kappa_g+2\Upsilon_1+\kappa_B}$ } \\[8pt]
\hline\\[-15pt]
$\Upsilon_4$ & {\Large $\frac{\kappa_{fcd}(1-\eta)+2}{4\kappa_f+\kappa_g+2\Upsilon_1+\kappa_B}$ }\\[8pt]
\hline\\[-15pt]
$\Upsilon_5$ & {\Large $\frac{(1-\eta)\kappa_{fcd}\min\{1,r\}}{(4\kappa_f +\kappa_g) \max\{\Delta_{\max},1\}+\kappa_B+2\Upsilon_1+2\Upsilon_2+(\kappa_H+\eta)(1-\eta)\kappa_{fcd}\min\{1,r\}}$ }\\[8pt]
\hline\\[-15pt]
$\Upsilon_6$ &  {\Large $\frac{\kappa_{fcd}(1-\eta)+2+2\sqrt{m}L_G/\sqrt{\kappa_{1,G}}}{(4\kappa_f +\kappa_g)\max\{\Delta_{\max},1\}+\kappa_B+2\Upsilon_1+2\Upsilon_2+(\kappa_H+\eta)(1-\eta)\kappa_{fcd}\min\{1,r\}}$} \\[8pt]
\hline\\[-15pt]
$\Upsilon_7$ &   {\Large $\frac{2}{4\kappa_f\max\{\Delta_{\max},1\}+2\Upsilon_1+2\Upsilon_2+(\kappa_H+\eta)(1-\eta)\kappa_{fcd}\min\{1,r\}}$} \\[8pt]
\hline
\end{tabular}
\end{table}

\section{Proofs of Section \ref{subsec:3.2}}\label{Appendix_1}

In this section, we present the proofs of the sample complexities for the oracle constructions introduced in Section~\ref{subsec:3.2}. We consider three estimation strategies: the sample average method, the Median-of-Means (MoM) method, and the finite-difference method.

\subsection{Sample average}\label{app:B1}

For the zeroth-order oracle, we need $\barf_k$ to satisfy \eqref{def:Ck}, \eqref{eq1}, and \eqref{eq3}. We analyze the~three conditions separately.

\vskip3pt
\noindent $\bullet$ \textbf{Sample complexity to satisfy \eqref{def:Ck}.} By Markov inequality and Burkholder-type inequality (Lemma \ref{append:lemma1} with i.i.d. case), we have for any $t>0$,
\begin{multline}\label{sample_avg_eq1}
P\rbr{e_k > t \mid \mF_{k-1}} = P\rbr{|\barf_k- f_k| > t \mid \mF_{k-1}} \leq \frac{\mE[|\barf_k- f_k|^{1+\delta}\mid \mF_{k-1}]}{t^{1+\delta}}\\
\leq \frac{2^{1-\delta}\sum_{\xi^f\in \xi_k^f}\mE[\big|F(\bx_k,\xi^f) - f_k\big|^{1+\delta}\mid \F_{k-1}]}{|\xi_k^f|^{1+\delta} \cdot t^{1+\delta}}   \leq \frac{2^{1-\delta} \tilde\Upsilon_f}{|\xi_k^f|^{\delta}\cdot t^{1+\delta}}.
\end{multline}
By setting $t = \epsilon_f+\kappa_f\Delta_k^{\alpha+2}$ and requiring the right-hand-side term to be bounded~by~$0.5p_f$, we obtain the sample complexity
\begin{equation}\label{pf:sample_avg_1}
|\xi_k^f| \geq \mathcal{O}\left(\left[\frac{1}{p_f}\right]^{\frac{1}{\delta}}\left[\frac{1}{ \epsilon_f+\kappa_f\Delta_k^{\alpha+2} }\right]^{\frac{1+\delta}{\delta}}\right).
\end{equation}
$\bullet$ \textbf{Sample complexity to satisfy \eqref{eq1}.} By Jensen's inequality, we have
\begin{multline*}
\mE[e_k\mid \mF_{k-1}]^{1+\delta} = \mE[|\barf_k-f_k|\mid \mF_{k-1}]^{1+\delta} =  \mE\left[\bigg|\frac{1}{|\xi_k^f|}\sum_{\xi^f\in \xi_k^f}F(\bx_k,\xi^f) - f_k\bigg| \mid \F_{k-1}\right]^{1+\delta} \\
\leq \mE\left[\bigg|\frac{1}{|\xi_k^f|}\sum_{\xi^f\in \xi_k^f} F(\bx_k,\xi^f) - f_k\bigg|^{1+\delta}\mid \F_{k-1}\right]  \stackrel{\eqref{sample_avg_eq1}}{\leq} \frac{2^{1-\delta}}{|\xi_k^f|^{\delta}}\tilde\Upsilon_f.
\end{multline*}
Thus, setting $\frac{2^{1-\delta}}{|\xi_k^f|^{\delta}}\tilde\Upsilon_f \leq \tilde{\epsilon}_f^{1+\delta}$ leads to
\begin{equation}\label{sample_avg_2}
|\xi_k^f| \geq \mathcal{O}\left( \left[\frac{1}{\tilde{\epsilon}_f}\right]^{\frac{1+\delta}{\delta}}\right).
\end{equation}
$\bullet$ \textbf{Sample complexity to satisfy \eqref{eq3}.} We note that
\begin{align*}
\mE\left[\big| e_k - \mE[e_k \mid \F_{k-1}] \big|^{1+\delta} \mid \F_{k-1}\right] & \leq 2^\delta \mE\left[ e_k^{1+\delta} + (\mE[e_k \mid \F_{k-1}] )^{1+\delta}\mid \F_{k-1}\right]\\
& \leq  2^{1+\delta}\mE[e_k^{1+\delta} \mid \F_{k-1}] \stackrel{\eqref{sample_avg_eq1}}{\leq} \frac{4}{|\xi_k^f|^{\delta}}\tilde\Upsilon_f.
\end{align*}
Thus, we only need $\frac{4}{|\xi_k^f|^{\delta}}\tilde{\Upsilon}_f \leq \Upsilon_f$ to satisfy \eqref{eq3}, implying that
\begin{equation}\label{sample_avg_3}
|\xi_k^f| \geq \mathcal{O}\left((\tilde{\Upsilon}_f/\Upsilon_f)^{1/\delta}\right).
\end{equation}
Combining \eqref{pf:sample_avg_1}, \eqref{sample_avg_2}, \eqref{sample_avg_3}, noting that \eqref{sample_avg_3} is independent of $(\epsilon_f, \tilde{\epsilon}_f)$, and applying the same analysis for $\barf_{s_k}$, we obtain the sample complexity for the zeroth-order oracle:$\quad$
\begin{equation}\label{sample_avg_obj_val}
\min\{|\xi_k^f|,|\xi_{s_k}^f|\} \geq \mathcal{O}\left( \left[\frac{1}{p_f}\right]^{\frac{1}{\delta}} \left[\frac{1}{\min\{\epsilon_f+\kappa_f\Delta_k^{\alpha+2},\tilde{\epsilon}_f\}}\right]^{\frac{1+\delta}{\delta}}\right).
\end{equation}
This verifies the results in \eqref{sample_avg_1} and \eqref{sample_avg_obj_val_1}.

\vskip3pt
\noindent$\bullet$ \textbf{Sample complexity to satisfy \eqref{def:Bk}.} We note that for any $t>0$,
\begin{multline}\label{sn:7}
P(\|\barg_k-g_k\|>t\mid \mF_{k-1}) \leq P\rbr{\max_{j\in[d]}|[\barg_k]_j - [g_k]_j|>\frac{t}{\sqrt{d}} \mid \mF_{k-1}} \\ \leq \sum_{j=1}^{d}P\rbr{|[\barg_k]_j - [g_k]_j|>\frac{t}{\sqrt{d}} \mid \mF_{k-1}} \stackrel{\eqref{sample_avg_eq1}}{\leq}\frac{d^{1.5+0.5\delta}2^{1-\delta}\tilde{\Upsilon}_g}{|\xi_{k}^g|^{\delta}\cdot t^{1+\delta}},
\end{multline}
where we use $\tilde{\Upsilon}_g$ to denote the upper bound of the $1+\delta$ moment of (entry-wise) estimate $\nabla F(\bx_k,\xi^g)$: for any $j\in[d]$,
\begin{equation*}
\mE[|[\barg_k]_j - [g_k]_j|^{1+\delta}\mid \mF_{k-1}] \leq \mE[\|\barg_k-g_k\|^{1+\delta}\mid \mF_{k-1}] \leq \tilde{\Upsilon}_g.
\end{equation*}
Thus, we let $t = \epsilon_g+\kappa_g\Delta_k^{\alpha+1}$ and require the right-hand-side term to be bounded by~$p_g$. This leads to the sample complexity to satisfy \eqref{def:Bk}:
\begin{equation}\label{sn:8}
|\xi_k^g| \geq \mathcal{O}\left( \left[\frac{d}{p_g}\right]^{\frac{1}{\delta}} \left[\frac{\sqrt{d}}{\epsilon_g+\kappa_g\Delta_k^{\alpha+1}}\right]^{\frac{1+\delta}{\delta}}\right).
\end{equation}
$\bullet$ \textbf{Sample complexity to satisfy \eqref{def:Ak}.} We follow the analysis in \eqref{sn:7} and just note that
\begin{align}\label{sn:10}
P(& \|\bnabla^2 f_k - \nabla^2 f_k\|>t\mid \mF_{k-1}) \leq P\rbr{\|\bnabla^2 f_k - \nabla^2 f_k\|_F>t\mid \mF_{k-1}} \nonumber\\
& \leq P\rbr{\max_{i,j\in[d]}|[\bnabla^2 f_k]_{i,j} - [\nabla^2 f_k]_{i,j}|>\frac{t}{d} \mid \mF_{k-1}} \nonumber\\
& \leq \sum_{i,j=1}^{d}P\rbr{|[\bnabla^2 f_k]_{i,j} - [\nabla^2 f_k]_{i,j}|>\frac{t}{d} \mid \mF_{k-1}} \stackrel{{\eqref{sample_avg_eq1}}}{\leq} \; \frac{d^22^{1-\delta}\tilde{\Upsilon}_h}{|\xi_k^h|^{\delta}}\left[\frac{d}{t}\right]^{1+\delta}.
\end{align}
Thus, we let $t = \epsilon_h+\kappa_h\Delta_k$ and require the right-hand-side term to be bounded by~$p_h$. This leads to the sample complexity to satisfy \eqref{def:Ak}:
\begin{equation}\label{sn:9}
|\xi_k^h| \geq \mathcal{O}\left(\left[\frac{d^2}{p_h}\right]^{\frac{1}{\delta}}\left[\frac{d}{\epsilon_h+\kappa_h\Delta_k}\right]^{\frac{1+\delta}{\delta}}\right).
\end{equation}
Finally, we combine \eqref{sn:8} and \eqref{sn:9}, and can verify the result in \eqref{sample_avg_gradhess_val_1}.

\subsection{Median of means}\label{app:B2} 

We follow the analysis in Appendix \ref{app:B1}. We first introduce some notations of the median-of-means estimator. Let us decompose the sample set $\xi_k^f$ into $K$ mutually exclusive batches $\xi_k^f = \xi_k^{f,1}\cup\cdots\cup\xi_k^{f,K}$, with each batch having batch size $|\xi_k^{f,j}| = b$, $\forall j\in [K]=\{1,\ldots,K\}$. The median-of-means estimator can be expressed as
\begin{equation*}
\barf_k \coloneqq \text{median}\cbr{\barf(\bx_k,\xi_{k}^{f,1}),\ldots,\barf(\bx_k,\xi_{k}^{f,K})}\quad \text{with}\quad \barf(\bx_k,\xi_{k}^{f,j}) = \frac{1}{|\xi_k^{f,j}|}\sum_{\xi^f\in \xi_k^{f,j}}F(\bx_k,\xi^f).
\end{equation*}
With these notations, we check the sample complexities $|\xi_k^f|, |\xi_{s_k}^f|$ to satisfy the three conditions \eqref{def:Ck}, \eqref{eq1}, and \eqref{eq3} for the zeroth-order oracle.

\vskip3pt
\noindent $\bullet$ \textbf{Sample complexity to satisfy \eqref{def:Ck}.} By the analysis \eqref{sample_avg_eq1}, for any $j\in [K]$ and $t>0$,
\begin{equation*}
P(|\barf(\bx_k,\xi_{k}^{f,j}) - f_k|>t \mid \mF_{k-1})\leq \frac{2^{1-\delta}\tilde{\Upsilon}_f}{b^\delta t^{1+\delta}} \eqqcolon p(t,b).
\end{equation*}
Since $|\barf_k-f_k|>t$ implies at least $K/2$ of the means $|\barf(\bx_k,\xi_{k}^{f,j}) - f_k|>t$, we have$\quad\quad$
\begin{align}\label{sn:2}
& \hspace{-0.9cm} P(|\barf_k-f_k|>t\mid \mF_{k-1}) \leq P\rbr{\text{Binomial}(K, \min(1,p(t,b))) \geq K/2} \nonumber\\
&\hspace{-0.9cm}  = P\cbr{\text{Binomial}(K, \min(1,p(t,b))) - \min(1,p(t,b))K \geq (0.5-\min(1,p(t,b)))K}.
\end{align}
Suppose
\begin{equation}\label{sn:1}
0.5- \min(1,p(t,b)) \geq 0.25 \Longleftrightarrow p(t,b)\leq 0.25,
\end{equation}
we then apply Hoeffding's inequality \citep{Hoeffding1963Probability} and have
\begin{align}\label{sn:6}
&\hskip-0.8cm P(|\barf_k-f_k|>t\mid \mF_{k-1}) \nonumber\\ 
&\hskip-0.8cm \leq P\cbr{\text{Binomial}(K, \min(1,p(t,b))) - \min(1,p(t,b))K \geq 0.25K} \leq \exp(-K/8).
\end{align}
Let us specify $t = \epsilon_f + \kappa_f\Delta_k^{\alpha + 2}$ and $\exp(-K/8)\leq 0.5p_f$; combine with \eqref{sn:1}; and~note that $|\xi_k^f| = K\cdot b$. Then, we know \eqref{def:Ck} is satisfied as long as
\begin{equation}\label{appd1_MoM1}
|\xi_k^f| \geq \mathcal{O}\left(\log\left(\frac{1}{p_f}\right)\left[\frac{1}{ \epsilon_f+\kappa_f\Delta_k^{\alpha+2} }\right]^{\frac{1+\delta}{\delta}}\right).
\end{equation}
$\bullet$ \textbf{Sample complexity to satisfy \eqref{eq1}.} We note that
\begin{align}\label{sn:4}
\mE[e_k\mid\mF_{k-1}] & = \int_{0}^{\infty}P(|\barf_k-f_k|>t\mid \mF_{k-1}) dt \nonumber\\
& = \left[\frac{4\cdot2^{1-\delta}\tilde{\Upsilon}_f}{b^\delta}\right]^{\frac{1}{1+\delta}} + \int_{\left[\frac{4\cdot2^{1-\delta}\tilde{\Upsilon}_f}{b^\delta}\right]^{\frac{1}{1+\delta}}}^{\infty}P(|\barf_k-f_k|>t\mid \mF_{k-1}) dt \nonumber\\
& \stackrel{\mathclap{\eqref{sn:2}}}{\leq} \left[\frac{4\cdot2^{1-\delta}\tilde{\Upsilon}_f}{b^\delta}\right]^{\frac{1}{1+\delta}} + \int_{\left[\frac{4\cdot2^{1-\delta}\tilde{\Upsilon}_f}{b^\delta}\right]^{\frac{1}{1+\delta}}}^{\infty} \sum_{j=K/2}^{K}\begin{pmatrix}
K\\[-6pt]
j
\end{pmatrix}\min(1,p(t,b))^j dt \nonumber\\
& \leq \left[\frac{4\cdot2^{1-\delta}\tilde{\Upsilon}_f}{b^\delta}\right]^{\frac{1}{1+\delta}} + \int_{\left[\frac{4\cdot2^{1-\delta}\tilde{\Upsilon}_f}{b^\delta}\right]^{\frac{1}{1+\delta}}}^{\infty} 2^K\min(1,p(t,b))^{K/2} dt \nonumber\\
& \leq \left[\frac{4\cdot2^{1-\delta}\tilde{\Upsilon}_f}{b^\delta}\right]^{\frac{1}{1+\delta}} + \int_{\left[\frac{4\cdot2^{1-\delta}\tilde{\Upsilon}_f}{b^\delta}\right]^{\frac{1}{1+\delta}}}^{\infty} (4p(t,b))^{K/2} dt \nonumber\\
& = \left[\frac{4\cdot2^{1-\delta}\tilde{\Upsilon}_f}{b^\delta}\right]^{\frac{1}{1+\delta}} + \left[\frac{4\cdot2^{1-\delta}\tilde{\Upsilon}_f}{b^\delta}\right]^{\frac{1}{1+\delta}} \cdot \frac{1}{1+\delta}\int_{0}^{1} u^{\frac{K}{2}-\frac{2+\delta}{1+\delta}} du.
\end{align}
Thus, by requiring $K\geq 4$ and $\left[\frac{4\cdot2^{1-\delta}\tilde{\Upsilon}_f}{b^\delta}\right]^{\frac{1}{1+\delta}}\leq 0.5\tilde{\epsilon}_f$, we know \eqref{eq1} is satisfied as long as
\begin{equation}\label{sn:3}
|\xi_k^f| \geq \mathcal{O}\left( \left[\frac{1}{\tilde{\epsilon}_f}\right]^{\frac{1+\delta}{\delta}}\right).
\end{equation}
$\bullet$ \textbf{Sample complexity to satisfy \eqref{eq3}.} Following the derivations in \eqref{sn:4} and noting~that \begin{equation*}
\mE[|\barf_k-f_k|^{1+\delta} \mid \F_{k-1} ] = \int_0^\infty P\left( |\barf_k-f_k| >t^{\frac{1}{1+\delta}} \mid \F_{k-1} \right) dt,
\end{equation*}
we immediately obtain that \eqref{eq3} is satisfied as long as
\begin{equation}\label{sn:5}
|\xi_k^f| \geq \mathcal{O}\left( (\tilde{\Upsilon}_f/\Upsilon_f)^{\frac{1}{\delta}}\right).
\end{equation}
Combining \eqref{appd1_MoM1}, \eqref{sn:3}, \eqref{sn:5}, noting that \eqref{sn:5} is independent of $(\epsilon_f, \tilde{\epsilon}_f)$, and applying the same analysis for $\barf_{s_k}$, we obtain the sample complexity for the zeroth-order oracle:
\begin{equation}\label{MoM_obj_val}
\min\{|\xi_k^f|,|\xi_{s_k}^f|\} \geq \mathcal{O}\left(\log\left(\frac{1}{p_f}\right)\left[\frac{1}{\min\{\epsilon_f+\kappa_f\Delta_k^{\alpha+2},\tilde{\epsilon}_f\}}\right]^{\frac{1+\delta}{\delta}}\right).
\end{equation}
This verifies the first result in \eqref{MoM_obj_val_1}.

\vskip3pt
\noindent$\bullet$ \textbf{Sample complexity to satisfy \eqref{def:Bk} and \eqref{def:Ak}.} For the gradient and Hessian estimates, we apply coordinate-wise MoM method. We use the gradient estimate as an example, while the Hessian estimate follows the same derivations. Combining the analysis \eqref{sn:7}~with~the complexity \eqref{appd1_MoM1}, we obtain
\begin{align*}
P(\|\barg_k - g_k\|\geq \epsilon_g & + \kappa_g\Delta_k^{\alpha+1} \mid \mF_{k-1}) \leq P\rbr{\max_{j\in[d]}|[\barg_k]_j - [g_k]_j|\geq \frac{\epsilon_g + \kappa_g\Delta_k^{\alpha+1}}{\sqrt{d}}\mid\mF_{k-1} }\\
& \leq \sum_{j=1}^{d}P\rbr{|[\barg_k]_j - [g_k]_j|\geq \frac{\epsilon_g + \kappa_g\Delta_k^{\alpha+1}}{\sqrt{d}}\mid\mF_{k-1}}\leq p_g,
\end{align*}
where the last inequality holds as long as
\begin{equation*}
|\xi_k^g| \geq \mathcal{O}\left(\log\left(\frac{d}{p_g}\right)\left[\frac{\sqrt{d}}{\epsilon_g+\kappa_g\Delta_k^{\alpha+1}}\right]^{\frac{1+\delta}{\delta}}\right).
\end{equation*}
A similar complexity for $|\xi_k^h|$ holds by combining the analysis \eqref{sn:10} with the complexity \eqref{appd1_MoM1}. This verifies the second result in \eqref{MoM_obj_val_1}.

\subsection{Finite difference} 

For objective value estimates $\barf_k$ and $\barf_{s_k}$, the sample complexities are specified in \eqref{sample_avg_obj_val}~and \eqref{MoM_obj_val} for the sample average method and the MoM method, respectively. We focus here only on the gradient and Hessian estimates, and use the sample average method to estimate function values as an example. 
The analysis for the MoM~method~\mbox{follows}~the~same~derivations but enjoys exponential tail bounds as in \eqref{sn:6}, which improves the dependence of the sample complexity on the failure probabilities from $(d/p_g)^{1/\delta}$, $(d^2/p_h)^{1/\delta}$ to $\log(d/p_g)$, $\log(d^2/p_h)$. 

We recall that the $j$-th gradient entry and the $(i,j)$-th Hessian entry are estimated~by
\begin{align*}
[\barg_k]_j & = \frac{1}{\sigma}\rbr{\barf(\bx_k + \sigma \boldsymbol{e}_j, \xi_k^g)- \barf(\bx_k,\xi_k^g)},\\
[\bar{\nabla}^2 f_k]_{i,j} & = 
\frac{1}{\sigma^2}\rbr{\barf(\bx_k + \sigma \boldsymbol{e}_i + \sigma \boldsymbol{e}_j, \xi_k^h) - \barf(\bx_k + \sigma \boldsymbol{e}_i, \xi_k^h) - \barf(\bx_k + \sigma \boldsymbol{e}_j, \xi_k^h) + \barf(\bx_k, \xi_k^h)}.
\end{align*}
$\bullet$ \textbf{Sample complexity to satisfy \eqref{def:Bk}.} By the analysis \eqref{sn:7}, we know it suffices to have
\begin{equation}\label{sn:11}
P\rbr{|[\barg_k]_j - [g_k]_j | > \frac{\epsilon_g+\kappa_g\Delta_k^{\alpha+1}}{\sqrt{d}}\mid \mF_{k-1}}\leq \frac{p_g}{d}\quad\quad \text{ for any } j\in[d].
\end{equation}
We note that
\begin{align*}
& |[\barg_k]_j - [g_k]_j |= \left| \frac{\barf(\bx_k + \sigma \boldsymbol{e}_j, \xi_k^g)- \barf(\bx_k,\xi_k^g)}{\sigma} - [g_k]_j \right| \\
& = \left| \frac{\barf(\bx_k + \sigma \boldsymbol{e}_j,\xi_k^g)- f(\bx_k + \sigma \boldsymbol{e}_j)}{\sigma} +\frac{f(\bx_k + \sigma \boldsymbol{e}_j)- f(\bx_k)}{\sigma} + \frac{f(\bx_k )- \barf(\bx_k,\xi_k^g)}{\sigma}  - [g_k]_j \right| \\
& \leq \left|\frac{f(\bx_k + \sigma \boldsymbol{e}_j)- f(\bx_k)}{\sigma}  - [g_k]_j\right| + \left| \frac{\barf(\bx_k + \sigma \boldsymbol{e}_j,\xi_k^g)- f(\bx_k + \sigma \boldsymbol{e}_j)}{\sigma}\right| + \left|\frac{f(\bx_k )- \barf(\bx_k,\xi_k^g)}{\sigma} \right| \\
& \leq \frac{L_{\nabla f} \sigma}{2} + \left| \frac{\barf(\bx_k + \sigma \boldsymbol{e}_j,\xi_k^g)- f(\bx_k + \sigma \boldsymbol{e}_j)}{\sigma}\right| + \left|\frac{f(\bx_k )- \barf(\bx_k,\xi_k^g)}{\sigma} \right|,
\end{align*}
where the last inequality is due to the Lipschitz continuity of the gradient: $|f(\bx_k + \sigma \boldsymbol{e}_j)- f(\bx_k)- \sigma[g_k]_j | \leq L_{\nabla f} \sigma^2/2$. Comparing the above display with \eqref{sn:11}, we can specify~$\sigma = \frac{\epsilon_g+\kappa_g\Delta_k^{\alpha+1}}{L_{\nabla f}\sqrt{d}}$ and know that \eqref{sn:11} is further implied by
\begin{equation*}
P\bigg(\max\{|\barf(\bx_k + \sigma \boldsymbol{e}_j,\xi_k^g)- f(\bx_k + \sigma \boldsymbol{e}_j)|, |\barf(\bx_k,\xi_k^g)- f(\bx_k )|\}
> \frac{(\epsilon_g+\kappa_g\Delta_k^{\alpha+1})^2}{4L_{\nabla f}d} \mid \mF_{k-1}\bigg) \leq \frac{p_g}{d},
\end{equation*}
which, according to \eqref{pf:sample_avg_1}, leads to the complexity
\begin{equation}\label{sn:12}
|\xi_k^g|\geq \O\rbr{\left[\frac{d}{p_g}\right]^{\frac{1}{\delta}}\left[\frac{L_{\nabla f}d}{(\epsilon_g+\kappa_g\Delta_k^{\alpha+1})^2}\right]^{\frac{1+\delta}{\delta}} }.
\end{equation}
By \eqref{appd1_MoM1}, the MoM method improves the above factor from $(d/p_g)^{1/\delta}$ to $\log(d/p_g)$.

\vskip3pt
\noindent$\bullet$ \textbf{Sample complexity to satisfy \eqref{def:Ak}.} By the analysis \eqref{sn:10}, we know it suffices to have
\begin{equation}\label{sn:13}
P\rbr{|[\bnabla^2 f_k]_{i,j} - [\nabla^2 f_k]_{i,j} | > \frac{\epsilon_h+\kappa_h\Delta_k}{d}\mid \mF_{k-1}}\leq \frac{p_h}{d^2}\quad\quad \text{ for any } i, j\in[d].
\end{equation}
We note that
\begin{align}\label{sn:14}
& |[\bar{\nabla}^2f_k]_{i,j} - [\nabla^2 f_k]_{i,j} | \nonumber\\
& = \left| \frac{\barf(\bx_k + \sigma \boldsymbol{e}_i + \sigma \boldsymbol{e}_j,\xi_{k}^h) - \barf(\bx_k + \sigma \boldsymbol{e}_i,\xi_{k}^h) - \barf(\bx_k + \sigma \boldsymbol{e}_j,\xi_{k}^h) + \barf(\bx_k,\xi_{k}^h)}{\sigma^2} - [\nabla^2 f_k]_{i,j} \right| \nonumber\\
& \leq \left|\frac{f(\bx_k + \sigma \boldsymbol{e}_i + \sigma \boldsymbol{e}_j) - f(\bx_k + \sigma \boldsymbol{e}_i) - f(\bx_k + \sigma \boldsymbol{e}_j) + f(\bx_k)}{\sigma^2}  - [\nabla^2 f_k]_{i,j} \right|  \nonumber\\
& \quad + \left| \frac{\barf(\bx_k + \sigma \boldsymbol{e}_i + \sigma \boldsymbol{e}_j,\xi_{k}^h)- f(\bx_k + \sigma \boldsymbol{e}_i + \sigma \boldsymbol{e}_j)}{\sigma^2}\right| + \left|\frac{\barf(\bx_k + \sigma \boldsymbol{e}_i,\xi_{k}^h)- f(\bx_k + \sigma \boldsymbol{e}_i)}{\sigma^2} \right| \nonumber\\
& \quad + \left| \frac{\barf(\bx_k + \sigma \boldsymbol{e}_j,\xi_{k}^h)- f(\bx_k + \sigma \boldsymbol{e}_j)}{\sigma^2}\right| + \left|\frac{\barf(\bx_k,\xi_{k}^h)- f(\bx_k)}{\sigma^2} \right|.
\end{align}
For the first term on the right hand side, by the Lipschitz continuity of the Hessian, we~have
\begin{align*}
\bigg| f(\bx_k + \sigma \boldsymbol{e}_i + \sigma \boldsymbol{e}_j) - f_k - \sigma([g_k]_i+[g_k]_j) - \frac{\sigma^2}{2}(\be_i+\be_j)^T\nabla^2 f_k(\be_i+\be_j)\bigg| & \leq \frac{\sqrt{2}L_{\nabla^2 f}\sigma^3}{3},\\
\bigg| f(\bx_k + \sigma \boldsymbol{e}_i) - f_k - \sigma[g_k]_i - \frac{\sigma^2}{2}\be_i^T\nabla^2 f_k\be_i\bigg| & \leq \frac{L_{\nabla^2 f}\sigma^3}{6},\\
\bigg| f(\bx_k + \sigma \boldsymbol{e}_j) - f_k - \sigma[g_k]_j - \frac{\sigma^2}{2}\be_j^T\nabla^2 f_k\be_j\bigg| & \leq \frac{L_{\nabla^2 f}\sigma^3}{6}.
\end{align*}
Combining the above three inequalities, we have
\begin{equation*}
\abr{f(\bx_k + \sigma \boldsymbol{e}_i + \sigma \boldsymbol{e}_j) - f(\bx_k + \sigma \boldsymbol{e}_i) - f(\bx_k + \sigma \boldsymbol{e}_j) + f(\bx_k) - \sigma^2[\nabla^2 f_k]_{i,j}} \leq \frac{5L_{\nabla^2 f}\sigma^3}{6}.
\end{equation*}
Plugging the above display into \eqref{sn:14}, and specifying $\sigma = (\epsilon_h+\kappa_h\Delta_k)/(L_{\nabla^2 f} d)$, we~know \eqref{sn:13} is further implied by
\begin{multline*}
P\bigg(\max\{|\barf(\bx_k + \sigma \boldsymbol{e}_i+\sigma\be_j,\xi_k^h)- f(\bx_k + \sigma \boldsymbol{e}_i+\sigma\be_j)|, |\barf(\bx_k+\sigma\be_i,\xi_k^h) - f(\bx_k+\sigma\be_i)|,\\
|\barf(\bx_k+\sigma\be_j,\xi_k^h) - f(\bx_k+\sigma\be_j)|, |\barf(\bx_k,\xi_k^h)- f(\bx_k)|\}> \frac{(\epsilon_h+\kappa_h\Delta_k)^3}{24L_{\nabla^2f}^2d^3}\mid\mF_{k-1}\bigg)\leq \frac{p_h}{d^2},
\end{multline*}
which, according to \eqref{pf:sample_avg_1}, leads to the complexity
\begin{equation}\label{sn:15}
|\xi_k^h|\geq \O\rbr{\left[\frac{d^2}{p_h}\right]^{\frac{1}{\delta}}\left[\frac{L_{\nabla^2 f}^2d^3}{(\epsilon_h+\kappa_h\Delta_k)^3}\right]^{\frac{1+\delta}{\delta}} }.
\end{equation}
By \eqref{appd1_MoM1}, the MoM method improves the above factor from $(d^2/p_h)^{1/\delta}$ to $\log(d^2/p_h)$. Combining the results \eqref{sn:12} and \eqref{sn:15}, we have verified the result in \eqref{sn:16}.

\section{Proofs of Section \ref{sec:4.1}}\label{append:A}

\subsection{Proof of Lemma \ref{lemma:diff_ared_pred_w_corr_step}}\label{lemma:append:A.4}

\begin{proof}

We have	
\begin{align}\label{eq:lemma_diff_w_soc_1}
& \big|\L_{\barmu_{k}}^{s_k}-\L_{\barmu_{k}}^{k}-\text{Pred}_k\big| \nonumber \\
& \qquad = \big|f_{s_k}+\barmu_{k}\|c_{s_k}\|-f_k-\barg_k^T\Delta\bx_k-\frac{1}{2} \Delta\bx_k^T\barH_k\Delta\bx_k-\barmu_{k}\|c_k+G_k\Delta\bx_k\|\big|  \nonumber \\
& \qquad \leq
\big|f_{s_k}-f_k-\barg_k^T\Delta\bx_k-\frac{1}{2} \Delta\bx_k^T\barH_k\Delta\bx_k\big|+\barmu_{T-1}\| c_{s_k}-c_k-G_k\Delta\bx_k\|,
\end{align}
where we have used Assumption \ref{assump:4-1}. First, we analyze the second term in \eqref{eq:lemma_diff_w_soc_1}.~Since~the SOC step is performed, we have $\bx_{s_k}=\bx_k+\Delta\bx_k+\bd_k$. For $1\leq i\leq m$, by the~\mbox{Taylor}~expansion, we obtain
\begin{equation*}
|c_{s_k}^i- c_k^i - (\nabla c_k^i)^T\Delta\bx_k| \stackrel{\eqref{def:correctional_step}}{=} |c_{s_k}^i - c^i(\bx_k+\Delta\bx_k) - (\nabla c_k^i)^T\bd_k| \leq L_G(\|\Delta\bx_k\|\|\bd_k\| + \|\bd_k\|^2).
\end{equation*} 
Therefore, $\| c_{s_k} -c_k-G_k\Delta\bx_k \|\leq \sqrt{m} L_G(\|\Delta\bx_k\|\|\bd_k\| + \|\bd_k\|^2)$. For $\|\bd_k\|$, we have
\begin{equation}\label{eq:length_of_correctional_step}
\|\bd_k\| \leq \|G_k^T[G_kG_k^T]^{-1}\|\|c(\bx_k+\Delta\bx_k)-c_k-G_k\Delta\bx_k\| \leq \frac{L_G}{\sqrt{\kappa_{1,G}}}\Delta_k^2.
\end{equation}
Combining the last two results and using the fact that $\Delta_k\leq \Delta_{\max}$, we have
\begin{equation}\label{eq:cubic_const}
\|c_{s_k}-c_k-G_k\Delta\bx_k\| 
\leq
m L_G\|\Delta\bx_k\|\|\bd_k\|+m L_G\|\bd_k\|^2
\leq\left(\frac{\sqrt{m}L_G^2}{\sqrt{\kappa_{1,G}}}+\frac{\sqrt{m}  L_G^3\Delta_{\max}}{\kappa_{1,G}}\right)\Delta_k^3.
\end{equation}
Next, we analyze the first term in \eqref{eq:lemma_diff_w_soc_1}. For some points $\phi_1$ between $[\bx_k+\Delta\bx_k,\bx_k+\Delta\bx_k+\bd_k]$ and $\phi_2$ between $[\bx_k,\bx_k+\Delta\bx_k]$, we have
\begin{align*}
f_{s_k} & = f(\bx_k+\Delta\bx_k+\bd_k) = f(\bx_k+\Delta\bx_k)+ \nabla f(\bx_k+\Delta\bx_k)^T\bd_k+\frac{1}{2}\bd_k^T\nabla^2 f(\phi_1)\bd_k \\
& = f_k+g_k^T\Delta\bx_k+\nabla f(\bx_k+\Delta\bx_k)^T\bd_k+\frac{1}{2}\Delta\bx_k^T\nabla^2 f(\phi_2)\Delta\bx_k +\frac{1}{2}\bd_k^T\nabla^2 f(\phi_1)\bd_k.
\end{align*}
Define $\tilde{\blambda}_k=-[G_kG_k^T]^{-1}G_k\nabla f(\bx_k+\Delta\bx_k)$. By the Taylor expansion, we have
\begin{align*} 
& \nabla f(\bx_k+\Delta\bx_k)^T\bd_k \stackrel{\mathclap{\eqref{def:correctional_step}}}{=}\; \tilde{\blambda}_k^T[c(\bx_k+\Delta\bx_k)-c_k-G_k\Delta\bx_k] \\
&\quad  =\sum_{i=1}^m\tilde{\blambda}_k^i[c^i(\bx_k+\Delta\bx_k)-c_k^i-(\nabla c_k^i)^T\Delta\bx_k] = \frac{1}{2}\sum_{i=1}^m\tilde{\blambda}_k^i\Delta\bx_k^T\nabla^2c^i(\phi_3^i)\Delta\bx_k,
\end{align*}
where the points $\{\phi_3^i\}_{i=1}^m$ are between $[\bx_k,\bx_k+\Delta\bx_k]$. Recall that $\barH_k=\bar{\nabla}^2f_k+\sum_{i=1}^{m}\barblambda_k^i\nabla^2c_k^i$. We combine the above two displays and have
\begin{align}\label{neq:4}
& \big| f_{s_k} - f_k  -\barg_k^T\Delta\bx_k-\frac{1}{2}\Delta\bx_k^T\barH_k\Delta\bx_k\big| \notag\\
& = \bigg|(g_k-\barg_k)^T\Delta\bx_k+\frac{1}{2}\Delta\bx_k^T\left(\nabla^2 f(\phi_2)-\nabla^2f_k\right)\Delta\bx_k+\frac{1}{2}\Delta\bx_k^T\left(\nabla^2 f_k-\bar{\nabla}^2f_k\right)\Delta\bx_k  \notag\\
& \quad\quad +\frac{1}{2}\sum_{i=1}^{m}\left(\blambda_k^i-\barblambda_k^i\right)\Delta\bx_k^T\nabla^2c_k^i\Delta\bx_k+\frac{1}{2}\sum_{i=1}^m\left(\tilde{\blambda}_k^i-\blambda_k^i\right)\Delta\bx_k^T\nabla^2c_k^i\Delta\bx_k \notag \\
& \quad\quad +\frac{1}{2}\bd_k^T\nabla^2 f\left(\phi_1\right)\bd_k+\frac{1}{2}\sum_{i=1}^m\tilde{\blambda}_k^i\Delta\bx_k^T\left(\nabla^2c^i(\phi_3^i)-\nabla^2c_k^i\right)\Delta\bx_k\bigg| \notag\\
& \leq \|g_k-\barg_k\|\|\Delta\bx_k\|+\frac{L_{\nabla^2 f}}{2}\|\Delta\bx_k\|^3+\frac{1}{2}\|\nabla^2 f_k-\bar{\nabla}^2f_k\|\|\Delta\bx_k\|^2 \notag\\
& \quad\quad + \frac{\sqrt{m}L_G}{2}(\|\blambda_k - \barblambda_k\| + \|\tilde{\blambda}_k-\blambda_k\|) \|\Delta\bx_k\|^2 + \frac{L_{\nabla f}}{2}\|\bd_k\|^2 + \frac{\sqrt{m}L_{\nabla^2 c}}{2}\|\tilde{\blambda}_k\| \|\Delta\bx_k\|^3 \notag\\
& \leq \|g_k-\barg_k\|\Delta_k+\frac{1}{2}\epsilon_h\Delta_k^2+ \rbr{\frac{L_{\nabla^2 f}+\kappa_h}{2}}\Delta_k^3 + \frac{\sqrt{m}L_G}{2}(\|\blambda_k - \barblambda_k\| + \|\tilde{\blambda}_k-\blambda_k\|) \Delta_k^2 \notag\\
&\quad + \frac{L_{\nabla f}}{2}\|\bd_k\|^2 + \frac{\sqrt{m}L_{\nabla^2 c}\|\tilde{\blambda}_k\|}{2} \Delta_k^3 \notag\\
& \stackrel{\mathclap{\eqref{eq:length_of_correctional_step}}}{\leq}\|g_k-\barg_k\|\Delta_k+\frac{1}{2}\epsilon_h\Delta_k^2+ \rbr{\frac{L_{\nabla^2 f}+\kappa_h}{2} + \frac{L_{\nabla f}L_G^2\Delta_{\max}}{2\kappa_{1,G}} + \frac{\sqrt{m}L_{\nabla^2 c}\|\tilde{\blambda}_k\|}{2} }\Delta_k^3  \nonumber\\
&\quad + \frac{\sqrt{m}L_G}{2}(\|\blambda_k - \barblambda_k\| + \|\tilde{\blambda}_k-\blambda_k\|) \Delta_k^2,
\end{align}
where the second inequality is by Assumption \ref{assump:4-1} and the third inequality is by the event $\A_k$ and the fact that $\Delta\bx_k\leq \Delta_k$ (note that the SOC step is performed only when $\alpha=1$). Furthermore, on the event $\B_k$, it follows from  Assumption \ref{assump:4-1} that 
\begin{align*}
& \|\blambda_k-\barblambda_k\|  \leq \|[G_kG_k^T]^{-1}G_k\|\|g_k-\barg_k\| \leq\frac{\kappa_g}{\sqrt{\kappa_{1,G}}}\Delta_k^2+ \frac{\epsilon_g}{\sqrt{\kappa_{1,G}}} \leq\frac{\kappa_g\Delta_{\max}}{\sqrt{\kappa_{1,G}}}\Delta_k+ \frac{\epsilon_g}{\sqrt{\kappa_{1,G}}},\\
& \|\blambda_k-\tilde{\blambda}_k\|  \leq \|[G_kG_k^T]^{-1}G_k\|\|g_k-\nabla f(\bx_k+\Delta\bx_k)\| \leq \frac{L_{\nabla f}}{\sqrt{\kappa_{1,G}}}\Delta_k,\\
& \|\tilde{\blambda}_k\| \leq \|[G_kG_k^T]^{-1}G_k\| \|\nabla f(\bx_k+\Delta\bx_k)\| \leq 
\frac{1}{\sqrt{\kappa_{1,G}}} (L_{\nabla f}\|\Delta\bx_k\| +\|g_k\|)\leq \frac{L_{\nabla f}\Delta_{\max}+\kappa_{\nabla f}}{\sqrt{\kappa_{1,G}}}.
\end{align*}
To simplify the first term in \eqref{neq:4}, we claim that under $\B_k$, $\|g_k-\barg_k\|\leq (\kappa_g+1)\Delta_k^2+\min\{\epsilon_g^{3/2}/\Delta_k,\epsilon_g\}$. To prove this, we notice that when $\epsilon_g^{3/2}/\Delta_k\geq \epsilon_g$, this claim is trivial~due to the definition of $\B_k$. When $\epsilon_g^{3/2}/\Delta_k< \epsilon_g$, we have $\epsilon_g^{1/2}<\Delta_k$. On the event $\B_k$, we~have $\|g_k-\barg_k\|\leq \epsilon_g+\kappa_g\Delta_k^2 < (\kappa_g+1)\Delta_k^2<(\kappa_g+1)\Delta_k^2+\min\{\epsilon_g^{3/2}/\Delta_k,\epsilon_g\}$. Thus, the claim is proved. Combining this claim with the above displays, we have
\begin{align}\label{eq:cubic_quad}
& \big|f_{s_k}-f_k  -\barg_k^T\Delta\bx_k-\frac{1}{2}\Delta\bx_k^T\barH_k\Delta\bx_k\big|  \notag \\
& \leq \epsilon_g^{3/2}+\left(\frac{1}{2}\epsilon_h+\frac{\sqrt{m}L_G}{\sqrt{\kappa_{1,G}}}\epsilon_g\right)\Delta_k^2
+\left(\kappa_g+1+\frac{L_{\nabla^2 f}+\kappa_h}{2}\right)\Delta_k^3 \notag \\
& \quad + \rbr{\frac{L_{\nabla f}L_G^2\Delta_{\max}}{2\kappa_{1,G}}+\frac{\sqrt{m}L_{\nabla^2 c}(L_{\nabla f}\Delta_{\max} + \kappa_{\nabla f}) }{2\sqrt{\kappa_{1,G}}} + \frac{\sqrt{m}L_G(\kappa_g\Delta_{\max} + L_{\nabla f})}{2\sqrt{\kappa_{1,G}}}}\Delta_k^3.
\end{align}
We complete the proof by combining \eqref{eq:lemma_diff_w_soc_1}, \eqref{eq:cubic_const}, and \eqref{eq:cubic_quad}.
\end{proof}

\subsection{Proof of Lemma \ref{lemma:guarantee_succ_step_eigen}}\label{lemma:append:A.6}

\begin{proof}
{\myred We first show that \eqref{delta:lemma_guarantee_succ_step_eigen} is well-defined. When $\alpha=1$, Assumption \ref{assump:epsilon} implies~$\epsilon>\Upsilon_6/\Upsilon_5 \cdot \epsilon_g + \Upsilon_7/\Upsilon_5 \cdot \epsilon_h$. For $k<T_\epsilon$, Definition \ref{def:stopping_time} implies $\max\{\|\nabla\L_k\|,\tau_k^+\}>\epsilon$.~Combining these two results yields $\max\{\|\nabla\L_k\|,\tau_k^+\} > \Upsilon_6/\Upsilon_5 \cdot \epsilon_g + \Upsilon_7/\Upsilon_5 \cdot \epsilon_h$, which can be rearranged to obtain $\Upsilon_5 \max\{\|\nabla\L_k\|,\tau_k^+\} > \Upsilon_6  \epsilon_g + \Upsilon_7\epsilon_h$. Thus, the right-hand side of \eqref{delta:lemma_guarantee_succ_step_eigen} is positive and the inequality is well-defined.}
To show the result, we consider two cases separately. \textbf{Case A:} $\max\{ \tau_k^+,\|\nabla\L_k\|\} = \|\nabla\L_k\|$ and \textbf{Case B:} $\max\{ \tau_k^+,\|\nabla\L_k\|\} = \tau_k^+$. 
\vskip4pt

\noindent$\bullet$ \textbf{Case A.} When $\alpha=1$, $\vartheta_{\alpha}=2\epsilon_f+\epsilon_g^{3/2}$. Thus, to show (\ref{nequ:3}a),  we only need to show $\text{Ared}_k-2\epsilon_f-\epsilon_g^{3/2} \leq \eta \text{Pred}_k$ can be satisfied by \eqref{delta:lemma_guarantee_succ_step_eigen}. By the algorithm design, we first consider the scenario in which the SOC step is not computed (i.e., $\bx_{s_k}=\bx_k+\Delta\bx_k$). We have
\begin{align}\label{nequ:10}
\text{Ared}_k & - 2\epsilon_f-\epsilon_g^{3/2}  = \bar{\L}_{\barmu_{k}}^{s_k}-\bar{\L}_{\barmu_{k}}^k - 2\epsilon_f-\epsilon_g^{3/2} \notag \\
& = \bar{\L}_{\barmu_{k}}^{s_k}- \L_{\barmu_{k}}^{s_k} +\L_{\barmu_{k}}^{s_k} - \L_{\barmu_{k}}^k +\L_{\barmu_{k}}^k-\bar{\L}_{\barmu_{k}}^k - 2\epsilon_f -\epsilon_g^{3/2} \notag \\
& = \bar{f}_{s_k} - f_{s_k} +\L_{\barmu_{k}}^{s_k} - \L_{\barmu_{k}}^k +f_k-\bar{f}_k - 2\epsilon_f -\epsilon_g^{3/2} \notag \\
& \stackrel{\mathclap{\text{Lemma }\ref{lemma:diff_ared_pred_wo_corr_step}}}{\leq}  |\bar{f}_{s_k} - f_{s_k}| + |f_k-\bar{f}_k| +\text{Pred}_k +\epsilon_g\Delta_k+\Upsilon_1\Delta_k^2 - 2\epsilon_f -\epsilon_g^{3/2}.
\end{align}
We observe that under $\A_k\cap\B_k$, \eqref{delta:lemma_guarantee_succ_step_eigen} implies $\Delta_k \leq \|\bnabla\L_k\|/\|\barH_k\|$. Thus,
\begin{equation}\label{nequ:11}
\text{Pred}_k\leq -\frac{\kappa_{fcd}}{2}\|\bnabla\L_k\|\min\left\{\Delta_k,\frac{\|\bnabla\L_k\|}{\|\barH_k\|}\right\} =-\frac{\kappa_{fcd}}{2}\|\bnabla\L_k\| \Delta_k.
\end{equation}
Combining \eqref{nequ:10} and \eqref{nequ:11}, and noticing that when $\alpha=1$ and $\C_k$ happens, $|f_k-\bar{f}_k|+ |\bar{f}_{s_k} - f_{s_k}|\leq 2\epsilon_f+2\kappa_f\max\{\Delta_{\max},1\}\Delta_k^2$,  we only need to show 
\begin{equation}\label{nequ:12}
2\kappa_f\max\{\Delta_{\max},1\}\Delta_k^2 +\epsilon_g\Delta_k+\Upsilon_1\Delta_k^2 \leq \frac{\kappa_{fcd}}{2}(1-\eta)\|\bnabla\L_k\| \Delta_k.
\end{equation}
To this end, we note that
\begin{align*}
\eqref{delta:lemma_guarantee_succ_step_eigen}\Rightarrow & \{(4\kappa_f+\kappa_g)\max\{\Delta_{\max},1\}+2\Upsilon_1\} \Delta_k\leq \kappa_{fcd}(1-\eta)\|\nabla\L_k\|-\left\{\kappa_{fcd}(1-\eta)+2\right\}\epsilon_g\\
\Rightarrow&\{4\kappa_f\max\{\Delta_{\max},1\}+\kappa_{fcd}(1-\eta)\kappa_g\max\{\Delta_{\max},1\}+2\Upsilon_1\}\Delta_k \\
&\hspace{5cm}\leq \kappa_{fcd}(1-\eta)\|\nabla\L_k\|-\left\{\kappa_{fcd}(1-\eta)+2\right\}\epsilon_g,
\end{align*}
since $\kappa_{fcd}(1-\eta)<1$. Rearranging the terms, we have
\begin{equation}\label{nequ:13}
(2\kappa_f \max\{\Delta_{\max},1\} +\Upsilon_1) \Delta_k 
\leq  \frac{\kappa_{fcd}}{2}(1-\eta)(\|\nabla\L_k\|-\kappa_g\max\{\Delta_{\max},1\}\Delta_k - \epsilon_g) -\epsilon_g.
\end{equation}
Under $\B_k$, we have $\|\bnabla\L_k\|\geq \|\nabla\L_k\|-\kappa_g\max\{\Delta_{\max},1\}\Delta_k-\epsilon_g$. This, together with \eqref{nequ:13}, implies
\begin{equation*}
(2\kappa_f \max\{\Delta_{\max},1\} +\Upsilon_1)\Delta_k\leq \frac{\kappa_{fcd}}{2}(1-\eta)\|\bnabla\L_k\| - \epsilon_g.
\end{equation*}
Multiplying $\Delta_k$ and rearranging the terms, we have shown \eqref{nequ:12}, thus proved (\ref{nequ:3}a). Since (\ref{nequ:3}a) holds, by the algorithm design, the SOC step will not be triggered in this iteration. Next, we show that  (\ref{nequ:3}b) holds. When $\alpha=1$, under $\A_k\cap\B_k$, we have $\max\{1,\|\barH_k\|\}\leq\kappa_B$. Therefore, it suffices to show $\|\bnabla\L_k\|\geq \eta \kappa_B\Delta_k$. Recalling $\kappa_{fcd}\leq 1, \eta<1$, we have
\begin{multline}\label{nequ:14}
(\eta\kappa_B+ \kappa_g\max\{\Delta_{\max},1\})\Delta_k
\leq \{(4\kappa_f+\kappa_g)\max\{\Delta_{\max},1\}+2\Upsilon_1+\kappa_B\} \Delta_k  \\
\stackrel{\mathclap{\eqref{delta:lemma_guarantee_succ_step_eigen}}}{\leq} \kappa_{fcd}(1-\eta)\|\nabla\L_k\|-\{\kappa_{fcd}(1-\eta)+2\}\epsilon_g
\leq \|\nabla\L_k\| - \epsilon_g.
\end{multline}
Since $\|\bnabla\L_k\| \geq \|\nabla\L_k\| - \kappa_g\max\{\Delta_{\max},1\}\Delta_k - \epsilon_g$ under $\B_k$, which combined with \eqref{nequ:14} leads to $\|\bnabla\L_k\|\geq \eta \kappa_B\Delta_k$, thus  (\ref{nequ:3}b) holds. 

\vskip4pt

\noindent$\bullet$ \textbf{Case B.} In this scenario, we show that (\ref{nequ:3}a) holds by considering two subcases. \textbf{Case B.1:} $\|c_k\|\geq r$ and \textbf{Case B.2:} $\|c_k\| < r$.

\vskip4pt
\noindent $\bullet$ \textbf{Case B.1.} In this case, the SOC step will not be performed, so we need to show that $\text{Ared}_k-2\epsilon_f-\epsilon_g^{3/2}\leq \eta \text{Pred}_k$ holds for $\bx_{s_k}=\bx_k+\Delta\bx_k$. Using the same analysis as in \textbf{Case A}, we have \eqref{nequ:10} and $ |f_k-\bar{f}_k|+ |\bar{f}_{s_k} - f_{s_k}|\leq 2\epsilon_f+2\kappa_f \max\{\Delta_{\max},1\}\Delta_k^2$ under $\C_k$ with $\alpha=1$. Since
\begin{equation*}
\text{Pred}_k\leq -\frac{\kappa_{fcd}}{2}\bartau_k^+\|c_k\|\Delta_k\leq -\frac{\kappa_{fcd}}{2}\bartau_k^+ r \Delta_k,
\end{equation*}
it suffices to show
\begin{equation}\label{neq:10}
2\kappa_f \max\{\Delta_{\max},1\}\Delta_k^2 +\epsilon_g\Delta_k+\Upsilon_1\Delta_k^2 \leq (1-\eta)\frac{\kappa_{fcd}}{2}\bartau_k^+ r \Delta_k.
\end{equation}
Now, we show that this can be achieved by $\Delta_k$ satisfying \eqref{delta:lemma_guarantee_succ_step_eigen}. We notice that
\begin{align*}
\tau_k^+&\stackrel{\mathclap{\eqref{delta:lemma_guarantee_succ_step_eigen}}}{\geq}\left(\frac{4\kappa_f \max\{\Delta_{\max},1\} +2\Upsilon_1+2\Upsilon_2}{(1-\eta)\kappa_{fcd}\min\{1,r\}}+\kappa_H+\eta\right)\Delta_k + \frac{(2+2\sqrt{m}L_G/\sqrt{\kappa_{1,G}})\epsilon_g+2\epsilon_h }{(1-\eta)\kappa_{fcd}\min\{1,r\}} \\ & \geq \left(\frac{4\kappa_f\max\{\Delta_{\max},1\}+2\Upsilon_1}{(1-\eta)r\kappa_{fcd}}+\kappa_H\right)\Delta_k + \frac{2\epsilon_g}{(1-\eta)r\kappa_{fcd}}+\epsilon_H \quad (\text{by the def. of }\epsilon_H).
\end{align*}
Rearranging the terms, we have
\begin{equation*}
(2\kappa_f\max\{\Delta_{\max},1\}+\Upsilon_1)\Delta_k + \epsilon_g  \leq (1-\eta)\frac{\kappa_{fcd}}{2}r\left(\tau_k^+  - \kappa_H  \Delta_k   - \epsilon_H \right).
\end{equation*}
Multiplying $\Delta_k$ on both sides, using the relation $\bartau_k^+ \geq \tau_k^+ - \kappa_H\Delta_k - \epsilon_H$ under $\A_k\cap\B_k$ (cf. Lemma \ref{lemma:tau_accurate}), we can show that \eqref{neq:10} holds.

\vskip4pt
\noindent$\bullet$ \textbf{Case B.2.} In this case, if $\text{Ared}_k-2\epsilon_f-\epsilon_g^{3/2}\leq \eta \text{Pred}_k$ holds for $\bx_{s_k}=\bx_k+\Delta\bx_k$, then there is nothing to prove. If not, the SOC step $\bd_k$ (cf. Section \ref{subsec:2.1.3}) will be performed. Thus, we only need to show that $\text{Ared}_k-2\epsilon_f-\epsilon_g^{3/2}\leq \eta \text{Pred}_k$ holds for $\bx_{s_k}=\bx_k+\Delta\bx_k+\bd_k$. Similar to the proof in \textbf{Case B.1}, but using the conclusion of Lemma \ref{lemma:diff_ared_pred_w_corr_step}, we have
\begin{align*}
\text{Ared}_k - 2\epsilon_f -\epsilon_g^{3/2} & \leq  |\bar{f}_{s_k} - f_{s_k}| + |f_k-\bar{f}_k| +\text{Pred}_k +\left(\frac{1}{2}\epsilon_h+\frac{\sqrt{m}L_G}{2\sqrt{\kappa_{1,G}}}\epsilon_g\right)\Delta_k^2+ \Upsilon_2\Delta_k^3 - 2\epsilon_f \\
&  \leq  \text{Pred}_k +\left(\frac{1}{2}\epsilon_h+\frac{\sqrt{m}L_G}{2\sqrt{\kappa_{1,G}}}\epsilon_g\right)\Delta_k^2+ (\Upsilon_2+2\kappa_f )\Delta_k^3 ,
\end{align*}
where the last inequality is due to the definition of $\C_k$ with $\alpha=1$. By the algorithm design, we have
\begin{equation*}
\text{Pred}_k\leq -\frac{\kappa_{fcd}}{2}\bartau_k^+\Delta_k^2,
\end{equation*}
thus we only need to show 
\begin{equation}\label{neq:11}
\left(\frac{1}{2}\epsilon_h+\frac{\sqrt{m}L_G}{2\sqrt{\kappa_{1,G}}}\epsilon_g\right)\Delta_k^2+ (\Upsilon_2+2\kappa_f )\Delta_k^3  \leq (1-\eta)\frac{\kappa_{fcd}}{2}\bartau_k^+\Delta_k^2.
\end{equation}
To this end, we note that with the definition of $\epsilon_H$,
\begin{align*}
\tau_k^+&  \stackrel{\mathclap{\eqref{delta:lemma_guarantee_succ_step_eigen}}}{\geq} \left(\frac{4\kappa_f\max\{\Delta_{\max},1\}+2\Upsilon_1+2\Upsilon_2}{(1-\eta)\kappa_{fcd}\min\{1,r\}}+\kappa_H+\eta\right) \Delta_k + \frac{(2+2 \sqrt{m} L_G/\sqrt{\kappa_{1,G}})\epsilon_g + 2\epsilon_h }{(1-\eta)\kappa_{fcd}\min\{1,r\}} \\
& \geq \left(\frac{	4\kappa_f+2\Upsilon_2}{(1-\eta)\kappa_{fcd}}+\kappa_H\right)\Delta_k + \frac{\epsilon_h +\epsilon_g \sqrt{m}L_G/\sqrt{\kappa_{1,G}}}{(1-\eta)\kappa_{fcd}}+\epsilon_H.
\end{align*}
Rearranging the terms and multiplying $\Delta_k^2$ on both sides, we have
\begin{equation*}
(2\kappa_f+\Upsilon_2)\Delta_k^3 + \left(\frac{1}{2}\epsilon_h + \frac{\sqrt{m}L_G}{2\sqrt{\kappa_{1,G}}}\epsilon_g\right)\Delta_k^2  \leq \frac{\kappa_{fcd}}{2}(1-\eta) (\tau_k^+ -\kappa_H\Delta_k - \epsilon_H)\Delta_k^2.
\end{equation*}
Since under $\A_k\cap \B_k$, we have $\bartau_k^+ \geq \tau_k^+ - \kappa_H\Delta_k -\epsilon_H$ (cf. Lemma \ref{lemma:tau_accurate}). This relation, together with the above display, has proved \eqref{neq:11}. Combining \textbf{Case B.1} and \textbf{Case B.2}, we have shown that  (\ref{nequ:3}a) holds in \textbf{Case B}. Next, we show that (\ref{nequ:3}b) holds in \textbf{Case B}. By the algorithm design, it is sufficient to show $\bartau_k^+\geq \eta\Delta_k$. We note that \eqref{delta:lemma_guarantee_succ_step_eigen} implies 
\begin{equation*}
\tau_k^+ \geq (\kappa_H+\eta)\Delta_k + \epsilon_H.
\end{equation*}
This, combined with the relation $\bartau_k^+ \geq \tau_k^+ - \kappa_H\Delta_k -\epsilon_H$ under $\A_k\cap \B_k$, leads to $\bartau_k^+\geq \eta\Delta_k$. We thus complete the proof.
\end{proof}

\subsection{Proof of Lemma \ref{lemma: count_iter}}\label{lemma:append:A.8}

\begin{proof}
(a) We define the sequence $\{\phi_k\}_k$ as
\begin{equation*}
\phi_k = \max\left\{\log_\gamma\left(\frac{\Delta_k}{\hat{\Delta}'}\right),0\right\}
\end{equation*}
and discuss the cases: \textbf{Case A:} $\Lambda_k\Theta_k=1$, \textbf{Case B:} $\Lambda_k(1-\Theta_k)=1$, and \textbf{Case C:} $\Lambda_k=0$.

\vskip4pt
\noindent$\bullet$ \textbf{Case A.} When $\Lambda_k\Theta_k=1$, we have $\min\{\Delta_k,\Delta_{k+1}\}\geq \hat{\Delta}'$ and $\Delta_{k+1}=\min\{\gamma\Delta_k,\Delta_{\max}\}$. Since $\Delta_{k+1}\geq \Delta_k \geq \hat{\Delta}'$, we have
\begin{equation*}
\log_\gamma\left(\frac{\Delta_{k+1}}{\hat{\Delta}'}\right)\geq \log_\gamma\left(\frac{\Delta_{k}}{\hat{\Delta}'}\right)\geq 0.
\end{equation*}
Therefore, $\phi_k= \log_\gamma ( \Delta_{k}/\hat{\Delta}')$ and $\phi_{k+1}= \log_\gamma ( \Delta_{k+1}/\hat{\Delta}')$. Since $\Delta_{k+1}\leq \gamma\Delta_k$, we further have
\begin{equation*}
\log_\gamma\left(\frac{\Delta_k}{\hat{\Delta}'}\right)+1 \geq  \log_\gamma\left(\frac{\Delta_{k+1}}{\hat{\Delta}'}\right),
\end{equation*}
thus $\phi_{k+1}-\phi_k\leq 1$.

\vskip4pt

\noindent$\bullet$ \textbf{Case B.} When $\Lambda_k(1-\Theta_k)=1$, we have $\min\{\Delta_k,\Delta_{k+1}\}\geq \hat{\Delta}'$ and $\Delta_{k+1} = \Delta_k/\gamma$. Therefore, we must have  $\Delta_k\geq\gamma\hat{\Delta}'$, which leads to
\begin{equation*}
\log_\gamma\left(\frac{\Delta_k}{\hat{\Delta}'}\right) \geq 1 \quad\text{and}\quad	\log_\gamma\left(\frac{\Delta_{k+1}}{\hat{\Delta}'}\right)=\log_\gamma\left(\frac{\Delta_k}{\hat{\Delta}'}\right)-1\geq 0.
\end{equation*}
Thus, $\phi_k =  \log_\gamma (\Delta_k/\hat{\Delta}' ) $, $\phi_{k+1} = \log_\gamma (\Delta_{k+1}/\hat{\Delta}')$ and $\phi_{k+1} -\phi_k = -1$.

\vskip4pt
\noindent$\bullet$ \textbf{Case C.} When $\Lambda_k=0$, we have $\max\{\Delta_k,\Delta_{k+1} \}\leq \hat{\Delta}'$. Thus, $\phi_k=\phi_{k+1} = 0$.

\vskip4pt
Combining the above three cases, we sum $\phi_{k+1}-\phi_k$ over $k=0,\cdots,T-1$ and get \vspace{-0.1cm}
\begin{equation}\label{neq:12}
\phi_{T}-\phi_0 = \sum_{k=0}^{T-1} \phi_{k+1}-\phi_k \leq \sum_{k=0}^{T-1} \Lambda_k\Theta_k - \sum_{k=0}^{T-1} \Lambda_k(1-\Theta_k).
\end{equation}
Noticing that $\phi_0= \log_\gamma (\Delta_0/\hat{\Delta}')$ (since $ \Delta_0> \hat{\Delta}'$) and $\phi_{T}\geq 0$, we rearrange the terms in \eqref{neq:12} and complete the proof.

\vskip4pt
\noindent (b) This is a direct corollary of (a), which can be proved by rearranging terms of (a).

\vskip4pt
\noindent (c) We define the sequence $\{\tilde{\phi}_k\}_k$ as
\begin{equation*}
\tilde{\phi}_k = \max\left\{\log_\gamma\left(\hat{\Delta}'/\Delta_k\right),0\right\}
\end{equation*}
and discuss three cases. \textbf{Case A:} $(1-\Lambda_k)\Theta_k=1$, \textbf{Case B:} $(1-\Lambda_k)(1-\Theta_k)=1$, \textbf{Case C:} $\Lambda_k=1$.

\vskip4pt
\noindent$\bullet$ \textbf{Case A.} When $(1-\Lambda_k)\Theta_k=1$, we have $\max\{\Delta_k,\Delta_{k+1}\}\leq \hat{\Delta}'$ and $\Delta_{k+1} = \min\{\gamma\Delta_k,\Delta_{\max}\}$. Since $\Delta_k\leq \hat{\Delta}'\leq \gamma^{-1}\Delta_{\max}$, we have $\gamma\Delta_k\leq\Delta_{\max}  $. Thus $\Delta_{k+1}=\gamma\Delta_k$. In this case, to ensure $\Delta_{k+1} \leq\hat{\Delta}'$, we must have $\Delta_k\leq \gamma^{-1}\hat{\Delta}'$. The above relations imply
\begin{equation*}
\log_\gamma\left(\frac{\hat{\Delta}'}{\Delta_k}\right) \geq 1\quad\text{and}\quad  \log_\gamma\left(\frac{\hat{\Delta}'}{\Delta_{k+1}}\right) =  \log_\gamma\left(\frac{\hat{\Delta}'}{\Delta_k}\right)  -1\geq 0.
\end{equation*}
Therefore, $ \tilde{\phi}_k = \log_\gamma (\hat{\Delta}'/\Delta_k) $, $ \tilde{\phi}_{k+1} = \log_\gamma (\hat{\Delta}'/\Delta_{k+1}) $, and $\tilde{\phi}_{k+1}-\tilde{\phi}_k = -1$.

\vskip4pt
\noindent$\bullet$ \textbf{Case B.} When $(1-\Lambda_k)(1-\Theta_k)=1$, we have $\max\{\Delta_k,\Delta_{k+1} \}\leq \hat{\Delta}'$ and $\Delta_{k+1} = \Delta_k/\gamma$. Thus,
\begin{equation*}
\log_\gamma\left(\frac{\hat{\Delta}'}{\Delta_k}\right)\geq 0 \quad\text{and}\quad  \log_\gamma\left(\frac{\hat{\Delta}'}{\Delta_{k+1}}\right) =  \log_\gamma\left(\frac{\hat{\Delta}'}{\Delta_k}\right)+1.
\end{equation*}
Therefore, $ \tilde{\phi}_k = \log_\gamma (\hat{\Delta}'/\Delta_k) $, $ \tilde{\phi}_{k+1} = \log_\gamma (\hat{\Delta}'/\Delta_{k+1}) $, and $\tilde{\phi}_{k+1}-\tilde{\phi}_k = 1$.

\vskip4pt
\noindent$\bullet$ \textbf{Case C.} When $\Lambda_k=1$, we have $\min\{\Delta_{k+1},\Delta_k\} \geq \hat{\Delta}'$, thus $\tilde{\phi}_{k+1}=\tilde{\phi}_k = 0$.

\vskip4pt
Combining the above cases, we sum  $\tilde{\phi}_{k+1}-\tilde{\phi}_k $ over $k=0,\cdots,T-1$ and get
\begin{equation*}
\tilde{\phi}_{T}-\tilde{\phi}_0  = \sum_{k=0}^{T-1}\tilde{\phi}_{k+1}-\tilde{\phi}_k \leq \sum_{k=0}^{T-1}(1-\Lambda_k)(1-\Theta_k) -\sum_{k=0}^{T-1}(1-\Lambda_k)\Theta_k.
\end{equation*}
Noticing that  $\tilde{\phi}_0 = \max \{\log_\gamma (\hat{\Delta}'/\Delta_0 ),0\}=0$ (since $\Delta_0=\Delta_{\max}\geq \hat{\Delta}'$) and $\tilde{\phi}_{T}\geq 0$, we rearrange the terms and complete the proof.

\vskip4pt
\noindent (d) We note that 
\begin{align*}
\sum_{k=0}^{T}(1-\Lambda_k)I_k  & = \sum_{k=0}^{T}(1-\Lambda_k)I_k \Theta_k  \leq \sum_{k=0}^{T}(1-\Lambda_k)\Theta_k 
\leq \sum_{k=0}^{T}(1-\Lambda_k)(1-\Theta_k) \\
& =  \sum_{k=0}^{T}(1-\Lambda_k)(1-\Theta_k)(1- I_k) \leq  \sum_{k=0}^{T}(1-\Lambda_k)(1- I_k).
\end{align*}
The first equality follows from Lemmas \ref{lemma:guarantee_succ_step_KKT} and \ref{lemma:guarantee_succ_step_eigen} that small and accurate iterations~must be sufficient, the second inequality is by Lemma  \ref{lemma: count_iter} (c), and the last equality is by the converse of Lemmas \ref{lemma:guarantee_succ_step_KKT} and \ref{lemma:guarantee_succ_step_eigen} that small and insufficient iterations must be inaccurate.
\end{proof}

\section{Proofs of Section \ref{sec:4_heavytail} }\label{append:B}

\subsection{Proof of Theorem \ref{thm: 1st}}\label{append:B1}

\begin{proof}
Since $P\{T_\epsilon\leq T-1\} =1- P\{T_{\epsilon}>T-1\}$, we will bound $P\{T_{\epsilon}>T-1\}$ in the proof. By the law of total probability, we have
\begin{multline}\label{eq:PAPB}
P\{T_{\epsilon}>T-1\} = P\underbrace{\left\{T_{\epsilon}>T-1, \sum_{k=0}^{T-1}(e_k+e_{s_k})> (2\epsilon_f+ 2s)T\right\}}_{A} \\
+ P\underbrace{\left\{T_{\epsilon}>T-1, \sum_{k=0}^{T-1}( e_k+e_{s_k} ) \leq (2\epsilon_f + 2s)T\right\}}_{B}.
\end{multline}
In what follows, we present lemmas that bound $P(A)$ and $P(B)$, separately. The proofs of these lemmas are presented in Appendices \ref{append:B2}, \ref{append:B3} and \ref{append:B4}.

\begin{lemma}[Burkholder-type inequality]\label{append:lemma1}
Let $X_k, k=0,1,\cdots,T-1$, be random variables with $\mE[X_k\mid X_{0:k-1}]=0$ and $\mE[|X_k|^{p}]<\infty$ for any $k\geq 0$ and some $p \in (1,2]$, then \vspace{-0.2cm}
\begin{equation}\label{append:eq2}
\mE\left[ \left| \sum_{k=0}^{T-1} X_k \right| ^{p} \right] \leq 2^{2-p} \sum_{k=0}^{T-1} \mE[|X_k|^{p}].
\end{equation}
\end{lemma}	

\begin{lemma}\label{append:lemma2} 
Under conditions of the probabilistic heavy-tailed zeroth-order oracle, for any $T\geq 1$, $s\geq 0$, and $\delta$ specified in \eqref{eq3}, we have
\begin{equation*}
P(A) \leq 2 \exp \left( - \frac{2(\epsilon_f-\tilde{\epsilon}_f+s)^2}{(3+\delta)^2 e^{1+\delta} \Upsilon_f^{\frac{2}{1+\delta}}} \cdot T \right)
 + \frac{3^{ 2+\delta } \Upsilon_f  }{(\epsilon_f -\tilde{\epsilon}_f + s)^{1+\delta}} \cdot T^{-\delta} .
\end{equation*}
\end{lemma}

\begin{lemma}\label{append:lemma3}
Under Assumptions \ref{assump:4-1}, \ref{assump:epsilon}, when $\hat{p}$ satisfies \eqref{thm:p} and $T$ satisfies \eqref{thm:T}, we have\vspace{-0.2cm}
\begin{equation}\label{eq:PB}
P(B) \leq \exp\left\{ -\frac{(p-\hat{p})^2}{2p^2}T\right\}.
\end{equation}
\end{lemma}    

Combining \eqref{eq:PAPB} with the conclusions of Lemmas \ref{append:lemma2} and \ref{append:lemma3}, we complete the proof.	
\end{proof}

\vspace{-0.6cm}
\subsection{Proof of Lemma \ref{append:lemma1}}\label{append:B2}

\begin{proof}

To prove the result, we fist show the following inequality: \begin{equation}\label{append:eq4}
|a + b|^p \leq |a|^p + p \cdot \text{sgn}(a)\cdot |a|^{p-1} b + 2^{2-p} |b|^p.
\end{equation}
When $a=0$, the inequality holds trivially as $2^{2-p}\geq 1$ for $p\in (1,2]$. Moreover, when $p=2$, the inequality also holds trivially by the observation that $\text{sgn}(a)|a|=a$. In what follows, we consider $a\neq 0$ and $p\in (1,2)$. We divide $|a|^p$ on both sides and get
\begin{equation}\label{append:eq5}
\left|1+ \frac{b}{a}\right|^p \leq 1 + p \frac{b}{a} + 2^{2-p}\left|\frac{b}{a}\right|^p,
\end{equation}
where we use the observation that $\text{sgn}(a)/|a| = 1/a$. Denoting $x=b/a$, it suffices to show that $f(x)= 1 + p x + 2^{2-p}\left|x\right|^p-\left|1+ x \right|^p \geq 0$. To this end, we consider three cases. \textbf{Case A:} $b/a>0$, \textbf{Case B:} $-1<b/a\leq 0$, and \textbf{Case C:} $b/a \leq -1$.

\vskip4pt
\noindent$\bullet$ \textbf{Case A:} $b/a>0$: We have $f(x) = 1+px+2^{2-p}x^p-(1+x)^p$ with $f(0) = 0, f'(0)=0$~and \begin{equation*}
f''(x) = p(p-1) \left[ \left( \frac{2}{x} \right)^{2-p} - \left( \frac{1}{1+x} \right)^{2-p} \right].
\end{equation*}
Since $1 + x > x/2 $ for all $x > 0$, when $p < 2$, we have \begin{equation*}
\left( \frac{2}{x} \right)^{2-p} > \left( \frac{1}{1+x} \right)^{2-p} \implies f''(x) > 0.
\end{equation*}
Thus, for all $x > 0$, we have $f'(x) > 0$, which then implies $f(x) >  0$, and thus \eqref{append:eq4} holds.

\vskip4pt
\noindent $\bullet$ \textbf{Case B:} $-1<b/a\leq 0$: We substitute $y=1+x$ so that $ 0 \leq y \leq 1 $. Then,
\begin{equation*}
f(x)=f(y-1)= 2^{2-p} (1 - y)^p + p (y - 1) + 1 - y^p \geq (1 - y)^p + p (y - 1) + 1 - y^p \coloneqq g(y).
\end{equation*}  
Then, we have $g(1)=0$, $g'(0)=g'(1)=0$, and 
\begin{equation*}
g''(y)= p (p-1) \left[ \left( \frac{1}{1 - y} \right)^{2-p} - \left( \frac{1}{y} \right)^{2-p} \right].
\end{equation*}
Note that $g''(y)$ has exactly one root at $y = 1/2$. Moreover, $ g''(y) < 0 $ for $y < 1/2 $ and $ g''(y) > 0 $ for $y >1/2$, so $ g'(y) $ has a minimum at $ y = \frac{1}{2} $. Since $ g'(0) = g'(1) = 0 $, we conclude that $ g'(y) < 0 $ for all $ 0 < y < 1 $. As $g(y)$ is monotonically decreasing on $(0,1)$ and $g(1)=0$, we have $ g(y) \geq 0 $ for all $0 \leq y \leq 1 $. Therefore, $f(x)\geq 0$.

\vskip4pt
\noindent$\bullet$ \textbf{Case C:} $b/a \leq -1$: We substitute $z = -1 - x $ so that $z\geq 0$. Then $f(x)=f(-z-1)\eqqcolon h(z)$ and \vspace{-0.2cm}
\begin{equation*}
h(z) = 2^{2-p} (1 + z)^p - p (1 + z) + 1 - z^p.
\end{equation*}
Then, we have $h(0) \geq 0, h'(1) = 0$, and  \vspace{-0.2cm}
\begin{equation*}
h'(z) = p (p-1) \left[ \left( \frac{2}{1 + z} \right)^{2-p} - \left( \frac{1}{z} \right)^{2-p} \right].
\end{equation*}
Note that $ h''(z) = 0 $ only at $ z = 1$. Moreover, $ h''(z) < 0 $ for $ z < 1 $ and $ h''(z) > 0 $ for $ z > 1 $, so $ h'(z) $ has a minimum at $z = 1 $. Since $ h'(1) = 0 $, we know $ h'(z) \geq 0 $ for all $z \geq 0 $, and hence $ h(z) \geq 0 $ for all $z \geq 0$. Combining the above three cases, we prove \eqref{append:eq4}.

Next, we prove \eqref{append:eq2} by induction. The inequality holds trivially for $T=1$, as $2^{2-p} \geq 1$ for $p \in (1,2] $. Suppose the inequality holds for $T=n$, we now consider $T=n+1$.~Applying \eqref{append:eq4} with $a =  \sum_{k=0}^{n-1} X_k $ and $b = X_{n}$, we have
\begin{equation*}
\left| \sum_{k=0}^{n} X_k \right| ^{p} \leq  \left| \sum_{k=0}^{n-1} X_k \right| ^{p} + p\cdot \text{sgn}(\sum_{k=0}^{n-1} X_k) \cdot \left| \sum_{k=0}^{n-1} X_k \right|^{p-1} \cdot X_{n} + 2^{2-p}  |X_{n}|^p .
\end{equation*}
Taking expectation on both sides, we have
\begin{equation*}
\mE\left[ \left| \sum_{k=0}^{n} X_k \right| ^{p} \right] \leq  \mE\left[  \left| \sum_{k=0}^{n-1} X_k \right| ^{p} \right] + p \mE\left[  \text{sgn}(\sum_{k=0}^{n-1} X_k) \left| \sum_{k=0}^{n-1} X_k \right|^{p-1} \cdot X_{n}  \right]  + 2^{2-p}  \mE\left[  |X_{n}|^p \right]  .
\end{equation*}
Using $\mE[X_{n}\mid X_{0:n-1}]=0$, we get
\begin{equation*}
\mE\left[ \left| \sum_{k=0}^{n} X_k \right| ^{p} \right] \leq  \mE\left[  \left| \sum_{k=0}^{n-1} X_k \right| ^{p} \right] + 2^{2-p}  \mE\left[  |X_{n}|^p \right] .
\end{equation*}
We complete the proof by induction.
\end{proof}

\subsection{Proof of Lemma \ref{append:lemma2}}\label{append:B3}

\begin{proof}
We observe that
\begin{align}\label{append:eq3}
P(A) & \leq P\left\{ \sum_{k=0}^{T-1} (e_k + e_{s_k}) > (2\epsilon_f + 2s)T\right\} \notag \\
& = P\left\{ \sum_{k=0}^{T-1}( e_k+e_{s_k}-\mE[e_k \mid\F_{k-1} ]-\mE[e_{s_k} \mid\F_{k-1/2}])\right. \notag  \\
&\hspace{3cm} \left. > (2\epsilon_f + 2s) T-\sum_{k=0}^{T-1} (\mE[e_k \mid\F_{k-1}]+\mE[e_{s_k} \mid\F_{k-1/2}])\right\} \notag \\
& \leq P\left\{ \sum_{k=0}^{T-1}( e_k -\mE[e_k \mid\F_{k-1} ]) > (\epsilon_f + s) T-\sum_{k=0}^{T-1}\mE[e_k \mid\F_{k-1}]\right\} \notag \\
& \quad + P\left\{ \sum_{k=0}^{T-1}( e_{s_k} -\mE[e_{s_k} \mid\F_{k-1/2} ]) > (\epsilon_f + s) T-\sum_{k=0}^{T-1}\mE[e_{s_k} \mid\F_{k-1/2}]\right\} \quad(\text{by union bound}) \notag \\
& \stackrel{\mathclap{\eqref{eq1}}}{\leq} P\left\{ \sum_{k=0}^{T-1}( e_k -\mE[e_k \mid\F_{k-1} ]) > (\epsilon_f -\tilde{\epsilon}_f + s) T\right\} \notag \\
& \quad + P\left\{ \sum_{k=0}^{T-1}( e_{s_k} -\mE[e_{s_k} \mid\F_{k-1/2} ]) > (\epsilon_f -\tilde{\epsilon}_f+ s) T\right\}.
\end{align}
Next we bound \eqref{append:eq3} by discussing two cases. \textbf{Case A:} $\delta \in (0,1]$ and \textbf{Case B:} $\delta \in (1,\infty)$.

\noindent$\bullet$ \textbf{Case A:} When $\delta \in (0,1]$. Since $e_k -\mE[e_k \mid\F_{k-1} ]$ is a martingale difference that satisfies \eqref{eq3}, it follows from the {\myred Markov} inequality and the conclusion of Lemma \ref{append:lemma1} that
\begin{multline*}
P\left\{ \sum_{k=0}^{T-1}( e_k -\mE[e_k \mid\F_{k-1} ]) > (\epsilon_f -\tilde{\epsilon}_f + s) T\right\} \leq \frac{\mE\left[\left|\sum_{k=0}^{T-1}( e_k -\mE[e_k \mid\F_{k-1} ])\right|^{1+\delta} \right]}{(\epsilon_f -\tilde{\epsilon}_f + s)^{1+\delta} T^{1+\delta}} \\
 \leq \frac{2^{1-\delta}\sum_{k=0}^{T-1}\mE\left[\left| e_k -\mE[e_k \mid\F_{k-1}]\right|^{1+\delta} \right]}{(\epsilon_f -\tilde{\epsilon}_f + s)^{1+\delta} T^{1+\delta}} 
 \leq \frac{2^{1-\delta}T\Upsilon_f}{(\epsilon_f -\tilde{\epsilon}_f + s)^{1+\delta} T^{1+\delta}} = \frac{2^{1-\delta}\Upsilon_f}{(\epsilon_f -\tilde{\epsilon}_f + s)^{1+\delta} } \cdot T^{-\delta}.    
\end{multline*}
We can similarly show 
\begin{align*}
P\left\{ \sum_{k=0}^{T-1}( e_{s_k} -\mE[e_{s_k} \mid\F_{k-1/2} ]) > (\epsilon_f -\tilde{\epsilon}_f + s) T\right\} \leq \frac{2^{1-\delta}\Upsilon_f}{(\epsilon_f -\tilde{\epsilon}_f + s)^{1+\delta} } \cdot T^{-\delta}.    
\end{align*}
Therefore, for $\delta \in (0,1]$ we have
\begin{equation*}
P(A) \leq \frac{2^{2-\delta}\Upsilon_f}{(\epsilon_f -\tilde{\epsilon}_f + s)^{1+\delta} } \cdot T^{-\delta}.
\end{equation*}
\noindent$\bullet$ \textbf{Case B:} When $\delta \in (1,\infty)$. We apply the martingale Fuk–Nagaev inequality \citep[Corollary 2.5]{Fan2017Deviation} {\myred (restated in Lemma \ref{Fuk–Nagaev})} and get
\begin{multline}\label{append:eq6}
P\left\{ \sum_{k=0}^{T-1}( e_k -\mE[e_k \mid\F_{k-1} ]) > (\epsilon_f -\tilde{\epsilon}_f + s) T\right\} \\
\leq \exp\left( -\frac{(\epsilon_f -\tilde{\epsilon}_f + s)^2 T^2}{2V}\right) + \frac{C}{(\epsilon_f -\tilde{\epsilon}_f + s)^{1+\delta} T^{1+\delta}},
\end{multline}
where
\begin{align*}
V & = \frac{1}{4}(3+\delta)^2e^{1+\delta} \left| \sum_{k=0}^{T-1} \mE \left[ \left(e_k-\mE\left[e_k\mid \F_{k-1}\right] \right)^{2}\mid \F_{k-1}\right] \right|, \\
C & = \left( 1+ \frac{2}{1 + \delta}\right)^{1+\delta} \left| \sum_{k=0}^{T-1} \mE \left[ \left| e_k-\mE\left[e_k\mid \F_{k-1}\right] \right|^{ 1+\delta }\mid \F_{k-1}\right] \right|. 
\end{align*}
By the Hölder's inequality and \eqref{eq3}, we have
$\mE \left[ \left(e_k-\mE\left[e_k\mid \F_{k-1}\right] \right)^{2}\mid \F_{k-1}\right] \leq \Upsilon_f^{\frac{2}{1+\delta}}$. Thus, 
\begin{equation}\label{append:eq7}
V \leq \frac{1}{4}(3+\delta)^2e^{1+\delta}\Upsilon_f^{\frac{2}{1+\delta}} T  \quad \text{and} \quad C \leq \left( 1+ \frac{2}{1 + \delta}\right)^{1+\delta}\Upsilon_f T.
\end{equation}
Combining \eqref{append:eq6} and \eqref{append:eq7}, we have
\begin{multline}\label{append:eq8}
P\left\{ \sum_{k=0}^{T-1}( e_k -\mE[e_k \mid\F_{k-1} ]) > (\epsilon_f -\tilde{\epsilon}_f + s) T\right\} \\
\leq \exp \left( - \frac{2(\epsilon_f-\tilde{\epsilon}_f+s)^2}{(3+\delta)^2 e^{1+\delta} \Upsilon_f^{\frac{2}{1+\delta}}} T \right)
+ \Upsilon_f \left( \frac{3+\delta}{(1+\delta)(\epsilon_f-\tilde{\epsilon}_f+s)}\right)^{ 1+\delta }\cdot T^{-\delta}.
\end{multline}
By similar argument, we have
\begin{multline}\label{append:eq9}
P\left\{ \sum_{k=0}^{T-1}( e_{s_k} -\mE[e_{s_k} \mid\F_{k-1/2} ]) > (\epsilon_f -\tilde{\epsilon}_f + s) T\right\} \\
\leq \exp \left( - \frac{2(\epsilon_f-\tilde{\epsilon}_f+s)^2}{(3+\delta)^2 e^{1+\delta} \Upsilon_f^{\frac{2}{1+\delta}}} T \right) 
+ \Upsilon_f \left( \frac{3+\delta}{(1+\delta)(\epsilon_f-\tilde{\epsilon}_f+s)}\right)^{ 1+\delta }\cdot T^{-\delta}.
\end{multline}
Combining \eqref{append:eq8} and \eqref{append:eq9}, we have 
\begin{equation*}
P(A) \leq 2 \exp \left( - \frac{2(\epsilon_f-\tilde{\epsilon}_f+s)^2}{(3+\delta)^2 e^{1+\delta} \Upsilon_f^{\frac{2}{1+\delta}}} T \right)
+ 2 \Upsilon_f \left( \frac{3+\delta}{(1+\delta)(\epsilon_f-\tilde{\epsilon}_f+s)}\right)^{ 1+\delta }\cdot T^{-\delta}.
\end{equation*}
Combining the conclusions in \textbf{Case A.} and \textbf{Case B.}, we have
\begin{align*}
P(A) & \leq 2 \exp \left( - \frac{2(\epsilon_f-\tilde{\epsilon}_f+s)^2}{(3+\delta)^2 e^{1+\delta} \Upsilon_f^{\frac{2}{1+\delta}}} T \right)
+ \frac{2 \Upsilon_f \cdot \max\left\{\left( \frac{3+\delta}{1+\delta}\right)^{ 1+\delta },2^{1-\delta}\right\}}{(\epsilon_f -\tilde{\epsilon}_f + s)^{1+\delta}} \cdot T^{-\delta} \\
& \leq 2 \exp \left( - \frac{2(\epsilon_f-\tilde{\epsilon}_f+s)^2}{(3+\delta)^2 e^{1+\delta} \Upsilon_f^{\frac{2}{1+\delta}}} T \right)
+ \frac{3^{2+\delta} \Upsilon_f}{(\epsilon_f -\tilde{\epsilon}_f + s)^{1+\delta}} \cdot T^{-\delta}, 
\end{align*}
as $(3+\delta)/(1+\delta)<3$ for $\delta>0$. 
We thus complete the proof.
\end{proof}

\subsection{Proof of Lemma \ref{append:lemma3}}\label{append:B4}

\begin{proof}
Denoting $d= \log_\gamma(\Delta_0/\hat{\Delta}) + 2$, by the law of total probability, we have
\begin{align*}
P(B) & = P\underbrace{\left\{T_{\epsilon}>T-1, \sum_{k=0}^{T-1}(\vartheta_\alpha +e_k+e_{s_k}) \leq (2\epsilon_f + \vartheta_\alpha + 2s) T, \sum_{k=0}^{T-1} \Lambda_kI_k\Theta_k \geq  \left(\hat{p}-\frac{1}{2}\right)T-\frac{d}{2} \right\}}_{B_1}\\
& + P\underbrace{\left\{T_{\epsilon}>T-1, \sum_{k=0}^{T-1}(\vartheta_\alpha + e_k+e_{s_k}) \leq (2\epsilon_f + \vartheta_\alpha +2s) T, \sum_{k=0}^{T-1} \Lambda_kI_k\Theta_k < \left(\hat{p}-\frac{1}{2}\right)T-\frac{d}{2}\right\} }_{B_2}.
\end{align*}
To prove $P(B_1)=0$, we will show that $T_{\epsilon}>T-1, \sum_{k=0}^{T-1}(\vartheta_\alpha +e_k+e_{s_k}) \leq (2\epsilon_f +\vartheta_\alpha + 2s) T$ together imply $\sum_{k=0}^{T-1} \Lambda_kI_k\Theta_k < (\hat{p}-1/2)T-d/2$. Suppose, on the contrary, $\sum_{k=0}^{T-1} \Lambda_kI_k\Theta_k \geq  (\hat{p}-1/2)T-d/2$ holds. Then unifying $\alpha=0$ and $\alpha=1$, we discuss~the following two cases. \textbf{Case A:} $\Lambda_kI_k\Theta_k=1$ and \textbf{Case B:} $\Lambda_kI_k\Theta_k=0$.

\vskip4pt
\noindent$\bullet$ \textbf{Case A:}  $\Lambda_kI_k\Theta_k=1$. We have
\begin{align*}
\L_{\barmu_k}^{k+1}-\L_{\barmu_k}^k & = \L_{\barmu_k}^{k+1}-\bar{\L}_{\barmu_k}^{k+1}+\bar{\L}_{\barmu_k}^{k+1}-\bar{\L}_{\barmu_k}^k+\bar{\L}_{\barmu_k}^k-\L_{\barmu_k}^k \\
& \leq | f_{ k+1 } - \barf_{k+1} | + | \barf_k - f_k | + \bar{\L}_{\barmu_k}^{k+1}-\bar{\L}_{\barmu_k}^k \\
& \leq -h_{\alpha}(\Delta_k)+ e_k + e_{s_k} +\vartheta_\alpha && (\text{by Lemma }\ref{lemma:suff_iter})\\
& \leq - h_{\alpha}(\hat{\Delta}') + e_k + e_{s_k} + \vartheta_\alpha . && (\text{since }\Delta_k \geq \hat{\Delta}'\text{ when }\Lambda_k=1)
\end{align*}
Dividing $\barmu_k$ on both sides and noticing that $\barmu_0\leq \barmu_k\leq \barmu_{T-1}$ and $\gamma^{-2}\hat{\Delta} < \hat{\Delta}'$, we have 
\begin{align*}
\frac{1}{\barmu_k}(\L_{\barmu_k}^{k+1}-f_{\inf})-\frac{1}{\barmu_k}(\L_{\barmu_k}^k-f_{\inf}) & \leq - \frac{1}{\barmu_k } h_{\alpha}(\hat{\Delta}')+\frac{1}{\barmu_k} (\vartheta_\alpha +e_k + e_{s_k} )\\
& < - \frac{1}{\barmu_{T-1} }  h_{\alpha}(\gamma^{-2}\hat{\Delta}) +\frac{1}{\barmu_0} (\vartheta_\alpha +e_k + e_{s_k}).
\end{align*}
\noindent$\bullet$ \textbf{Case B:} $\Lambda_kI_k\Theta_k=0$. We consider two subcases. \textbf{Case B.1:} (\ref{nequ:3}a) holds in the $k$-th iteration; and \textbf{Case B.2:} (\ref{nequ:3}a) does not hold the $k$-th iteration.

\vskip4pt
\noindent$\bullet$ \textbf{Case B.1:} In this case, we have
\begin{align}\label{nequ:30}
\L_{\barmu_k}^{k+1}-\L_{\barmu_k}^k & = \L_{\barmu_k}^{k+1}-\bar{\L}_{\barmu_k}^{k+1}+\bar{\L}_{\barmu_k}^{k+1}-\bar{\L}_{\barmu_k}^k+\bar{\L}_{\barmu_k}^k-\L_{\barmu_k}^k \notag \\
& \leq | f_{k+1 } - \barf_{k+1} | + |  \barf_k - f_k | + \eta \text{Pred}_k +\vartheta_\alpha \qquad (\text{by} (\ref{nequ:3}\text{a}))\notag \\
& \leq e_k + e_{s_k} +\vartheta_\alpha.
\end{align}
The last inequality is because $\text{Pred}_k< 0$ but we can no longer guarantee (\ref{nequ:3}b).

\vskip4pt
\noindent$\bullet$ \textbf{Case B.2:} In this case, since $\bx_{k+1}=\bx_k$, we have $\L_{\barmu_k}^{k+1}-\L_{\barmu_k}^k=0$. 

We combine \textbf{Case B.1} and \textbf{Case B.2} and divide $\barmu_k$ on both sides. Since $\barmu_0\leq \barmu_k$, we have
\begin{equation}\label{nequ:31}
\frac{1}{\barmu_k}(\L_{\barmu_k}^{k+1}-f_{\inf})-\frac{1}{\barmu_k}(\L_{\barmu_k}^k-f_{\inf})  \leq \frac{1}{\barmu_0} (\vartheta_\alpha + e_k + e_{s_k}).
\end{equation}
Using the results in \eqref{nequ:31} , we sum over $k=0,1,\cdots,T-1$. Since $\barmu_{k}\leq \barmu_{k+1}$, we have
\begin{equation*}
\frac{1}{\barmu_{T-1}}(\L_{\barmu_{T-1}}^{T}-f_{\inf})-\frac{1}{\barmu_0}(\L_{\barmu_0}^0 - f_{\inf})  \leq - \frac{h_\alpha(\gamma^{-2}\hat{\Delta})}{\barmu_{T-1} } \sum_{k=0}^{T-1}\Lambda_kI_k\Theta_k+ \sum_{k=0}^{T-1}\frac{ \vartheta_\alpha + e_k + e_{s_k} }{\barmu_0}.
\end{equation*}
Rearranging the terms, using the relations $\sum_{k=0}^{T-1}(\vartheta_\alpha +e_k + e_{s_k}) \leq (2\epsilon_f +\vartheta_\alpha+ 2s) T$ and $\sum_{k=0}^{T-1} \Lambda_kI_k\Theta_k \geq  (\hat{p}-1/2)T-d/2$, we have
\begin{equation}\label{nequ:32}
\frac{1}{\barmu_{T-1}}(\L_{\barmu_{T-1}}^{T}-f_{\inf}) \\
\leq \frac{1}{\barmu_0}(\L_{\barmu_0}^0 - f_{\inf})   -\frac{h_\alpha(\gamma^{-2}\hat{\Delta})}{\barmu_{T-1}}   \left(\hat{p}-\frac{1}{2}\right)T +  \frac{2\epsilon_f + \vartheta_\alpha+ 2s}{\barmu_0}T + \frac{h_\alpha(\gamma^{-2}\hat{\Delta}) d}{ 2\barmu_{T-1}}.
\end{equation}
However, the conditions of $\hat{p}$ and $T$ imply that the right-hand side of \eqref{nequ:32} is negative. This yields a contradiction, since the left-hand side is non-negative. We thus proved $P(B_1)=0$. Next, we bound $P(B_2)$. We observe that \begin{align*}
P(B_2) & \leq P\left\{T_{\epsilon}>T-1,  \sum_{k=0}^{T-1} \Lambda_kI_k\Theta_k < \left(\hat{p}-\frac{1}{2}\right)T-\frac{d}{2}\right\} \\
& = P\underbrace{\left\{T_{\epsilon}>T-1,  \sum_{k=0}^{T-1}I_k<\hat{p}T, \sum_{k=0}^{T-1} \Lambda_kI_k\Theta_k < \left(\hat{p}-\frac{1}{2}\right)T-\frac{d}{2} \right\}}_{B_{21}}\\
&\quad +  P\underbrace{\left\{T_{\epsilon}>T-1,  \sum_{k=0}^{T-1}I_k \geq \hat{p}T, \sum_{k=0}^{T-1} \Lambda_kI_k\Theta_k < \left(\hat{p}-\frac{1}{2}\right)T-\frac{d}{2} \right\}}_{B_{22}}.
\end{align*}
For $P(B_{21})$, we have
\begin{equation}
P(B_{21}) \leq P\left\{\sum_{k=0}^{T-1}I_k<\hat{p}T \right\}  \stackrel{\text{Lemma } \ref{lemma:I_k}}{\leq} \exp\left\{-\frac{(p-\hat{p})^2}{2p^2}T\right\}.
\end{equation}
For $P(B_{22})$, if follows from Lemma \ref{lemma: prob=0} that $P(B_{22})=0$. Combining the above results, we have 
\begin{equation*}
P(B) \leq \exp\left\{-\frac{(p-\hat{p})^2}{2p^2}T\right\}.
\end{equation*}
We thus complete the proof.
\end{proof}

\section{Proofs of Section \ref{sec:4_subexp}}\label{append:C}

\subsection{Proof of Theorem \ref{thm: 1st_subexp}}\label{thm:Append:C.2}

\begin{proof}
Same as the proof of Theorem \ref{thm: 1st} in Appendix \ref{append:B}, we also bound $P\{T_{\epsilon}>T-1\}$ via \eqref{eq:PAPB} and analyze $P(A)$ and $P(B)$ separately. To bound $P(A)$, we observe that
\begin{align*}
P(A) & \leq P\left\{ \sum_{k=0}^{T-1}( e_k+e_{s_k}) \geq (2\epsilon_f + 2s)T\right\} \\
& = P\left\{ \sum_{k=0}^{T-1}( e_k+e_{s_k}-\mE[e_k \mid\F_{k-1} ]-\mE[e_{s_k} \mid\F_{k-1/2}])\right. \\
&\hspace{3cm} \left.\geq ( 2\epsilon_f +2 s) T-\sum_{k=0}^{T-1} (\mE[e_k \mid\F_{k-1}]+\mE[e_{s_k} \mid\F_{k-1/2}])\right\} \\
& = P\left\{ \sum_{k=0}^{T-1} \lambda ( e_k+e_{s_k}-\mE[e_k \mid\F_{k-1}]-\mE[e_{s_k} \mid\F_{k-1/2}])\right. \\
&\hspace{1cm} \left. \geq \lambda( 2\epsilon_f + 2s) T-\sum_{k=0}^{T-1} \lambda(\mE[e_k \mid\F_{k-1}]+\mE[e_{s_k} \mid\F_{k-1/2}])\right\} \;\text{for any } \lambda\in\left[0,\frac{1}{b}\right]\\
& \stackrel{\mathclap{\eqref{eq1}}}{\leq} P\left\{ \sum_{k=0}^{T-1} \lambda ( e_k+e_{s_k}-\mE[e_k\mid\F_{k-1}]-\mE[e_{s_k}\mid\F_{k-1/2}]) \geq \lambda(2\epsilon_f - 2\tilde{\epsilon}_f+ 2s)T\right\}\\
& \leq \mE\left[e^{\sum_{k=0}^{T-1} \lambda ( e_k+e_{s_k}-\mE[e_k\mid\F_{k-1}]-\mE[e_{s_k}\mid\F_{k-1/2}])}\right]\cdot e^{ - \lambda(2\epsilon_f -2\tilde{\epsilon}_f + 2s) T} \quad\text{(by Markov inequality)}.
\end{align*}
Next, we prove
\begin{equation*}
\mE\left[e^{\sum_{k=0}^{T-1} \lambda ( e_k+e_{s_k}-\mE[e_k\mid\F_{k-1}]-\mE[e_{s_k}\mid\F_{k-1/2}])}\right] \leq  e^{\frac{(2v)^2\lambda^2}{2}T}.
\end{equation*}
To show this, we first note that
\begin{align}\label{nequ:19}
& \mE[e^{ \lambda (e_k + e_{s_k} -\mE[e_k\mid \F_{k-1}]- \mE[e_{s_k}\mid \F_{k-1/2}]) } \mid \F_{k-1}] \notag \\
& \leq \left( \mE[e^{ 2\lambda (e_k -\mE[e_k\mid \F_{k-1}])}\mid \F_{k-1}] \cdot \mE[e^{2\lambda (e_{s_k} - \mE[e_{s_k}\mid \F_{k-1/2}]) } \mid \F_{k-1}]\right)^{1/2} \;(\text{by Hölder's inequality}) \notag\\
& \leq \left( \mE[e^{ 2\lambda (e_k -\mE[e_k\mid \F_{k-1}])}\mid \F_{k-1}] \cdot \mE[\mE[e^{2\lambda (e_{s_k} - \mE[e_{s_k}\mid \F_{k-1/2}]) } \mid \F_{k-1/2}] \mid \F_{k-1}]\right)^{1/2} \;(\text{{\myred by tower property}}) \notag\\
& \stackrel{\mathclap{\eqref{eq4}}}{\leq} \left( e^{2\lambda^2v^2}\cdot e^{2\lambda^2v^2}  \right)^{1/2} = e^{\frac{\lambda^2(2v)^2}{2}}\quad \text{for all }\lambda\in\left[0,\frac{1}{2b}\right].
\end{align} 
Thus, given $\F_{k-1}$, $e_k + e_{s_k}$ follows a sub-exponential distribution with parameters $(2v,2b)$. Since
\begin{align*}
& \mE\left[e^{\sum_{k=0}^{T-1} \lambda ( e_k+e_{s_k}-\mE[e_k\mid\F_{k-1}]-\mE[e_{s_k}\mid\F_{k-1/2}])}\right] \\
& = \mE\left[\mE\left[e^{\sum_{k=0}^{T-1} \lambda ( e_k+e_{s_k}-\mE[e_k\mid\F_{k-1}]-\mE[e_{s_k}\mid\F_{k-1/2}])}\mid \F_{T-2}\right]\right] \notag \\
& = \mE\left[e^{\sum_{k=0}^{T-2} \lambda ( e_k+e_{s_k}-\mE[e_k\mid\F_{k-1}] -\mE[e_{s_k}\mid\F_{k-1/2}])}\cdot\mE\left[e^{\lambda ( e_{T-1}+e_{s_{T-1}}-\mE[e_{T-1}\mid\F_{T-2}]-\mE[e_{s_{T-1}}\mid\F_{T-3/2}])}\mid \F_{T-2}\right]\right] \notag\\
& \stackrel{\mathclap{\eqref{nequ:19}}}{\leq} \mE\left[e^{\sum_{k=0}^{T-2} \lambda ( e_k+e_{s_k}-\mE[e_k\mid\F_{k-1}]-\mE[e_{s_k}\mid\F_{k-1/2}])}\right]\cdot e^{\frac{\lambda^2(2v)^2}{2}},
\end{align*}
we can inductively show
\begin{equation*}
\mE\left[e^{\sum_{k=0}^{T-1} \lambda ( e_k+e_{s_k}-\mE[e_k\mid\F_{k-1}]-\mE[e_{s_k}\mid\F_{k-1/2}])}\right] \leq  e^{\frac{(2v)^2\lambda^2}{2}T} \quad \text{for any }\lambda\in \left[0,\frac{1}{2b}\right].
\end{equation*}
Therefore, we have $P(A) \leq  \exp\left\{(2v^2\lambda^2)T - 2 \lambda(\epsilon_f -\tilde{\epsilon}_f + s) T\right\}$. To find the optimal $\lambda$ to minimize the right-hand side, we define $l(\lambda) = (2v^2\lambda^2)T - 2 \lambda(\epsilon_f -\tilde{\epsilon}_f + s) T$. Since $e^{l(\lambda)}$ is monotonically increasing with respect to $l(\lambda)$, we only need to minimize $l(\lambda)$, which is achieved by $\lambda^* = (\epsilon_f-\tilde{\epsilon}_f + s)/(2v^2)$. Recalling that $\lambda\in [0,1/(2b)]$, when  $(\epsilon_f-\tilde{\epsilon}_f + s)/(2v^2) \leq 1/(2b)$, we have
\begin{equation*}
\min_{\lambda} e^{l(\lambda)} = e^{l(\lambda^*)} =e^{-\frac{(\epsilon_f-\tilde{\epsilon}_f+s)^2}{2v^2}T},
\end{equation*}
otherwise
\begin{equation*}
\min_{\lambda} e^{l(\lambda)} =e^{l( 1/(2b))} \leq e^{-\frac{\epsilon_f-\tilde{\epsilon}_f+s}{2b}T}.
\end{equation*}
Combining the above two displays, we have
\begin{align}\label{nequ:27}
P(A) \leq \exp\left\{-\frac{1}{2}\min\left(\frac{(\epsilon_f -\tilde{\epsilon}_f + s)^2}{v^2},\frac{\epsilon_f -\tilde{\epsilon}_f + s}{b}\right)T\right\}.
\end{align}    
Note that the proof of Lemma \ref{append:lemma3} is independent {\myred of} the probabilistic heavy-tailed zeroth-order oracle design, thus we can bound $P(B)$ in the same approach as in Lemma \ref{append:lemma3} and obtain \eqref{eq:PB}. Combining \eqref{eq:PB} and \eqref{nequ:27}, we complete the proof.
\end{proof}

\subsection{Proof of Lemma \ref{lemma:p_f}}\label{lemma:Append:C.1}

\begin{proof}
It follows from the definition of $\C_k$ and the union bound that
\begin{equation*}
P(\C_k \mid \F_{k-1/2}) \geq 1- P(e_k \geq \epsilon_f +\kappa_f\Delta_k^{\alpha+2} \mid  \F_{k-1}) - P(e_{s_k} \geq \epsilon_f +\kappa_f\Delta_k^{\alpha+2} \mid  \F_{k-1/2}),
\end{equation*}
where we also use the fact that $P(e_k \geq \epsilon_f +\kappa_f\Delta_k^{\alpha+2} \mid  \F_{k-1}) =P(e_k \geq \epsilon_f +\kappa_f\Delta_k^{\alpha+2} \mid  \F_{k-1/2}) $.
We consider the probability $P(e_k \geq \epsilon_f +\kappa_f\Delta_k^{\alpha+2} \mid  \F_{k-1})$ first. We have
\begin{align*}
P(e_k \geq & \epsilon_f +\kappa_f\Delta_k^{\alpha+2} \mid  \F_{k-1}) \notag\\
& = P(e_k -\mE[e_k\mid \F_{k-1}]\geq \epsilon_f +\kappa_f\Delta_k^{\alpha+2} -\mE[e_k\mid \F_{k-1}]\mid  \F_{k-1}) \notag\\
& = P(e^{ \lambda (e_k -\mE[e_k\mid \F_{k-1} ])} \geq e^{\lambda(\epsilon_f +\kappa_f\Delta_k^{\alpha+2} -\mE [e_k\mid \F_{k-1}])} \mid  \F_{k-1})\notag\\
& \stackrel{\mathclap{\eqref{eq1}}}{\leq} P(e^{ \lambda (e_k -\mE[e_k\mid \F_{k-1}]) } \geq e^{\lambda(\epsilon_f- \tilde{\epsilon}_f)} \mid  \F_{k-1})\notag\\
& \leq  \mE[e^{ \lambda (e_k -\mE[e_k\mid \F_{k-1}]) } \mid \F_{k-1}]\cdot  e^{-\lambda(\epsilon_f- \tilde{\epsilon}_f)}\quad\text{(by {\myred Markov} inequality)}\notag\\
& \stackrel{\mathclap{\eqref{eq4}}}{\leq} e^{\frac{\lambda^2v^2}{2}-\lambda(\epsilon_f- \tilde{\epsilon}_f)},	
\end{align*}
where the second equality is due to the monotonicity of $e^x$ and holds for any $\lambda\in[0,1/b]$. We apply the same approach as in Appendix \ref{thm:Append:C.2} to find the optimal $\lambda$ that minimizes the right-hand side and get
\begin{equation}\label{nequ:21}
P(e_k \geq \epsilon_f +\kappa_f\Delta_k^{\alpha+2} \mid  \F_{k-1}) \leq \exp\left\{-\frac{1}{2}\min\left(\frac{(\epsilon_f-\tilde{\epsilon}_f)^2}{v^2} ,\frac{\epsilon_f-\tilde{\epsilon}_f}{b} \right)\right\}.
\end{equation}
Applying the same analysis to $P(e_{s_k} \geq \epsilon_f +\kappa_f\Delta_k^{\alpha+2} \mid  \F_{k-1/2})$, we can still get \eqref{nequ:21}. Combining them together, we show \eqref{eq5} and complete the proof.	
\end{proof}

\section{Useful Inequalities}\label{append:F}

\begin{lemma}[Azuma-Hoeffding inequality, \cite{Hoeffding1963Probability, Azuma1967Weighted}]\label{Azuma-Hoeffding}
Suppose $X_k,k = 0,1, \cdots $ is a submartingale and $|X_k - X_{k-1}| < p$ for some constant $p >0$ almost surely. Then for all $T,c>0$, we have
\begin{equation*}
P(X_T-X_0 \leq - c) \leq \exp\left\{ -\frac{c^2}{2 T p^2}\right\}.
\end{equation*}
\end{lemma}

\begin{lemma}[Fuk–Nagaev inequality, {\cite[Corollary 2.5]{Fan2017Deviation}}]\label{Fuk–Nagaev}
Assume $X_k,k=0,1,\cdots$ is a sequence of random variables such that $\mE[X_k\mid \mathcal{F}_{k-1}]=0$. Denote by $S_T = \sum_{k=0}^{T-1}X_k$~and $\tilde{S}_T = \sum_{k=0}^{T-1}\mE[X_k^2\mid \mathcal{F}_{k-1}]$. Let $\delta\geq1$ and assume $\|\mE[|X_k|^{1+\delta} | \mathcal{F}_{k-1}]\|_{\infty} < \infty$ for all~$ k =0,1,\cdots, T-1$, where $\|\cdot\|_{\infty}$ denotes the essential supremum of a random variable. Then for all $x > 0$,
\begin{equation*}
P\left( S_T \ge x \right) \le \exp\left\{ -\frac{x^2}{2V} \right\} + \frac{C}{x^{1+\delta}}, 
\end{equation*}
where
\begin{equation*}
V = \frac{1}{4}(3+\delta)^2 e^{1+\delta} \, \|\tilde{S}_T \|_{\infty}
\quad \text{and} \quad
C = \left(1 + \frac{2}{1+\delta}\right)^{1+\delta} \left\| \sum_{i=0}^{T-1} \mE[|X_k|^{1+\delta} | \mathcal{F}_{k-1}] \right\|_{\infty}.
\end{equation*}
\end{lemma}

\end{document}